\documentclass[11pt,letterpaper]{article}

\usepackage{subcaption}
\usepackage{tikz}
\usetikzlibrary{arrows}
\usetikzlibrary{arrows.meta}

\usepackage{pgfplots}

\usepackage[margin=1in]{geometry}       
\usepackage{amsmath}
\usepackage{hyperref}
\hypersetup{
	colorlinks=true,
	linkcolor=blue,
	citecolor=blue,
	filecolor=magenta,      
	urlcolor=cyan,
	linktoc=page
}

\usepackage{amsthm}
\usepackage{amssymb}
\usepackage{dsfont}

\usepackage{enumitem}
\usepackage{placeins}
\usepackage{graphicx}
\usepackage{epstopdf}

\usepackage{xcolor}

\usepackage{stmaryrd}

\usepackage{mathrsfs}

\theoremstyle{plain}
\newtheorem{theorem}{Theorem}[section]
\newtheorem{corollary}[theorem]{Corollary}
 
\newtheorem{prob}[theorem]{Problem} 
\newtheorem{prop}[theorem]{Proposition}
\newtheorem{lemma}[theorem]{Lemma}

\theoremstyle{definition}
\newtheorem{definition}[theorem]{Definition}
\newtheorem{remark}[theorem]{Remark}

\newcommand{\Z}{\mathbb{Z}}
\newcommand{\Q}{\mathbb{Q}}
\newcommand{\N}{\mathbb{N}}
\newcommand{\R}{\mathbb{R}}
\newcommand{\p}{\mathbb{P}}
\newcommand{\E}{\mathbb{E}}

\newcommand{\sset}{\subset}
\newcommand{\lf}{\left}
\newcommand{\rg}{\right}

\newcommand{\ga}{\gamma}

\newcommand{\ep}{\epsilon}

\newcommand{\sig}{\sigma}

\newcommand{\eps}{\epsilon}

\newcommand{\Rd}{\mathbb{R}^4_\uparrow}

\newcommand{\Geo}{\operatorname{Geo}}
\newcommand{\Gai}{\operatorname{Gamma}^{-1}}

\newcommand{\II}[1]{\llbracket#1 \rrbracket}

\newcommand{\cB}{\mathcal{B}}
\newcommand{\cC}{\mathcal{C}}
\newcommand{\cD}{\mathcal{D}}
\newcommand{\cE}{\mathcal{E}}
\newcommand{\cF}{\mathcal{F}}
\newcommand{\cG}{\mathcal{G}}

\newcommand{\cH}{\mathcal{H}}

\newcommand{\cK}{\mathcal{K}}
\newcommand{\cL}{\mathcal{L}}
\newcommand{\cM}{\mathcal{M}}

\newcommand{\cP}{\mathcal{P}}
\newcommand{\cQ}{\mathcal{Q}}

\newcommand{\cX}{\mathcal{X}}

\newcommand{\DL}{\mathsf{DL}}
\newcommand{\LG}{\mathsf{LG}}
\newcommand{\EL}{\mathsf{EL}}
\newcommand{\GL}{\mathsf{GL}}
\newcommand{\KPZ}{\mathsf{KPZ}}

\newcommand{\supp}{\text{supp}}

\usepackage{MnSymbol}

\DeclareMathOperator*{\argmax}{arg\,max}

\newcommand{\eqd}{\stackrel{d}{=}}

\newcommand{\cvgd}{\stackrel{d}{\to}}

\newcommand{\cvgp}{\stackrel{\mathbb P}{\to}}
\newcommand{\X}{\times}

\DeclareMathOperator{\SRW}{SRW}
\DeclareMathOperator{\NW}{NW}

\renewcommand{\P}{\mathbb{P}}
\newcommand{\fh}{\mathfrak{h}}
\newcommand{\Exp}{\operatorname{Exp}}
\newcommand{\UC}{\operatorname{UC}}

\newcommand{\apar}{\gamma}

\DeclareEmphSequence{\bfseries}

\usepackage[nottoc,numbib]{tocbibind}
\title{The directed landscape in half-space}

\author{
	Duncan Dauvergne
	\thanks{Department of Mathematics, University of Toronto, Canada. e-mail: duncan.dauvergne@utoronto.ca}
	\and
	Lingfu Zhang
	\thanks{The Division of Physics, Mathematics and Astronomy, Caltech, Pasadena, CA, USA. e-mail: lfzhang@caltech.edu}
}

\begin{document}

	\maketitle
	\begin{abstract}
		We prove that two half-space models in the KPZ universality class, exponential last passage percolation and a family of Poisson-avoiding metrics generalizing colored TASEP, converge to a common scaling limit. This scaling limit is the directed landscape in half-space, a random directed metric in the half-plane indexed by a parameter which determines the strength of the boundary interaction.

        As part of our analysis, we characterize the half-space directed landscape in terms of the half-space KPZ fixed point, and prove convergence of geodesics. We also give an explicit construction of joint stationary measures (or horizons) in half-space for the log-gamma polymer, the KPZ equation, exponential and geometric last passage percolation, and the directed landscape itself.
        \end{abstract}
	
	\setcounter{tocdepth}{1}
	\tableofcontents

	\section{Introduction}
	The KPZ (Kardar-Parisi-Zhang) universality class is a broad family of one-dimensional random growth models and two-dimensional random metric and polymer models which are expected to exhibit the same universal behavior under rescaling. Physical predictions for scaling exponents in the KPZ class go back to the 1980s, e.g., see \cite{huse1985huse, kardar1986dynamic}. 
	The past thirty years have seen a period of intense and fruitful research on this class, propelled by the discovery of a handful of exactly solvable models. The seminal work of 
	\cite{baik1999distribution} on the longest increasing subsequence in a random permutation was the first work to harness this exact solvability, confirming the characteristic KPZ scaling exponents and proving a one-point limit theorem for the model. Since then, similar limit theorems have been shown for certain exclusion processes, directed polymers, last passage percolation models, and the KPZ equation itself, e.g., see \cite{johansson2000shape, PSale, tracy2009asymptotics, amir2011probability} for a few highlights. The full scaling limit for random growth models in the KPZ class, the KPZ fixed point, was constructed by \cite{matetski2016kpz}, and the full scaling limit for random metrics and coupled random growth, the  directed landscape, was constructed by \cite{DOV}. For more background, we refer the reader to \cite{quastel2011introduction, corwin2012kardar, romik2015surprising, borodin2016lectures, zygouras2022some, ganguly2021random} and references therein.

	The papers discussed above are all set in \textit{full-space}, i.e., they study random growth models on $\R$, or analogously, metric and polymer models on $\R^2$. It is natural to try to understand how the behavior of these models changes in the presence of a boundary. The simplest way to introduce a boundary is to consider growth models on $\R_{\ge 0}$ or metric models defined in a half-plane. Collectively, these are often referred to as \textit{half-space} models. The behavior of half-space models depends on the strength of the boundary interaction, with predictions for a phase transition in the boundary parameter dating back to \cite{kardar1985}. 
	
	Remarkably, many of the most tractable KPZ models retain at least some exact solvability when appropriately defined in half-space. However, the models are almost always harder to understand in half-space than in full-space and so while many beautiful theorems have been proven in this setting, there has always been a lag in our collective understanding. The half-space setting was first studied rigorously by Baik and Rains, who established one-point limits and phase transitions when both the start and end points are on the boundary for certain exactly solvable last passage models \cite{BR1,BR2,BR3}. Since then, there has been a series of breakthroughs which move beyond these one-point laws in the zero temperature setting. The papers \cite{SI,BBCS1,BBCS2} prove multi-point convergence in the spatial direction when the starting point is on the boundary; \cite{OSZ, BBCW,HeB, bisi2021geometric, IMSsol} establish solvability and one-point limit laws for positive temperature models; and \cite{BC,BFO,CW,CK,BCY} give explicit descriptions for stationary measures, both in half-space and on a strip (with boundary conditions on both sides).
	Finally, an unexpected recent breakthrough of \cite{Zhang} derives formulas for the half-space KPZ fixed point from general initial conditions.
    (See also \cite{dGMRW} for another set of exact formulas for half-space random growth models with general initial conditions.)
	
	The main content of this paper is to construct the directed landscape in half-space, which we expect to be the complete scaling limit for random metric models and coupled random growth in half-space. Our scaling limit depends on a parameter which determines the strength of the boundary interaction.  We construct the directed landscape in half-space by proving convergence of two half-space KPZ models to the same limit: exponential last passage percolation (LPP), and the colored totally asymmetric simple exclusion process (TASEP).
	
	The construction of the full-space directed landscape in \cite{DOV} goes through an analysis of the Airy line ensemble \cite{PSale,CH}, which is connected to random metrics through the Robinson–Schensted–Knuth (RSK) correspondence.
	A similar construction could potentially be carried out in half-space since the RSK correspondence still applies in that setting. Indeed, there has been recent progress on line ensembles and their convergence in half-space \cite{dimitrov2025half,dimitrov2026pinned,das2026pinning}. On the other hand, due to the lack of symmetry in half-space (e.g., no translation invariance), the rich boundary behavior, and more complicated exact formulas, completing a construction of the half-space directed landscape analogously to \cite{DOV} would at a minimum be a formidable technical accomplishment and likely require many new ideas.   
	
	Our strategy in the present paper differs completely from the line ensemble approach of \cite{DOV}. 
	We first find joint stationary measures for exponential LPP and use this to construct the half-space stationary horizon, which couples all stationary measures for the half-space directed landscape. The analogous object in full-space, the stationary horizon, was constructed in \cite{BuS,BSS24}. Our construction of joint stationary measures also works for the log-gamma polymer and the KPZ equation, so we obtain results for these models in passing. We then prove tightness of exponential LPP, and relate our subsequential limits back to the KPZ fixed point. These first two steps give us a collection of subsequential limits for each boundary parameter. Finally, we use these subsequential limits to derive a characterization of the half-space directed landscape in terms of a triangle inequality, independence of time increments, and KPZ fixed point marginals. The proof proceeds by showing that if we have two four-parameter processes satisfying a triangle inequality, independence of time increments, and KPZ fixed point marginals, and the half-space stationary horizon is stationary for \textit{one} of these processes (as is the case for our subsequential limits), then the two processes must be equal in law. 
	
	The proofs in this paper build on several key ideas introduced over the past few years together with several crucial novel aspects.
	In our construction of the half-space stationary horizon, we introduce new and strong symmetries for exactly solvable KPZ models in half-space, extending techniques from \cite{BC}.
	The proof of tightness harnesses a key identity from \cite{BW} and our stationary measure construction, together with probabilistic ideas that build on \cite{basu2014last, BBSloc, dauvergne2021scaling, ganguly2023optimal}. The idea and proof framework for the characterization stems from our recent work \cite{DZ24} in the full-space setting, with the half-space KPZ fixed point from \cite{Zhang} as an input. The key quantitative estimates differ significantly from those in \cite{DZ24}, and are proven using new methods.
	
	In the rest of the introduction, we present our main results, along with our proof strategy and some key ingredients which may be of independent interest.

	\subsection{The directed landscape in half-space}
	
	\subsubsection{Construction via convergence}   \label{sssec:constru}
	
	For each $\alpha>0$, the \emph{half-space exponential last passage percolation (LPP)} with boundary parameter $\alpha$ is defined as follows.
	Let $\Z^2_\ge=\{(i,j): i, j \in \Z, i\ge j\}$, and let $\{X(u)\}_{u \in \Z^2_\ge}$ be an array of independent random variables such that
	\[
	X(i,j)\sim \Exp(1) \;\; \text{ for } \;\; i>j\in\Z, \qquad \qquad
	X(i,i)\sim \Exp(\alpha) \;\; \text{ for } \;\; i\in \Z.
	\]
	Next, for $u, v \in \Z^2$ write $u \le v$ if $u_1 \le v_1, u_2 \le v_2$. For $u \le v \in \Z^2_\ge$ define the passage time 
	\begin{equation}
		\label{E:ELPP-first-def}
	X(u; v) = \max_{\pi}\sum_{w\in \pi} X(w),
	\end{equation}
	where the maximum is over all up-right vertex paths $\pi$ from $u$ to $v$ contained in the half-space $\Z^2_\ge$. We set $X(u; v) = -\infty$ for all other choices of $u, v$.
	To take a scaling limit, we let $\mathbb{H} = \R_{\ge 0} \times \R$ and
	\begin{equation}   \label{eq:defnH2}
	\mathbb{H}^2_\uparrow := \{(x, s; y, t) : (x, s), (y, t) \in \mathbb{H}, s < t\} \subset \R^4.
	\end{equation}
	For each $n\in\N$ and $u = (x, s) \in \mathbb{H}$ define the scaling $(x, s)_n = (\lfloor ns + 2^{5/3} n^{2/3} x \rfloor, \lfloor ns \rfloor)$, which puts limiting coordinates back in $\Z^2_\ge$. We also set $h_n(u) = 4ns + 2^{8/3} n^{2/3} x$, which is the limit shape for $X(0,0; u_n)$ above the fluctuation scale. Finally, define $\cL_n:\mathbb{H}^2_\uparrow \to \R \cup \{-\infty\}$ by
	\begin{equation}  \label{eq:cLn}
		\cL_n(u; v) := 2^{-4/3} n^{-1/3}\left(X(u_n; v_n)  - [h_n(v) -h_n(u)]\right).
	\end{equation}
Note that $\cL_n$ has a dependence on the boundary parameter $\alpha$.
	\begin{theorem}   \label{thm:expconv}
		In \eqref{eq:cLn}, take either
		\begin{enumerate}[nosep, label=(\roman*)]
			\item fixed $\alpha > \frac{1}{2}$, or
			\item $\alpha = \frac{1}{2}-2^{-4/3}\rho n^{-1/3}$ for some $\rho\in\R$.
		\end{enumerate}
		Then as $n\to\infty$, $\cL_n$ converges in distribution in the uniform-on-compact topology to a random continuous function from $\mathbb{H}^2_\uparrow \to \R$. In case (i) the limit is independent of $\alpha$.
	\end{theorem}
	
	\begin{definition}     \label{thm:defhdl}
		We let $\cL^\rho:\mathbb{H}^2_\uparrow\to \R$ denote the limit from Theorem \ref{thm:expconv}, where $\rho=-\infty$ in case (i) and $\rho\in \R$ in (ii). The random function $\cL^\rho$ is the  \emph{half-space directed landscape with (boundary) parameter $\rho$}. 
	\end{definition}

The parameter $\rho$ in Definition \ref{thm:defhdl} corresponds to the strength of the boundary. As we increase $\rho$, the boundary strength increases and geodesics tend to spend more time on the boundary, see Section \ref{ssec:scaling} for more discussion and Section \ref{ssec:geodesic-cvgence} for definitions of path length and geodesics. The case when $\rho = -\infty$ is a subcritical setting when the boundary has no weight. In this special case, the function $\cL^{-\infty}$ should have an alternate description through optimizing over paths restricted to $[0, \infty) \times \R$ in a full-space directed landscape. For $\rho > -\infty$, we believe there should be a local time representation for the boundary effect, see Problem \ref{prob:measurability} and surrounding discussion.

	We also prove convergence of a family of \textbf{multi-level Poisson-avoiding metrics} to the same limit. We describe the single-level model here, which is equivalent to the basic coupling for the totally asymmetric simple exclusion process (TASEP) in half-space. 
	
	Half-space TASEP is a Markov process on particle configurations on $\N$, where each location $x\in \N$ is either empty or occupied by a particle. We define the dynamics through a collection of Poisson clocks. Let $\Z_{\ge 0} = \{0, 1, \dots \}$, and consider a Poisson process $\Pi$ on $\Z_{\ge 0} \times \R$, where $\Pi$ has intensity $1$ on $\N \times \R$ and intensity $\alpha>0$ on $\{0\} \times \R$. We think of $\Pi_x := \Pi|_{\{x\} \times \R}$ as a Poisson clock at site $x$. 
	Whenever $\Pi_x$ rings, if site $x$ is occupied by a particle and site $x+1$ is empty, then the particle at site $x$ jumps to $x+1$.
	Whenever $\Pi_0$ rings, if site $1$ is empty, we add an extra particle at site $1$, i.e., a particle jumps from an infinite-particle reservoir to site $1$. Given a particle configuration at time $t$, we define the associated \emph{height function} $h_t:\Z_{\ge 0} \to \R$ by letting
	$$
	h_t(x) - h_t(x-1) = 2 \cdot \mathds{1}(\text{site } x \text{ is empty at time }  t) - 1, \qquad x \in \N.
	$$
	This rule determines $h_t$ up to the choice of $h_t(0)$. While it is convenient to allow $h_0(0)$ to remain arbitrary, to ensure that TASEP induces a nice dynamics on the space of height functions for $t > 0$, we set 
	$$
	h_t(0) = h_{t^-}(0) + 2\cdot \mathds{1}(\text{a particle jumps onto site $1$ at time $t$}).
	$$
	With this choice, TASEP dynamics on heights are equivalent to the rule that a local minimum of $h$ at site $x \in \Z_{\ge 0}$ flips to a local maximum according to the Poisson clock at site $x$. If we couple multiple TASEP evolutions using the same Poisson clocks, this is known as the \emph{basic coupling} \cite{liggett1976coupling}. It is equivalent to \emph{colored TASEP}, see \cite{amir2011tasep, ACH} for background.
	
	The basic coupling of TASEP can be recast as a random metric. Given $\Pi$, we define the \textbf{Poisson-avoiding metric} 
	$$
	H: (\Z_{\ge 0} \times \R)^2 \to \Z_{\ge 0}
	$$ 
	as follows. For $u =(x, s), v = (y, t) \in \Z_{\ge 0} \times \R$ with $s \le t$ let
	$
	\mathbb{M}(u, v)
	$
	be the set of all functions $f:[s,t]\to \Z_{\ge 0}$ with $f(s)=x$ and $f(t)=y$. Define
	\begin{equation}
		\label{E:PA-def}
		H(u; v) = \min \left\{ \int_s^t |f'| + 2 \#(\{(f(r), r) : r \in (s, t]\} \cap \Pi) : f \in \mathbb{M}(u, v)\right\},
	\end{equation}
	where $\int_s^t |f'|$ is the total variation of $f$, and set $H(v; u) = H(u; v)$. It is straightforward to check that $H$ defines a pseudometric, i.e., a metric where distinct points are allowed to have distance $0$.
To explain the name for the Poisson-avoiding metric, note that it is enough to take the minimum above over $f \in \mathbb{M}(u, v)$ which do not hit $\Pi$ on the interval $(s, t)$, since we can always avoid a Poisson point at the cost of increasing total variation by $2$. See Figure \ref{fig:FPP} for an illustration.

Interpreting TASEP as a \textit{discrete Burgers equation} with random force given by $\Pi$, the integral in \eqref{E:PA-def} is then the \textit{action} of $f\in\mathbb{M}(u, v)$ (see e.g., \cite{Lions,EKMS} on general stochastic Burgers or Hamilton-Jacobi equations). For an initial TASEP height function $h_0:\Z_{\ge 0} \to \Z$, we can define
	\begin{equation}   \label{eq:h0Hvar}
	h_t(x; h_0) = \min_{y\in \Z_{\ge 0}} h_0(y) + H(y,0;x,t).
	\end{equation}
	This gives the evolution of TASEP started from $h_0$ using the Poisson clocks $\Pi$. We now state our main convergence theorem for Poisson-avoiding metrics. Here we take a lattice spacing parameter $\ep \to 0$, following \cite{matetski2016kpz}.
	
	\begin{theorem}  \label{thm:colorTASEPconv}
		For $\ep>0$, define the rescaling $(x, s)_\eps = (\lfloor 2\eps^{-1}x \rfloor, 2\eps^{-3/2}s)$ and define $H_\eps:\mathbb H^2_\uparrow \to \R$ by
		\[
		H_\eps(x, s; y, t) = -\eps^{1/2}H((x, s)_\eps; (y, t)_\eps) + \eps^{-1}(t-s).
		\]
		Take either (i) fixed $\alpha > \frac{1}{2}$ and $\rho=-\infty$, or (ii) $\alpha = \frac{1}{2}-\rho\eps^{1/2}/2$ for some $\rho\in\R$.
		Then as $\eps\to 0$, $H_\eps$ converges in distribution to $\cL^\rho$ in the uniform-on-compact topology.
	\end{theorem}
	Theorem \ref{thm:colorTASEPconv} is a special case of the convergence of multi-level Poisson-avoiding metrics, given in Theorem \ref{thm:mPAM}. A quick corollary of Theorem \ref{thm:colorTASEPconv} gives joint convergence of TASEP from different initial conditions in the basic coupling (and when coupled through multi-level metrics).
	
	TASEP from a fixed initial condition can be constructed as a sequence of growing balls in exponential LPP, see Section \ref{ssec:kpzfp} for details. However, the basic coupling of TASEP is not the same as the coupling through exponential LPP, so there is not a direct translation between exponential LPP and Poisson-avoiding metrics. Instead, the reason we are able to derive Theorem \ref{thm:colorTASEPconv} from the proof of Theorem \ref{thm:expconv} is because our proof gives a characterization of the half-space directed landscape from its KPZ fixed point marginals, which correspond to limits of TASEP from a single initial condition.
	
	\begin{remark}[Symmetric environments]
		\label{R:symmetrized}
	A small extension of Theorem \ref{thm:expconv} and Theorem \ref{thm:colorTASEPconv} should give convergence of the same models in symmetric environments. In the exponential setting, the symmetric environment corresponds to setting $X(i, j) = X(j, i)$ for all $(i, j) \in \Z^2$ and in the Poisson-avoiding metric we can extend to a symmetric environment by letting $\Pi_x = \Pi_{-x}, x \in \N$. Distances in symmetric environments match those of half-space environments unless we are forced to cross the axis of symmetry, in which case distances are equivalent to maximizing over paths in a half-space environment which are forced to touch a boundary. For brevity, we have not pursued this extension here. We note in passing that \emph{colored symmetrized TASEP}, which dates back to Harris \cite{Ha77} (see \cite{He2025} for a modern presentation and more background) can be described from the symmetrized Poisson-avoiding metric with an analogue of the formula \eqref{eq:h0Hvar}.
\end{remark}

\begin{remark}[Other models]
The full-space analogue of Theorem \ref{thm:expconv} is shown in \cite{dauvergne2021scaling}, and the full-space analogue of Theorem \ref{thm:colorTASEPconv} is shown in \cite{ACH} (see also \cite{DZ24} for a different proof). There are many other models where directed landscape convergence is known and for which a half-space version exists, e.g., geometric or Poisson LPP \cite{dauvergne2021scaling}, the KPZ equation \cite{wu2023kpz}, the stochastic six vertex model and colored ASEP \cite{ACH}, the log-gamma polymer \cite{zhang2025convergence}. See also \cite{virag2025actions} for  a unified conceptual framework for treating LPP and polymer models. The most important ingredient needed to prove half-space directed landscape convergence for these models using our present framework is KPZ fixed point convergence. While it is possible that the techniques of \cite{Zhang} could extend to the LPP models listed above, beyond these cases it is likely significant new ideas are needed.
\end{remark}

	\begin{figure}[t]
		\centering
		\resizebox{0.6\textwidth}{!}{
			\centering
			\begin{tikzpicture}
				\foreach \i in {0,...,8}
				{
					\draw [line width=.1pt] (\i,0) -- (\i,4);
					
				}
				
				\begin{scriptsize}
					
					\foreach \i in {0,...,8}
					{
						\draw (\i,0) node[anchor=north]{\i};
					}
				\end{scriptsize}
				
				\draw [line width=1pt, color=blue] (-0.2,2) -- (0.2,2);
				\draw [line width=1pt, color=blue] (-0.2,1) -- (0.2,1);
				\draw [line width=1pt, color=blue] (-0.2,3.5) -- (0.2,3.5);
				\draw [line width=1pt, color=blue] (1-0.2,2.9) -- (1+0.2,2.9);
				\draw [line width=1pt, color=blue] (1-0.2,0.8) -- (1+0.2,0.8);
				\draw [line width=1pt, color=blue] (2-0.2,2) -- (2+0.2,2);
				\draw [line width=1pt, color=blue] (3-0.2,2.4) -- (3+0.2,2.4);
				\draw [line width=1pt, color=blue] (3-0.2,0.4) -- (3+0.2,0.4);
				\draw [line width=1pt, color=blue] (4-0.2,1.5) -- (4+0.2,1.5);
				\draw [line width=1pt, color=blue] (4-0.2,3.2) -- (4+0.2,3.2);
				\draw [line width=1pt, color=blue] (4-0.2,0.2) -- (4+0.2,0.2);     
				\draw [line width=1pt, color=blue] (5-0.2,0.9) -- (5+0.2,0.9);
				\draw [line width=1pt, color=blue] (5-0.2,2.3) -- (5+0.2,2.3);      
				\draw [line width=1pt, color=blue] (6-0.2,2) -- (6+0.2,2);
				\draw [line width=1pt, color=blue] (6-0.2,0.3) -- (6+0.2,0.3);
				\draw [line width=1pt, color=blue] (6-0.2,3.8) -- (6+0.2,3.8);
				\draw [line width=1pt, color=blue] (6-0.2,3.1) -- (6+0.2,3.1);    
				\draw [line width=1pt, color=blue] (7-0.2,1.4) -- (7+0.2,1.4);
				\draw [line width=1pt, color=blue] (7-0.2,2.5) -- (7+0.2,2.5);
				\draw [line width=1pt, color=blue] (8-0.2,2.9) -- (8+0.2,2.9);
				\draw [line width=1pt, color=blue] (8-0.2,0.8) -- (8+0.2,0.8);
				\draw [line width=1pt, color=blue] (8-0.2,1.2) -- (8+0.2,1.2);          
				
				\draw [line width=1pt, color=red] (5,0) -- (5,0.6) -- (4,0.6) -- (4,1) -- (3,1) -- (3,2.2) -- (2,2.2) -- (2,3) -- (3,3) -- (3,3.4) -- (4,3.4) -- (4,4); 
				\draw [color=red, fill=red] (5,0) circle (1.5pt);
				\draw [color=red, fill=red] (4,4) circle (1.5pt);
			\end{tikzpicture}
		}
		
		\caption{An illustration of the Poisson avoiding metric. The blue horizontal bars are the Poisson process $\Pi$. The red path avoids $\Pi$, and has minimal total variation among all paths between the same endpoints.}  
		\label{fig:FPP}
	\end{figure}
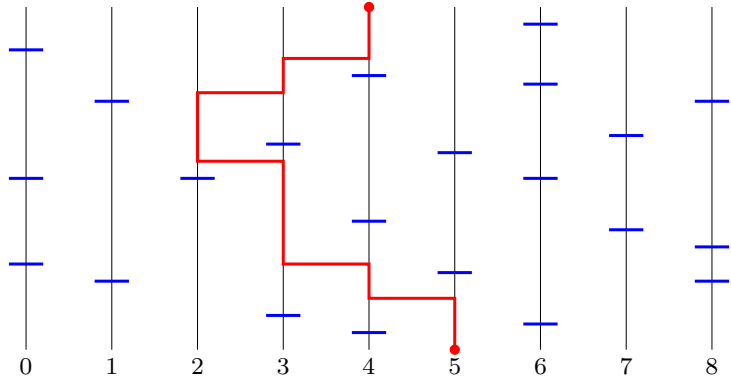

	\subsubsection{Scaling, shape theorem, and continuity}
    \label{ssec:scaling}
	As part of the proof of Theorem \ref{thm:expconv}, we establish various symmetries, metric properties, continuity, and regularity for the directed landscape. Many of these are summarized in Proposition \ref{P:properties-elpp} and Proposition \ref{P:only-easy-consequences}. We highlight a few properties here.
	
	First, we have the following $1$:$2$:$3$ scale invariance, from Proposition \ref{P:properties-elpp}.3. For any $q > 0, \rho \in \R \cup \{-\infty\}$:
	\begin{equation}
		\label{E:KPZ-scaling-intro}
\cL^{q \rho}(x, s; y, t) \eqd q^{-1} \cL^\rho(q^2x, q^3 s; q^2 y, q^3 t),
	\end{equation}
	where the equality in law is as random continuous functions on $\mathbb H^2_\uparrow$. Equation \eqref{E:KPZ-scaling-intro} implies that up to scaling, all the limits $\cL^\rho, \rho  \in (0, \infty)$ are equivalent, as are all the limits $\cL^\rho, \rho \in (-\infty, 0)$. In other words, up to scaling there are only four distinct limit objects: a \emph{subcritical} limit $\cL^{-\infty}$, a \emph{slightly subcritical} limit $\cL^{-1}$, a \emph{critical} limit $\cL^0$, and a \emph{slightly supercritical} limit $\cL^1$. By \eqref{E:KPZ-scaling-intro} when we zoom in on the slightly subcritical or slightly supercritical limits through $1$:$2$:$3$ rescaling, we recover the critical limit, which is scaling invariant. If we zoom out in the slightly subcritical limit, we recover the subcritical limit, which is again scaling invariant. Zooming out in the slightly supercritical limit in the $1$:$2$:$3$ scaling results in the metric blowing up. If we account for this blow-up in the scaling, then there should instead be a limiting transition to a Gaussian limit which can be described from a single Brownian motion at the boundary. We do not pursue this direction here.

	Next, the half-space directed landscape $\cL^\rho$ satisfies continuity and shape theorems, which are half-space analogues of \cite[Propositions 10.5, Corollary 10.7]{DOV}. We state a weaker version of these results here (see Proposition \ref{P:only-easy-consequences}.1 and 2 below).
	\begin{theorem}   \label{thm:dlsc}
		Fix $\rho\in \R\cup\{-\infty\}$, a compact $K\subset \mathbb{H}^2_\uparrow$, and $\eps > 0$. Then on $K$, almost surely the half-space directed landscape $\cL^\rho(x, s; y, t)$ is H\"older-$(1/2-\eps)$ continuous in $x$ and $y$ and H\"older-$(1/3-\eps)$ continuous in $s$ and $t$.
		
		Moreover, for any $b > 0$ there exists $c > 0$ and a random number $M > 0$ satisfying
		$
		\p(M > a) \le 2\exp(-ca^{3/2})
		$
		for all $a > 0$, such that for all $u = (x, t; y, t + s) \in \mathbb H^2_\uparrow$ with $s \le b$ we have
		$$
		\lf|\cL^\rho(u) + (x-y)^2/s \rg|\le M s^{1/3} \log^{2}\lf(\frac{2\|u\|_2}{s}\rg).
		$$
	\end{theorem}
	
	\subsubsection{Random geometry and geodesic convergence}
	\label{ssec:geodesic-cvgence}
	As in \cite{DOV, dauvergne2021scaling}, uniform-on-compact convergence of a model to the directed landscape can be used to show convergence of geodesics. As in the full-space directed landscape, geodesics in the half-space directed landscape $\cL^\rho$ exist because of the metric composition law, recorded as Proposition \ref{P:only-easy-consequences}.3 in the sequel.
	\begin{prop}   \label{thm:dlcomp}
		For each $\rho\in\R \cup \{-\infty\}$, almost surely we have
		\[
		\cL^\rho(x,r;y,t) = \max_{z\ge 0} \cL^\rho(x,r;z,s) + \cL^\rho(z,s;y,t),
		\]
		for each $(x,r;y,t) \in \mathbb{H}^2_\uparrow$ and $r<s<t$.
	\end{prop}
	The metric composition law allows us to construct geodesics in $\cL^\rho$. We first define path length.
	\begin{definition}
		Fix $\rho\in\R \cup \{-\infty\}$. For a  continuous path $\pi:[s,t]\to \R_{\ge 0}$, its \emph{length in $\cL^\rho$} is
		\begin{equation}  \label{eq:defncLmet}
	\|\pi\|_{\cL^\rho}=\inf_{k\in \N} \inf_{s=t_0<t_1<\cdots < t_k=t} \sum_{i=1}^k \cL^\rho(\pi(t_{i-1}), t_{i-1}; \pi(t_i), t_i).
		\end{equation}
		A path $\pi$ is a \emph{geodesic} from $(\pi(s), s)$ to $(\pi(t), t)$ if $\|\pi\|_{\cL^\rho} = \cL^\rho(\pi(s),s; \pi(t),t)$.
	\end{definition}
	By the metric composition law, $\pi$ is a directed geodesic if the equality in \eqref{eq:defncLmet} holds before taking any infima, i.e., for any $s=t_0<t_1<\cdots < t_k=t$. As in the full-space setting, geodesics exist between any pair of points, are H\"older-$2/3^-$ continuous and are almost surely unique. 
	
	\begin{theorem}  \label{thm:geounicont}
	Fix $\rho \in \R \cup \{-\infty\}$.
	\begin{enumerate}[nosep, label=\arabic*.]
		\item Almost surely, for all $(p; q) \in \mathbb H^2_\uparrow$, there exists at least one geodesic in $\cL^\rho$ from $p$ to $q$.
		\item Almost surely, all geodesics in $\cL^\rho$ are H\"older-$(2/3-\eps)$ continuous for all $\eps > 0$.
		\item For any fixed $(p; q) \in \mathbb H^2_\uparrow$, almost surely there is a unique geodesic in $\cL^\rho$ from $p$ to $q$.
	\end{enumerate}	
\end{theorem}

Theorem \ref{thm:geounicont} is the half-space counterpart of \cite[Theorem 1.7]{DOV}. It is proven in Proposition \ref{P:geos-exist} and Proposition \ref{P:geo-unique}. The arguments for geodesic existence and continuity are essentially as in \cite{DOV}. Uniqueness is slightly harder in the half-space setting since we cannot directly appeal to Brownian absolute continuity of spatial profiles and requires a new argument. Given existence and uniqueness of geodesics, we can argue along the lines of \cite{DOV, dauvergne2021scaling} to prove convergence of geodesics in all models. After some initial setup, we will treat Poisson-avoiding metrics and exponential LPP together for the purposes of the next theorem. Here for a function $f:[s, t] \to \R$ we write $\mathfrak{g} f := \{(f(r), r) : r \in [s, t]\}$.

We start with exponential LPP, recalling \eqref{eq:cLn} and the scaling $(x, s)_n = (\lfloor ns + 2^{5/3} n^{2/3} x \rfloor, \lfloor ns \rfloor)$. We say that $f:[s, t] \to \R$ is a geodesic for $\cL_n$ if
$
(\mathfrak{g} f)_n \subset \pi, 
$
where $\pi$ is a geodesic in the environment $X$ from $(f(s), s)_n$ to $(f(t), t)_n$, where we view up-right lattice paths as subsets of $\Z^2$. 

Moving to Poisson-avoiding metrics, we say that $f:[s, t] \to \R$ is a geodesic for $H_\eps$ if 
$\mathfrak{g} \tilde f = (\mathfrak{g} f)_\eps$ for a geodesic $\tilde f$ in $H$ from $(f(s), s)_\eps$ to $(f(t), t)_\eps$.

	\begin{theorem}  \label{thm:expLPPgeoconv}
Fix $\rho \in \R \cup \{-\infty\}$. Either let $\cM_n$ equal $\cL_n$ in Theorem \ref{thm:expconv}, or else let $\cM_n$ equal $H_{\ep_n}$ in Theorem \ref{thm:colorTASEPconv} for a sequence $\ep_n \to 0$. Consider a coupling where $\cM_n \to \cL^\rho$ almost surely, uniformly on compact sets. Then there exists a set $\Omega$ of probability $1$ such that on $\Omega$, the following assertions hold.
\begin{enumerate}[nosep]
	\item Consider a sequence $(x_n, s; y_n, t) \to (x, s; y, t) \in \mathbb H^2_\uparrow$ and a sequence of geodesics $\pi_n$ in $\cM_n$ from $(x_n, s)$ to $(y_n, t)$. Then $\pi_n$ is precompact in the uniform topology and any subsequential limit of $\pi_n$ is a geodesic in $\cL^\rho$ from $(x, s)$ to $(y, t)$.
	\item If $\cL^\rho$ contains a unique geodesic $\pi$ from $(x, s)$ to $(y, t)$, then $\pi_n \to\pi$.
\end{enumerate}
	\end{theorem}
In the above theorem, we could also allow the time coordinates to vary with $n$. Doing so requires replacing uniform convergence of paths with Hausdorff convergence of graphs.

	\subsection{Broader ingredients: stationary horizon, tightness and characterization}
	
	As discussed at the beginning of the introduction, the proof of Theorem \ref{thm:expconv} consists of three key steps: constructing joint stationary measures (i.e., horizons), proving tightness bounds, and a characterization theorem. Each of these steps yields results which are of interest in their own right. We discuss some highlights here, referring the reader further down the paper for complete statements. 
	
	\subsubsection{Half-space stationary horizon}
	
	The first step in the proof of Theorem \ref{thm:expconv} is the construction of joint stationary measures for exponential LPP in half-space (i.e., the half-space exponential horizon). The proof also works for the log-gamma polymer and geometric LPP, and by taking degenerations we can recover joint stationary measures for the KPZ equation and the directed landscape. 
	
	For the purposes of the introduction, we focus on the $1$:$2$:$3$ scaling limit for all of these models, which gives joint stationary measures for the half-space directed landscape. Let $\mathcal D(\alpha)$ be the set of functions $F:[\alpha, \infty) \times \R_{\ge 0} \to \R$ such that
	\begin{itemize}[nosep]
		\item $F(\cdot, 0) = 0$ and $F(\alpha, \cdot)$ is continuous, and
		\item the function
		$$
		\mu_F([\lambda_1 , \lambda_2] \times [x_1, x_2]) = F(\lambda_1, x_1)+F(\lambda_2, x_2)-F(\lambda_2, x_1)-F(\lambda_1, x_2)
		$$
		defines a $\sigma$-finite Borel measure on $[\alpha, \infty) \times \R_{\ge 0}$.
	\end{itemize}
	The space $\mathcal D(\alpha)$ is a Polish space, under the topology of uniform-on-compact convergence on $F(\alpha, \cdot)$ and vague convergence of $\mu_F$.

	\begin{theorem}  \label{thm:stahri}
		Fix $\rho \in \R \cup \{-\infty\}$, and let $\cL^\rho$ be the half-space directed landscape of parameter $\rho$. Then there exists a random function $\cH^\rho \in \mathcal D(2\rho \vee 0)$ with explicit marginals given in Proposition \ref{P:half-space-horizon-marginals} below, such that if we consider 
		\[
		\cH_t^\rho(\lambda, x) := \sup_{y\ge 0} \cH^\rho(\lambda, y) + \cL^\rho(y,0;x,t),
		\]
		then the function$(\lambda, x) \mapsto \cH_t^\rho(\lambda, x) - \cH_t^\rho(\lambda, 0)$ is equal in law to $\cH^\rho$. We call $\cH^\rho$ the \emph{half-space stationary horizon} with (boundary) parameter $\rho$.
	\end{theorem}
	From the marginals in Proposition \ref{P:half-space-horizon-marginals}, the process $\cH^\rho$ satisfies the following slope asymptotics:
	$$
	\lim_{x \to \infty} \frac{1}{x}\cH^\rho(\lambda, x) = \begin{cases}
		\lambda, \qquad &\lambda > 2\rho \vee 0, \\
		-(2\rho \vee 0), \qquad &\lambda = 2\rho\vee 0.
	\end{cases}
	$$
	We expect that the law of $\cH^\rho$ is uniquely determined by stationarity and the slope condition above, and that it can alternately be described in terms of Busemann functions. Finally, the description for $\cH^\rho$ is quite tractable. If we want to describe the marginal of $\cH^\rho$ restricted to $k$ values of $\lambda_1, \dots , \lambda_k$, then we can do this by combining a last passage problem across $2k$ independent Brownian motions on $\R_{\ge 0}$ of drifts $\pm \lambda_1, \dots, \pm \lambda_k$ with a last passage problem involving an array of $k^2$ independent exponential random variables (see the model of half-space exponential-Brownian LPP in Section \ref{ss:hfsh}). When $k = 1$, this description matches a description from \cite{BC}. The comparable marginal in full-space is similar, but simpler. It can be described using either a last passage process with $k$ independent Brownian motions on $\R$ of drift $\lambda_1, \dots, \lambda_k$ \cite{BuS, BSS24}, or else $k$ independent Brownian motions on $\R_{\ge 0}$ and an array of ${k \choose 2}$ exponential random variables \cite{bates2025permutation}.
	
	The explicit construction of the half-space stationary horizon and constructions of similar horizons for other models are given in Section \ref{sec:horizons}. Our constructions actually give more general versions which apply to both symmetric and half-space models. Finally, we remark that while the statement of Theorem \ref{thm:stahri} uses the half-space directed landscape $\cL^\rho$ as an input, in the sequel we will build the half-space stationary horizon before constructing the half-space directed landscape. We construct $\cH^\rho$ through its marginals (Section \ref{ss:hfsh}) by rescaling the exponential horizon, then after proving tightness for exponential LPP, state a version of Theorem \ref{thm:stahri} for subsequential limits (Proposition \ref{P:properties-elpp}.10).

	\subsubsection{Moderate deviation bounds}
	The second step in the proof of Theorem \ref{thm:expconv} is to establish tightness in the uniform-on-compact topology. The main task here is to establish one-point and two-point moderate deviation tail bounds for exponential LPP. The bounds we prove are optimal, up to constants, in the parameter ranges that are relevant for taking a scaling limit. However, they also extend far beyond these ranges (sometimes suboptimally), which may be of use in studying prelimiting models directly.  
    
    For the purposes of the introduction, we only state the one-point tail bounds on the diagonal, as this is where the situation differs most dramatically from the full-space setting. The reader should compare these with the full-space bounds of Ledoux and Rider, see Theorem \ref{T:W-tail-bounds} below.
    We refer the reader to Theorem \ref{T:half-space-full-tails} for tail bounds off of the diagonal, Proposition \ref{P:two-point-bd} for a two-point spatial tail bound, and Proposition \ref{P:temporal-est} for a two-point temporal tail bound. The limiting versions of all these bounds are contained in Proposition \ref{P:properties-elpp}.

	Recall the model of half-space exponential LPP from \eqref{E:ELPP-first-def}.
	Define
	\begin{equation}  \label{eq:mualphan}
	\mu_{\alpha}(n) = \begin{cases}
		4 n, \qquad &\alpha \ge 1/2. \\
		\frac{n}{\alpha(1 - \alpha)}, \qquad &\alpha < 1/2,
	\end{cases}
	\end{equation}
	which gives the leading order for the passage time $X(1, 1; n, n)$
	(see e.g., \cite{BBCS1}).
\begin{theorem}
\label{T:one-point-bds}
For any $\alpha_0>0$, there exists a constant $c>0$, such that whenever the boundary parameter satisfies $\alpha\ge \alpha_0$ we have the following bounds. 
\begin{enumerate}
    \item (Upper tail) For all $n \in \N, \eps \in (0, n^{-1/3}]$ we have
$$
\P(X(1, 1; n, n) \ge \mu_\alpha(n) + \eps n) \le 2\exp\left(-c \eps^{3/2} n \left(1 \wedge \frac{\eps^{1/2}}{(1/2 - \alpha)\vee 0}\right) \right).
$$
    \item (Shallow lower tail when $\alpha < 1/2$) For all $n \in \N, \eps \in (0, 1]$ such that $\alpha \le 1/2 - n^{-1/3}$ and $\eps n \le \mu_\alpha(n) - 4n$, we have
\begin{align}
\label{E:X-type-above-4}
	\P\Big(X(1, 1; n, n) \le \mu_\alpha(n) - \eps n\Big) \le 2\exp\left(- \frac{c\eps^2 n}{1/2 - \alpha} \right).
\end{align}
\item (Full lower tail for $\alpha \ge 1/2$, deep lower tail for $\alpha < 1/2$) For all $n \in \N, \eps \in (0, 1]$ we have
\begin{align}
\label{E:X-type}
	\P\Big(X(1, 1; n, n) \le 4n - \eps n\Big) \le 2\exp\left(- c \eps^3 n^2\right).
\end{align}	
\end{enumerate}
\end{theorem}	
\begin{remark}
The situation here is most interesting in the supercritical case when $\alpha < 1/2$. In that setting, we can see that there is a transition from Gaussian behavior in the tail to Tracy-Widom behavior as we move deeper. This reflects a transition from the tail behavior being controlled by deviations of the boundary weights in the shallow tail, and the bulk weights in the deep tails.
(See \cite[Section 3]{ESMixing} for the Gaussian tail in the supercritical case.)
\end{remark}
\begin{remark}
The upper bound on $\eps$ in part $1$ of Theorem \ref{T:one-point-bds} is artificial and should extend all the way to $\eps < 1$, at which point we transition to the large deviation regime. This comes from a minor technical constraint due to suboptimal tail bounds in full-space in certain regimes (see Theorem \ref{T:W-tail-bounds}). It can easily be rectified, but we do not pursue this here.     
\end{remark}

	Given tightness, we can take a subsequential limit $\cQ^\rho:\cH_\uparrow^2\to \R$ of $\cL_n$ from \eqref{eq:cLn}. We would like to show that any such $\cQ^\rho$ has the same law.
	For this, we rely on two key properties of $\cQ^\rho$: the joint stationary measure (i.e., the half-space stationary horizon from Theorem \ref{thm:stahri}), and half-space KPZ fixed point marginals.

	\subsubsection{Characterization using the half-space KPZ fixed point}
	
		Define the state space
	\begin{equation}
		\label{E:UC0+}
		\UC_{0,+} := \Big\{f: \R_{\ge 0} \to \R \cup \{-\infty\} : f \text{ upper semicontinuous}, f \not \equiv - \infty, \quad \limsup_{x \to \infty} \frac{f(x)}{x^2}\le 0  \Big\}
	\end{equation}
	with the topology of local hypograph convergence, see \cite[Section 3.1]{matetski2016kpz} for details. The KPZ fixed point in half-space is a Markov process on $\UC_{0,+}$, whose transition probabilities (on a dense subset of $\UC_{0,+}$) were obtained in \cite{Zhang} by taking a scaling limit of TASEP in half-space. We note that there are some minor technical issues to show that the transition probabilities from \cite{Zhang} indeed define a Feller process with the above state space. In the present paper, we handle this through a variational characterization of the transition probability from time $s$ to $t$. Taking $\cQ^\rho$ as before, for any $f\in \UC_{0,+}$, the law of the process
	\begin{equation}   \label{eq:fvari}
		y \mapsto \max_x f(x) + \cQ^\rho(x,s; y ,t) .
	\end{equation}
	matches the transition probabilities defined in \cite{Zhang}. This variational formulation follows since TASEP can equivalently be described as a process of metric balls of increasing radius in exponential LPP, which yields a discrete variational formula. We pass this discrete formula to the limit in Section \ref{ssec:kpzfp}.
	
	Now, the above variational characterization for the KPZ fixed point implies that for any two distinct limits $\cQ^\rho$, marginals of the form \eqref{eq:fvari} are equal in law.
	However, it is not clear that this implies uniqueness of the subsequential limit, as can be seen by considering discrete examples (e.g. consider different couplings of TASEP).

	We will prove (in Section \ref{sec:chara}) that if we have two processes $\cQ^\rho, \tilde \cQ^\rho$ such that both processes have half-space KPZ fixed point marginals, satisfy a triangle inequality, and have independent time increments and \textit{one} of these processes is stationary for the half-space stationary horizon, then $\cQ^\rho \eqd \tilde \cQ^\rho$. This will complete the proof of Theorem \ref{thm:expconv}. Moreover, since only one of the processes requires stationary horizon marginals, a posteriori we may state a characterization theorem without including this condition.
	\begin{theorem}   \label{thm:chawdl}
		Fix $\rho\in \R$. Consider a random function $\cM^\rho: \Q^4 \cap \mathbb{H}^2_\uparrow \to \R$ satisfying the following three conditions:
		\begin{enumerate}
			\item (Triangle inequality) For any $o = (x, s)$, $p = (y, r)$, $q = (z, t) \in \Q_{\ge 0} \times \Q$ with $s<r<t$, we have
			$\cM^\rho(o; p) + \cM^\rho(p; q) \le \cM^\rho(o; q)$.
			\item (Independent increments) For any $s_1<s_2<\cdots<s_k \in \Q$, the increments $\cM^\rho(\cdot, s_i; \cdot, s_{i+1}):\Q_{\ge 0}^2 \to \R, i \in 1, \ldots, k-1$ are independent.
			\item (KPZ fixed point marginals) For any $f: \R_{\ge 0} \to \R\cup \{-\infty\}$, with $f(\R)\subset\Q\cup\{-\infty\}$ and $f^{-1}(\R)$ being finite, $s<t \in \Q$, and finite $Y\subset \Q_{\ge 0}$, we have
			\[
			(\max_{x\in\Q_{\ge 0}} f(x) + \cM^\rho(x,s;y,t): y\in Y)
			\eqd ( \max_{x\in\R_{\ge 0}} f(x) + \cL^\rho(x,s;y,t): y\in Y).
			\]
		\end{enumerate}
		Then $\cM^\rho$ has a continuous extension to $\mathbb H^2_\uparrow$, which has the same law as $\cL^\rho$.
	\end{theorem}
	This is the half-space analog of the full-space directed landscape characterization, given in \cite{DZ24}. This characterization also allows us to prove convergence of Poisson-avoiding metrics, even though in that setting we do not currently have access to explicit joint stationary measures.	
	
	\begin{remark}[$\rho = -\infty$] The subcritical case when $\rho = -\infty$ is not included in Theorem \ref{thm:chawdl}. While the theorem can still be obtained in this setting, the proof would be quite different since the behavior of the stationary measures at the boundary is locally a Brownian-Bessel process, rather than locally Brownian. To avoid adding in another long proof, we instead have opted for a weaker and faster characterization (Corollary \ref{C:-infty-case}) that suffices for our main convergence theorems. 
		
	We note that the subcritical case should actually be the simplest to characterize, since in this setting the structure of stationary measures suggests that no interior geodesic point lies on the boundary. We have included a suggested approach to this in our open problems section.
	\end{remark}
	
	\subsection*{Organization of the remaining text}
	In Section \ref{sec:horizons} we construct the joint stationary measures, or \emph{horizons}, for various models in half-space.
	In Section \ref{sec:tightness} we prove one-point and two-point moderate deviation bounds for exponential LPP. In Section \ref{sec:fixed} we prove tightness of exponential LPP from these moderate deviation bounds and prove key properties of the subsequential limits: continuity, regularity, and KPZ fixed point marginals.
	Section \ref{sec:chara} finishes the convergence of exponential LPP given tightness (Theorem \ref{thm:expconv}) by establishing the characterization theorem (Theorem \ref{thm:chawdl}).
	In Section \ref{sec:TASEP} we show convergence of multi-level Poisson-avoiding metrics, by proving tightness and using the characterization theorem.
	Existence, uniqueness, and convergence of geodesics are proven in Section \ref{sec:convgeo}.
	Finally, we provide a list of open problems in Section \ref{sec:openproblems}, with a focus on the question of whether the boundary contributes any randomness in the limit.

	\subsection*{Notation and conventions}
	For $a<b\in \R\cup\{-\infty, \infty\}$, we write $\II{a, b}=[a,b]\cap \Z$.
	For $u=(u_1, u_2), v=(v_1, v_2)\in \Z^2$, we write $u \le v$ if $u_1 \le v_1$ and $u_2 \le v_2$. 
	We also denote $\Z^2_\ge=\{(i,j): i, j \in \Z, i\ge j\}$. For $x,y\in \R\cup\{-\infty, \infty\}$, we denote $x\vee y=\max\{x,y\}$ and $x\wedge y = \min\{x, y\}$.
	Unless otherwise noted, all Brownian motions and Brownian bridges have diffusivity $2$.
	For any $x>0$, the notation $O(x)$ denotes a quantity $y$ satisfying $|y|<Cx$, for some constant $C>0$. Unless otherwise noted, such constant $C$ is universal.
	
	\textbf{Distributions.} We let $\Gai(\theta)$ be the inverse gamma distribution with parameter $\theta>0$, with support $[0,\infty)$ and Lebesgue density $\Gamma(\theta)^{-1}x^{-\theta-1}e^{-1/x}$ for $x>0$.
	We let $\Geo(\theta)$ be the geometric distribution with parameter $\theta\in(0,1)$, which is supported on $\Z_{\ge 0}$, with the probability of $k$ being $(1-\theta)\theta^k$.
	We let $\Exp(\theta)$ be the exponential distribution with parameter $\theta>0$, with support $[0,\infty)$ and Lebesgue density $\theta\exp(-\theta x)$ for $x>0$. We also interpret $\Exp(\infty)=0$.

    \textbf{Spaces.} We use $\UC_+$ to denote the space of all upper semicontinuous functions $f:\R_{\ge 0} \to \R\cup\{-\infty\}$, with the topology of local hypograph convergence.
    We let $\UC_{0,+}$ denote the subspace of $\UC_+$, defined from \eqref{E:UC0+}.

	\subsection*{Acknowledgements}
	D.D.\ was supported by an NSERC Discovery Grant and a Sloan fellowship.
	L.Z.\ was supported by NSF grant DMS-2505625 and a Sloan fellowship. We thank Xincheng Zhang for many helpful discussions, and Ofer Busani, Sayan Das, Dominik Schmid, and Evan Sorensen for many valuable comments and suggestions on the first version of this paper.

\section{Horizons for half-space KPZ models}   \label{sec:horizons}

In this section, we identify the joint stationary measures for various half-space versions of models in the KPZ universality class: the log-gamma polymer, geometric and exponential LPP, and the KPZ equation. Following the terminology introduced in \cite{BuS} for the full suite of stationary measures of the full-space directed landscape, we call our constructions \emph{(half-space) horizons}.
In particular, for the log-gamma polymer, geometric and exponential LPP, and the KPZ equation, our half-space horizons couple the stationary measures found in \cite{BC}. 
We also construct their universal scaling limit, the \emph{half-space stationary horizon}, which will serve as the joint stationary measure for the half-space directed landscape to be constructed.

In the full-space setting, joint stationary measures were first constructed for multi-species TASEP in \cite{FPAMJB} using queueing theory, building on \cite{DJLS,FFK,Angel}. See \cite{MJb,ANP} for further developments on joint stationary measures in full-space interacting particle systems. In the setting of exponential last passage percolation, joint stationary measures were first constructed by \cite{FSj} by discovering an adaptation of the queueing-theoretic ideas of \cite{FPAMJB}. We refer the reader to \cite{GRS,SepCor,BuS,BSS24, dauvergne2024directed, bates2025permutation,GRSS,EFS,CGS} for results on other solvable last passage and polymer models, as well as to the directed landscape.

Our construction and proof of stationarity are presented in detail for the log-gamma polymer model. 
More precisely, we consider half-space log-gamma polymer models with general parameters.
A key step is a new joint invariance property under permutations of parameters (see Lemma \ref{lem:invp}; see also Remark \ref{rem:fsana} regarding its full-space analogues in \cite{bates2025permutation,engel2025discrete,D22hidden}). 
We repeatedly use this property to permute the parameters in order to obtain joint stationarity and consistency of our construction (see Lemma \ref{lem:sta} and Lemma \ref{lem:consis}, respectively).
The proofs for the other models follow by degeneration from the log-gamma polymer. 

For the constructions, it is also helpful to work with extensions of half-space models to full-space models with symmetrized weights. 
We begin by introducing the extensions that we will need.

\subsection{Extended half-space models}

  In this subsection, we introduce a particular log-gamma polymer model in a symmetrized field with inhomogeneous weight parameters, along with symmetrized exponential and geometric LPP models.

We consider the following model, which we call \emph{the extended half-space log-gamma polymer with general parameters}.

\begin{definition}
\label{D:gamma}
Let $X:\Z^2_\ge \to \R_{\ge 0}\cup \{\infty\}$ be a field of independent random variables distributed as follows.
Take real parameters $\alpha$ and $\{\apar_i\}_{i\in\Z}$.

For any $i\in \Z$, take $X(i,i)\sim \Gai(\alpha+\apar_i)$ if $\alpha+\apar_i>0$, and $X(i,i)=\infty$ otherwise.

For any $i > j \in \Z$, take $X(i,j)\sim \Gai(\apar_i+\apar_j)$ if $\apar_i+\apar_j>0$, and $X(i,j)=\infty$ otherwise. 
We extend the weights $X$ to all of $\Z^2$ by the symmetry condition $X(i,j)=X(j,i)$ for all $i,j$.

For any $u,v\in \Z^2$ with $u \le v$, define the \emph{partition function}
\[
Z_X(u;v)=\sum_{\pi}\prod_{(i,j)\in \pi}X(i,j),
\]
where the sum is over all up-right paths $\pi$ from $u$ to $v$, subject to the following condition:
\begin{equation}
\label{E:diagonal-condition}
\text{for any $(i,i)\in\pi$, necessarily $(i-1,i)\notin\pi$.}
\end{equation}
In words, paths may cross the diagonal at most once, moving from below the diagonal to above it.
\end{definition}

When both $u,v\in\Z^2_\ge$, we recover the usual half-space log-gamma polymer. 
However, we note that while the field $\{X(i,j)\}_{(i,j)\in\Z^2}$ is symmetric, the above log-gamma model is not; in general,
$Z_X(p,q;r,s)\neq Z_X(q,p;s,r)$.

We now turn to the zero-temperature models. 
Focusing on exactly solvable models, we define \emph{the extended half-space exponential/geometric last passage percolation (LPP) with general parameters}. 
\begin{definition}
\label{D:geo}
Let $X:\Z^2 \to \R_{\ge 0}\cup\{\infty\}$ be a field of random variables that are independent subject to the symmetry condition $X(i,j)=X(j,i)$ for all $i,j$, and distributed as follows.

\textbf{Exponential setting.}
Take real parameters $\alpha$ and $\{\apar_i\}_{i\in\Z}$. For any $i\in\Z$, take $X(i,i)\sim \Exp(\alpha+\apar_i)$ if $\alpha+\apar_i>0$, and $X(i,i)=\infty$ otherwise.  
For any $i\neq j\in\Z$, take $X(i,j)=X(j,i)\sim \Exp(\apar_i+\apar_j)$ if $\apar_i+\apar_j>0$, and $X(i,j)=X(j,i)=\infty$ otherwise.

\textbf{Geometric setting.}
Take positive parameters $\alpha$ and $\{\apar_i\}_{i\in\Z}$. For any $i\in\Z$, take $X(i,i)\sim \Geo(\alpha\apar_i)$ if $\alpha\apar_i\in(0,1)$, and $X(i,i)=\infty$ otherwise.  
For any $i\neq j\in\Z$, take $X(i,j)=X(j,i)\sim \Geo(\apar_i\apar_j)$ if $\apar_i\apar_j\in(0,1)$, and $X(i,j)=X(j,i)=\infty$ otherwise.

For any $u\le v$ in $\Z^2$, define the \emph{passage time}
\[
X(u;v)=\max_{\pi}\sum_{(i,j)\in\pi}X(i,j),
\]
where the maximum is taken over all up-right paths $\pi$ from $u$ to $v$ satisfying \eqref{E:diagonal-condition}.
\end{definition}

Note that, unlike in the log-gamma setting, the constraint \eqref{E:diagonal-condition} can technically be dropped. Indeed, without this constraint we would have the symmetry
\[
X(p,q;r,s)=X(q,p;s, r)
\]
for all $p,q,r,s$, so we may assume $(p,q)\in\Z^2_\ge$. 
In this case, if $(r,s)\in\Z^2_\ge$, the symmetry of the environment guarantees that a maximizing path for $X$ never crosses the diagonal. If $(r,s)\not\in\Z^2_\ge$, then a maximizing path crosses the diagonal exactly once from below to above, and therefore never violates the constraint \eqref{E:diagonal-condition}. 
We nevertheless keep the constraint \eqref{E:diagonal-condition} in the definition, since this formulation is convenient for constructing bi-infinite stationary measures.

\subsection{Invariance under parameter permutation}
	
The key to constructing joint stationary measures for the models above is a parameter-permutation lemma. This can be viewed as an extension of the better-known symmetry of geometric RSK in the half-space log-gamma environment under parameter permutations, which was used to great effect in \cite{BC} to construct stationary measures.

Take the environment $X$ from Definition \ref{D:gamma}. 
Let $\hat\apar_0=\apar_1$ and $\hat\apar_1=\apar_0$, and let $\hat\apar_i=\apar_i$ for all other $i\in\Z$.
Let $\hat{X}$ be defined the same way as $X$ in Definition \ref{D:gamma}, but with parameters $\alpha$ and $\{\hat\apar_i\}_{i\in\Z}$. 
Let $Z_{\hat{X}}(p,q;r,s)$ be the partition function under the environment $\hat{X}$.

\begin{lemma}\label{lem:invp}
We have
\[
\{Z_X(p,q;r,s)\}_{p\le r,\; q\le s,\; p,q\neq 1,\; r,s\neq 0}
\eqd
\{Z_{\hat{X}}(p,q;r,s)\}_{p\le r,\; q\le s,\; p,q\neq 1,\; r,s\neq 0}.
\]
\end{lemma}

\begin{remark}[Full-space analogues]\label{rem:fsana}
A full-space exponential LPP version of Lemma~\ref{lem:invp} was proven independently as the base case of \cite[Theorem~3.1]{bates2025permutation}, where it is used to give an alternative construction of the joint stationary measures for the full-space model. 
The corresponding versions in the settings of full-space log-gamma polymer and geometric LPP are given in \cite[Theorems 1.6, 1.10]{engel2025discrete}. 
Alternatively, the full-space analogues can be viewed as an extension of part of \cite[Theorems~1.5 and~1.10]{D22hidden}.
\end{remark}
	
To prove this, we require the following lemma.

\begin{lemma}\label{lem:coupr}
Let $I=\llbracket a,b\rrbracket$ for some $-\infty \le a < b \le \infty$. 
Take real parameters $\apar_0<\apar_1$ and $\{\beta_i\}_{i\in I}$, such that $\apar_0+\beta_i>0$ for each $i\in I$.

Let $\{A_i,B_i\}_{i\in I}$ be independent random variables with
$A_i\sim \Gai(\apar_0+\beta_i)$, $B_i\sim \Gai(\apar_1+\beta_i)$.
Similarly, let $\{C_i,D_i\}_{i\in I}$ be independent random variables with
$C_i\sim \Gai(\apar_1+\beta_i)$, $D_i\sim \Gai(\apar_0+\beta_i)$.

Then these two families of random variables can be coupled so that almost surely, for any $x\le y\in I$,
\begin{equation}\label{eq:coupr}
\sum_{j=x}^y \prod_{i=x}^j A_i \prod_{i=j}^y B_i
=
\sum_{j=x}^y \prod_{i=x}^j C_i \prod_{i=j}^y D_i .
\end{equation}
\end{lemma}

Identities of this type (in the LPP setting) are related to the classical stationary $M/M/1$ queue in queueing theory; see, e.g., \cite[Section~2.5]{timonotes}. 
In the log-gamma setting, the identity can also be extracted as a special case of \cite[Theorem~1.10]{D22hidden}. 
Since the proof in the present setting is considerably simpler than that theorem, we include it here for completeness.

	\begin{proof}[Proof of Lemma \ref{lem:coupr}]
By Kolmogorov's extension theorem (which allows us to extend to infinite intervals) and translation invariance it suffices to prove this when $a = 1, b = n$ for any $n \in \N$. Next, define a new environment $(\hat A_i, \hat B_i, i \in\II{1, n})$ by the following two rules for all $k \in \II{1, n}$:
$$
\prod_{i=1}^k \hat B_i = \sum_{j=1}^k \prod_{i=1}^j A_i \prod_{i=j}^k B_i, \qquad \prod_{i=1}^k \hat A_i\hat B_i = \prod_{i=1}^k A_i B_i.
$$
It is easy to see that these equations uniquely determine $\hat A_i, \hat B_i$. The output $(\hat A_i, \hat B_i, i \in\II{1, n})$ is one of the two `tableaux' coming out of the geometric RSK correspondence. We similarly define $(\hat C_i, \hat D_i, i\in\II{1, n})$ from $C_i, D_i$. 
Now,
	\begin{align}
		\nonumber
\sum_{j=1}^y \prod_{i=1}^j A_i \prod_{i=j}^y B_i &=	\prod_{i=1}^y \hat B_i, \quad &&1 \le y \le n, \\
\label{E:rsk-iso}
\sum_{j=x}^y \prod_{i=x}^j A_i \prod_{i=j}^y B_i &=	\sum_{j=x}^y \prod_{i=x}^j \hat A_i \prod_{i=j}^y \hat B_i, \quad &&2 \le x \le y \le n,
		\end{align}
		and similarly for $C_i, D_i, \hat C_i, \hat D_i$. This is the two-line positive temperature RSK isometry and can easily be checked by hand. In greater generality it is also shown in \cite{NY04}, \cite[Theorem 3.3]{D22hidden}, or \cite[Theorem 1.1]{corwininvariance}. Therefore to complete the proof of \eqref{eq:coupr} it suffices to observe that the joint law of $(\hat A_i, \hat B_i)$ is the same as the joint law of $(\hat C_i, \hat D_i)$. This amounts to an invariance of geometric RSK applied to a log-gamma environment under permutation of parameters, shown in \cite{COSZ}.
	\end{proof}

	\begin{proof}[Proof of Lemma \ref{lem:invp}]

		Without loss of generality, we assume that $\apar_0<\apar_1$. (In the case where $\apar_0=\apar_1$, this equality is trivial).
		
		We construct a coupling between $X$ and $\hat{X}$, as follows.
		For any $(i,j)\in (\Z\setminus\{0,1\})^2$, we let $\hat{X}(i,j)=X(i,j)$.
		We also let $\hat{X}(0,1)=\hat{X}(1,0)=X(0,1)=X(1,0)$.
		
		Now, let $A_i=X(0,i)$, $B_i=X(1,i)$ for $i<0$, and $A_0=X(0,0)$, $B_0=X(1,1)$, and $A_i=X(0,i+1)$, $B_i=X(1,i+1)$ for $i>0$.
	Let $C_i=\hat{X}(0,i)$, $D_i=\hat{X}(1,i)$ for $i<0$, and $C_0=\hat{X}(0,0)$, $D_0=\hat{X}(1,1)$, and $C_i=\hat{X}(0,i+1)$, $D_i=\hat{X}(1,i+1)$ for $i>0$.
	Define $\beta_i=\apar_i$ for $i<0$, $\beta_0=\alpha$, and $\beta_i=\apar_{i+1}$ for $i>0$.
		
	We claim that we can couple the random variables $A_i, B_i, C_i, D_i$ so that \eqref{eq:coupr} holds for all $x \le y \in \Z$. This will complete the construction of the coupling between $X, \hat{X}$. Define $J = \{i\in \Z: \apar_0+\beta_i>0\}$. First, for each $i\notin J$ we simply take $A_i=D_i = \infty$ and $B_i=C_i$. Next, write $J = \bigcup_{i \in F} I_i$, where each $I_i = \II{a_i, b_i}$ is a maximal discrete interval in $J$, i.e., $a_i - 1, b_i + 1 \notin J$ whenever $a_i, b_i$ is finite. By applying Lemma \ref{lem:coupr} separately on each interval $I_i$ and combining the couplings using the independence of the random variables $A_i, B_i$ we can guarantee that \eqref{eq:coupr} also holds whenever $x \le y$ are in a common interval $I_i$. Finally, if $x \le y$ are such that $\II{x, y} \not \subset J$, then \eqref{eq:coupr} is immediate since both sides are equal to $\infty$ almost surely.

Next, define the cross-shaped region $\mathtt{X} = (\{0,1\}\times \Z)\cup(\Z\times \{0,1\})$ and set $\partial^- \mathtt{X} = \{0\}\times (\Z\setminus\{1\})\cup (\Z\setminus\{1\})\times\{0\}$ and $\partial^+ \mathtt{X} = \{1\}\times(\Z\setminus\{0\})\cup (\Z\setminus\{0\})\times\{1\}$. The lower (resp. upper) boundary $\partial^- \mathtt{X}$ (resp.~$\partial^+ \mathtt{X}$) comprises all points that can be the entry point (resp.~exit point)  for an up-right path starting in $\mathtt{X}^c$. Define partition functions
		\[
		Z_{X;\tt{X}}(u; v)= \sum_{\pi}\prod_{(i,j)\in \pi} X(i,j), \qquad 
		Z_{\hat{X};\tt{X}}(u; v)= \sum_{\pi}\prod_{(i,j)\in \pi} \hat{X}(i,j),
		\]
		where the sum is over all up-right paths from $u$ to $v$, restricted to the cross-shaped set $\mathtt{X}$, and subject to the condition \eqref{E:diagonal-condition}.
		We claim that in our coupling,
		\begin{equation}
			\label{E:R-hat-R}
				Z_{X;\tt{X}}(u; v) = Z_{\hat{X};\tt{X}}(u; v), \qquad \text{ for all } \;\; u \in \partial^- \mathtt{X}, \;\;v \in \partial^+ \mathtt{X}, \quad u \le v.
		\end{equation}
		For this, we first show that 
		\begin{equation}
		\label{E:vertical-bar}
		Z_X(0,x;1,y) = Z_{\hat{X}}(0,x;1,y), \quad \text{ for all } \quad x\le y \in \Z, x \neq 1, y\neq 0.
		\end{equation}
		This is immediate from \eqref{eq:coupr} when either $x, y \le -1$ or else $x, y \ge 2$. When $x \le 0$ and $y \ge 1$, we have the computation
		\begin{align*}
		Z_X(0,x;1,y) &= \sum_{j=x}^0 \prod_{i=x}^j X(0, i) \prod_{i=j}^y X(1, i) + \sum_{j=2}^y \prod_{i=x}^j X(0, i) \prod_{i=j}^y X(1, i) \\
		&= X(1, 0) \sum_{j=x}^0 \prod_{i=x}^j A_i \prod_{i=j}^{y-1} B_i + X(0, 1) \sum_{j=2}^y \prod_{i=x}^{j-1} A_i \prod_{i=j-1}^{y-1} B_i \\
		&= X(1, 0) \sum_{j=x}^{y-1} \prod_{i=x}^j A_i \prod_{i=j}^{y-1} B_i.
	\end{align*}
Here the first equality uses condition \eqref{E:diagonal-condition}, the second equality is relabelling and the final equality uses that $X(1, 0) = X(0, 1)$. Executing a similar computation for $Z_{\hat{X}}$, using that $\hat{X}(1, 0) = \hat{X}(0, 1) = X(0, 1)$ and applying \eqref{eq:coupr} gives the result. 
	Similarly,
		\begin{equation}
		\label{E:horizontal-bar}
		Z_X(x,0;y,1) = Z_{\hat{X}}(x,0;y,1), \qquad \text{ for all } \quad x\le y \le -1 \in \Z, \qquad 2\le x\le y \in \Z.
		\end{equation}
		Now, \eqref{E:vertical-bar} and \eqref{E:horizontal-bar} readily imply that $Z_{X;\tt{X}}(0, x;1, y)=Z_{\hat{X};\tt{X}}(0, x;1,y)$ for any $x \ne 1, y \ne 0$ and that $Z_{X;\tt{X}}(x,0;y, 1)=Z_{\hat{X};\tt{X}}(x,0;y, 1)$, for any $y\le -1$ or $2 \le x$. Also note that $Z_{X;\tt{X}}(x,0;y,1)=Z_{\hat{X};\tt{X}}(x, 0;y,1)=0$, for any $x \le -1$ and $y\ge 1$. Next, if $x \le 0$ and $y \ge 1$ then we can write
		$$
		Z_{X;\tt{X}}(0, x; y, 1) = Z_X(0, x; 1, y) + Z_X(0, x; 1, -1) X(1, 0) Z_X(2, 0; y, 1).
		$$
		Here this decomposition follows by splitting the partition function for $Z_{X;\tt{X}}(0, x; y, 1)$ into paths that hit the diagonal $x = y$ and paths that do not. The paths that hit the diagonal can be mapped bijectively onto the set of paths satisfying \eqref{E:diagonal-condition} from $(0, x)$ to $(y, 1)$ by reflecting across the diagonal at the final hitting location, giving the contribution $Z_X(0, x; 1, y)$. The paths that do not hit the diagonal necessarily pass through the points $(1, -1), (1, 0)$ and $(2, 0)$, giving the second piece of the decomposition. A similar decomposition for $Z_{\hat{X};\tt{X}}(0, x; y, 1)$, together with  \eqref{E:vertical-bar}, \eqref{E:horizontal-bar} and the fact that $X(1, 0) = \hat X(1, 0)$ gives that $Z_{X;\tt{X}}(0, x; y, 1) = Z_{\hat{X};\tt{X}}(0, x; y, 1)$.
		
		 Finally, if $x \le 0$ and $y \ge 2$ then the diagonal condition \eqref{E:diagonal-condition} implies
		 $$
		 Z_{X;\tt{X}}(x, 0; 1, y) = Z_X(x, 0; -1, 1) X(0, 1) Z_X(0, 2; 1, y),
		 $$
		 and similarly for $Z_{\hat{X};\tt{X}}$, which again by \eqref{E:vertical-bar}, \eqref{E:horizontal-bar} and the fact that $X(0, 1) = \hat{X}(0, 1)$  similarly yields the identity $Z_{X;\tt{X}}(x, 0; 1, y)  = Z_{\hat{X};\tt{X}}(x, 0; 1, y)$. This completes the proof of \eqref{E:R-hat-R}.
		 
		 Now, consider any $p \le r$ and $q \le s$ with $p, q \ne 1$ and $r, s \ne 0$. By the metric composition law and \eqref{E:R-hat-R}, we have that 
		  \begin{equation}
		 	\label{E:Z-hat-Z}
		 	Z_X(p, q; r, s) = Z_{\hat{X}}(p, q; r, s),
		 \end{equation}
		 since both sides can be decomposed into sums of products of partitions functions that either stay outside of $\mathtt{X}$, or else entirely within $\mathtt{X}$. In more detail, for an up-right path $\pi$ from $(p, q)$ to $(r, s)$, we can uniquely decompose $\pi$ into pieces $\pi_1, \dots, \pi_k$ where each piece $\pi_i$ is either entirely contained in $\mathtt{X}$, or in $\mathtt{X}^c$, and $k$ is minimal. Let $\pi_i^+, \pi_i^-$ be the start and end vertices for $\pi_i$. We call the vertex sequence $\sigma(\pi) = (\pi_1^-, \pi_1^+; \dots, \pi_k^-, \pi_k^+)$ the signature of $\pi$. Then we can write
		 \begin{equation}
		 Z_X(p, q; r, s) = \sum_{\sigma = (\pi_1^-, \pi_1^+; \dots ;\pi_{k(\sigma)}^-, \pi_{k(\sigma)}^+)} \prod_{i=1}^{k(\sigma)} [Z_X(\pi_i^-; \pi_i^+) \mathds{1}(\pi_i^- \notin \mathtt{X}) + Z_{X;\tt{X}}(\pi_i^-; \pi_i^+) \mathds{1}(\pi_i^- \in \mathtt{X})],
		\end{equation}
		where the sum is over all potential path signatures. A similar decomposition holds for $Z_{\hat{X}}$. All terms above are invariant under switching $Z_X \mapsto Z_{\hat{X}}$ and $Z_{X;\tt{X}} \mapsto Z_{\hat{X};\tt{X}}$, yielding \eqref{E:Z-hat-Z}. For the $Z_X$ terms, this uses that $ Z_X(\pi_i^-; \pi_i^+) \mathds{1}(\pi_i^- \notin \mathtt{X})$ can only ever use weights in $\tt{X}^c$, and these weights are the same in $X$ and $\hat{X}$. For the $Z_{X;\tt{X}}$ terms, this uses \eqref{E:R-hat-R}, together with the fact that $\pi_i^- \in \partial^- \mathtt{X}$ and $\pi_i^+ \in \partial^+ \mathtt{X}$ for any signature. This uses that our restrictions on $p, q, r, s$ imply $(p, q) \in \mathtt{X}^c \cup \partial^- \mathtt{X}$ and $(r, s) \in \mathtt{X}^c \cup \partial^+ \mathtt{X}$.
	\end{proof}
	
	We also have the following LPP version.
	
	For exponential or geometric LPP as in Definition \ref{D:geo}, let $\hat{X}$ be defined in the same way as $X$, but with parameters $\alpha$ and $\{\hat\apar_i\}_{i\in\Z}$. 
    Here $\hat\apar_0=\apar_1$ and $\hat\apar_1=\apar_0$, and let $\hat\apar_i=\apar_i$ for all other $i\in\Z$.
    Define the passage times $\hat{X}(u;v)$ analogously to how $X(u;v)$ was defined in Definition \ref{D:geo}, with $\hat{X}$ in place of $X$.
	\begin{lemma}   \label{lem:invz}
		We have 
		$$\{X(p,q;r,s)\}_{p\le r; q\le s; p,q\neq 1; r,s \neq 0}\eqd\{\hat{X}(p,q;r,s)\}_{p\le r; q\le s; p,q\neq 1; r,s \neq 0}.
		$$
	\end{lemma}
	The proof is essentially identical to the proof of Lemma \ref{lem:invp}, with the $(+, \times)$-algebra replaced by the $(\max, +)$-algebra everywhere, using the following LPP version of  Lemma \ref{lem:coupr}.
    \begin{lemma}\label{lem:couprLPP}
Let $I=\llbracket a,b\rrbracket$ for some $-\infty \le a < b \le \infty$. Take  $\{A_i,B_i\}_{i\in I}$ and $\{C_i,D_i\}_{i\in I}$ to be two families of independent random variables, such that either
\begin{itemize}
    \item[(1)] $A_i\sim \Exp(\apar_0+\beta_i)$, $B_i\sim \Exp(\apar_1+\beta_i)$, $C_i\sim \Exp(\apar_1+\beta_i)$, $D_i\sim \Exp(\apar_0+\beta_i)$, where $\apar_0<\apar_1$ and $\{\beta_i\}_{i\in I}$ are real parameters, such that $\apar_0+\beta_i>0$ for each $i\in I$; or
    \item[(2)] $A_i\sim \Geo(\apar_0\beta_i)$, $B_i\sim \Geo(\apar_1\beta_i)$, $C_i\sim \Geo(\apar_1\beta_i)$, $D_i\sim \Geo(\apar_0\beta_i)$, where $\apar_0>\apar_1$ and $\{\beta_i\}_{i\in I}$ are positive parameters, such that $\apar_0\beta_i\in (0,1)$ for each $i\in I$.
\end{itemize}
Then these two families of random variables can be coupled so that almost surely, for any $x\le y\in I$,
\[
\max_{j\in\II{x,y}} \sum_{i=x}^j A_i + \sum_{i=j}^y B_i
=
\max_{j\in\II{x,y}} \sum_{i=x}^j C_i + \sum_{i=j}^y D_i .
\]
\end{lemma}
This uses the zero-temperature version of the two-line RSK isometry for a single path \eqref{E:rsk-iso}, which was first shown in \cite{NY04, biane2005littelmann} before being rediscovered in \cite{DOV}. It also follows from formally changing the algebra in \eqref{E:rsk-iso}. The invariance of the output of RSK under parameter permutation when applied to an array of geometric or exponential random variables is classical and was first observed in \cite{johansson2003discrete}. 
We omit the details of the proofs of both lemmas.
	
	\subsection{Half-space log-gamma horizon}
	\label{S:half-space-gamma-horizon}
	Next, we construct the joint stationary measures for the homogeneous extended half-space log-gamma polymer.
	This couples the stationary measures constructed in \cite{BC}.  We expect that our construction identifies the full set of joint stationary measures for the extended half-space log-gamma polymer, but we do not pursue this direction here. 
    Throughout this section, we fix a bulk parameter $\theta > 0$ and a boundary parameter $\alpha \in (-\theta, \infty)$.

Setting notation, for each $t, j\in \Z$, define
	\begin{equation}   \label{eq:Lshapedef}
	[j,t]_\textsf{L} = 
	\begin{cases}
		(t+j, t),\quad j \ge 0 ; \\
		(t, t-j),\quad j <0,
	\end{cases}    
	\end{equation}
	so that $[\cdot,t]_\textsf{L}$ parametrizes a bi-infinite $L$-shaped region with its corner at $(t, t)$.
    Next, for a function $f:\Z \to \R_+$ and an environment $E:\Z^2 \to \R_{\ge 0}\cup\{\infty\}$, define
\[
	E^{f}(u) = \begin{cases}
	f(j)/f(j - \operatorname{sgn}(j)), \qquad &u = [j,0]_\textsf{L}\; \text{ for some } j \in \Z	 \\
	E(u), \qquad &\text{else}.
\end{cases}
\]
	Here $\operatorname{sgn}(0) = 0$, so that $E^f(0,0) = 1$. Let $Z_E^{f}$ denote partition functions for the environment $E^{f}$. 
    
    Now, we say that a family of random functions  $(F_i:\Z \to \R_+, i \in I)$ for a finite index set $I$ is \emph{jointly stationary} for the extended half-space log-gamma polymer with boundary parameter $\alpha$ and equal bulk parameters $\apar_i = \theta, i \in \Z$, if for $X$ being the field as given in Definition \ref{D:gamma} with such parameters $\alpha$ and $\{\gamma_i\}_{i\in\Z}$, and any $h \in \N$, we have:
    \begin{equation}
    \label{E:stationary-defn}
   \frac{F_i(j)}{F_i(0)} \eqd \frac{Z^{F_i}_X(0,0; [j, h]_\textsf{L})}{Z^{F_i}_X(0,0; [0, h]_\textsf{L})},
    \end{equation}
where the equality in law is joint in all $i \in I, j \in \Z$. 

We next define our jointly stationary families. Fix any parameters $\alpha$ and $\apar_1, \dots, \apar_k$ with 
$-\theta<\apar_1< \apar_2< \cdots < \apar_k\le 0\wedge \alpha$.
Consider the model in Definition \ref{D:gamma}, with given values of $\alpha, \apar_i, i\in\II{1, k}$ and $\apar_i=\theta$ for each $i\ge 2k+1$ or $i\le 0$. We fix a small $\eps > 0$ and set $\apar_{2k+1-i}=\eps-\apar_i$ for each $i\in\II{1, k}$. 
	For each $i\in\II{1, k}$ and $j\in \Z$, we let
	\begin{equation}
		\label{E:Resp}
	R^{\LG,\epsilon}_i(j) = \frac{Z_X(2k+1-i, i; [j,2k]_\textsf{L} )}{Z_X(2k+1-i, i; 2k, 2k)}.
	\end{equation}
	\begin{prop}
		\label{P:stationary-existence}
	The collection of functions $R^{\LG,\epsilon} = \{R_i^{\LG,\epsilon}\}_{i=1}^k$ has a distributional limit $R^{\LG} = \{R_i^{\LG}\}_{i=1}^k$ as $\epsilon \to 0$. 
	\end{prop}
	
	\begin{proof}
	First let $i < k$. As $\epsilon \to 0$, all random variables $X(v)$ in the environment that contribute to the partition functions defining $R^{\LG,\epsilon}_i$ stay finite (and change continuously in their parameters) except for the variable $X(2k + 1 -i, i)$, which cancels out in the ratio \eqref{E:Resp}. When $i = k$, the argument is the same except in the case $\apar_k = \alpha$. In this case, since both the terms $X(k + 1, k)$ and $X(k+1, k+1)$ blow up, the fraction of the weight in the partition function in \eqref{E:Resp} using $X(k+1, k+1)$ goes to $1$ as $\eps \to 0$, and so
	$$
	\lim_{\eps \to 0} \frac{Z_X(k+1, k; [j,2k]_\textsf{L})}{Z_X(k+1, k; 2k, 2k)} = 	\lim_{\eps \to 0} \frac{Z_X(k+1, k + 1; [j,2k]_\textsf{L})}{Z_X(k+1, k + 1; 2k, 2k)}.
	$$
On the right-hand side, the blow-up term $X(k+1, k+1)$ now cancels out as before.
\end{proof}

	The process $R^{\LG}$ is jointly stationary in the sense of \eqref{E:stationary-defn}. We first prove this when $h = 1$, where we can reframe this stationarity as follows. For each $i\in\II{1, k}$ and $j\in \Z$, let
	\begin{equation}
		\label{E:Resp+ell}
	R^{\LG, \eps, +}_i(j) = \frac{Z_X(2k+1-i, i; [j,2k+1]_\textsf{L})}{Z_X(2k+1-i, i; 2k+1, 2k+1)}.
	\end{equation}
	Exactly as in Proposition \ref{P:stationary-existence}, the functions $R^{\LG,\eps,+} = \{R^{\LG,\eps,+}_i\}_{i=1}^k$ have a joint distributional limit as $\epsilon \to 0$, which we call $R^{\LG,+} = \{R^{\LG,+}_i\}_{i=1}^k$. 

    \begin{lemma}  \label{lem:sta} We have $R^{\LG} \eqd R^{\LG,+}$.
	\end{lemma}
	\begin{proof}
    Our overall strategy is to use Lemma \ref{lem:invp} to permute the parameters repeatedly.
    
		For $\ell\in\II{1, 2k+1}$ let $\sigma_\ell \in S_{2k + 1}$ be the permutation $(1, 2, \dots, \ell - 1, 2k + 1, \ell, \dots, 2k)$. Note that $\sigma_\ell^{-1} = (1, 2, \dots, \ell - 1, \ell + 1, \dots, 2k + 1, \ell)$, and extend $\sigma_\ell$ to a permutation of $\Z$ by setting it equal to the identity off of $\II{1, 2k+1}$.
		 Letting $\{\apar_i\}_{i\in\Z}$ be the parameters used in the definition \eqref{E:Resp}, let $X^{(\ell)}, Z_{X^{(\ell)}}$ be environments and partition functions defined with the parameter sequence $\{\apar_{\sigma(i)}\}_{i\in\Z}$. Note the implicit dependence on $\eps$. Finally, for $\ell, \tilde \ell \in \II{1, 2k + 1}$ set
		\[
			R^{\eps, \ell, \tilde \ell}_i(j) = \frac{Z_{X^{(\ell)}}(\sigma_{\tilde \ell}^{-1}(2k+1-i), \sigma_{\tilde \ell}^{-1}(i); [j,2k+1]_\textsf{L})}{Z_{X^{(\ell)}}(\sigma_{\tilde \ell}^{-1}(2k+1-i), \sigma_{\tilde \ell}^{-1}(i); 2k+1, 2k+1)}.
		\]
		Observe that $R^{\eps, 2k + 1, 2k+1} = R^{\LG,\eps, +}$ and $R^{\eps, 1, 1} \eqd R^{\LG,\eps}$. Therefore to prove the lemma it suffices to show that for any $\ell\in\II{1,2k}$, we have
		\begin{equation}
\label{E:eps-ell-sigma}
\lim_{\eps \to 0} R^{\eps, \ell, \ell} \eqd \lim_{\eps \to 0}	R^{\eps, \ell + 1, \ell + 1}. 
		\end{equation}
First fix $\eps > 0$. By Lemma \ref{lem:invp}, we can couple the environments $X^{(\ell)}, X^{(\ell + 1)}$ so that
$$
Z_{X^{(\ell)}}(p, q; r, s) = Z_{X^{(\ell+1)}}(p, q; r, s)
$$ 
whenever $p, q \ne \ell + 1$ and $r, s \ne \ell$. In this coupling we have $R^{\eps, \ell, \ell+1}_i = R^{\eps, \ell + 1, \ell+1}_i$ for all $i \in \II{1, k}$ and $j \in \Z$, since $\sigma_{\ell+1}^{-1}(2k + 1 - i), \sigma_{\ell +1}(i) \ne \ell + 1$ for any such $i$. To complete the proof of \eqref{E:eps-ell-sigma} it suffices to show that 
$\lim_{\ep \to 0} R^{\eps, \ell, \ell}_i \eqd \lim_{\ep \to 0} R^{\eps, \ell, \ell +1}_i$ jointly in all $i$. By examining the definition of $\sigma_\ell^{-1}$, this equality holds almost surely before sending $\eps \to 0$ other than for the unique index $i = \min(\ell, 2k + 1- \ell)$. At this point, to ease notation we will just check the case when $i = \ell$ as the case when $i = 2k + 1- \ell$ is similar.  
It is enough to verify that for any fixed $j$, in probability we have
\begin{equation}
	\label{E:limeps}
	\lim_{\ep \to 0} \frac{X^{(\ell)}(2k + 2 - \ell, \ell) Z_{X^{(\ell)}}(\sigma_{\ell}^{-1}(2k+1-\ell), \sigma_{\ell}^{-1}(\ell); [j,2k+1]_\textsf{L})}{Z_{X^{(\ell)}}(\sigma_{\ell +1}^{-1}(2k+1-\ell), \sigma_{\ell+1}^{-1}(\ell); [j,2k+1]_\textsf{L})} = 1.
\end{equation}
We have $(\sigma_{\ell}^{-1}(2k+1-\ell), \sigma_{\ell}^{-1}(\ell)) = (2k+2-\ell, \ell+1)$ whereas $(\sigma_{\ell+1}^{-1}(2k+1-\ell), \sigma_{\ell+1}^{-1}(\ell)) = (2k+2-\ell, \ell)$.
Moreover, as $\eps \to 0$, the number $X^{(\ell)}(2k+2-\ell, \ell+1)$ blows up to $\infty$, whereas none of the other cells contributing to the partition functions in \eqref{E:limeps} blow up. This yields \eqref{E:limeps}.  
	\end{proof}

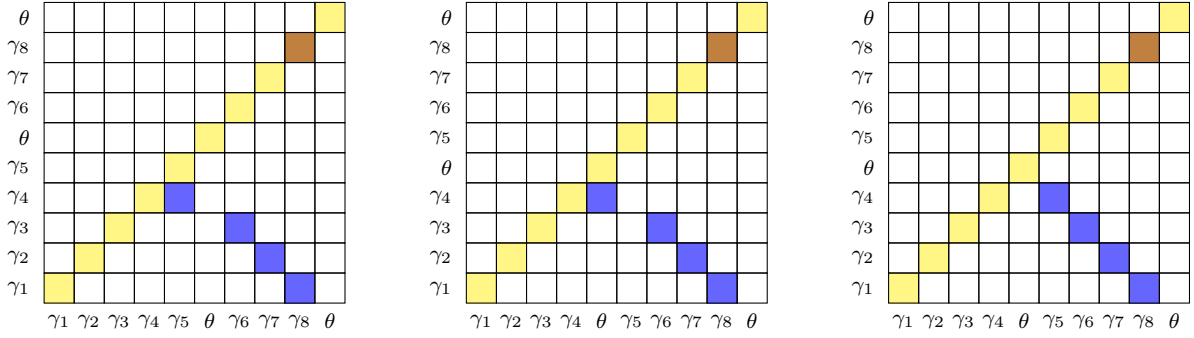
\begin{figure}[t]
    \centering 

\begin{subfigure}{0.32\textwidth}
\centering
\begin{tikzpicture}[scale=0.4]
\def\s{1}

\fill[blue!60] (9*\s - 1*\s, -9*\s) rectangle ++(\s,\s);
\fill[blue!60] (9*\s - 2*\s, -8*\s) rectangle ++(\s,\s);
\fill[blue!60] (9*\s - 3*\s, -7*\s) rectangle ++(\s,\s);
\fill[blue!60] (9*\s - 5*\s, -6*\s) rectangle ++(\s,\s);

\foreach \col in {0,...,9} {
    \fill[yellow!60] (9*\s - \col*\s, -\col*\s) rectangle ++(\s,\s);
}

\fill[brown] (8*\s, -1*\s) rectangle ++(\s,\s);

\foreach \row in {0,...,9} {
    \foreach \col in {0,...,9} {
        \draw (9*\s - \col*\s, -\row*\s) rectangle ++(\s,\s);
    }
}

\begin{scriptsize}
\node at (9*\s - 0*\s + 0.5*\s, -9*\s - 0.6*\s) {$\theta$};
\node at (9*\s - 1*\s + 0.5*\s, -9*\s - 0.6*\s) {$\gamma_{8}$};
\node at (9*\s - 2*\s + 0.5*\s, -9*\s - 0.6*\s) {$\gamma_{7}$};
\node at (9*\s - 3*\s + 0.5*\s, -9*\s - 0.6*\s) {$\gamma_{6}$};
\node at (9*\s - 4*\s + 0.5*\s, -9*\s - 0.6*\s) {$\theta$};
\node at (9*\s - 5*\s + 0.5*\s, -9*\s - 0.6*\s) {$\gamma_{5}$};
\node at (9*\s - 6*\s + 0.5*\s, -9*\s - 0.6*\s) {$\gamma_{4}$};
\node at (9*\s - 7*\s + 0.5*\s, -9*\s - 0.6*\s) {$\gamma_{3}$};
\node at (9*\s - 8*\s + 0.5*\s, -9*\s - 0.6*\s) {$\gamma_{2}$};
\node at (9*\s - 9*\s + 0.5*\s, -9*\s - 0.6*\s) {$\gamma_{1}$};

\node[left] at (-0.2*\s, -9*\s + 0.5*\s) {$\gamma_{1}$};
\node[left] at (-0.2*\s, -8*\s + 0.5*\s) {$\gamma_{2}$};
\node[left] at (-0.2*\s, -7*\s + 0.5*\s) {$\gamma_{3}$};
\node[left] at (-0.2*\s, -6*\s + 0.5*\s) {$\gamma_{4}$};
\node[left] at (-0.2*\s, -5*\s + 0.5*\s) {$\gamma_{5}$};
\node[left] at (-0.2*\s, -4*\s + 0.5*\s) {$\theta$};
\node[left] at (-0.2*\s, -3*\s + 0.5*\s) {$\gamma_{6}$};
\node[left] at (-0.2*\s, -2*\s + 0.5*\s) {$\gamma_{7}$};
\node[left] at (-0.2*\s, -1*\s + 0.5*\s) {$\gamma_{8}$};
\node[left] at (-0.2*\s, 0*\s + 0.5*\s) {$\theta$};
\end{scriptsize}

\end{tikzpicture}
\end{subfigure}
\hfill
\begin{subfigure}{0.32\textwidth}
\centering
\begin{tikzpicture}[scale=0.4]
\def\s{1}

\fill[blue!60] (9*\s - 1*\s, -9*\s) rectangle ++(\s,\s);
\fill[blue!60] (9*\s - 2*\s, -8*\s) rectangle ++(\s,\s);
\fill[blue!60] (9*\s - 3*\s, -7*\s) rectangle ++(\s,\s);
\fill[blue!60] (9*\s - 5*\s, -6*\s) rectangle ++(\s,\s);

\foreach \col in {0,...,9} {
    \fill[yellow!60] (9*\s - \col*\s, -\col*\s) rectangle ++(\s,\s);
}

\fill[brown] (8*\s, -1*\s) rectangle ++(\s,\s);

\foreach \row in {0,...,9} {
    \foreach \col in {0,...,9} {
        \draw (9*\s - \col*\s, -\row*\s) rectangle ++(\s,\s);
    }
}

\begin{scriptsize}
\node at (9*\s - 0*\s + 0.5*\s, -9*\s - 0.6*\s) {$\theta$};
\node at (9*\s - 1*\s + 0.5*\s, -9*\s - 0.6*\s) {$\gamma_{8}$};
\node at (9*\s - 2*\s + 0.5*\s, -9*\s - 0.6*\s) {$\gamma_{7}$};
\node at (9*\s - 3*\s + 0.5*\s, -9*\s - 0.6*\s) {$\gamma_{6}$};
\node at (9*\s - 4*\s + 0.5*\s, -9*\s - 0.6*\s) {$\gamma_{5}$};
\node at (9*\s - 5*\s + 0.5*\s, -9*\s - 0.6*\s) {$\theta$};
\node at (9*\s - 6*\s + 0.5*\s, -9*\s - 0.6*\s) {$\gamma_{4}$};
\node at (9*\s - 7*\s + 0.5*\s, -9*\s - 0.6*\s) {$\gamma_{3}$};
\node at (9*\s - 8*\s + 0.5*\s, -9*\s - 0.6*\s) {$\gamma_{2}$};
\node at (9*\s - 9*\s + 0.5*\s, -9*\s - 0.6*\s) {$\gamma_{1}$};

\node[left] at (-0.2*\s, -9*\s + 0.5*\s) {$\gamma_{1}$};
\node[left] at (-0.2*\s, -8*\s + 0.5*\s) {$\gamma_{2}$};
\node[left] at (-0.2*\s, -7*\s + 0.5*\s) {$\gamma_{3}$};
\node[left] at (-0.2*\s, -6*\s + 0.5*\s) {$\gamma_{4}$};
\node[left] at (-0.2*\s, -5*\s + 0.5*\s) {$\theta$};
\node[left] at (-0.2*\s, -4*\s + 0.5*\s) {$\gamma_{5}$};
\node[left] at (-0.2*\s, -3*\s + 0.5*\s) {$\gamma_{6}$};
\node[left] at (-0.2*\s, -2*\s + 0.5*\s) {$\gamma_{7}$};
\node[left] at (-0.2*\s, -1*\s + 0.5*\s) {$\gamma_{8}$};
\node[left] at (-0.2*\s, 0*\s + 0.5*\s) {$\theta$};
\end{scriptsize}

\end{tikzpicture}
\end{subfigure}
\hfill
\begin{subfigure}{0.32\textwidth}
\centering

\begin{tikzpicture}[scale=0.4]
\def\s{1}

\fill[blue!60] (9*\s - 1*\s, -9*\s) rectangle ++(\s,\s);
\fill[blue!60] (9*\s - 2*\s, -8*\s) rectangle ++(\s,\s);
\fill[blue!60] (9*\s - 3*\s, -7*\s) rectangle ++(\s,\s);
\fill[blue!60] (9*\s - 4*\s, -6*\s) rectangle ++(\s,\s);

\foreach \col in {0,...,9} {
    \fill[yellow!60] (9*\s - \col*\s, -\col*\s) rectangle ++(\s,\s);
}

\fill[brown] (8*\s, -1*\s) rectangle ++(\s,\s);

\foreach \row in {0,...,9} {
    \foreach \col in {0,...,9} {
        \draw (9*\s - \col*\s, -\row*\s) rectangle ++(\s,\s);
    }
}

\begin{scriptsize}
\node at (9*\s - 0*\s + 0.5*\s, -9*\s - 0.6*\s) {$\theta$};
\node at (9*\s - 1*\s + 0.5*\s, -9*\s - 0.6*\s) {$\gamma_{8}$};
\node at (9*\s - 2*\s + 0.5*\s, -9*\s - 0.6*\s) {$\gamma_{7}$};
\node at (9*\s - 3*\s + 0.5*\s, -9*\s - 0.6*\s) {$\gamma_{6}$};
\node at (9*\s - 4*\s + 0.5*\s, -9*\s - 0.6*\s) {$\gamma_{5}$};
\node at (9*\s - 5*\s + 0.5*\s, -9*\s - 0.6*\s) {$\theta$};
\node at (9*\s - 6*\s + 0.5*\s, -9*\s - 0.6*\s) {$\gamma_{4}$};
\node at (9*\s - 7*\s + 0.5*\s, -9*\s - 0.6*\s) {$\gamma_{3}$};
\node at (9*\s - 8*\s + 0.5*\s, -9*\s - 0.6*\s) {$\gamma_{2}$};
\node at (9*\s - 9*\s + 0.5*\s, -9*\s - 0.6*\s) {$\gamma_{1}$};

\node[left] at (-0.2*\s, -9*\s + 0.5*\s) {$\gamma_{1}$};
\node[left] at (-0.2*\s, -8*\s + 0.5*\s) {$\gamma_{2}$};
\node[left] at (-0.2*\s, -7*\s + 0.5*\s) {$\gamma_{3}$};
\node[left] at (-0.2*\s, -6*\s + 0.5*\s) {$\gamma_{4}$};
\node[left] at (-0.2*\s, -5*\s + 0.5*\s) {$\theta$};
\node[left] at (-0.2*\s, -4*\s + 0.5*\s) {$\gamma_{5}$};
\node[left] at (-0.2*\s, -3*\s + 0.5*\s) {$\gamma_{6}$};
\node[left] at (-0.2*\s, -2*\s + 0.5*\s) {$\gamma_{7}$};
\node[left] at (-0.2*\s, -1*\s + 0.5*\s) {$\gamma_{8}$};
\node[left] at (-0.2*\s, 0*\s + 0.5*\s) {$\theta$};
\end{scriptsize}

\end{tikzpicture}
\end{subfigure}
    
    \caption{Illustrations of $R^{\eps,6,6}$, $R^{\eps,5,6}$, $R^{\eps,5,5}$ (from left to right), with $k=4$. The brown box denotes $(2k+1, 2k+1)=(9,9)$, the yellow boxes denote the rest of the diagonal, while the blue boxes denote all the $(\sigma_{\tilde \ell}^{-1}(2k+1-i), \sigma_{\tilde \ell}^{-1}(i))$. Note that here $\gamma_1+\gamma_8=\gamma_2+\gamma_7=\gamma_3+\gamma_6=\gamma_4+\gamma_5=\eps$.}
    \label{fig:proofsta}
\end{figure}

	\begin{corollary}  \label{cor:sta} The process $R^{\LG}$ is jointly stationary, in the sense of \eqref{E:stationary-defn}.
	\end{corollary}
	
	\begin{proof}
	This is a straightforward induction on $h$, with the base case being Lemma \ref{lem:sta}. The inductive step then follows by applying a metric composition law on the line $j \mapsto [j, 2k + h]_\textsf{L}$.
	\end{proof}

    By Corollary \ref{cor:sta}, we call $R^{\LG}$ the \emph{(extended) half-space log-gamma joint stationary process with boundary parameter $\alpha$, bulk parameter $\theta$, and slope parameters $(\apar_1, \ldots, \apar_k)$}.
	Note that when $k=1$, the process $R^{\LG}_1$ on $\N$ precisely recovers the stationary measure in \cite[Section 1.2.2]{BC} (the parameters $\apar, u, v$ there correspond to $\theta, \alpha, \apar_1$, respectively).
	
	Next, we show that our construction above is consistent. 
    
    Let $R^{\LG}$ be as above, and fix $m \in \II{1,k}$. Let $\tilde{R}^{\LG} = \{\tilde{R}^{\LG}_i\}_{i=1}^{k-1}$ be the half-space log-gamma joint stationary process with the same boundary and bulk parameters $\alpha$ and $\theta$, and slope parameters $(\tilde\apar_1,\ldots, \tilde\apar_{k-1})$, where $\tilde\apar_i=\apar_i$ for $1\le i<m$, and $\tilde\apar_i=\apar_{i+1}$ for $m\le i <k$. 
    
	\begin{lemma}  \label{lem:consis}
	We have $\{R^{\LG}_i\}_{i\in\II{1,k}, i \ne m} \eqd \tilde R^{\LG}$.
	\end{lemma}
	
	\begin{proof}
	The proof is similar to the proof of Lemma \ref{lem:sta}. We first define permutations $\sigma_\ell \in S_{2k + 1 - m}$ for $\ell\in\II{1, 2k}$. First, for $\ell \in\II{m, 2k}$ set $\ell^* = \ell + 1- m$ and define
	$$
	\sigma_\ell = (1, \dots, \ell^* - 1, 2k + 1 - m, \ell^*, \dots, 2k - m), \qquad \sigma_\ell^{-1} = (1, \dots, \ell^* - 1, \ell^* + 1, \dots, 2k + 1 - m, \ell^*).
	$$
	Then, for $\ell \in\II{1, m-1}$ define
	\begin{align*}
	\sigma_\ell &= (2k + 1 - m, 1, \dots, \ell-1, m, \ell, \dots, m -1, m + 1, \dots, 2k - m), \qquad \\
	\sigma_\ell^{-1} &= (2, \dots, \ell, \ell + 2, \dots, m + 1, \ell + 1, m + 2, \dots, 2k + 1 - m, 1).
	\end{align*}
We extend $\sigma_\ell$ to a permutation of $\Z$ by setting it equal to the identity off of $\II{1, 2k+1-m}$.
 As in the proof of Lemma \ref{lem:sta}, we define $Z_{X^{(\ell)}}$ to be partition functions defined with the parameter sequence $\{\apar_{\sigma(i)}\}_{i\in\Z}$. For $\ell, \tilde \ell \in \II{1, 2k + 1 -m}$ set
		\[
			R^{\eps, \ell, \tilde \ell}_i(j) = \frac{Z_{X^{(\ell)}}(\sigma_{\tilde \ell}^{-1}(2k+1-i), \sigma_{\tilde \ell}^{-1}(i); [j,2k]_\textsf{L})}{Z_{X^{(\ell)}}(\sigma_{\tilde \ell}^{-1}(2k+1-i), \sigma_{\tilde \ell}^{-1}(i); 2k, 2k)}.
		\]
 Similarly to before, we have that $\{R^{\eps, 2k, 2k}_i\}_{i\in\II{1, k}, i \ne m} = \{R^{\LG,\eps}_i\}_{i\in\II{1, k}, i \ne m}$ and $\{R^{\eps, 1, 1}_i\}_{i\in\II{1, k}, i \ne m} \eqd \tilde R^{\LG,\eps}$. Therefore to prove the lemma, again we need to prove that for each $\ell\in\II{1,2k-1}$,
\[
\lim_{\eps \to 0} \{R^{\eps, \ell, \ell}_i\}_{i\in\II{1, k}, i \ne m} \eqd \lim_{\eps \to 0}\{R^{\eps, \ell+1, \ell+1}_i\}_{i\in\II{1, k}, i \ne m}.
\]
The rest of this proof proceeds essentially as in the proof of Lemma \ref{lem:sta}. We omit the details.
	\end{proof}

\begin{remark}
		\label{R:special-case-gamma*}
	In the special case when $\alpha = \apar_k \le 0$, we can alternately define $R^{\LG,\epsilon}$, $R^{\LG}$ using only $2k - 1$ lines. Indeed, as before fix parameters $\alpha$ and $\theta, \apar_1, \dots, \apar_k$ with 
	$-\theta<\apar_1< \apar_2< \cdots < \apar_k = \alpha\le 0$.
    We take a small $\eps > 0$ and set $\bar\apar_i=\apar_i$ for each $i\in\II{1, k-1}$,  $\bar\apar_{2k-i}=\eps-\apar_i$ for each $i\in\II{1, k}$, and $\bar\apar_i=\theta$ for any other $i$.
    Let $Z_{\bar{X}}$ be the partition function defined with parameters $\alpha$ and $\{\bar\apar_i\}_{i\in\Z}$.
    For each $i\in\II{1, k}$ and $j\in \Z$, we let
	\[
		Q^{\LG,\epsilon}_i(j) = \frac{Z_{\bar{X}}(2k-i, i; [j,2k-1]_\textsf{L} )}{Z_{\bar{X}}(2k-i, i; 2k-1, 2k-1)},
	\]
	and let $Q^{\LG}$ be the distributional limit of $Q^{\LG,\epsilon}$ as $\eps\to 0$. This limit exists by following the proof of Proposition \ref{P:stationary-existence}, except we do not need to take extra care for the case $i = k$. Moreover, we can use the strategy of Lemma \ref{lem:consis} to check that $Q^{\LG} \eqd R^{\LG}$, when the two models have the same parameters. While the $Q^{\LG}$-description is simpler than the $R^{\LG}$-description for our stationary law (and still exists for the degeneration in the upcoming sections), in the sequel we focus on the $R^{\LG}$-description as it is more general.
	\end{remark}

    Now, given the consistency in Lemma \ref{lem:consis}, we can apply Kolmogorov's extension theorem to construct a family of stationary measures indexed by any countable collection of slope parameters. Further using right-continuity, we can get a family of all slope parameters $\apar \in (-\theta, \alpha \wedge 0]$. To do so, we require some mild discussion of topology. The simplest solution is to view our functions as measures. 
    More precisely, we set up the next definition, which can also be applied to the other horizons. 

\begin{definition}
\label{D:horizon-topology}
Let $L$ be an interval in $\R$ and let $J \in \{\R, \Z\}$, and fix a distinguished point $\gamma_0 \in L$. Define $\mathcal{D}(L, J, \gamma_0)$ to be the set of functions $F:L \times J \to \R$, such that (denoting $F_\gamma(j)=F(\gamma,j)$)
\begin{itemize}[nosep]
    \item $F_{\gamma_0}$ is continuous if $J = \R$;
    \item $F_\gamma(0) = 0$ for all $\gamma$;
    \item there is a $\sigma$-finite measure $\mu_F$ on $L \times J$ such that for any $\gamma_1<\gamma_2\in L$ and $j_1<j_2\in J$, we have
    \begin{equation}
        \label{E:F-measure}
        \mu_F((\gamma_1, \gamma_2] \times ((j_1, j_2]\cap J)) = F(\gamma_1, j_1) + F(\gamma_2, j_2) - F(\gamma_1, j_2) - F(\gamma_2, j_1). 
    \end{equation}
\end{itemize}
We say that $F^n \to F$ in $\mathcal{D}(L, J, \gamma_0)$ if $F^n_{\gamma_0} \to F_{\gamma_0}$ in the uniform-on-compact topology, and $\mu_{F^n} \to \mu_F$ vaguely. This turns $\mathcal{D}(L, J, \gamma_0)$ into a Polish space.
\end{definition}

\begin{theorem}
\label{T:half-space-log-gamma-horizon}
Let $\theta > 0$ and $\alpha \in (-\theta, \infty)$. Then there exists a random function $\cH^{\LG}$, such that $-\log \cH^{\LG}\in \mathcal{D}((-\theta, \alpha \wedge 0], \Z, \alpha \wedge 0)$, and for any finite set $I = (\apar_1 < \cdots < \apar_k) \subset (-\theta, \alpha \wedge 0]$, the restriction $\cH^{\LG}|_I$ is equal in law to the half-space log-gamma joint stationary measure given in Proposition \ref{P:stationary-existence}. In particular, $\cH^{\LG}$ gives a coupling of all the log-gamma joint stationary measures identified above. We call $\cH^{\LG}$ the \textbf{(extended) half-space log-gamma horizon}.
\end{theorem}

\begin{proof}
For the proof, let $\hat \Q = \Q \cap (-\theta, \alpha \wedge 0]$.
First, by Kolmogorov's extension theorem, we can define a function $\cH^{\LG}:\hat \Q \times \Z \to \R$ such that for all finite $I \subset \hat \Q$, $\cH^{\LG}|_I$ is equal in law to the process $R^{\LG}$ in Proposition \ref{P:stationary-existence}. Here the consistency of the marginals is given by Lemma \ref{lem:consis}. 

Now, for any $\apar_1 < \apar_2 \in \hat \Q$ and $j_1 < j_2 \in \Z$ we claim that
\begin{equation} \label{E:LGV}
    \cH^{\LG}(\apar_1,j_2) \cH^{\LG}(\apar_2,j_1) - \cH^{\LG}(\apar_1,j_1) \cH^{\LG}(\apar_2,j_2) \ge 0.
\end{equation}
It suffices to consider this for the process $R^{\LG,\eps}$ in Proposition \ref{P:stationary-existence}, with $k = 2$. In this setting, by \eqref{E:Resp}, \eqref{E:LGV} is equivalent to the statement that
\begin{equation*}
   Z_X(4, 1; [j_2,4]_\textsf{L} ) Z_X(3, 2; [j_1,4]_\textsf{L} ) -  Z_X(4, 1; [j_1,4]_\textsf{L} ) Z_X(3, 2; [j_2,4]_\textsf{L}) \ge 0.
\end{equation*}
Now, any pair of up-right paths from $(4, 1)$ to $[j_1,4]_\textsf{L}$ and from $(3, 2)$ to $[j_2,4]_\textsf{L}$ necessarily cross, and so by the Lindstr\"om-Gessel-Viennot lemma, the left-hand side above is equal to the partition function for pairs of disjoint paths from $(4, 1)$ to $[j_2,4]_\textsf{L}$ and from $(3, 2)$ to $[j_1,4]_\textsf{L}$. In particular, it is non-negative, as desired. 

Next, observe that $\cH^{\LG}$ is almost surely continuous at every point $(\gamma, j) \in \hat \Q \times \Z$. Indeed, since this is a countable set it suffices to prove almost sure continuity for every fixed $(\gamma, j)$. Moreover, if we let $j_1 = 0$ in \eqref{E:LGV}, then since $\cH^{\LG}(\gamma,0) = 1$ for each $\gamma\in\hat\Q$, we see that $\cH^{\LG}(\gamma,j)$ is monotone in $\gamma$, so it suffices to show continuity in law of $\cH^{\LG}(\cdot, j)$ for each $j\in \Z$. This is immediate from the description in Proposition \ref{P:stationary-existence}, which is continuous in the slope parameters.

We claim that if we let $F=-\log \cH^{\LG}$, then such $F$ (hence $\cH^{\LG}$) can be extended to a function in $\mathcal D((-\theta, \alpha \wedge 0], \Z, \alpha \wedge 0)$. First, by \eqref{E:LGV}, the right-hand side of \eqref{E:F-measure} is always non-negative for this $F$. Next, the continuity of $\cH^{\LG}$ (hence $F$) on $\Q \cap (-\theta, \alpha \wedge 0] \times \Z$ guarantees that the right-hand side of \eqref{E:F-measure} can be extended to a measure on $(-\theta, \alpha \wedge 0] \times \Z$, and hence $F$ can be defined accordingly. This $F$ will satisfy $F_\gamma(0) = 0$ for all $\gamma$, since this holds for all $\gamma \in \hat \Q$.

To complete the proof, it simply suffices to show that for any finite set $I \subset (-\theta, \alpha \wedge 0]$, that $\cH^{\LG}|_I$ is equal in law to the function from Proposition \ref{P:stationary-existence}. This holds for $I \subset \hat \Q$ by our construction, and for general $I$ since the construction in Proposition \ref{P:stationary-existence} is continuous in law as a function of the slope parameters, and moreover, for any sequence $I_n \to I$ from below, we have the continuity
$$
\cH^{\LG}|_I = \lim_{n\to\infty} \cH^{\LG}|_{I_n} 
$$
from continuity of measure for $\mu_F$.
\end{proof}

\begin{remark}
	\label{R:consistent-family}
Note that we do not expect that the $\cH^{\LG}$ in Theorem \ref{T:half-space-log-gamma-horizon} is continuous, but only right-continuous in the slope parameter. Indeed, the analogous full-space process is not (see e.g., \cite{bates2025permutation} in the setting of LPP). A left-continuous version can also be defined, and the two processes together should have a geometric interpretation as left- or right-continuous Busemann processes for the extended half-space log-gamma polymer.

We leave the construction of these Busemann processes as an open problem. The construction of Busemann processes is also closely connected to the question of uniqueness for joint stationary measures, which we have not addressed here.

 Note that we expect a qualitative difference between the Busemann functions in half-space and full-space models: when $\alpha < 0$,  we expect that the set of allowable angles for semi-infinite polymer measures (i.e., positive temperature versions of semi-infinite geodesics) is of the form $[0, \pi/4 - \hat\alpha) \cup \{\pi/4\}$ for some $0<\hat\alpha<\pi/4$. This comes from the fact that diagonal is highly attractive in this regime, and this phenomenon also manifests through the presence of a facet in the limit shape. The same phenomenon is expected in all the other models in the next few subsections.
	\end{remark}

	\subsection{Half-space geometric and exponential LPP horizons}
    \label{SS:half-space-zero-temp-horizons}
	With completely analogous proofs as in Section \ref{S:half-space-gamma-horizon}, we can obtain joint stationary measures for exponential and geometric LPP. The joint stationary measures for exponential LPP can alternately be obtained through taking a zero temperature limit of the log-gamma polymer.

For a function $f:\Z \to \R$, and an environment $E:\Z^2 \to \R\cup\{\infty\}$, define
\begin{equation}     \label{eq:Efdef}
	E^{f}(u) = \begin{cases}
	f(j) - f(j - \operatorname{sgn}(j)), \qquad &u = [j,0]_\textsf{L}\; \text{ for some } j \in \Z,	 \\
	E(u), \qquad &\text{else}.
\end{cases}    
\end{equation}
	Here $\operatorname{sgn}(0) = 0$, so that $E^{f}(0, 0) = 0$. Let $E^{f}(u;v)$ denote the passage time from $u$ to $v$ for the environment $E^{f}$. 
    
    Now, we say that a family of functions  $(F_i:\Z \to \R, i \in I)$ for a finite index set $I$ is \emph{jointly stationary} for the extended half-space exponential or geometric last passage percolation with boundary parameter $\alpha$ and equal bulk parameters $\gamma_i=\theta$, $i\in\Z$, if for $X$ being the field as given in Definition \ref{D:geo} with such parameters $\alpha$ and $\{\gamma_i\}_{i\in\Z}$, and any $h \in \N$, we have:
    \begin{equation}
    \label{E:stationary-lpp}
   F_i(j) - F_i(0) \eqd X^{F_i}(0,0; [j,h]_\textsf{L}) - X^{F_i}(0,0; [0,h]_\textsf{L}).
    \end{equation}
where the equality in law is joint in all $i \in I, j \in \Z$. 
	
	We now define our joint stationary measures.
    Consider the models in Definition \ref{D:geo}.
	In the model of extended half-space exponential LPP, as in the log-gamma polymer we take $\apar_i=\theta$ for each $i\ge 2k+1$ or $i\le 0$, where $\theta>0\vee -\alpha$. 
	We also assume that $-\theta<\apar_1< \apar_2< \cdots < \apar_k\le 0\wedge \alpha$, and let $\apar_{2k+1-i}=\eps-\apar_i$, for some $\eps>0$, for each $i\in\II{1, k}$, exactly as in the log-gamma model. In extended half-space geometric LPP, all parameters must now be positive, and the analogous constraints are
	\begin{equation*}
	\theta < (1/\alpha) \wedge 1, \qquad 1/\theta>\apar_1>  \cdots > \apar_k \ge (\alpha \vee 1), \qquad \apar_{2k + 1 - i} = (1 - \eps)/\apar_i, \;\; i\in\II{1, k}.
	\end{equation*}
	Now, for each $i\in\II{1, k}$ and $j\in \Z$, we let $R^{\EL,\epsilon}_i(j)$ or $R^{\GL,\epsilon}_i(j)$ (in the exponential or geometric settings, respectively) be equal to
	\[
	X(2k+1-i, i; [j,2k]_\textsf{L}) - X(2k+1-i, i; 2k, 2k),
	\]
	and define $R^{\EL} = \{R_i^{\EL}\}_{i=1}^k \eqd \lim_{\eps\to 0}\{R^{\EL,\eps}_i\}_{i=1}^k$ or $R^{\GL} = \{R_i^{\GL}\}_{i=1}^k \eqd \lim_{\eps\to 0}\{R^{\GL,\eps}_i\}_{i=1}^k$, where the existence of the limit can be justified as in Proposition \ref{P:stationary-existence}.

	Similarly to Corollary \ref{cor:sta}, we have the following stationarity result. 
	
	\begin{lemma}   \label{lem:staLPP}
The process $R^{\EL}$ (resp.~$R^{\GL}$) is jointly stationary for the extended half-space exponential (resp.~geometric) LPP with boundary parameter $\alpha$ and bulk parameter $\theta$, in the sense of \eqref{E:stationary-lpp}.
	\end{lemma}
	
	\begin{proof}
	The proof is exactly as in Lemma \ref{lem:sta} and  Corollary \ref{cor:sta} with the following changes:
	\begin{itemize}[nosep]
		\item We use Lemma \ref{lem:invz} in place of Lemma \ref{lem:invp}. 
		\item We switch algebras: $(+, \times) \mapsto (\max, +)$, and partition functions by passage times.\qedhere
	\end{itemize}
	\end{proof}

	We call $R^{\EL}$ (resp.~$R^{\GL}$) the \emph{(extended) half-space exponential (resp.~geometric) LPP joint stationary process, with boundary parameter $\alpha$, bulk parameter $\theta$, and slope parameters $(\apar_1,\ldots, \apar_k)$}.
    Note that when $k=1$, $R^{\EL}_1$ or $R^{\GL}_1$ restricted to $\N$ is the same as the stationary measures for half-space LPP given in \cite[Section 3]{BC}.
    
	Similar to Lemma \ref{lem:consis}, this construction is also consistent, and as in Theorem \ref{T:half-space-log-gamma-horizon} can be used to construct a full horizon for the model.

\begin{theorem}
\label{T:lpp-horizons}
\begin{enumerate}
    \item Let $\theta > 0, \alpha \in (-\theta, \infty)$. Then there exists a random function $\cH^{\EL}$ such that $-\cH^{\EL}\in \mathcal{D}((-\theta, \alpha \wedge 0], \Z, \alpha \wedge 0)$, and for any finite set $I = (\apar_1 < \cdots < \apar_k) \subset (-\theta, \alpha \wedge 0]$, the restriction $\cH^{\EL}|_I$ is equal in law to the exponential LPP joint stationary process defined above. In particular, $\cH^{\EL}$ gives a coupling of all the (extended) half-space exponential LPP joint stationary processes identified above. We call $\cH^{\EL}$ the \textbf{(extended) half-space exponential horizon}.
    \item Let $\theta \in (0, 1), \alpha \in (0, 1/\theta)$. Then there exists a random function $\cH^{\GL}\in \mathcal{D}([\alpha \vee 1, 1/\theta), \Z, \alpha \vee 1)$ such that for any finite set $I = (\apar_1 < \cdots < \apar_k) \subset [\alpha \vee 1, 1/\theta)$, the restriction $\cH^{\GL}|_I$ is equal in law to the geometric LPP joint stationary process defined above. In particular, $\cH^{\GL}$ gives a coupling of all the (extended) half-space geometric LPP joint stationary processes identified above. We call $\cH^{\GL}$ the \textbf{(extended) half-space geometric horizon}. 
\end{enumerate}
\end{theorem}

\begin{proof}
We only prove the exponential LPP setting, as the other one is essentially identical (with sign changes as appropriate). The proof of consistency goes through exactly as in Lemma \ref{lem:consis}, making the same two changes as in the proof of Lemma \ref{lem:staLPP}, and we omit the details. 

Given this, we proceed exactly as in the proof of Theorem \ref{T:half-space-log-gamma-horizon}, by first defining $\cH^{\EL}:\hat \Q \times \Z \to \R$ and then extending $-\cH^{\EL}$ to an element of $\mathcal{D}((-\theta, \alpha \wedge 0], \Z, \alpha \wedge 0)$. All steps go through verbatim, except instead of \eqref{E:LGV}, we need to check that for any $\apar_1 < \apar_2 \in \hat \Q$ and $j_1 < j_2 \in \Z$, we have
\[
    \cH^{\EL}(\gamma_1,j_2) + \cH^{\EL}(\gamma_2,j_1) - \cH^{\EL}(\gamma_1,j_1) - \cH^{\EL}(\gamma_2,j_2) \ge 0.
\]
Using the $\cH^{\EL,\eps}$-representation with $k = 2$, this is equivalent to the statement that
\begin{equation*}
   X(4, 1; [j_2,4]_\textsf{L} ) + X(3, 2; [j_1,4]_\textsf{L} ) \ge  X(4, 1; [j_1,4]_\textsf{L}) +X(3, 2; [j_2,4]_\textsf{L}).
\end{equation*}
Let $\pi = (\pi_1, \dots, \pi_k)$ be a geodesic from $(4, 1)$ to $[j_1,4]_\textsf{L}$, and let $\tau = (\tau_1, \dots, \tau_\ell)$ be a geodesic from $(3, 2)$ to $[j_2,4]_\textsf{L}$. These necessarily meet at a point $v = \pi_j = \tau_{j'}$. Then the paths
$$
(\pi_1, \dots, \pi_j, \tau_{j' + 1}, \dots, \tau_\ell), \qquad  (\tau_1, \dots, \tau_{j'}, \pi_{j + 1}, \dots, \pi_k)
$$
go from $(4, 1)$ to $[j_2,4]_\textsf{L}$ and $(3, 2)$ to $[j_1,4]_\textsf{L}$. Moreover, by switching ends we have not broken the condition \eqref{E:diagonal-condition}. This gives the desired inequality.
\end{proof}

\subsection{Half-space stationary horizon}   \label{ss:hfsh}

In this section, we take the 1:2:3-scaling limit of the half-space exponential horizon to construct the (extended) half-space stationary horizon. 
	
	Take $\rho\in\R \cup \{-\infty\}$, and slopes $\lambda_1>\lambda_2>\cdots>\lambda_k \ge 2\rho \vee 0$. 
	For any $n\in \N$, let $\theta=1/2$, $\alpha=-2^{-4/3}\rho n^{-1/3}$ (if $\rho\in\R$) or $\alpha>0$ fixed (if $\rho=-\infty$), and $\gamma_i=-2^{-7/3}n^{-1/3}\lambda_i$.
	Let $R^{\EL}$ be the half-space exponential LPP joint stationary process with these parameters. Then we define
	\[
	R_i^{(n)}(x) = 2^{-4/3}n^{-1/3}[R_i^{\EL}( 2^{5/3}n^{2/3}x) - 2^{8/3}n^{2/3}|x|],
	\]
	for each $i\in\II{1, k}$ and $x\in \R$. Here we extend each function $R_i^{\EL}$ to all of $\R$ by linearly interpolating between integers. The processes $R^{(n)}=\{R_i^{(n)}\}_{i=1}^k$ have a limit as $n\to \infty$, which can be described by a \emph{half-space exponential-Brownian LPP} problem which mixes a fully discrete and a semi-discrete environment. Note that in this setting, because of a coordinate shift the half-space in question is everything below the shifted diagonal $y = x + 2k + 1$. A polymer version of this kind of mixed problem has appeared before, see \cite[Definition 6.33]{barraquand2020half} and surrounding discussion, as well as the description of the half-space KPZ horizon in Section \ref{ssec:hsKPZhor}. 
    
    To define this problem, we use the language of cadlag LPP, introduced in \cite{dauvergne2021scaling}.
	
Let $f = \{f_i:\R \to \R, i \in I\}$ be a sequence of cadlag functions, where $I$ is an interval in $\Z$. The environment $f$ defines a finitely additive signed measure $df$ on $\R \times \Z$ through
	$$
	df([x, y] \times \{i\}) = f_i(y) - f_i(x^-),
	$$
	where $f_i(x^-) = \lim_{y \uparrow x} f_i(y)$. For $u = (x, n), v = (y, m) \in \R\times\Z$ with $x\le y$ and $n\le m$, a path $\pi$ from $u$ to $v$ is a union of closed intervals
	$$
	[x_i, x_{i+1}] \times \{i\}, \qquad i\in\II{n, m-1}, \qquad x = x_n \le x_{n+1} \le \cdots \le x_m = y.
	$$
For $u \le v$, we now define the \emph{passage times} as
	$$
	f(u; v) = \sup_{\pi:u \to v} df(\pi), \qquad f^{\Delta_k}(u; v) = \sup_{\pi:u \to v, \pi \cap \Delta_k \ne \emptyset} df(\pi).
	$$
	In the first definition, the supremum is over all paths from $u$ to $v$. In the second definition, the supremum is over paths from $u$ to $v$ which hit the shifted diagonal $\Delta_k = \{(x, x + k): x \in \Z\}$.

	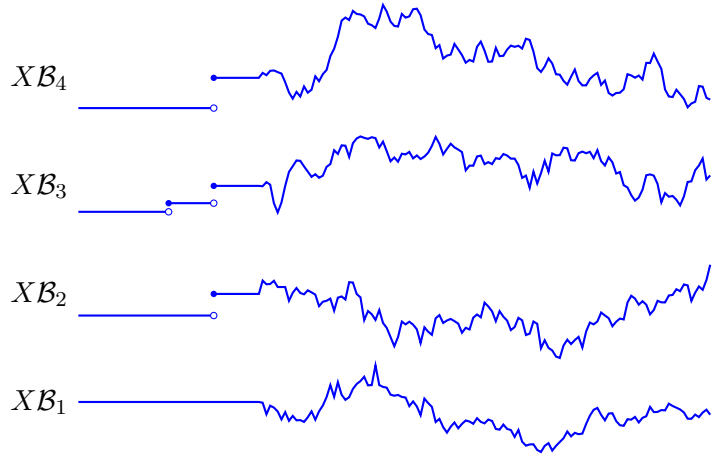
\begin{figure}[t]
		\centering

\begin{tikzpicture}[x=6cm,y=1.15cm]
    \draw[blue, thick]
    plot coordinates {
(-0.4, 0)
(0.0000, 0.0000)
(0.0083, -0.0084)
(0.0167, -0.1420)
(0.0250, -0.0432)
(0.0333, -0.0651)
(0.0417, -0.1099)
(0.0500, -0.2014)
(0.0583, -0.1175)
(0.0667, -0.2183)
(0.0750, -0.1611)
(0.0833, -0.2123)
(0.0917, -0.2097)
(0.1000, -0.2308)
(0.1083, -0.1771)
(0.1167, -0.1084)
(0.1250, -0.2051)
(0.1333, -0.1087)
(0.1417, -0.0404)
(0.1500, 0.0568)
(0.1583, 0.1955)
(0.1667, 0.0597)
(0.1750, 0.2295)
(0.1833, 0.0835)
(0.1917, 0.0245)
(0.2000, 0.0553)
(0.2083, 0.1509)
(0.2167, 0.2083)
(0.2250, 0.2415)
(0.2333, 0.2922)
(0.2417, 0.1928)
(0.2500, 0.1950)
(0.2583, 0.4231)
(0.2667, 0.1958)
(0.2750, 0.1744)
(0.2833, 0.1655)
(0.2917, 0.0845)
(0.3000, 0.0721)
(0.3083, 0.0814)
(0.3167, 0.0585)
(0.3250, 0.0513)
(0.3333, -0.0478)
(0.3417, 0.0065)
(0.3500, -0.0518)
(0.3583, -0.1529)
(0.3667, 0.0393)
(0.3750, -0.0125)
(0.3833, -0.0563)
(0.3917, -0.2318)
(0.4000, -0.1954)
(0.4083, -0.2911)
(0.4167, -0.3544)
(0.4250, -0.2864)
(0.4333, -0.2373)
(0.4417, -0.3042)
(0.4500, -0.2535)
(0.4583, -0.2140)
(0.4667, -0.2264)
(0.4750, -0.2525)
(0.4833, -0.1873)
(0.4917, -0.2302)
(0.5000, -0.2469)
(0.5083, -0.2036)
(0.5167, -0.2198)
(0.5250, -0.3120)
(0.5333, -0.3204)
(0.5417, -0.2612)
(0.5500, -0.3120)
(0.5583, -0.3245)
(0.5667, -0.4045)
(0.5750, -0.3238)
(0.5833, -0.4594)
(0.5917, -0.5055)
(0.6000, -0.4912)
(0.6083, -0.4893)
(0.6167, -0.5556)
(0.6250, -0.5785)
(0.6333, -0.5074)
(0.6417, -0.4765)
(0.6500, -0.5624)
(0.6583, -0.5081)
(0.6667, -0.4628)
(0.6750, -0.4197)
(0.6833, -0.3410)
(0.6917, -0.3977)
(0.7000, -0.3568)
(0.7083, -0.3727)
(0.7167, -0.3663)
(0.7250, -0.2552)
(0.7333, -0.1587)
(0.7417, -0.0905)
(0.7500, -0.0687)
(0.7583, -0.1488)
(0.7667, -0.2124)
(0.7750, -0.1339)
(0.7833, -0.1301)
(0.7917, -0.2215)
(0.8000, -0.2690)
(0.8083, -0.2759)
(0.8167, -0.1993)
(0.8250, -0.2046)
(0.8333, -0.1727)
(0.8417, -0.2373)
(0.8500, -0.1835)
(0.8583, -0.1650)
(0.8667, -0.0570)
(0.8750, -0.1170)
(0.8833, -0.1152)
(0.8917, -0.0935)
(0.9000, -0.0973)
(0.9083, -0.0197)
(0.9167, -0.0465)
(0.9250, -0.0659)
(0.9333, -0.0867)
(0.9417, -0.1212)
(0.9500, -0.1934)
(0.9583, -0.1944)
(0.9667, -0.1461)
(0.9750, -0.1747)
(0.9833, -0.1223)
(0.9917, -0.0894)
(1.0000, -0.1474)
    };

\draw[blue, thick]
    plot coordinates {
(-0.4, 1)
(-0.1, 1)
};
    
    \draw[blue, thick]
    plot coordinates {
(-0.1, 1.25)
(0.0000,1.2500)
(0.0083,1.4043)
(0.0167,1.3618)
(0.0250,1.3648)
(0.0333,1.4020)
(0.0417,1.3297)
(0.0500,1.3300)
(0.0583,1.3299)
(0.0667,1.1704)
(0.0750,1.2636)
(0.0833,1.3184)
(0.0917,1.2611)
(0.1000,1.2454)
(0.1083,1.2918)
(0.1167,1.2679)
(0.1250,1.2422)
(0.1333,1.1090)
(0.1417,1.1593)
(0.1500,1.1709)
(0.1583,1.1959)
(0.1667,1.0561)
(0.1750,1.2068)
(0.1833,1.2210)
(0.1917,1.1855)
(0.2000,1.3860)
(0.2083,1.3818)
(0.2167,1.2485)
(0.2250,1.2111)
(0.2333,0.9469)
(0.2417,1.0519)
(0.2500,1.0157)
(0.2583,0.9490)
(0.2667,1.0638)
(0.2750,0.9141)
(0.2833,0.8567)
(0.2917,0.7506)
(0.3000,0.6452)
(0.3083,0.7518)
(0.3167,0.9306)
(0.3250,0.8977)
(0.3333,0.9284)
(0.3417,0.8945)
(0.3500,0.8403)
(0.3583,0.6813)
(0.3667,0.8632)
(0.3750,0.8135)
(0.3833,0.7771)
(0.3917,0.9692)
(0.4000,1.0050)
(0.4083,0.9136)
(0.4167,0.9427)
(0.4250,0.9607)
(0.4333,0.9353)
(0.4417,0.9064)
(0.4500,0.8006)
(0.4583,0.8584)
(0.4667,0.8389)
(0.4750,0.9755)
(0.4833,1.0123)
(0.4917,0.9797)
(0.5000,1.1494)
(0.5083,1.1053)
(0.5167,1.0480)
(0.5250,0.9343)
(0.5333,1.0509)
(0.5417,0.9863)
(0.5500,0.9417)
(0.5583,0.8318)
(0.5667,0.8673)
(0.5750,0.9734)
(0.5833,0.9446)
(0.5917,0.7748)
(0.6000,0.7899)
(0.6083,0.9308)
(0.6167,0.9046)
(0.6250,0.6219)
(0.6333,0.7137)
(0.6417,0.6255)
(0.6500,0.5961)
(0.6583,0.5257)
(0.6667,0.5092)
(0.6750,0.6533)
(0.6833,0.7208)
(0.6917,0.6704)
(0.7000,0.7161)
(0.7083,0.7532)
(0.7167,0.6297)
(0.7250,0.8144)
(0.7333,0.9188)
(0.7417,0.8980)
(0.7500,0.9312)
(0.7583,1.0350)
(0.7667,1.0004)
(0.7750,0.9696)
(0.7833,0.8767)
(0.7917,0.9290)
(0.8000,0.9651)
(0.8083,0.9544)
(0.8167,1.1507)
(0.8250,1.2290)
(0.8333,1.1965)
(0.8417,1.1257)
(0.8500,1.1656)
(0.8583,1.1259)
(0.8667,1.0138)
(0.8750,1.0636)
(0.8833,1.1702)
(0.8917,1.1528)
(0.9000,1.1519)
(0.9083,1.2322)
(0.9167,1.2385)
(0.9250,1.3046)
(0.9333,1.2817)
(0.9417,1.3596)
(0.9500,1.3491)
(0.9583,1.1696)
(0.9667,1.2654)
(0.9750,1.3365)
(0.9833,1.2976)
(0.9917,1.4120)
(1.0000,1.5878)
    };

\draw[blue, thick]
plot coordinates {
(-0.4, 2.2)
(-0.2, 2.2)
};

\draw[blue, thick]
plot coordinates {
(-0.2, 2.3)
(-0.1, 2.3)
};

\draw[blue, thick]
plot coordinates {
(-0.1, 2.5)
(0.0000,2.5000)
(0.0083,2.4961)
(0.0167,2.5434)
(0.0250,2.5143)
(0.0333,2.2990)
(0.0417,2.1924)
(0.0500,2.3072)
(0.0583,2.5004)
(0.0667,2.6620)
(0.0750,2.7274)
(0.0833,2.7943)
(0.0917,2.7899)
(0.1000,2.7252)
(0.1083,2.6660)
(0.1167,2.6043)
(0.1250,2.6349)
(0.1333,2.5940)
(0.1417,2.6177)
(0.1500,2.7205)
(0.1583,2.7519)
(0.1667,2.9125)
(0.1750,2.8701)
(0.1833,2.9689)
(0.1917,2.9651)
(0.2000,2.8350)
(0.2083,2.9934)
(0.2167,3.0492)
(0.2250,3.0719)
(0.2333,3.0560)
(0.2417,3.0441)
(0.2500,3.0630)
(0.2583,3.0528)
(0.2667,2.9448)
(0.2750,3.0118)
(0.2833,2.9011)
(0.2917,2.7845)
(0.3000,2.8286)
(0.3083,2.7796)
(0.3167,2.8720)
(0.3250,2.8643)
(0.3333,2.8931)
(0.3417,2.8181)
(0.3500,2.8523)
(0.3583,2.8104)
(0.3667,2.8394)
(0.3750,2.8833)
(0.3833,2.9620)
(0.3917,3.0665)
(0.4000,2.9866)
(0.4083,2.8639)
(0.4167,2.9284)
(0.4250,2.9878)
(0.4333,2.9458)
(0.4417,2.9199)
(0.4500,2.8909)
(0.4583,2.6800)
(0.4667,2.7349)
(0.4750,2.7331)
(0.4833,2.8402)
(0.4917,2.8501)
(0.5000,2.7515)
(0.5083,2.7749)
(0.5167,2.7521)
(0.5250,2.8215)
(0.5333,2.8109)
(0.5417,2.8686)
(0.5500,2.8514)
(0.5583,2.8151)
(0.5667,2.8154)
(0.5750,2.7421)
(0.5833,2.5926)
(0.5917,2.6189)
(0.6000,2.4773)
(0.6083,2.6648)
(0.6167,2.5280)
(0.6250,2.6471)
(0.6333,2.7912)
(0.6417,2.6641)
(0.6500,2.8749)
(0.6583,2.8537)
(0.6667,2.8560)
(0.6750,2.9420)
(0.6833,2.9351)
(0.6917,2.7971)
(0.7000,2.7932)
(0.7083,2.8785)
(0.7167,2.8968)
(0.7250,2.8778)
(0.7333,2.8946)
(0.7417,2.8209)
(0.7500,2.9134)
(0.7583,2.7864)
(0.7667,2.7360)
(0.7750,2.8315)
(0.7833,2.7724)
(0.7917,2.6428)
(0.8000,2.6604)
(0.8083,2.5076)
(0.8167,2.4822)
(0.8250,2.3931)
(0.8333,2.3297)
(0.8417,2.3654)
(0.8500,2.4557)
(0.8583,2.6231)
(0.8667,2.6329)
(0.8750,2.6188)
(0.8833,2.5225)
(0.8917,2.3105)
(0.9000,2.4170)
(0.9083,2.3750)
(0.9167,2.3834)
(0.9250,2.3394)
(0.9333,2.2745)
(0.9417,2.3770)
(0.9500,2.5453)
(0.9583,2.5522)
(0.9667,2.6815)
(0.9750,2.7358)
(0.9833,2.8171)
(0.9917,2.5740)
(1.0000,2.6211)
};

\draw[blue, thick]
plot coordinates {
(-0.4, 3.4)
(-0.1, 3.4)
};
\draw[blue, thick]
plot coordinates {
(-0.1, 3.75)
(0.0000,3.7500)
(0.0083,3.8072)
(0.0167,3.7807)
(0.0250,3.8312)
(0.0333,3.8354)
(0.0417,3.8159)
(0.0500,3.7994)
(0.0583,3.6991)
(0.0667,3.5802)
(0.0750,3.5091)
(0.0833,3.5872)
(0.0917,3.5340)
(0.1000,3.6506)
(0.1083,3.5787)
(0.1167,3.6128)
(0.1250,3.7678)
(0.1333,3.7108)
(0.1417,3.7932)
(0.1500,3.8412)
(0.1583,4.0005)
(0.1667,4.0863)
(0.1750,4.2245)
(0.1833,4.4118)
(0.1917,4.4472)
(0.2000,4.4185)
(0.2083,4.5243)
(0.2167,4.4894)
(0.2250,4.4831)
(0.2333,4.4612)
(0.2417,4.3766)
(0.2500,4.3595)
(0.2583,4.3266)
(0.2667,4.4374)
(0.2750,4.5932)
(0.2833,4.5162)
(0.2917,4.4002)
(0.3000,4.3510)
(0.3083,4.3438)
(0.3167,4.4896)
(0.3250,4.4468)
(0.3333,4.4609)
(0.3417,4.5439)
(0.3500,4.5499)
(0.3583,4.4743)
(0.3667,4.3101)
(0.3750,4.2568)
(0.3833,4.2456)
(0.3917,4.0927)
(0.4000,4.1284)
(0.4083,3.9730)
(0.4167,4.0071)
(0.4250,4.0189)
(0.4333,4.1424)
(0.4417,4.0786)
(0.4500,4.0131)
(0.4583,3.9250)
(0.4667,3.9573)
(0.4750,3.9386)
(0.4833,4.1089)
(0.4917,4.1691)
(0.5000,4.0649)
(0.5083,3.9350)
(0.5167,3.9391)
(0.5250,4.0972)
(0.5333,4.0009)
(0.5417,4.0050)
(0.5500,4.0773)
(0.5583,4.1232)
(0.5667,4.1085)
(0.5750,4.1213)
(0.5833,4.0886)
(0.5917,4.2025)
(0.6000,4.1885)
(0.6083,4.0545)
(0.6167,3.8625)
(0.6250,3.8185)
(0.6333,3.8810)
(0.6417,3.9756)
(0.6500,3.9163)
(0.6583,3.7920)
(0.6667,3.7862)
(0.6750,3.7147)
(0.6833,3.6708)
(0.6917,3.7627)
(0.7000,3.8459)
(0.7083,3.7129)
(0.7167,3.6628)
(0.7250,3.6629)
(0.7333,3.7369)
(0.7417,3.7785)
(0.7500,3.8081)
(0.7583,3.6682)
(0.7667,3.7187)
(0.7750,3.5743)
(0.7833,3.5452)
(0.7917,3.5748)
(0.8000,3.5777)
(0.8083,3.6180)
(0.8167,3.7923)
(0.8250,3.7778)
(0.8333,3.7516)
(0.8417,3.7765)
(0.8500,3.7301)
(0.8583,3.9235)
(0.8667,3.9212)
(0.8750,4.0303)
(0.8833,3.9401)
(0.8917,3.8535)
(0.9000,3.6142)
(0.9083,3.5346)
(0.9167,3.5278)
(0.9250,3.6127)
(0.9333,3.5613)
(0.9417,3.5062)
(0.9500,3.4089)
(0.9583,3.4533)
(0.9667,3.4733)
(0.9750,3.6551)
(0.9833,3.6475)
(0.9917,3.5184)
(1.0000,3.4981)
};

\fill[blue] (-0.1,1.25) circle (1.2pt);
\draw[blue,fill=white] (-0.1,1) circle (1.2pt);

\fill[blue] (-0.1,2.5) circle (1.2pt);
\draw[blue,fill=white] (-0.1,2.3) circle (1.2pt);

\fill[blue] (-0.1,3.75) circle (1.2pt);
\draw[blue,fill=white] (-0.1,3.4) circle (1.2pt);

\fill[blue] (-0.2,2.3) circle (1.2pt);
\draw[blue,fill=white] (-0.2,2.2) circle (1.2pt);

\node[left] at (-0.4,3.75) {$X\cB_4$};
\node[left] at (-0.4,2.5) {$X\cB_3$};
\node[left] at (-0.4,1.25) {$X\cB_2$};
\node[left] at (-0.4,0) {$X\cB_1$};
\end{tikzpicture}		
		\caption{An illustration of the half-space exponential-Brownian LPP environment (for marginals of the half-space stationary horizon), with $k=2$.} 
		\label{fig:hfsh}
	\end{figure}

	\begin{prop}
		\label{P:half-space-horizon-marginals}
	As $n\to\infty$, the process $R^{(n)}$ has a distributional limit $R^{\DL}=\{R^{\DL}_i\}_{i=1}^k$, under the uniform-on-compact topology. 
    We can describe the process $R^{\DL}$ explicitly as follows.
	
Let $\cB = (\cB_1, \cB_2, \ldots, \cB_{2k})$ be $2k$ independent Brownian motions on $\R_{\ge 0}$ of diffusivity $2$, drift $\lambda_i$ for $\cB_i, i\in\II{1, k}$ and $-\lambda_i$ for $\cB_{2k+1-i}$. Next, let $\{X(-i, j) : i\in\II{1, k}, j \in \II{i +1, 2k +1 - i}\}$ be an array of exponential random variables, independent of each other and $\mathcal B$, satisfying 
\[
X(-i, j) \sim \begin{cases}
\operatorname{Exp}(\lambda_i/2 - \lambda_j/2), \qquad &j \in \II{i + 1, k}, \\
\operatorname{Exp}(\lambda_i/2 + \lambda_{2k + 1 -j}/2), \qquad &j\in\II{k + 1, 2k - i}, \\
\operatorname{Exp}(\lambda_i/2 - \rho), \qquad &j=2k + 1- i.
\end{cases}
\]
We turn this array into a 
sequence of cadlag functions $X = (X_1, \dots, X_{2k})$ on $(-\infty, 0]$ by the rule that $X_j(0) = 0$ for all $j$, and the measure $dX$ has atoms of size $X(-i, j)$ at each coordinate $(-i, j)$ where $X(-i, j)$ is defined (and is supported on these points). Equivalently, the path $X_j$ has a jump of size $X(-i, j)$ at location $-i$.

Let $X \cB$ be the concatenation of $X$ and $\cB$. (See Figure \ref{fig:hfsh} for an illustration.) Then:
$$
R^{\DL}_i(x) = \begin{cases}
X\cB(-i, i; x, 2k) - X\cB(-i, i; 0, 2k), \qquad &x \ge 0, \\
X\cB^{\Delta_{2k+1}}(-i, i; |x|, 2k) - X\cB(-i, i; 0, 2k), \qquad &x \le 0.
\end{cases}
$$
\end{prop}

In the above proposition, if $\lambda_k = 2 \rho$, then the exponential random variable $X(-k, k + 1)$ is $\operatorname{Exp}(0)$, which we interpret as $\infty$. This is similar to the finite case. In this case, $R^{\DL}_k$ should be viewed as a limit, where we  approximate $X(-k, k + 1)$ by a sequence of numbers $\to\infty$. This is equivalent to setting
$$
R^{\DL}_k(x) = X\cB(-k+1, k + 1; x, 2k) - X\cB(-k+1, k + 1; 0, 2k).
$$

\begin{proof}[Proof of Proposition \ref{P:half-space-horizon-marginals}]
We first give an explicit description of the processes $R^{(n)}$ in terms of a cadlag LPP. Starting with the unscaled process $R^{\EL}$, following Section \ref{SS:half-space-zero-temp-horizons} we can write it in terms of a collection of $2k$ lines of independent exponential random variables $\{W(i, j), i \in \N, j\in\II{1, 2k}\}$. 

Namely, we let $W(i, i) \sim \Exp(\alpha + \gamma_i)$ for all $i\in \II{k+1, 2k}$, and $W(i, j) \sim \Exp(\gamma_i + \gamma_j)$ for all $i > j$ with $i+j\ge 2k+2$. We take $W(i, j)=0$ for the remaining $(i, j)$.
Here $\alpha$ and $\gamma_i$ for $i\in \II{1, k}$ are defined as in the beginning of this subsection, and $\gamma_i=-\gamma_{2k+1-i}$ for $i \in \II{k+1, 2k}$, and $\gamma_i=\theta=1/2$ for $i\ge 2k+1$.
If $\lambda_k=2\rho$, thereby $\alpha+\gamma_{k+1}=0$, we interpret $W(k+1,k+1)$ as $\infty$.

We can turn this array into a cadlag environment $W = (W_1, \dots, W_{2k})$: setting $W_j(2k) = 0$ for all $j$, and let $dW$ have atoms of size $W(i, j)$ at every point $(i, j)$. Note that there is an implicit dependence on $n$ in the environment parameters here.

We now shift and rescale the environment $W$ to construct $R^{(n)}$. Define 
$$
\hat W_i(x) = \begin{cases}
	2^{-4/3} n^{-1/3}W_i(x + 2k+1), \qquad &x < 0, \\
	2^{-4/3} n^{-1/3}[W_i(2^{5/3} n^{2/3} x + 2k + 1) - 2^{8/3} n^{2/3} x], \qquad &x \ge 0.
\end{cases}
$$
Then since cadlag LPP commutes with scalings of the environment and shifts of all lines by a common function, we can write
$$
R^{(n)}_i(x) = \begin{cases}
	\hat W(-i, i; x, 2k) - \hat W(-i, i; 0, 2k), \qquad &x \ge 0, \\
	\hat W^{\Delta_{2k+1}}(-i, i; |x|, 2k) - \hat W(-i, i; 0, 2k), \qquad &x \le 0.
\end{cases}
$$
Here to obtain the expression for $R^{(n)}_i(x), x \le 0$ we have used that LPP in a symmetric environment can alternately be described as LPP in a half-space environment, where the paths are forced to touch the diagonal. At this point, it simply remains to observe that $\hat W \cvgd X \cB$, in the uniform-on-compact topology as $n \to \infty$. This uses Donsker's invariance principle for $x \ge 0$, and is immediate even before $n \to \infty$ for $x \le 0$.
\end{proof}
	
	We will check in Proposition \ref{P:properties-elpp} that the process $R^{\DL}$ is jointly stationary for any subsequential limit of exponential LPP. For now, we finish this section by extending it to a horizon.

    \begin{theorem}
    \label{T:half-stationary-horizon}
    Let $\rho \in \R\cup\{-\infty\}$. Then there exists a process $\cH^{\DL}\in \mathcal D([2\rho \vee 0, \infty), \R, 2\rho \vee 0)$, which we call the \emph{(extended) half-space stationary horizon}, such that for any finite set $I=(\lambda_1>\cdots>\lambda_k) \subset [2\rho \vee 0, \infty)$, the process $\cH^{\DL}|_I$ has the same distribution as $R^{\DL}$ from Proposition \ref{P:half-space-horizon-marginals}.
    \end{theorem}

    \begin{proof}
    Let $\hat \Q := \Q \cap [2\rho \vee 0, \infty)$.
     The consistency of the construction in Proposition \ref{P:half-space-horizon-marginals} follows immediately from Theorem \ref{T:lpp-horizons}.1 upon taking a limit. This allows us to define $\cH^{\DL}$ restricted to the set $\hat \Q \times \R$. Moreover, for $\lambda_1 > \lambda_2 \in \hat \Q$ and $x_1 < x_2 \in \R$, we have that
     $$
     \cH^\DL(\lambda_1, x_1) + \cH^\DL(\lambda_2, x_2) - \cH^\DL(\lambda_1, x_2) - \cH^\DL(\lambda_2, x_1) \le 0,
     $$
     since this inequality holds in the prelimit by the proof of Theorem \ref{T:lpp-horizons}. Finally, from the construction in Proposition \ref{P:half-space-horizon-marginals}, the process $\cH^{\DL}|_I$ is continuous in law as a function of the finite set $I$. Then this allows us to extend $\cH^{\DL}$ to all of $[2\rho \vee 0, \infty) \times \R$ and check that the resulting construction has marginals given by Proposition \ref{P:half-space-horizon-marginals} on arbitrary finite sets.
    \end{proof}

    	\subsubsection{Invariance under parameter permutation: exponential-Brownian setting}
	We next state a degeneration of Lemma \ref{lem:coupr}, i.e., invariance under parameter permutation for the log-gamma polymer model, in the exponential-Brownian LPP setting.
	This will be used several times later in this paper.
	
	Define a two-line exponential-Brownian environment as follows. Take any $\lambda_0>\lambda_1$, and $\{\beta_i\}_{i\in \Z_-}$, such that $\lambda_0<2\beta_i$ for each $i\in \Z_-$.
	Take $X(i,j)\sim \Exp(-\lambda_j/2+\beta_i)$ for each $i\in \Z_-$ and $j\in\{0,1\}$, and define $X_0, X_1:(-\infty, 0] \to \R$ by setting $X_i(0) = 0$ and letting $dX$ be the atomic measure with atoms $X(i,j)$ at $(i, j)$. Let $\cB_0, \cB_1$ be Brownian motions on $\R_{\ge 0}$ with diffusivity $2$ and slopes $\lambda_0, \lambda_1$, respectively. All these random variables are taken to be independent of each other, and we let $X\cB$ be the concatenation of $(X_0, X_1)$ and $(\cB_0, \cB_1)$. Define $X'$, $\cB'$, and their concatenation $X'\cB'$ in the same way, except with the parameters $\lambda_0, \lambda_1$ switched.
	
	\begin{lemma}   \label{lem:exp-bro-swap}
		We have $X\cB(x, 0; y, 1)\eqd X'\cB'(x, 0; y, 1)$, jointly as functions of $x \le y \in \R$. 
	\end{lemma}
	\begin{proof}
	This directly follows from Lemma \ref{lem:invz} (with $I=\Z$) after taking a limit transition from exponential LPP to Brownian LPP in the range $\R_{\ge 0}$.
We omit the details, as they are similar to the proof of Proposition \ref{P:half-space-horizon-marginals}.
	\end{proof}

	We also state a useful corollary. Here we condense notation and write $X\cB(x, y) = X\cB(x, 0; y, 1)$ and $X'\cB'(x, y) = X'\cB'(x, 0; y, 1)$ for the passage times.
    For $0\le x \le y$, we also write $\cB(x, y)=X\cB(x, y)$ and $\cB'(x, y)=X'\cB'(x, y)$, since it is actually independent of $X$ or $X'$.
	\begin{corollary}  \label{cor:exp-bro}
		Take $X_*\sim \Exp(\lambda_0/2-\lambda_1/2)$.
		We can couple $\cB_0, \cB_1, X_*$ and $\cB_0', \cB_1'$ together, such that $\cB_0, \cB_1, X_*$ are independent of each other, and $\cB_0', \cB_1'$ are independent of each other, and
		\begin{align}
\label{eq:exp-bro-c1}
	\cB(x,y) &= \cB'(x,y), \qquad \text{ for each  } 0\le x \le y, \\
\label{eq:exp-bro-c2}
	\cB_1'(x) &= \cB_1(x) \vee \max_{0\le y \le x} \cB_1(x) - \cB_1(y) + \cB_0(y) - X_* , \\
 \label{eq:exp-bro-c3}
	X_* &= \max_{x\ge 0} \cB_0'(x)-\cB_1'(x).
		\end{align}		
	\end{corollary}
	\begin{proof}
		The coupling can be obtained as follows.
		From Lemma \ref{lem:exp-bro-swap}, we get a coupling between $X\cB$, $X'\cB'$,
		such that $X\cB(x, y)=X'\cB'(x, y)$ almost surely for all $x \le y$.
		Under this coupling, we send $\beta_{-1}\to \lambda_0/2$ from above, so that $X(-1, 0) \to \infty$ and $X'(-1, 1) \to\infty$, and let $X_*=X(-1,1)$. Therefore we get a coupling between $\cB_0, \cB_1, X_*$ and $\cB_0', \cB_1'$.
		We next check that this coupling satisfies all the conditions. 
		
		The independence and \eqref{eq:exp-bro-c1} are obvious.
		As for \eqref{eq:exp-bro-c2}, since $X'(-1, 1) \to \infty$ as $\beta_{-1}\downarrow \lambda_0/2$, 
		\begin{equation}
			\label{E:B1B'}
\cB_1'(x) = \lim_{\beta_{-1}\downarrow \lambda_0/2}X'\cB'(-1,x)-X'\cB'(-1,-1) = \lim_{\beta_{-1}\downarrow \lambda_0/2}X\cB(-1,x)-X\cB(-1,-1).
		\end{equation}
		On the other hand, before sending $\beta_{-1}\downarrow \lambda_0/2$, we have
		\[
		X\cB(-1,x)=
		\left(\cB_1(x)+X(-1,0)+X(-1,1)\right) \vee \left( \max_{0\le y \le x} \cB_1(x) - \cB_1(y) + \cB_0(y) + X(-1,0)\right).
		\]
		By subtracting $X\cB(-1,-1)=X(-1,0)+X(-1,1)$ from both sides and taking the limit, \eqref{eq:exp-bro-c2} follows.
		
		As for \eqref{eq:exp-bro-c3}, before sending $\beta_{-1}\downarrow \lambda_0/2$, for all $x$ large enough we have
		\[
		X_*=X(-1,1) = X\cB(-1,-1)+\cB(0,x) - X\cB(-1,x),
		\]
		since the slope of $\cB_0$ is larger than that of $\cB_1$.
		Thus since $X\cB(\cdot, \cdot)=X'\cB'(\cdot, \cdot)$ almost surely,
		\[
		X_* = X'\cB'(-1,-1)+\cB'(0,x) - X'\cB'(-1,x),
		\]
		for all $x$ large enough. Now sending $\beta_{-1}\downarrow \lambda_0/2$, and using \eqref{E:B1B'}, we have that 
		\[
		X_*=\cB'(0,x)-\cB_1'(x),
		\]
		for all $x$ large enough. Since the slope of $\cB_1'$ is larger than that of $\cB_0'$, almost surely the process $\cB_0' - \cB_1'$ has a unique maximum, and for all $x$ large enough we have
		\[
		\cB'(0,x) = \cB_1'(x) + \max_{y\ge 0} - \cB_1'(y) + \cB_0'(y).
		\]
		Thus we get \eqref{eq:exp-bro-c3}.
	\end{proof}
	\begin{remark}
	We note that Corollary \ref{cor:exp-bro} may alternately be proven from Brownian Burke's theorem (see e.g., \cite[Theorem 2]{OYbb}).
	\end{remark}
    
\subsection{Half-space KPZ horizon}   \label{ssec:hsKPZhor}

We end this section by constructing the (extended) half-space KPZ horizon, via taking a scaling limit of the half-space log-gamma polymer to the half-space KPZ equation.

Take parameters $\alpha$ and $\apar_1, \dots, \apar_k$ satisfying 
$\apar_1< \apar_2< \cdots < \apar_k\le 0\wedge \alpha$,
and let $\theta=\frac{1}{2}+\sqrt{n}$ for $n\in\N$ that is large enough.
For $R^{\LG}=\{R^{\LG}_i\}_{i=1}^k$ from Proposition \ref{P:stationary-existence}, we define the rescaling:
\[
R^{(n)}_i(x) = \log\left( R^{\LG}_i(n^{1/2}x) \right) + \frac{n^{1/2}\log(n)|x|}{2},
\]
for each $i\in\II{1, k}$ and $x\in\R$.
Here we extend each $R^{\LG}_i$ to be defined on $\R$, by linearly interpolating between integers.
	  
	\begin{prop}
		\label{P:half-space-KPZ-horizon-marginals}
	As $n\to\infty$, the process $R^{(n)}$ has a distributional limit $R^{\KPZ}=\{R^{\KPZ}_i\}_{i=1}^k$, under the uniform-on-compact topology. 
    We can describe the process $R^{\KPZ}$ explicitly as follows.
	
Let $\cB = (\cB_1, \cB_2, \ldots, \cB_{2k})$ be $2k$ independent Brownian motions on $\R_{\ge 0}$ of diffusivity $1$, with drift $-\gamma_i$ for $\cB_i$ and $\gamma_i$ for $\cB_{2k+1-i}$, for each $i\in\II{1, k}$.
We consider the O'Connell-Yor polymer partition function (see \cite{OYbb}) across $\{\cB_i\}_{i=1}^{2k}$, defined as
\[
Z_{\cB}(x,i;y,j) = \idotsint_{x=x_i\le \cdots \le x_j \le x_{j+1}=y} \exp\left( \sum_{i'=i}^j \cB_{i'}(x_{i'+1})-\cB_{i'}(x_{i'}) \right) dx_{i+1}\cdots dx_j,
\]
for any $1\le i \le j \le 2k$, and $0\le x \le y$.

Next, let $\{X(-i, j) : i\in\II{1, k}, j\in\II{i, 2k +1 - i}\}$ be an array of inverse gamma random variables, independent of each other and $\cB$, satisfying 
\[X(-i, j) \sim \begin{cases}
1 & j =i,\\
\Gai(\gamma_j-\gamma_i), \qquad &j\in\II{ i + 1, k}, \\
\Gai(-\gamma_i- \gamma_{2k + 1 -j}), \qquad &j\in\II{k + 1, 2k - i}, \\
\Gai(-\gamma_i+\alpha), \qquad &j =  2k + 1- i.
\end{cases}
\]
Then:
\[
R^{\KPZ}_i(x)=
\begin{cases}
\log\left(\frac{\sum_{i\le \ell \le 2k} \big(\sum_{\pi\in P[i,\ell]}\prod_{v\in \pi}X(v)\big)Z_{\cB}(0,\ell; x, 2k)}{
\sum_{\pi\in P[i,2k]}\prod_{v\in \pi}X(v) }\right), \qquad & x\ge 0,\\
\log\left(\frac{\sum_{i\le \ell \le 2k} \big(\sum_{\pi\in \hat{P}[i,\ell]}\prod_{v\in \pi}X(v)\big)Z_{\cB}(0,\ell; |x|, 2k)}{
\sum_{\pi\in \hat{P}[i,2k]}\prod_{v\in \pi}X(v) }\right), \qquad & x< 0,
\end{cases}
\]
where $P[i,\ell]$ is the collection of all up-right paths in $\Z^2$ from $(-i,i)$ to $(-1,\ell)$, and  
$\hat{P}[i,\ell]$ is the subset of such paths that intersect $\{(x,x+2k+1): x\in \Z\}$.
\end{prop}

Similar to Proposition \ref{P:half-space-horizon-marginals} and the finite case (in the proof of Proposition \ref{P:stationary-existence}), if $\gamma_k=\alpha$, $X(-k,k+1)\sim\Gai(0)$, which we interpret as $\infty$. Then $R^{\KPZ}_k$ should be viewed as a limit, where we  approximate $X(-k, k + 1)$ by a sequence of numbers $\to\infty$.

The proof of Proposition \ref{P:half-space-KPZ-horizon-marginals} follows from tracking the construction of $R^{\LG}$ in Proposition \ref{P:stationary-existence}, and the rescaled convergence of the sums of the logarithms of i.i.d.~$(\Gai\big(\frac{1}{2}+\sqrt{n}+\gamma_i\big)$ (resp.~$\Gai\big(\frac{1}{2}+\sqrt{n}-\gamma_i\big)$) random variables to $\cB_i$ (resp.~$\cB_{2k+1-i}$) for each $i\in\II{1,k}$ (via Donsker's invariance principle).
We omit the details.

Note that when $k=1$, $R^{\KPZ}_1$ is precisely the stationary measure of the half-space KPZ equation, as given in \cite[Definition 1.2]{BC}.

We have that $R^{\KPZ}$ is jointly stationary for the half-space KPZ equation, which we recall now.
Following the setup in \cite{BC}, the half-space KPZ equation, subject to Neumann boundary condition
with parameter $\alpha$ is the following:
\[
\partial_t \mathbf{F}(x,t) = \frac{1}{2}\partial_x^2 \mathbf{F}(x,t) + \frac{1}{2}(\partial_x \mathbf{F}(x,t))^2 + \xi(x,t),
\]
\[
\partial_x \mathbf{F}(x,t)|_{x=0} = \alpha,
\]
where $\xi$ is space-time white noise. This is defined for $x, t \ge 0$, via the Hopf-Cole transform $\mathbf{F}(x,t)=\log \mathbf{Z}(x,t)$, where $\mathbf{Z}$ solves the half-line stochastic heat equation (SHE) with Robin boundary condition:
\[
\partial_t \mathbf{Z}(x,t) = \frac{1}{2}\partial_x^2 \mathbf{Z}(x,t) + \mathbf{Z}(x,t)\xi(x,t),
\]
\[
\partial_x \mathbf{Z}(x,t)|_{x=0} = \left( \alpha - \frac{1}{2}\right) \mathbf{Z}(0,t),
\]
with initial data being $\exp(H(x,0))$ for $x\ge 0$.
The precise notion of a solution to these equations is given in \cite[Definition 5.1]{BC}.
\begin{lemma}
For the above $R^{\KPZ}=\{R^{\KPZ}_i\}_{i=1}^k$, 
let $\mathbf{F}_i(x,t)$ be the solution to the half-space KPZ equation with boundary parameter $\alpha$ and initial data $R_i^{KPZ}|_{\R\ge 0}$, for each $i\in\II{1, k}$, and let them be coupled using the same space-time white noise $\xi$.
Then for any $t>0$, $\{\mathbf{F}_i(\cdot, t)-\mathbf{F}_i(0,t)\}_{i=1}^k$ has the same distribution as $\{R^{\KPZ}_i|_{\R\ge 0}\}_{i=1}^{k}$.
\end{lemma}
The proof of this is by taking the scaling limit from Lemma \ref{lem:sta}, and is verbatim from the derivation of \cite[Theorem 1.2]{BC} from  \cite[Theorem 1.8]{BC}.  We omit the details.

Finally, as in the many other settings, we can extend $R^{\KPZ}$ to a horizon.
    \begin{theorem}
    \label{T:half-KPZ-horizon}
    Let $\alpha \in \R$. Then there exists a random function $\cH^{\KPZ}$, which we call the \emph{(extended) half-space KPZ horizon}, such that $-\cH^{\KPZ}\in \mathcal D( (-\infty, 0\wedge \alpha], \R, \alpha \wedge 0)$, and for any finite set $I=(\gamma_1<\cdots<\gamma_k) \subset (-\infty, 0\wedge \alpha]$, the process $\cH^{\KPZ}|_I$ is has the same distribution as $R^{\KPZ}$ from Proposition \ref{P:half-space-KPZ-horizon-marginals}.
    \end{theorem}

    \begin{proof}
    Let $\hat \Q := \Q \cap (-\infty, 0\wedge \alpha]$.
     The consistency of the construction in Proposition \ref{P:half-space-KPZ-horizon-marginals} follows immediately from Lemma \ref{lem:consis} upon taking a limit. Thus we can define $\cH^{\KPZ}$ restricted to the set $\hat \Q \times \R$. Moreover, for $\gamma_1 < \gamma_2 \in \hat \Q$ and $x_1 < x_2 \in \R$, we have that
     $$
     \cH^{\KPZ}(\gamma_1, x_1) + \cH^{\KPZ}(\gamma_2, x_2) - \cH^{\KPZ}(\gamma_1, x_2) - \cH^{\KPZ}(\gamma_2, x_1) \le 0,
     $$
     since this inequality holds in the prelimit in the proof of Theorem \ref{T:half-space-log-gamma-horizon}. Finally, from the construction in Proposition \ref{P:half-space-KPZ-horizon-marginals}, the process $\cH^{\KPZ}|_I$ is continuous in law as a function of the finite set $I$. Then we can extend $\cH^{\KPZ}$ to all of $(-\infty, 0\wedge \alpha] \times \R$ and check that the resulting construction has marginals given by Proposition \ref{P:half-space-KPZ-horizon-marginals} on arbitrary finite sets.
    \end{proof}

\section{Moderate deviation estimates for exponential LPP}  \label{sec:tightness}

In this section, we prove moderate deviation estimates for exponential LPP in half-space. Our main goal with these estimates will be to prove tightness in the 1:2:3 scaling when the diagonal weights are critical or subcritical. However, in many settings the estimates will give information beyond these regimes and may be of independent interest. 

In principle, many tail bounds in this section should be available from exact distribution formulas (in e.g., \cite{BBCS1,BFO,Zhang,dGMRW}). However, as in the full-space case, such formulas can be involved, and analysis of them can be highly non-trivial tasks. Therefore, our approach combines exact-solvability, together with geometric and probabilistic arguments.

In Section \ref{SS:one-point-tails-diagonal}, we prove one-point tail bounds for passage times along the diagonal, culminating in Theorem \ref{T:one-point-bds}. Here we prove our estimates in the greatest generality, and our final estimates are optimal (up-to-constants) in most regimes. Our method here uses an identity of Barraquand and Wang \cite{BW} which relates point-to-line passage times in half-space to point-to-point passage times in a full-space problem with a boundary condition. Given this identity, we derive moderate deviation estimates on point-to-line passage times in half-space by appealing to the well-known estimates in full space. We then use a novel geometric method to transfer point-to-line estimates in half-space to the point-to-point estimates we need.
(See also Remark \ref{rem:alter-route} on another potential route to get upper bound, using a different identity from \cite{BBCS1}.)

In Section \ref{SS:off-diagonal-tails}, we extend these tail bounds to give one-point tail bounds on off-diagonal entries. The method is soft, based on \textit{path-crossing inequalities} for comparing models with different boundary conditions or on different domains. The path-crossing inequalities (Proposition \ref{P:weighted-quadrangle} and its corollaries) are straightforward and may be of general interest. For example, they imply that two half-space models with the same bulk weights and monotonically coupled boundary weights differ most in their behavior exactly at the boundary.

Finally, in Section \ref{SS:two-point-bounds}, we prove two-point moderate deviation bounds which will eventually lead to tightness. There are two bounds here: 
\begin{itemize}[nosep]
    \item a spatial two-point bound (Proposition \ref{P:two-point-bd}) which translates to a H\"older-$1/2^-$ estimate and a Gaussian tail bound in the limit, and
    \item a temporal two-point bound (Proposition \ref{P:temporal-est}) which translates to a H\"older-$1/3^-$ estimate and a $\exp(-c x^{3/2})$-tail bound in the limit.
\end{itemize}
The temporal bound follows from the spatial bound and the one-point estimates using a standard method, e.g.,\ as in \cite[Lemma 10.4]{DOV}. 

For the spatial bound, estimates weaker than Gaussian (yet still sufficient for tightness) can be obtained quickly from one-point estimates, for instance via the geometric arguments in \cite[Theorem 3]{BGtim} or \cite[Lemma 2.3]{SSZinf}.
In the full-space setting, Gaussian spatial tail bounds have also been derived using Gibbs resampling arguments (see, e.g., \cite[Theorem 1.11]{HamBr}, \cite[Corollary 1.3]{calvert2019brownian}, \cite[Lemma 6.1]{dauvergne2021bulk}, \cite[Example 1.7]{dauvergne2024wiener}).
Our proof of the Gaussian spatial tail bound follows a different approach, based on comparison with stationary initial conditions, whose tails are easier to control.
Ideas of comparison with stationarity have appeared in the literature (see, e.g., \cite[Section 2]{BBSloc} for local comparison in total variation distance, and \cite[Section 2]{BFu}, \cite[Corollary 5.9]{MSZ} for finite geodesic fluctuations and coalescence). To our knowledge, this is the first time such a method has been used to derive sharp two-point bounds, particularly in the half-space setting, which requires the extended half-space stationary measures from Section \ref{sec:horizons} as input.
This method is quite flexible, and we expect that it can be applied to other models where a full range of stationary measures is available, but which may not have as much symmetry or exact solvability as full-space models.

In this section, for any $E:\Z^2\to \R_{\ge 0}$, we write the LPP passage time from $u$ to $v$ under $E$ as
\[
E(u;v)  = \max_{\pi:u\to v} \sum_{w\in \pi} E(w),
\]
where the maximum is over all up-right paths (from $u$ to $v$). 
For $E$ defined only on a subset of $\Z^2$, the passage time is taken by extending $E$ to $\Z^2$, with $E(v)=0$ for each $v$ outside the subset.
 
\subsection{One-point tail bounds on the diagonal}
\label{SS:one-point-tails-diagonal}

Recall the field $X:\Z_{\ge}^2\to \R_{\ge 0}$ with boundary parameter $\alpha>0$ from Section \ref{sssec:constru}.
We prove Theorem \ref{T:one-point-bds} in this subsection.
Note that the methods in this subsection can also be used to prove optimal moderate deviation bounds on the random variables of the form $X(1, 1; u)$ for arbitrary $u$. We have chosen to focus our efforts on the case when $u = (n, n)$ in order to reduce some technicality.

As discussed previously, our starting point for proving one-point tail bound is an identity from \cite{BW} relating point-to-line partition functions in the half-space log-gamma polymer to point-to-point partition functions in an inhomogeneous full-space log-gamma polymer. Taking a zero-temperature limit of this identity recovers an analogous identity for exponential LPP.

We state the identity in a less general form than is given in \cite{BW}. 
Similar to the previous section, for an array of non-negative weights $U:\Z^2\to \R_{\ge 0}$ and $u \le v\in\Z^2$, define the point-to-point partition function
$$
Z_U(u; v) = \sum_{\pi: u \to v} \prod_{w\in \pi} U(w),
$$
where the sum is over all up-right paths.
Now, for $u\in\Z^2$ and $(n, m) \in \Z^2_\ge$ with $u \le (n, m)$. We define the \textbf{trapezoidal point-to-line} partition function
$$
Z^\mathsf{T}_U(u; n, m) = \sum_{i=0}^{m-1} Z_U(u; n + i, m - i).  
$$
Note that unlike in the previous section, here we do not impose the constraint \eqref{E:diagonal-condition}.
\begin{theorem}[\protect{Special case of \cite[Theorem 1.4]{BW}}]
    \label{T:barraquand-wang}
Take any parameters $\alpha, \beta > 0$. Let $U, V:\Z^2\to \R_{\ge 0}$ be arrays of independent random variables, 
such that for each $j \in \Z$ and $i \in \Z \setminus \{1\}$,
$$
U(1, j) \sim \Gai(\alpha), \qquad U(i, j) \sim \Gai(\beta);
$$
and for each $(i,j)\in \Z^2_{\ge}$, 
$$
V(i, i) \sim \Gai(\alpha), \qquad V(i, j) \sim \Gai(\beta),
$$
and $V(i,j)=0$ for each $(i,j)\in\Z^2\setminus \Z^2_{\ge}$.
Then for each $(1,1)\le (n,m)\in\Z^2_{\ge}$, we have
$$
Z^\mathsf{T}_V(1,1; n, m) \eqd Z_U(1,1; n, m).
$$
\end{theorem}

Taking a limit transition, we recover the following identity for exponential LPP.

\begin{corollary}
\label{C:elpp-result}
Take $\alpha>0$ and the field $X$ from Section \ref{sssec:constru}.
Let $W:\Z^2\to\R_{\ge 0}$ be an array of independent random variables, such that for each $j\in\Z$ and $i \in \Z \setminus \{1\}$,
$$
W(1, j) \sim \Exp(\alpha), \qquad W(i, j) \sim \Exp(1).
$$
Then for each $(1,1)\le (n,m)\in\Z^2_{\ge}$, we have
$$
\max_{0 \le i \le m - 1}X(1, 1; n + i, m - i) \eqd W(1, 1; n, m).
$$
\end{corollary}

\begin{proof}
In the setting of Theorem \ref{T:barraquand-wang}, take $\alpha = \eps \tilde \alpha, \beta = \eps$ for some $\eps > 0$, and call the resulting environments used in that theorem $U_\eps$ and $V_\eps$. Then
$$
\epsilon \log Z^\mathsf{T}
_{V_\eps}(1, 1; n, m) \eqd 
\epsilon \log Z
_{U_\eps}(1, 1; n, m),
$$
which recovers the corollary with $\tilde \alpha$ upon taking $\eps\to 0$.
\end{proof}

\begin{remark}   \label{rem:alter-route}
An identity different from but similar to Corollary \ref{C:elpp-result} is \cite[Lemma 6.1]{BBCS1}, which relates two point-to-point passage times in half-space exponential LPP models.
This identity can also lead to upper tail bounds in half-space (see \cite[Section 4.3]{DominikMaxCurrent}, in particular Lemma 4.14 there), by using stationary exponential LPP exit-point or transversal fluctuation estimates (from e.g., \cite[Section 4.2]{BFu}, \cite[Section 2.6]{emrah2020right}, \cite[Theorem 3.1]{EJS21}, \cite[Theorem 2.4]{Bha}, or \cite[Lemma 5.6]{MSZ}).
\end{remark}

To derive Theorem \ref{T:one-point-bds} from Corollary \ref{C:elpp-result}, we need bounds on $W(1, 1; n, m)$, which is simply a full-space exponential passage time starting from a one-sided exponential random walk. 

Below we let $Y:\Z^2\to \R_{\ge 0}$ is an array of i.i.d.\ $\operatorname{Exp}(1)$ random variables.
Then we have
\begin{equation}
\label{E:metric-composition}
\begin{split}
W(1, 1; n, m) &= \max_{j \in \II{1, m}} [W(1, 1; 1, j) + W(2, j;n, m) 
]\\
&\eqd \max_{j \in \II{1, m}} [A(j) + Y(1, j; n-1, m)],
\end{split}
\end{equation}
where $A:\Z_{\ge 0}\to\R_{\ge 0}$ is a random walk with $\operatorname{Exp}(\alpha)$-increments.

Finding moderate deviation estimates on random variables of the above form is well-studied, e.g., see  \cite{Bha, emrah2020right}, but the estimates there do not quite apply in our setting where we need to vary $\alpha$ with $n$. Instead, we will re-derive a moderation deviation estimate by hand.

The random variable $Y(1,1;n, m)$ (resp.~$\max_{i\in\II{-n+1,n-1}}Y(1,1;n+i,n-i)$) equal in law to the largest eigenvalue in Laguerre Unitary (resp.~Orthogonal) Ensemble random matrices. An almost complete description for the tails of this random variable was developed by Ledoux and Rider \cite{ledoux2010small}, and this is the starting point for our analysis. 
To state their result, we let 
\begin{equation}   \label{eq:defmunm}
\mu(n, m) = (\sqrt{m} + \sqrt{n})^2
\end{equation}
be the limit shape function for exponential LPP. We will also write $\mu(n)=\mu(n,n)=4n$.

\begin{theorem}    [\protect{Part of \cite[Theorem 2]{ledoux2010small}}]
    \label{T:W-tail-bounds}
With $Y$ as above, there are absolute constants $C, c > 0$ such that for all $\eps \in (0, 1]$ and $n \ge m \in \N$ we have:
\begin{align*}
    \P\Big(Y(1, 1; n, m) \le \mu(n, m) (1- \eps)\Big) &\le C \exp\left( - c \eps^3 n m  \left( \tfrac{1}{\eps} \wedge \sqrt{\tfrac{n}{m}}\right)\right), \\
    \P\Big(Y(1, 1; n, m) \ge \mu(n, m) (1 + \eps)\Big) &\le C \exp\left( - c  \eps^{3/2}\sqrt{n m} \left( \tfrac{1}{\sqrt{\eps}} \wedge \left(\tfrac{n}{m}\right)^{1/4} \right)\right), \\
    \P\Big(\max_{i\in\II{-n+1,n-1}}Y(1,1;n+i,n-i) \ge 4n (1 + \eps)\Big) &\le C \exp\left( - c  \eps^{3/2}n\right).
\end{align*}
\end{theorem}

Note that the tails above are suboptimal when $\epsilon > \sqrt{m/n}$, as can be easily seen by setting $m = 1$. At this point, the true behavior moves into the large deviation regime. This difference is ultimately what leads to the $\ep \le n^{-1/3}$ constraint in Theorem \ref{T:one-point-bds}.1. The constraint  could be removed with an analysis of the large deviations of $Y(1, 1; n, m)$ which is uniform in $m/n$. This constraint disappears in the scaling limit, where the relevant regime is $\ep = O(n^{-2/3})$.

The upper tail bound on $Y(1, j;n, m)$ can be upgraded to an upper tail bound on the supremum of the process $j \mapsto Y(1, j;n, m)$ on intervals at the limiting scale.
Such estimates were first obtained in \cite[Proposition 10.5]{basu2014last} in the setting of Poissonian LPP, and the estimate below can be found as a special case of \cite[Theorem 4.2(ii)]{BGZ} (see also e.g., \cite[Proposition B.1]{HSMod} and \cite[Proposition A.2]{ZhangOE}). 
\begin{lemma}
\label{L:point-to-segment}
For every $\kappa > 1$, there exists a constant $c > 0$ such that for all $n, m \in \N$ with $\kappa^{-1} \le m/n \le \kappa$ and $\eps \in (0, 1)$, we have
$$
\P\Big(\max_{j \in \Z : |j| \le n^{2/3}} Y(1, j; n, m)  - \mu(n, m+1-j) \ge \eps m)\Big) \le 2 \exp\left( - c \eps^{3/2}m\right).
$$
\end{lemma}

By taking a union bound, we can turn Lemma \ref{L:point-to-segment} into a parabolic profile bound.

\begin{corollary}
\label{C:parabolic-profile}
Take $\kappa \in (0, 1/2)$ and $n, m \in \N$ with $n/m \in [\kappa, \kappa^{-1}]$. Consider the process 
$$
\bar Y(j) = Y(1, j; n, m) - \mu(n, m+1- j), \qquad j\in\II{1,m},
$$
and fix a point $j_0 \in [1, (1-\kappa) m]$. Take any $\delta_0 > 0$. There exists some $c > 0$ depending only on $\kappa, \delta_0$, such that for all $\eps \in (0, n^{-1/3})$ we have
$$
\P\left(\max_{j \in \II{1, m}} \bar Y(j) - \frac{\delta_0|j-j_0|^2}{m} \ge \eps m\right) \le 2\exp(- c \eps^{3/2}m). 
$$
\end{corollary} 

\begin{proof}
	In this proof, we use $c>0$ to denote a small constant that may depend on $\kappa, \delta_0$, and its value may change from line to line.
    We also assume that $m\eps^{3/2}$ is large enough (depending on $\kappa, \delta_0$), since otherwise the bound is trivial.
	
    First, for any $j \in \II{1, m}$, by the upper bound in Theorem \ref{T:W-tail-bounds} we have $\P(\bar Y(j) > \delta_0\kappa^2 m/4) \le 2e^{-c\sqrt{m}}$. Therefore by taking a union bound, we have that
	\begin{equation}
		\label{E:first-simplif}
	\P\left(\max_{j \in \II{1, m}} \bar Y(j) - \frac{\delta_0|j-j_0|^2}{m} \ge \eps m\right) \le 2m\exp(-c \sqrt{m}) + \P\left(\max_{j : |j - j_0| \le \kappa m/2} \bar Y(j) - \frac{\delta_0|j-j_0|^2}{m} \ge \eps m\right).
	\end{equation}
Next, for $k \in \Z$, define
$$
Y^+(k) = \max_{|j| \le m^{2/3}} \bar Y(j_0 + k\lfloor m^{2/3} \rfloor + j).
$$
Then letting $I = \{k \in \Z: (|k|-1)\lfloor m^{2/3} \rfloor \le (\kappa/2 \wedge \delta_0) m\}$, the probability on the right-hand side of \eqref{E:first-simplif} is bounded above by 
$$
\sum_{k \in I} \P(Y^+(k) > \eps m + \delta_0 (|k|-1)^2 m^{1/3}) \le 2 \exp(-c \eps^{3/2}m),
$$
where the final bound uses Lemma \ref{L:point-to-segment}. Combining this bound with \eqref{E:first-simplif} and using the upper bound on $\eps$ gives the result.
\end{proof}

We now combine Theorem \ref{T:W-tail-bounds} and Corollary  \ref{C:parabolic-profile} to give upper and lower tail bounds on the random variable $W(1,1; n, n)$ in the regimes we care about. Recall $\mu_\alpha(n)$ from \eqref{eq:mualphan}.
\begin{prop}
\label{P:W-upper-bd}
Let $\alpha_0 > 0$. Then for all $\eps \in (0, n^{-1/3})$, when $\alpha  \ge \alpha_0$ we have
$$
\P(W(1, 1; n, n) \ge \mu_\alpha(n) + \eps n) \le 2\exp\left(-c \eps^{3/2} n\left(\frac{\eps^{1/2}}{(1/2 - \alpha)\vee 0}\wedge 1\right) \right),
$$
where $c$ is a constant depending only on $\alpha_0$.
\end{prop}

\begin{proof}
	In this proof, we use $c>0$ to denote a small constant that may depend on $\alpha_0$, and its value may change from line to line. We also assume that $n\eps^{3/2}$ is large enough (depending on $\alpha_0$), since otherwise the bound is trivial.

We will use the metric composition law \eqref{E:metric-composition}, and work with the processes $A(j)$ and $B(j) := Y(1, j; n - 1, n)$. 
Let $\mu_A(j) = \alpha^{-1} j$ and $\mu_B(j) = \mu(n, n-j)$ be their deterministic approximations, and let $\bar A = A - \mu_A, \bar B = B - \mu_B$. Define
$$
x_0 = 
\left(\frac{n(1 - 2 \alpha)}{(1 - \alpha)^2} \right) \mathds{1}(\alpha < 1/2), \qquad j_0 = \lfloor x_0 \rfloor.
$$
We can compute by Taylor expansion that for $j \in \II{0,n}$,
\begin{align*}
\mu_A(j) + \mu_B(j) \le \mu_\alpha(n) - \delta_0\frac{(j-j_0)^2}{n},
\end{align*}
where $\delta_0 > 0$ is a constant depending only on $\alpha_0$. Next, we have the following standard random walk bounds:
\begin{align*}
\P\Big(\bar A(j) - \bar A(j_0) \le \eps n/4 + \delta_0\frac{(j-j_0)^2}{2n} \text{ for all } j \in \II{0, n} \Big) &\ge 1 - 2\exp \left(-\frac{c\eps^2 n^2}{n^{2/3} + \eps n} \right), \\
\P(\bar A(j_0) \ge \eps n/4) &\ge 1 - 2 \exp\Big(-\frac{c \eps^2 n^2}{j_0}\Big).
\end{align*}
From Corollary \ref{C:parabolic-profile} we have
\[
\P\left(\max_{j \in \II{0, n}} \bar B(j) - \frac{\delta_0|j-j_0|^2}{2n} \ge \eps n/2\right) \le 2\exp(- c \eps^{3/2}n). 
\]
Combining these bounds and using that $n^{-2/3}<\eps<n^{-1/3}$, we get
$$
\P\left( \max_{j \in \II{0, n}}  A(j) +  B(j) \ge \mu_\alpha(n) + \eps n \right) \le 2 \exp\left(-\frac{c \eps^2 n^2}{j_0}\right) + 2 \exp(-c\eps^{3/2} n ).
$$
Simplifying using the formula for $j_0$ gives the result.
\end{proof}

\begin{proof} [Proof of part 1 of Theorem \ref{T:one-point-bds}]
Using Corollary \ref{C:elpp-result} and following the notations, we have
$$
X(1, 1; n, n) \le \max \{X(1, 1; n + i, n - i) : 0 \le i \le n- 1\} \eqd W(1, 1; n, n).
$$
Then Proposition \ref{P:W-upper-bd} immediately implies the upper bound in Theorem \ref{T:one-point-bds}.1.
\end{proof}

We move on to the lower bound. We will start by using the triangle inequality to give a simple lower tail estimate on $W(1, 1; n, m)$. For $n\ge m\in \N$ and $\alpha>0$ define
\begin{equation}   \label{eq:defmunma}
\mu_\alpha(n, m) = \begin{cases}
    \frac{m}{\alpha} + \frac{n}{1 - \alpha}, \qquad &\alpha < \frac{\sqrt{m/n}}{1 + \sqrt{m/n}} \\
    \mu(n, m), \qquad &\alpha \ge \frac{\sqrt{m/n}}{1 + \sqrt{m/n}},
\end{cases} 
\end{equation}
which is the shape function for half-space LPP. Equivalently, $\mu_\alpha$ is the smallest concave function which dominates $\mu$ everywhere, and satisfies $\mu_\alpha(n, n) = \mu_\alpha(n)$ on the diagonal. 
\begin{lemma}
\label{L:alpha0-tail}
Take $n/2 \le m \le n$. For any $\eps \in (0, 1)$ we have the bound
$$
\P(W(1, 1; n, m) < \mu(n, m) - \eps n) \le 2\exp\left(- c \eps^3 n^2 \right),
$$
where $c>0$ is a universal constant.
Moreover, for any $\alpha_0 \in (0, 1/2)$, $\alpha \in [\alpha_0, 1/2 - n^{-1/3}]$, and $\eps \in (0, 1)$ with $\mu_\alpha(n, m) - \eps n \ge \mu(n, m)$, we have 
$$
\P(W(1, 1; n, m) < \mu_\alpha(n, m) - \eps n) \le 2\exp\left(-  \frac{c\eps^2 n}{1/2 - \alpha} \right).
$$
where $c > 0$ is a constant depending on only on $\alpha_0$. 
\end{lemma}

\begin{proof}
The first bound follows from the inequality 
$$
W(1, 1; n, m) \ge W(2, 1; n, m) \eqd Y(1, 1; n-1, m),
$$
together with the lower bound in Theorem \ref{T:W-tail-bounds} and the bounds on $m$. 

We move on the second bound. 
Below we use $C, c>0$ to denote large and small constants that may depend on $\alpha_0$, and their values may change from line to line.
First, it suffices to prove the bound assuming that $\eps  n (1/2 - \alpha)$ is large enough depending on $\alpha_0$, since otherwise the bound is trivial.  Set
$$
x_0 = \left(m - \frac{n \alpha^2}{(1-\alpha)^2} \right)\vee 0
<C n(1/2- \alpha).
$$
By \eqref{E:metric-composition}, letting $j_0 = \lfloor 1 + x_0 \rfloor$, we have the triangle inequality
$$
W(1, 1; n, m) \ge A_0 + B_0, \qquad A_0 := W(1, 1; 1, j_0), \quad B_0 := W(2, j_0; n, m).
$$
Observe that $j_0/\alpha + \mu(n, m - j_0) - \mu_\alpha(n, m) >- \eps n/3$ by Taylor expansion and our lower bound on $\eps n$. Therefore
\begin{align*}
\P(W(1, 1; n, m) < \mu_\alpha(n, m) - \eps n) &\le \P( A_0 < j_0/\alpha - \eps n/3) +  \P(B_0 < \mu(n, m - j_0) - \eps n/3) \\
&\le 2\exp(- c \eps^2 n^2/j_0) + 2\exp(- c \eps^3n^2) \\
&\le 2 \exp(-c \eps^2 n/(1/2 - \alpha)).
\end{align*}
Here in the second line we have used a standard random walk tail bound on $A_0$, and we have applied Theorem \ref{T:W-tail-bounds} to the second term. We have used that $\frac{n}{m - j_0}<C$ to simplify the bound from Theorem \ref{T:W-tail-bounds}. In the final line, we have used the upper bound on $x_0$, together with the lower bound of $\eps n (1/2 - \alpha).$ 
\end{proof}

By Corollary \ref{C:elpp-result}, the lower bounds in Lemma \ref{L:alpha0-tail} is equivalent to half-space point-to-line passage time lower bounds.
We must convert them to point-to-point bounds, to finish the proof of Theorem \ref{T:one-point-bds}.
For this, we require a moderation deviation estimate on constrained passage times. 

For a simply connected set $\mathbb{U} \subset \Z^2$, and any $F:\Z^2:\to \R_{\ge 0}$, and $u, v \in \mathbb{U}$, let
$$
F^\mathbb{U}(u; v) = \max_{\pi:u \to v, \pi \subset \mathbb{U}} \sum_{w\in \pi} F(w)
$$
be the passage time from $u$ to $v$ along up-right paths constrained to stay in the set $\mathbb{U}$. Bounds on constrained passage times are important tools in the study of certain events for unconstrained LPP, e.g., see \cite{basu2014last, BGtim, BGZ}. Ganguly and Hegde \cite{ganguly2023optimal} found upper and lower moderate deviation estimates with the correct exponents for constrained passage times in parallelograms under a set of general KPZ scaling assumptions on the weights. We state their result only for the lower tail under homogeneous exponential LPP (for which their assumptions are valid).

For $\ell>0$ and $n\in \N$, let
 $$
\mathbb{U}_{n, \ell} = \{v = (x-t, x + t) \in \Z^2 : |t| \le \ell n^{2/3}/2, \; 1 \le x \le n\} 
 $$
 denote a parallelogram of height $n$ and width $\ell n^{2/3}$.

 \begin{prop}[\protect{Special case of \cite[Theorem 5]{ganguly2023optimal}}]
     \label{P:constrained-estimate}
 There exist universal constants $C, c, \theta_0, n_0 > 0$ such that for all $\theta > \theta_0$, $n > n_0$ and $C\theta^{-1} < \ell < 2 n^{1/3}$, we have
 \begin{align*}
    \P(Y^{\mathbb{U}_{n,\ell}}(1, 1; n, n) - 4 n \le - \theta n^{1/3}) \le \exp \left(- c (\ell \theta^{5/2} \wedge \theta^3) \right).
 \end{align*}
 \end{prop}

We can now combine Lemma \ref{L:alpha0-tail} and Proposition \ref{P:constrained-estimate} to prove the two lower bounds in Theorem \ref{T:one-point-bds}.

\begin{proof}[Proof of parts 2 and 3 of Theorem \ref{T:one-point-bds}]
We assume that $n$ is large enough (depending on $\alpha_0$), since otherwise the bounds are trivial. 
Below we use $c>0$ to denote a small constant that may depend on $\alpha_0$, and its value may change from line to line.

We will prove both bounds with $2n + 1$ in place of $n$; the extension to even $n$ easily follows from monotonicity of in $n$. We will also prove the two bounds in parallel. Note that for the bound \eqref{E:X-type}, it suffices to consider the case $\alpha = \infty$, when all diagonal weights are $0$, since the passage time $X(1, 1; n, n)$ is monotone decreasing in $\alpha$. Therefore throughout the proof, we assume that either:
\begin{itemize}[nosep]
    \item $\alpha = \infty$, and we are aiming to prove \eqref{E:X-type}. Here we assume that $\eps n^{2/3}$ is large enough, since otherwise the bound is trivial. 
    \item $\alpha \in [\alpha_0, 1/2- (2n+1)^{-1/3}]$, and we are aiming to prove \eqref{E:X-type-above-4}. Here we assume that $\frac{\eps^2n}{1/2-\alpha}$ is large enough (depending on $\alpha_0$), since otherwise the bound is trivial. We also assume that $\eps \le \frac{(1 - 2\alpha)^2}{\alpha(1-\alpha)}$ to ensure that $\eps n \le \mu_\alpha(2n+1) -  4(2n+1)$.\\
    Note that the above two assumptions imply that $(1-2\alpha)^3n$ is large enough (depending on $\alpha_0$), thus we can let $\alpha \in [\alpha_0, 1/2- n^{-1/3}]$. They also imply that $\eps n (1-2\alpha)$ is large enough (depending on $\alpha_0$).
\end{itemize}
To start, note that for each $m\in\II{0,n-1}$,
$$
X(1, 1; 2n + 1, 2n + 1) \ge X(1, 1; n + m, n - m) + X(n + m + 1, n - m + 1; 2n + 1, 2n + 1).
$$
The two random variables on the right-hand side above are i.i.d., and so it suffices to show that for some $m$ chosen judiciously from $\epsilon$, we can bound
\begin{equation}
\label{E:Big-R-bd}
   \P\Big(X(1, 1; n + m, n - m) \le \mu_\alpha(n) - \eps n\Big) 
\end{equation}
by the right-hand side of \eqref{E:X-type} or \eqref{E:X-type-above-4}. 

We next choose $m$ for which we prove \eqref{E:Big-R-bd}.
We will aim to play off the limit shape cost of increasing $m$ against the freedom that paths get from being far away from the diagonal, allowing us to appeal to Proposition \ref{P:constrained-estimate}. 
Thus we choose $m$ such that
\[
\eps n/24 < \mu_\alpha(n) - \mu_\alpha(n + m, n-m) < \eps n/12.
\]
Such $m$ exists when $n$ is large: indeed, $\mu_\alpha(n) - \mu_\alpha(n + m, n-m)$ equals $\mu(n)-\mu(n+m,n-m)\in [m^2/n, 2m^2/n]$ when $\alpha=\infty$, and equals $\frac{m(1 - 2\alpha)}{\alpha(1-\alpha)}$ when $\alpha \in [\alpha_0, 1/2 - n^{-1/3}]$ and $\mu_\alpha(n + m, n-m) \ge 4n$.
So we can choose $m$, satisfying 
\begin{align*}
0.1\sqrt{\eps}n<m<n/2, \qquad &\alpha = \infty, \\
\frac{0.1\alpha_0\ep n}{1/2 - \alpha}<m<n/2, \qquad &\alpha \in [\alpha_0, 1/2 - n^{-1/3}], \;\; \eps \le \frac{(1 - 2\alpha)^2}{\alpha(1-\alpha)}.
\end{align*}

We bound \eqref{E:Big-R-bd} using the following strategy.
Let 
\begin{align*}
X_{\mathsf{tot}} &= \max_{j \in \II{0, n - m-1}} X(1, 1; n + m + j, n - m - j),\\
X_{\mathsf{low}} &= \max_{t \in \II{1,n-m}} X(1, 1; 2m + t-1, t) + 4(n-m-t),\\
X_{\mathsf{up}} &= \max_{t \in \II{1, n-m}, j \in \II{0, n-m-t}} X^-(2m + t, t; n + m + j, n - m - j) - 4(n-m-t).
\end{align*}
Here $X^-:\Z^2\to \R_{\ge 0}$ is taken to be $X^-(i,j)=X(i,j)\mathds{1}(i>j)$, for each $(i,j)\in\Z^2$.
Then obviously we have $X_{\mathsf{low}} + X_{\mathsf{up}}\ge X_{\mathsf{tot}} \eqd W(1,1;n+m,n-m)$.
We also let
\begin{align*}
T &= \min \{t \in \II{1, n-m} :X(1, 1; 2m + t-1, t) + 4(n - m -t) > \mu_\alpha(n) -\eps n/2\},
\end{align*}
and let $T= \infty$ if the above set is empty.

We will prove the following statements:
\begin{itemize}
    \item[(i)] There exists a universal constant $C_0>0$, such that $\P(X_{\mathsf{up}} \le C_0n^{1/3})\ge 1/2$.
    \item[(ii)] For each $t\in \II{1,n-m}$, we have
\[
\P(\tilde X(2m + T, T; n + m, n - m) - 4(n -m - T) \le - \eps n/2 \mid T=t ) \le 2\exp\left(- c ( \eps^{5/2} mn\wedge \eps^3 n^2) \right). 
\]
\end{itemize}
Assuming (i), by taking $\eps>6C_0n^{-2/3}$ and using the FKG inequality, we have
\[
\P(X_{\mathsf{tot}} \le \mu_\alpha(n) - \eps n/3) \ge \P(X_{\mathsf{low}} \le \mu_\alpha(n) - \eps n/2, X_{\mathsf{up}} \le C_0n^{1/3}) \ge \P(X_{\mathsf{low}} \le \mu_\alpha(n) - \eps n/2)/2.
\]
On the other hand, by the choice of $m$, we have
\begin{align*}
\P(X_{\mathsf{tot}} \le \mu_\alpha(n) - \eps n/3) \le & \P(X_{\mathsf{tot}} \le  \mu_\alpha(n+m,n-m) - \eps n/4) \\ 
=& \P(W(1,1;n+m,n-m) \le \mu_\alpha(n+m,n-m) - \eps n/4)    \\ 
\le & \begin{cases}
 2\exp(-c\eps^3n^2), \quad &\alpha = \infty, \\
2\exp(-c\eps^2n/(1/2-\alpha)), \quad &\alpha \in [\alpha_0, 1/2 - n^{-1/3}], \;\; \eps \le \frac{(1 - 2\alpha)^2}{\alpha(1-\alpha)},
\end{cases}
\end{align*}
where the last inequality is by Lemma \ref{L:alpha0-tail} .
Thus we get an upper bound of $\P(X_{\mathsf{low}} \le \mu_\alpha(n) - \eps n/2)=\P(T=\infty)$. 
Now by further assuming (ii), we can upper bound \eqref{E:Big-R-bd} by
\begin{align*}
    &\P(T=\infty) + \sum_{t\in\II{1,n-m}} \P(\tilde X(2m + t, t; n + m, n - m) - 4(n -m - t) \le - \eps n/2 , t ) \\
\le & \P(T=\infty) +  2\exp\left(- c ( \eps^{5/2} mn \wedge \eps^3 n^2) \right) \\
\le & \begin{cases}
 2\exp(-c\eps^3n^2), \quad &\alpha = \infty, \\
2\exp(-c\eps^2n/(1/2-\alpha)), \quad &\alpha \in [\alpha_0, 1/2 - n^{-1/3}], \;\; \eps \le \frac{(1 - 2\alpha)^2}{\alpha(1-\alpha)}.
\end{cases}
\end{align*}
Here for the last inequality, in the first case we used $m>c\sqrt{\eps}n$ to deduce that $\eps^{5/2} mn>c\eps^3 n^2$, and in the second case we used that
\[
\eps^{5/2} mn \wedge \eps^3 n^2 > \frac{c\eps^{7/2}n^2}{1/2-\alpha} > \frac{c\eps^2n}{1/2-\alpha},
\]
where the first inequality is due to $m>c\eps n/(1/2-\alpha)$ and $\eps \le \frac{(1 - 2\alpha)^2}{\alpha(1-\alpha)}$, and the second inequality is due to that $\frac{\eps^2n}{1/2-\alpha}$ is large enough and $\eps \le \frac{(1 - 2\alpha)^2}{\alpha(1-\alpha)}$.

\textbf{Proof of (i).}
From the fact that $\{X^-(u;v)\}_{u\le v\in\Z^2}$ is stochastically dominated by $\{Y(u;v)\}_{u\le v\in\Z^2}$, and translation invariance of the field $Y$, it suffices to show that, for a universal constant $C_0>0$,
\begin{equation} \label{eq:iinY}
\P\left( \max_{t\in\II{1,n-m}, j\in\II{0,n-m-t}} Y(t,t;n-m+j,n-m-j) - 4(n-m-t) \le C_0n^{1/3} \right) \ge 1/2.
\end{equation}
Similar maximal inequalities for LPP has be proved in the literature, see e.g., \cite[Proposition 4.3]{BGZ}, whose proof can be adapted to get \eqref{eq:iinY}.
Here we instead present a proof of \eqref{eq:iinY} using ideas from \cite[Section 10.2]{dauvergne2021scaling}.
Define 
\[
S=\max \Big\{ t \in \II{1,n-m} : \max_{j \in \II{0, n-m-t}}  Y(t,t;n-m+j,n-m-j) - 4(n-m-t) > C_0n^{1/3} \Big\},
\]
and set $S=\infty$ if the set is empty. Then we want to show that $\P(S=\infty)\ge 1/2$ by taking $C_0$ large.

We argue by contradiction, and assume that $\P(S<\infty)\ge 1/2$.
For each $t\in\II{1,n-m}$, the event $S=t$ is measurable with respect to the field $Y$ in $\II{t, 2n-2m-t}\times \II{t, n-m}$, and is independent of $Y(0,0;t-1,t)$.
Then by the lower bound in Theorem \ref{T:W-tail-bounds}, we have that
$\P( Y(0,0;t-1,t) \ge 4t-C_0n^{1/3}/2 \mid S=t  ) \ge 1/2$,
when $C_0$ is large enough. We thus conclude that
\[
\P\Big(\max_{j \in \II{0, n-m}}  Y(0,0;n-m+j,n-m-j) \ge 4(n-m) + C_0n^{1/3}/2 \Big) > 1/4.
\]
However, by taking $C_0$ large, this contracts the last bound in Theorem \ref{T:W-tail-bounds}.
Thus the assumption that $\P(S<\infty)\ge 1/2$ cannot hold, and (i) is true.

\textbf{Proof of (ii).}
Note that the event $T=t$ is measurable with respect to the field $X$ in $\II{1, 2m+t-1} \times \II{1,t}$, thus is independent of $\tilde X(2m + t, t; n + m, n - m) - 4(n -m - t)$.
Therefore it suffices to lower bound $\tilde X(2m + t, t; n + m, n - m) - 4(n -m - t)$ for each $t\in\II{1, n-m}$.
When $n-m-t<n^{1/3}$ such a lower bound is obvious, since then $\eps n/2>n-m-t$ by the assumption that $\eps n^{2/3}$ is large enough from the beginning of this proof.

It now remains to consider the case where $n-m-t\ge n^{1/3}$. Such a bound follows from Proposition \ref{P:constrained-estimate}. 
Indeed, we consider $\mathbb{U}_{n-m-t+1,\ell}$ with $\ell = m (n - m - t+1)^{-2/3} \wedge (n-m-t+1)^{1/3}$.
Then $Y^{\mathbb{U}_{n-m-t+1, \ell}}(1,1;n-m-t+1, n-m-t+1)$ is stochastically dominated by $\tilde X(2m + t, t; n + m, n - m)$. 
By taking $\theta = \eps n (n - m - t+1)^{-1/3}/2$, and noting that both $\eps n^{2/3}$ and $\eps m$ are large enough in either case, we have that $\theta$ and $\theta\ell$ are large enough.
Thus we can apply Proposition \ref{P:constrained-estimate} to get (ii).
\end{proof}

\subsection{Path-crossing inequalities and tails off the diagonal}
\label{SS:off-diagonal-tails}

Our next goal is to extend the tail bounds in Theorem \ref{T:one-point-bds} off the diagonal in a way that will allow us to prove one-point tightness. The bounds are optimal up-to-constants for $\alpha \ge 1/2 - n^{-1/3}$, but our $\alpha$-dependence can be improved with some straightforward geometric arguments when $\alpha \le 1/2 - n^{-1/3}$. We do not purse that direction here.

Our core idea for the off diagonal extension is that, as we move further away from the diagonal, the passage times are closer to the full-space LPP passage times. We can make this idea rigorous through path-crossing inequalities, which we will state and prove in a general setting before specializing to half-space exponential LPP.

Consider a weighted finite directed acyclic graph $G = (V, E, M)$, where $V$ and $E$ are vertices and edges, and $M:V \to \R$ are real-valued weights. For two vertices $u, v \in G$ we write $u \prec v$ if there is at least one directed (vertex) path $\pi$ from $u$ to $v$, i.e.,a sequence of vertices $(u = \pi_0, \pi_1, \dots, \pi_k = v)$ with $(\pi_{i-1}, \pi_i) \in E$ for all $i\in\II{1, k}$. We write $|\pi|_M = M(\pi_0) + \cdots + M(\pi_k)$ for the length of $\pi$ with the weight assignment $M$. For $u \prec v$, define the passage time
$$
G(u ; v) = \max_{\pi:u \to v} |\pi|_M
$$
where the maximum is over all directed paths $\pi$ from $u$ to $v$. 
Our general path-crossing inequality is as follows. 

\begin{prop}
\label{P:weighted-quadrangle}
  Consider a weighted directed acyclic graph $G = (V, E, M)$, and let $H = (V, E, M')$ be another one with the same vertices and edges, but different weights satisfying $M' \le M$, with equality on a set $W \subset V$. Consider vertices $x_1 \prec y_1$ and $x_2 \prec y_2$ and suppose that we have the following {\bf path-crossing} property:
\begin{itemize}
    \item Let $\tau$ be any path from $x_1$ to $y_1$ and $\pi = (\pi_0, \dots, \pi_k)$ be any path from $x_2$ to $y_2$. Then if $\pi_i \notin W$ for some $i \in \II{1, k-1}$, there are indices $i^- < i < i^+$ such that the vertices $\pi(i^-), \pi(i^+)$ are on the path $\tau$.
\end{itemize}
Then
$$
H(x_1; y_1) + G(x_2 ; y_2)  \le G(x_1; y_1) + H(x_2 ; y_2).
$$  
\end{prop}

\begin{proof}
 Let $\tau = (\tau_0 = x_1,\dots, \tau_\ell = y_1)$ and $\pi = (\pi_0 = x_2, \dots, \pi_k = y_2)$ be paths such that
   $
   |\pi|_M = G(x_2 ; y_2)$ and $|\tau|_{M'} = H(x_1 ; y_1).
   $
   If $\pi \subset W$, then $G(x_2; y_2) = H(x_2 ; y_2)$, and so the inequality follows since $H(x_1 ; y_1) \le G(x_1 ; y_1)$. If $\pi\not\subset W$, then let $i^-, i^+ \in \II{0,k}$ be the smallest and largest indices, respectively, such that $\pi_{i^-}. \pi_{i^+} \in \tau$, and let $j^-, j^+ \in \II{0,\ell}$ be such that $\pi_{i^\pm} = \tau_{j^\pm}$. The indices $i^-, i^+$ exist by the path-crossing property above. Consider the paths
   $$
   \pi' = (\pi_0, \dots, \pi_{i^- - 1}, \tau_{j^-}, \dots, \tau_{j^+}, \pi_{i^+ + 1}, \dots, \pi_k), \qquad \tau' = (\tau_0, \dots, \tau_{j^- - 1}, \pi_{i^-}, \dots, \pi_{i^+}, \tau_{i^+ + 1}, \dots, \tau_\ell).
   $$
   Then $\tau'$ is a path from $x_1$ to $y_1$ and $\pi'$ is a path from $x_2$ to $y_2$. We have that
   \begin{align*}
      |\tau|_{M'} + |\pi|_{M} &= \sum_{i=0}^\ell M'(\tau_i) + \sum_{i=0}^k M(\pi_i) \\
      &= \sum_{i=0}^\ell M'(\tau_i) + \sum_{i=0}^{i^- - 1} M'(\pi_i)  + \sum_{i=i^-}^{i^+} M(\pi_i) + \sum_{i=i^+ + 1}^{\ell} M'(\pi_i) \\
      &\le |\tau'|_{M} + |\pi'|_{M'}.
   \end{align*}
   Here the second equality uses that $\pi_i \in W$ whenever $i \notin \II{i^-,i^+}$. This follows from the construction of $i^\pm$ and the path-crossing property. The final inequality uses that $M' \le M$.  
Finally, since $|\tau'|_{M} \le G(x_1, y_1)$ and $|\pi'|_{M'} \le H(x_2, y_2)$, the conclusion follows. 
\end{proof}

\begin{corollary}
\label{C:containment-quadrangle}
Consider a weighted directed acyclic graph $G = (V, E, M)$, and we let $H = (W, E|_{W\times W}, M|_{W})$ be the subgraph of $G$ induced by a vertex set $W \subset V$. Consider vertices $x_1 \prec y_1$ and $x_2 \prec y_2 \in H$ and suppose that we have the following {\bf path-crossing} property:
\begin{itemize}
    \item Let $\tau$ be any path in $H$ from $x_1$ to $y_1$ and $\pi = (\pi_0, \dots, \pi_k)$ be any path in $G$ from $x_2$ to $y_2$. Then if $\pi_i \notin W$ for some $i \in \II{1,k-1}$, there are indices $i^- < i < i^+$ such that the vertices $\pi(i^-), \pi(i^+)$ are on the path $\tau$.
\end{itemize}
Then
$$
H(x_1; y_1) + G(x_2; y_2) \le G(x_1; y_1) + H(x_2 ; y_2).
$$
\end{corollary}

\begin{proof}
Consider Proposition \ref{P:weighted-quadrangle}, and let $H_r = (V, E, M_r)$, where $M_r = M$ on $W$ and equals $r$ off of $W$. The inequality above follows from the inequality for $H_r, G$ after taking $r \to -\infty$, since in this limit $H_r(x; y) \to H(x; y)$ for all $x, y$.
\end{proof}

We next apply Proposition \ref{P:weighted-quadrangle} and Corollary \ref{C:containment-quadrangle} to  half-space exponential LPP. Let $Y:\Z^2\to \R_{\ge 0}$ be an array of i.i.d.\ $\Exp(1)$ random variables.
For each $\alpha \in \R_+\cup\{\infty\}$, define $X_\alpha:\Z_\ge^2\to \R_{\ge 0}$ such that
$X_\alpha(i, j) = Y(i, j)$ for $i > j \in \Z$ and $X_\alpha(i, i) = Y(i, i)/\alpha$ for $i \in \Z$ (interpreting $X_\infty(i,i)=0$).

The next corollary is an immediate application of the above path-crossing inequalities.
\begin{corollary}
\label{C:exp-quadrangle}
For $i = 1, 2$, consider points $u_i = (x_i, n) \le v_i = (y_i, m) \in \Z^2_\ge, i = 1, 2$ with $x_1 \le x_2$ and $y_1 \le y_2$. Then:
\begin{align*}
    X_\infty(u_1; v_1) + Y(u_2; v_2) &\le Y(u_1; v_1) + X_\infty(u_2; v_2) \\
    X_\beta(u_1; v_1) + X_\alpha(u_2; v_2) &\le X_\alpha(u_1; v_1) + X_\beta(u_2; v_2), \qquad \text{ for any } \alpha \le \beta \in \R_+\cup\{\infty\}.
\end{align*}
\end{corollary}

We now extend the tail bounds on the diagonal in Theorem \ref{T:one-point-bds} to give moderation deviation tail bounds for half-space exponential LPP. 
Recall the functions $\mu$ and $\mu_\alpha$ from the \eqref{eq:defmunm} and \eqref{eq:defmunma}.
\begin{theorem}
	\label{T:half-space-full-tails} 
For $n, m\in \N$ and any $\alpha>0$, let
$$
\Delta_\alpha(n) = \mu_\alpha(n) - \mu(n)
$$
be the gap in the limit shape on the diagonal, and
$$
\nu_\alpha(n, m) = \mu(n, m) +  \Delta_\alpha(m).
$$
Fix $\alpha_0>0$, and take any $\alpha \ge \alpha_0$.
Then there exists a constant $c> 0$ depending only on $\alpha_0$, such that for all $v + (1, 1), v + (n, m) \in \Z^2_\ge$ with $n/m \in [\alpha_0, \alpha_0^{-1}]$ and all $\ep \in (0, m^{-1/3}]$ we have
\begin{align}
\label{E:half-space-upper-bd}
	\P\Big(X_\alpha(v + (1, 1); v + (n, m)) \ge \nu_\alpha(n, m) + \eps n\Big) &\le 2\exp\left(-c \eps^{3/2}n \left(1 \wedge \frac{\eps^{1/2}}{(1/2 - \alpha)\vee 0}\right) \right).
    \end{align}
Moreover, for all $v + (1, 1), v + (n, m) \in \Z^2_\ge$ with $n/m \in [\alpha_0, \alpha_0^{-1}]$ and all $\ep \in (0, 1)$ we have
    \begin{align}
	\label{E:half-space-lower-bd}
	\P\Big(X_\alpha(v + (1, 1); v + (n, m)) \le \mu(n, m) - \eps n\Big) &\le 2 \exp\left( - c\eps^3 n^2 \right).
\end{align}	
\end{theorem}

\begin{proof}
By translation invariance, we can assume that $v = (t, 0)$ for some $t \in \Z_{\ge 0}$. For each $i, j \in \Z$, we let 
$$
\bar X_\alpha(i, 1; j, m) = X_\alpha(i, 1; j, m) - \nu_\alpha(j-i + 1, m), \qquad \bar X_\infty(i, 1; j, m) = X_\infty(i, 1; j, m) - \mu(j-i + 1, m).
$$
We start with \eqref{E:half-space-upper-bd}.  Let $u_1 = (1, 1), v_1 = (m, m)$, $u_2 = (t + 1, 1), v_2 = (t + n, m)$. By Corollary \ref{C:exp-quadrangle}, 
\begin{align*}
\bar X_\alpha(u_2; v_2) \le \bar X_\alpha(u_1; v_1) + \bar X_{\infty}(u_2; v_2) - \bar X_{\infty}(u_1; v_1) \le \bar X_\alpha(u_1; v_1) + Y(u_2; v_2)-\mu(n,m) - \bar X_{\infty}(u_1; v_1).
\end{align*}
Therefore by a union bound
\begin{align*}
&\P(\bar X_\alpha(u_2; v_2) \ge \eps n) \\
\le\; &\P(\bar X_\alpha(u_1; v_1) \ge \eps n/3) + \P(Y(u_2, v_2) \ge \mu(n,m) + \eps n/3) + \P(\bar X_{\infty}(u_1; v_1) \le -\eps n/3)).
\end{align*}
We can bound the first and third terms using Theorem \ref{T:one-point-bds}, and the second term using Theorem \ref{T:W-tail-bounds}.
These give \eqref{E:half-space-upper-bd}.

We move on to the lower tail bound \eqref{E:half-space-lower-bd}, using the same $u_i, v_i$ notation. By monotonicity, it suffices to prove this bound for $X_\alpha$ replaced by $X_\infty$. By Corollary \ref{C:exp-quadrangle}, we have the inequality
\[
   \bar X_{\infty}(u_2; v_2) + Y(u_1; v_1) - \mu(m) \ge Y(u_2; v_2) - \mu(n,m) + \bar X_{\infty}(u_1; v_1),
\]
and so
\begin{multline*}
			\P(\bar X_{\infty}(u_2; v_2) \le -\eps n, Y(u_1; v_1) \le \mu(m)+\eps n/2)  \\ \le \P(Y(u_2; v_2) \le \mu(n,m) -\eps n/4) + \P(\bar X_{\infty}(u_1; v_1) \le -\eps n/4).
\end{multline*}
The right-hand side above is bounded above by the right-hand side of \eqref{E:half-space-lower-bd} by combining the lower bounds in Theorem \ref{T:one-point-bds} with Theorem \ref{T:W-tail-bounds}. On the other hand, we can apply the FKG inequality to the left-hand side above to get that, whenever $\eps n^{2/3}$ is large enough,
\begin{align*}
\P(\bar X_{\infty}(u_2; v_2) \le -\eps n) &\le 2\P(\bar X_{\infty}(u_2; v_2) \le -\eps n)\P( Y(u_1; v_1) \le \mu(m)+\eps n/2)\\
&\le 2\P(\bar X_{\infty}(u_2; v_2) \le -\eps n, Y(u_1; v_1) \le \mu(m)+\eps n/2),
\end{align*}
where the first inequality holds by Theorem \ref{T:W-tail-bounds}. When $\eps n^{2/3}$ is not large enough, we note that \eqref{E:half-space-lower-bd} holds obviously.
Thus we finish the proof.
\end{proof}

We finish this subsection with a corollary that applies to exponential LPP in a symmetric environment. For this, we extend $X_\alpha$ to $\Z^2$ by setting $X_\alpha(i, j) = X_\alpha(j, i)$ for all $i, j \in\Z$.

\begin{corollary}
\label{C:symmetric-version}
For all $v + (1, 1), v + (n, m) \in \Z^2$ with $n/m \in [\alpha_0, \alpha_0^{-1}]$ and $\ep \in (0, 1)$, \eqref{E:half-space-lower-bd} holds.
\end{corollary}

\begin{proof}
By symmetry, it remains to consider the case where $v + (1, 1) = (t, 1)$ and $v + (n, m) = (t + n, m)$ for some $t\in \Z$, $t\ge 2$, and $t+ n < m$. Let $(i, i)$ be the point on the diagonal that is closest to the line segment connecting $(t, 1)$ to $(m, t + n)$ (choose an arbitrary one if there is a tie). Then we must have $i\ge 2$, and
\begin{align*}
    X_\alpha(t, 1; t + n, m) \ge X_\alpha(t, 1; i, i-1) + X_\alpha(i, i; t + n, m) = X_\alpha(t, 1; i, i-1) + X_\alpha(i, i; m, t + n).
\end{align*}
By lower bounding each of the two terms using Theorem \ref{T:half-space-full-tails}, the conclusion follows.
\end{proof}

\subsection{Two-point tail bounds}
\label{SS:two-point-bounds}

Our next goal is to prove two-point tail bounds, both in space and in time. Our two-point tail bounds have a narrower parameter range than in the one-point setting, but will still be sharp when we pass to the limit. Along the way, we will prove spatial profile bounds that will be helpful when we apply the metric composition law.

To start, we recall from Lemma \ref{lem:staLPP} the following representation for the (extended) half-space exponential LPP stationary measure. 
As we do not need to consider the joint stationary process here or the stationary measure with the maximal slope parameter (i.e., under the notations of Lemma \ref{lem:staLPP},  we only consider the case where $k=1$ and $\gamma_1<\alpha$), we can make the presentation more explicit. 

Take the field $X:\Z_{\ge}^2\to \R_{\ge 0}$ with boundary parameter $\alpha>0$ from Section \ref{sssec:constru}, and extend it to $\Z^2$ by setting $X(j,i)=X(i,j)$ for each $i, j \in \Z$.
Take another parameter $\beta \in ((1/2 - \alpha)\vee 0, 1/2)$, and consider independent exponential random variables
$$
Y(2, 2) \sim \operatorname{Exp}(\beta + \alpha- 1/2), \qquad Y(j, 2) \sim  \operatorname{Exp}(1/2 + \beta), \qquad Y(j, 1) \sim \operatorname{Exp}(1/2 - \beta), \qquad j\ge 3,
$$ 
and let $Y(2, 1) = 0$. Define the process $Z^+:\Z \to \R_{\ge 0}$ by
$$
Z^+(j) = Y(2, 1; 2 + j, 2) - Y(2, 2) = \max_{2 \le k \le j + 2} \left(\sum_{i=2}^k Y(i, 1) + \sum_{i=k}^{j+2} Y(i, 2)\right) - Y(2, 2) ,
\qquad j \ge 0,
$$
$$
Z^+(j) = \sum_{i=3}^{|j|+2} Y(i, 2), \qquad j \le -1.
$$
Finally, we set $Z^-(j) = Z^+(-j)$. Then both of the initial conditions $Z^\pm$ are stationary for the symmetric environment $X$, in the sense of \eqref{E:stationary-lpp}. 
Indeed, under the notation of Lemma \ref{lem:staLPP}, $Z^+$ is $R^{\EL}_1$ with boundary parameter $\alpha-1/2$, bulk parameter $1/2$, and slope parameter $(-\beta)$ and $k=1$.
Because $Z^+$ is stationary and the environment is symmetric, its reflection $Z^-$ is also stationary.

The next lemma records two-point bounds and modulus of continuity estimates for the stationary measures $Z^\pm$.
\begin{lemma}
\label{L:basic-stationary-est}
Fix $\alpha_0 \in (0, 1/2)$. There is $c>0$ depending only on $\alpha_0$ such that the following is true. Take parameters $\alpha, \beta$ such that $\alpha > \alpha_0$, $\beta \in (2(1/2 - \alpha)\vee 0, 1/2 - \alpha_0)$. Set
$$
\bar Z^+(i) =\begin{cases}
	Z^+(i) - \frac{i}{1/2 - \beta}, \qquad i \ge 0 \\
	Z^+(i) - \frac{|i|}{1/2 + \beta}, \qquad i < 0.
\end{cases}
$$
 Then for any $i < j \in \Z$ and $m > 0$ we have that
\begin{equation}
	\label{E:two-point-stat}
\P\Big(\max_{r < s \in \II{i,j}} |\bar Z^+(s) - \bar Z^+(r)| \ge m\Big) \le 2\exp \left(-\frac{cm^2}{m + (j - i)} \right).
\end{equation}
Moreover, for any $j_0 \in \Z, \delta > 0,$ and $m > 0$ we have that
\begin{equation}
	\label{E:shape-bd-stat}
	\P(|\bar Z^+(j_0 + i) - \bar Z^+(j_0)| \le m + \delta i^2 \text{ for all } i \in \Z) \ge 1 - 2\exp \left(-\frac{cm^2}{m + \sqrt{m/\delta}} \right).
\end{equation}
\end{lemma}

To prove Lemma \ref{L:basic-stationary-est}, we require the following statement, which can be viewed as Burke's theorem for exponential random walks.
\begin{prop}
    \label{P:burke-exponential}
In the setting above, if we replace $Y(2, 2)$ with an independent random variable $\sim \operatorname{Exp}(2 \beta)$, then $(Z^+(j), j \ge 0)$ has the law of a random walk whose increments are $\operatorname{Exp}(1/2 - \beta)$.
\end{prop}
\begin{proof}
This follows from the exponential case of Lemma \ref{lem:couprLPP} (with $a=2, b=\infty$, $\gamma_1=\beta$, each $\beta_i=1/2$, and sending $\gamma_0$ to $-\beta$ from above).    
\end{proof}

\begin{proof}[Proof of Lemma \ref{L:basic-stationary-est}]
In this we use $c>0$ to denote a small constant that depends only on $\alpha_0$, and its value may change from line to line.
We first prove \eqref{E:two-point-stat}. Note that it suffices to prove it when $ij\ge 0$, as the general case then follows by the triangle inequality. For $j \in \N$ define random walks with mean-$0$ increments:
$$
Y_1(j) = \sum_{i=2}^{j+2} Y(i, 1) - \frac{j}{1/2 - \beta}, \qquad Y_2(j) = \sum_{i=2}^{j+2} Y(i, 2) - \frac{j}{1/2 + \beta}.
$$
By standard random walk bounds (i.e.,Doob's maximal inequality applied to $e^{\lambda Y_k}$ and a Chernoff bound), for any $m > 0$ and $0 \le i \le j$, and $\iota = 1, 2$ we have the estimate
\begin{equation}
	\label{E:doob}
\P\Big(\max_{r \in \II{i,j}} |Y_\iota(r) - Y_\iota(i)| \vee  |Y_\iota(r) - Y_\iota(j)|\ge m\Big) \le 2\exp \left(-\frac{cm^2}{m + (j - i)} \right),
\end{equation}
This immediately implies the bound \eqref{E:two-point-stat} for $i<j \le 0$.

We move on to bound when $0 \le i < j$. For any $0 \le r < s$, we have
\begin{align*}
\bar Z^+(s) - \bar Z^+(r)  \le\; &\max_{r+3 \le k \le s + 2} \left( \sum_{\ell=r+3}^k Y(\ell, 1) + \sum_{\ell=k}^{s+2} Y(\ell, 2) \right) - \frac{(s-r)}{1/2 - \beta} \\
\le\; &\max_{r+1 \le k \le s} \left( Y_1(k) - Y_1(r) + Y_2(s) - Y_2(k-1) \right) + \frac{1}{1/2 -\beta}.
\end{align*}
Together with \eqref{E:doob}, this implies the upper tail bound in \eqref{E:two-point-stat} (i.e., \eqref{E:two-point-stat} with the absolute value removed). Here we have used that $1/(1/2 -\beta) \le \alpha_0^{-1}$ to remove the constant term.

The lower tail bound in \eqref{E:two-point-stat} when $0 \le i < j$ is more delicate. For this, we use Proposition \ref{P:burke-exponential}.
Consider the field $Y^*:\II{2,\infty}\times\II{1,2}\to \R_{\ge 0}$, which equals $Y$ except for that $Y^*(2,2)\sim \Exp(2\beta)$, and is independent from $Y$.
Define
$$
Z^*(j) = Y^*(2, 1; 2 + j, 2) - Y^*(2,2).
$$
The process $\bar Z^*(j) = Z^*(j) - j/(1/2 - \beta)$ is now simply a centered exponential random walk, per Proposition \ref{P:burke-exponential}, and so satisfies the tail bound \eqref{E:two-point-stat}.

If $2 \beta \le \beta + \alpha - 1/2$, then we can couple the environments so that $Y^*(2,2) > Y(2, 2)$ almost surely. In this coupling, by Proposition \ref{P:weighted-quadrangle}, for all $0 \le r < s$ we have that
$$
Z^*(s) - Z^*(r) \le Z^+(s) - Z^+(r),
$$
so the lower tail in \eqref{E:two-point-stat} follows from the lower tail for $\bar Z^*$. Now suppose $2 \beta \ge \beta + \alpha - 1/2$, and define 
\begin{align*}
    p(x) &= \P\Big(\max_{r < s \in \II{i,j}} \bar Z^+(s) - \bar Z^+(r) \le - m \Big\mid  Y(2,2) = x\Big) \\
    &= \P\Big( \max_{r < s \in \II{i,j}}  \bar Z^*(s) - \bar Z^*(r) \le - m \Big\mid  Y^*(2,2) = x\Big).
\end{align*}

For any $\kappa > 0$, we can then write
\begin{align}
\nonumber
\P\Big(\max_{r < s \in \II{i,j}} &\bar Z^+(s) - \bar Z^+(r) \le - m\Big) \\
\nonumber
&= (\beta + \alpha - 1/2) \int_0^\infty p(x) e^{-(\beta + \alpha - 1/2) x} dx \\
\nonumber
&\le e^{(\beta - \alpha + 1/2)\kappa} 2 \beta \int_0^\kappa p(x) e^{-2 \beta x} dx  + (\beta + \alpha - 1/2)\int_\kappa^\infty e^{-(\beta + \alpha - 1/2) x} dx \\
\label{E:RHS-integral}
&\le e^{(\beta - \alpha + 1/2)\kappa} 2\exp \left(-\frac{cm^2}{m + (j - i)} \right)  + e^{-(\beta + \alpha - 1/2) \kappa}.
\end{align}
In the second inequality, we have evaluated the second integral, and used the tail bound \eqref{E:two-point-stat} for $\bar Z^*$ on the first integral. Now, by the lower bound $\beta \ge 2(1/2 - \alpha)\vee 0$ we have that $e^{(\beta - \alpha + 1/2)\kappa} \le e^{3(\beta + \alpha - 1/2) \kappa}$. 
Therefore by taking appropriate $\kappa$, we can upper bound  \eqref{E:RHS-integral} by the right-hand side of \eqref{E:two-point-stat}.

Finally, for the bound \eqref{E:shape-bd-stat}, by a union bound we can write
\begin{align*}
&\P(|\bar Z^+(j_0 + i) - \bar Z^+(j_0)| > m + \delta i^2 \text{ for some } i \in \Z)  \\
\le &\sum_{k \in \N} \P\Big(\max_{j \in \Z : |j| \le 2^k \sqrt{m/\delta}}|\bar Z^+(j_0 + j) - \bar Z^+(j_0)| > 2^{2k-2} m\Big) \\
\le &\sum_{k \in \N} 2 \exp \left( -\frac{ c2^{4k} m^2}{2^k \sqrt{m/\delta} + 2^{2k-2} m}\right).
\end{align*}
Here the final bound is by \eqref{E:two-point-stat}. Simplifying this expression gives the result.
\end{proof}

Next, we will use Lemma \ref{L:basic-stationary-est} to bound the two-point spatial tail. The basis for comparison is the following inequality.
Recall the $L$-shaped region notation $[\cdot, \cdot]_{\textsf{L}}$ from \eqref{eq:Lshapedef}. 
For $E:\Z^2\to\R$ and $u,v\in\Z^2$, with $[u]_{\textsf{L}}\le [v]_{\textsf{L}}$, we denote
\[
E[u;v]_{\textsf{L}} = E([u]_{\textsf{L}};[v]_{\textsf{L}}).
\]
Moreover, for $f:\Z\to\R$, recall the notation $E^f$ from \eqref{eq:Efdef}.
\begin{lemma}
\label{L:basis-for-comparison}
Take $x \in \Z$ and $y_1 \le y_2 \in \Z, t \in \N$ such that $[x,1]_{\textsf{L}}\le [y_1,t]_{\textsf{L}}$ and $[x,1]_{\textsf{L}}\le [y_2,t]_{\textsf{L}}$, and take $f:\Z \to \R_{\ge 0}$. Then for any $E:\Z^2\to\R$,
$$
E[x, 1; y_2, t]_{\textsf{L}} - E[x, 1; y_1, t]_{\textsf{L}} \le E^f[0,0; y_2, t]_{\textsf{L}} - E^f[0,0; y_1, t]_{\textsf{L}},
$$
on the event where a geodesic in $E^f$ from $[0,0]_{\textsf{L}}$ to $[y_1, t]_{\textsf{L}}$ intersects $\{[r, 1]_{\textsf{L}} : r \ge x\}$. Similarly, 
$$
E[x, 1; y_2, t]_{\textsf{L}} - E[x, 1; y_1, t]_{\textsf{L}} \ge E^f[0,0; y_2, t]_{\textsf{L}} - E^f[0,0; y_1, t]_{\textsf{L}},
$$
on the event where a geodesic in $E^f$ from $[0,0]_{\textsf{L}}$ to $[y_2, t]_{\textsf{L}}$ intersects $\{[r, 1]_{\textsf{L}} : r \le x\}$.
\end{lemma}

\begin{proof}
We only prove the first inequality, as the second has a symmetric proof. First, it suffices to prove the inequality with $E^f$ on both sides, since the environment $E^f$ equals $E$ on the set $\N^2$. Next, let $\gamma$ be a geodesic in $E^f$ from $[0,0]_{\textsf{L}}$ to $[y_1, t]_{\textsf{L}}$ using a point $[r,1]_{\textsf{L}}$ for some $r \ge x$, and let $\pi$ be a geodesic from $[x, 1]_{\textsf{L}}$ to $[y_2, t]_{\textsf{L}}$. Since $r \ge x$ and $y_1 \le y_2$, these two geodesics necessarily meet, at a point $p$. Let $\ga'$ be the path following $\gamma$ up to the point $p$, and $\pi$ thereafter, and let $\pi'$ follow $\pi$ up to the point $p$ and $\gamma$ thereafter. Then
$$
E^f[0,0; y_1, t]_{\textsf{L}} + E^f[x, 1; y_2, t]_{\textsf{L}} = |\gamma'|_{E^f} + |\pi'|_{E^f} \le E^f[0,0; y_2, t]_{\textsf{L}} + E^f[x, 1; y_1, t]_{\textsf{L}},
$$
where the last inequality uses that $\ga'$ is a path from $[0,0]_\textsf{L}$ to $[y_2, t]_\textsf{L}$ and $\pi'$ is a path from $[x, 1]_\textsf{L}$ to $[y_1, t]_\textsf{L}$.
\end{proof}

In order to apply Lemma \ref{L:basis-for-comparison}, we will analyze the location of the geodesics from stationary initial conditions. The proof uses the following profile bound for symmetric LPP, a simple consequence of Corollary \ref{C:parabolic-profile}. For this lemma, recall $\mu$ from \eqref{eq:defmunm}, and for $u, v \in \Z^2$ with $[u]_\textsf{L}\le [v]_\textsf{L}$ we set
$$
\mu[u; v]_\textsf{L} = \mu([v]_\textsf{L} - [u]_\textsf{L} + (1,1)),
$$
and take $\Delta_\alpha(n) = \mu_\alpha(n) - \mu(n)$ as in Theorem \ref{T:half-space-full-tails}.
Also recall the symmetric field $X:\Z^2\to\R_{\ge 0}$ from the beginning of this subsection, and denote
\[
\bar X[u; v]_\textsf{L} = X[u; v]_\textsf{L} - \mu[u; v]_\textsf{L}.
\]
\begin{lemma}
\label{L:Xalpha-bd}
Take any $\alpha_0 \in (0, 1)$, $\delta_0 > 0$, $n \in \N$ and $j_0, x \in \Z_{\ge 0}$ with $x, j_0 \le n/10$. Take parameter $\alpha \ge \alpha_0$, and set (for $j\in\II{1-n, n+x-1}$)
$\bar X(j) = \bar X[j, 1; x, n]_\textsf{L}$.
Then there exists $c > 0$ depending only on $\alpha_0, \delta_0$ such that for all $\eps \in (0, n^{-1/3})$ we have
$$
\P\left(\max_{j\in\II{0,n+x-1}} \bar X(j) - \frac{\delta_0(j-j_0)^2}{n} \ge \Delta_\alpha(n) + \eps n\right) \le 2\exp\left(-c \eps^{3/2}n \left(1 \wedge \frac{\eps^{1/2}}{(1/2 - \alpha)\vee 0}\right) \right).
$$
\end{lemma}

\begin{proof}
Take $X_\infty=X$ on $\Z^2\setminus\{(i,i):i\in\Z\}$, and $X_\infty(i,i)=0$ for each $i\in\Z$.
Let $\bar X_\infty(j)=X_\infty[j, 1; x, n]_\textsf{L} - \mu[j, 1;x, n]_\textsf{L}$.
Then by Corollary \ref{C:exp-quadrangle}, we have
$$
\max_{j\in\II{0,n+x-1}} \bar X(j) - \bar X_\infty(j) \le   X[0, 1; 0, n]_\textsf{L} - X_\infty[0, 1; 0, n]_\textsf{L}.
$$
By upper bounding $\max_{j\in\II{0,n+x-1}}\bar X_\infty(j)$ using Corollary \ref{C:parabolic-profile} and the fact that $X_\infty$ is dominated by full-space passage times, and upper bounding the right-hand side using the one-point tail bounds in Theorem \ref{T:one-point-bds}, the conclusion follows.
\end{proof}

By combining the profile bound in Lemma \ref{L:Xalpha-bd}, standard fluctuation bounds on the stationary measures $Z^\pm$ (Lemma \ref{L:basic-stationary-est}) and the lower tail bounds in Corollary \ref{C:symmetric-version}, we can get precise estimates on the argmax location when applying metric composition with a stationary measure for symmetric exponential LPP. For this next lemma, we have restricted our parameter ranges to simplify both the statements and the proof.
\begin{lemma}
	\label{L:location-bd}
    Take $n \in \N$ and $x \in \Z_{\ge 0}$ with $x \le n/100$. Take parameters $\alpha, \beta$ satisfying that
$\alpha \ge 1/2 - n^{-1/6}$ and $\beta \in (2(1/2 - \alpha)\vee 0, 1/100)$. Define
$$
y^\pm(\beta) = \argmax_{y\in\II{-n, n + x}} f^\pm(y, \beta), \quad \text{where }\quad f^\pm(y, \beta) = \frac{|y|}{1/2 \mp \operatorname{sgn}(y)\beta} + \mu[y-\operatorname{sgn}(y), 1; x, n]_\textsf{L}.
$$
Then we have $y^+(\beta)>0$, and
$|y^\pm(\beta) - (x \pm 8 \beta n)|\le 100 \beta^2 n + 1$.
Moreover, there are universal constants $C, c>0$ such that the following statements hold. Let
\[
y^\pm_* = \argmax_{y\in\II{-n,n+x}} (Z^\pm(y) + X[y-\operatorname{sgn}(y),1; x,n]_\textsf{L}) .
\]
For any $\eps > C (1/2 - \alpha)\vee 0$ we have
\[
\P(|y^+_* - y^+(\beta)|\ge \eps n),\quad \P(|y^-_* - y^-(\beta)|\ge \eps n\;\text{and}\;y^-_*>0)\le 2\exp(- c \eps^3 n) + 2 \exp(-c\sqrt{n}).
\]
\end{lemma}

\begin{proof}
	The estimates on $y^\pm(\beta)$ follow by Taylor expanding $f(y, \beta)$ near $y = x$, $\beta = 0$. We can also compute that for all $j \in \II{-n, n + x}$, we have that 
\begin{equation}   \label{eq:ffypm}
f^\pm(j, \beta) - f^\pm(y^\pm(\beta)) \le - \frac{\delta_0(j-y^\pm(\beta))^2}{n},    
\end{equation}
where $\delta_0$ is a universal constant. 

Below we use $c>0$ to denote a universal constant, and its value may change from line to line.
Recall $\bar Z^+$ from Lemma \ref{L:basic-stationary-est} and $\bar X$ from Lemma \ref{L:Xalpha-bd}. We also let $\bar Z^-(i)=\bar Z^+(-i)$ for each $i\in\Z$.
By the profile bounds on $\bar X$ from Lemma \ref{L:Xalpha-bd} with $j_0 = (y^\pm(\beta)-1)\vee 0$, we get
$$
\P\left(\max_{j \in\II{1,n+x}: |j - y^\pm(\beta)| \ge \eps n} \bar X(j-1) - \frac{\delta_0 (j-1-(y^\pm(\beta)-1)\vee 0))^2}{4n} \ge \Delta_\alpha(n) + \frac{\delta_0}{8} \eps^2 n\right) \le 2\exp\left(-c\eps^3 n \right),
$$
where $(8/\delta_0)^{1/2}(1/2 - \alpha)\vee 0 < \eps < (8/\delta_0)^{1/2}n^{-1/6}$. 
Note that the left-hand side above is non-increasing in $\delta$, so when $\eps \ge  (8/\delta_0)^{1/2}n^{-1/6}$ we can still upper bound it by $2\exp(-c\sqrt{n})$.
Besides, by \eqref{E:shape-bd-stat} in Lemma \ref{L:basic-stationary-est}  with $j_0 = y^\pm(\beta), \delta = \frac{\delta_0}{4n}$ and $m=\frac{\delta_0}{8} \eps^2 n$ for any $\eps>0$, we get that
$$
\P\left(\max_{j \in\II{1,n+x}: |j - y^\pm(\beta)| \ge \eps n}  \bar Z^\pm(j) - \bar Z^\pm(y^\pm(\beta)) - \frac{\delta_0(j-y^\pm(\beta))^2}{4n}  \ge \frac{\delta_0}{8} \eps^2 n \right) \le 2\exp\left(-c\eps^3 n \right).
$$
By taking $C$ to be a large enough universal constant and $\eps>C(1/2 - \alpha)\vee 0$, we have that
\[
\Delta_\alpha(n) = O((1/2 - \alpha)\vee 0)^2 n \le \frac{\delta_0}{8} \eps^2 n.
\]
Thus from the three bounds above we conclude that,
\begin{align*}
&\P\left(\max_{j \in\II{1,n+x}: |j - y^+(\beta)| \ge \eps n}  \bar Z^+(j) - \bar Z^+(y^+(\beta)) + \bar X(j-1)  - \frac{\delta_0(j-y^+(\beta))^2}{2n}  \ge \frac{3\delta_0}{8} \eps^2 n \right), \\
&\P\left(\max_{j \in\II{1,n+x}: j - y^-(\beta) \ge \eps n}  \bar Z^-(j) - \bar Z^-(y^-(\beta)) + \bar X(j-1)  - \frac{\delta_0(j-y^-(\beta))^2}{2n}  \ge \frac{3\delta_0}{8} \eps^2 n \right) \\ &\le 2\exp\left(-c\eps^3 n \right) + 2 \exp(-c\sqrt{n}).
\end{align*}
With \eqref{eq:ffypm} and the lower bound on 
$\bar X(y^\pm(\beta))$ from Corollary \ref{C:symmetric-version}, 
we finish the proof.
(Here we note that $y^+_*>0$ almost surely, since $Z^+(j)\ge Z^+(-j)$ and $X[j-1,1; x,n]_\textsf{L}\ge X[-j+1,1; x,n]_\textsf{L}$ for each $j\in\N$.)
\end{proof}

We are now ready to state and prove the two-point spatial bound. 
\begin{prop}
	\label{P:two-point-bd}
Take $n \in \N$, parameter $\alpha \ge 1/2 - n^{-1/6}$, and $x, y, z_0 \in \II{0, n/100}$, $a > 0$ satisfying
\begin{equation}
	\label{E:constraints}
C_0 \max \left(\frac{z_0^2}{n}, \frac{z_0 |y-x|}{n}, z_0(1/2 - \alpha)\vee 0 \right) < a < C_* z_0,
\end{equation}
where $C_0>0$ is a universal constant, and $C_*>0$ is any large enough number.
Then for a constant $c>0$ depending only on $C_*$, we have
$$
\P\left(\max_{z \in\II{0,z_0}} \left|\bar X[x, 1; y + z, n]_\textsf{L} - \bar X[x, 1; y, n]_\textsf{L} \right| \ge a \right) \le 2\exp(- c a^2/z_0) + 2 \exp(-c \sqrt{n}).
$$
\end{prop}

Before proceeding with the proof, it is worth noting that for the purpose of tightness, the constraints \eqref{E:constraints} will not give any restrictions in the limit. Indeed, asymptotically we will take $z_0 = w n^{2/3}, |y-x| = O(n^{2/3}), (1/2 - \alpha)\vee 0 = O(n^{-1/3}), a = b n^{1/3}$ and so in these ranges our constraints become
$$
C_0 (w^2\vee w) < b < C_* w n^{1/3},
$$
and the tail bound becomes sub-Gaussian: $2\exp(-c b^2/w)$. For $w$ large and $b \le C_0 (w^2\vee w)$ the above bound no longer applies, but in this case the Tracy-Widom one-point bounds can actually provide a sharper estimate (up to the constant in the exponent).

\begin{proof}[Proof of Proposition \ref{P:two-point-bd}]
	We prove the upper tail, i.e., the bound without the absolute value, as the lower tail follows from symmetric arguments. 
    
    Set $\beta_0 = (x-y)/(8n)$, and take parameter $\beta=\beta_0\vee 2(1/2 - \alpha)\vee 0+(c_0a/z_0) \wedge 1/200$, where $c_0$ is a small enough universal constant.
    We then have $|\beta_0|\le 1/800$, and $\beta \in (2(1/2 - \alpha)\vee 0, 1/100)$.
    Below we use $c>0$ to denote a constant small enough depending on $C_*$ and all other universal constants (including $C_0, c_0$), and its value may change from line to line.

    Recall $y^+(\beta)$ from Lemma \ref{L:location-bd}. 
    In particular, if we denote $\eps=(y^+(\beta)-(x+1))/n$, we have
    \[
    \eps \ge y/n+8\beta - 100\beta^2 -x/n -2/n.
    \]
    From our choice of $\beta$, we have that $\beta-\beta_0<1/100$, so
    \[
    \eps+2/n\ge y/n+8\beta - 100\beta^2 -x/n\ge 8(\beta-\beta_0) - 200 (\beta-\beta_0)^2 - 200\beta_0^2> 6(\beta-\beta_0)- 200\beta_0^2.
    \]
    By taking $C_0$ large depending on $c_0$, we have $200\beta_0^2<200a/(64C_0z_0) < c_0a/z_0$. Also note that $200\beta_0^2<1/200$.    
    Thus $200\beta_0^2<\beta-\beta_0$. By plugging this into the right-hand side above, we get $\eps+2/n > 5(\beta-\beta_0)\ge 5(c_0a/z_0\wedge 1/200)$. Since  $C_0((1/2-\alpha)\vee 1/n) <a/z_0 < C_*$ due to the constraints of $a$, we further have $\eps+2/n > ca/z_0$ and $\eps+2/n > 5c_0C_0((1/2-\alpha)\vee 1/n)$.
    As $C_0$ is large enough depending on $c_0$,  we can apply to Lemma \ref{L:location-bd} get that
\[
    \P\bigg(\argmax_{j \in \II{-n, n+y}} (Z^+(j) + X[j-\operatorname{sgn}(j), 1; y, n]_\textsf{L}) \ge x+1 \bigg)  \ge  1- 2\exp(-c(a/z_0)^3 n ) - 2\exp(-c\sqrt{n}).\]

On the event in the above probability, Lemma \ref{L:basis-for-comparison} implies that for all $z \in \II{0, z_0}$, we have
$$
X[x, 1; y + z, n]_{\textsf{L}} - X[x, 1; y, n]_{\textsf{L}} \le X^{Z^+}[0,0; y + z, n]_{\textsf{L}} - X^{Z^+}[0,0; y, n]_{\textsf{L}} \eqd Z^+(y+z) - Z^+(y).
$$
Here the final equality is as a process in $z$, from Lemma \ref{lem:staLPP}.
Then 
\begin{multline}   \label{E:big-blomp}
\P\left(\max_{z \in\II{0,z_0}} \bar X[x, 1; y + z, n]_\textsf{L} - \bar X[x, 1; y, n]_\textsf{L} \ge a \right)\le 
2\exp(- c ( a/z_0)^3n) + 2 \exp(- c \sqrt{n}) \\ 
+ \P\left( \max_{z \in\II{0, z_0}}\bar Z^+(y+z) - \bar Z^+(y) + \frac{z}{1/2-\beta}-\mu[x, 1; y + z, n]_\textsf{L} + \mu[x, 1; y, n]_\textsf{L}\ge a \right).
\end{multline}
For $z \in\II{0, z_0}$, we have
\begin{align*}
    \mu[x, 1; y + z, n]_\textsf{L} - \mu[x, 1; y, n]_\textsf{L} &= 2z + \frac{(x-y)z}{2n} - \frac{z^2}{4 n} + O(z^3 n^{-2} + |x-y|^2 z n^{-2}), \\
\frac{z}{1/2 - \beta} &= 2z + \frac{(x-y)z}{2n} + 4 (\beta-\beta_0)z  + O((\beta-\beta_0)^2 z + |x-y|^2 z n^{-2}).
\end{align*}
Note that $\beta-\beta_0 \le 2(1/2-\alpha)\vee 0 + |x-y|/(8n) + c_0a/z_0$.
Thus by taking $C_0$ large enough and $c_0$ small enough, we have that
\[
a/2 >  \frac{z}{1/2-\beta}-\mu[x, 1; y + z, n]_\textsf{L} + \mu[x, 1; y, n]_\textsf{L},
\]
Therefore we can apply \eqref{E:two-point-stat} in Lemma \ref{L:basic-stationary-est} (and use that $a<C_*z_0$) to bound the last line of \eqref{E:big-blomp} by $2\exp(-ca^2/z_0)$.
Finally, note that $\exp(-c(a/z_0)^3 n ) < \exp(-ca^2/z_0)$ since $a > C_0z_0^2/n$, we get the final bound.
\end{proof}

Next, we can combine the two-point spatial tail bound in Proposition \ref{P:two-point-bd}, the profile bound in Lemma \ref{L:Xalpha-bd}, and the one-point tail in Theorem \ref{T:one-point-bds} to give a sharp temporal two-point estimate.

\begin{prop}
\label{P:temporal-est}
Take $n < n+r \le 2n \in \N$, and parameter $\alpha \ge 1/2 - n^{-1/6}$. For any $x, y_2\in \II{0, n/1000}$, denote $y_1 = \lfloor x + n(y_2-x)/(n + r) \rfloor$, and 
$u_1 = [x, 0; y_1, n]_\textsf{L}$, $u_2 = [x, 0; y_2, n + r]_\textsf{L}$.
Take any $\epsilon > 0$ satisfying the constraints:
\begin{equation}
	\label{E:constraints-phrased}
	C_0 \max \left(\frac{|x-y_2|^2}{n^2}, ((1/2 - \alpha)\vee 0)^2\right) < \epsilon < C_* r^{-1/3},
\end{equation}
where $C_0>0$ is a universal constant, and $C_*>0$ is any large enough number. 
Then for a constant $c>0$ depending only on $C_*$, we have
$$
\P\left( \left|\bar X(u_1) - \bar X(u_2)\right| \ge \eps r \right) \le 2\exp\left(-c \eps^{3/2} r \right).
$$
\end{prop}

\begin{proof}
First, it suffices to prove the bound when $r \le \eps_0 n$, where $\eps_0$ is a small universal constant to be determined. The tail bound for general $r \le n$ then follows by applying the $r\le \eps_0 n$ case many times and a triangle inequality. We may also assume $r$ is greater than a large constant, since otherwise the bound holds trivially.
Below we use $c>0$ to denote a constant small enough depending on $C_*$ and all other universal constants (including $C_0, \eps_0$), and its value may change from line to line.
	
Denote $v = [y_1, n+1; y_2, n+r]_\textsf{L}$. Then we have the triangle inequality
\begin{align*}
\bar X_\alpha(u_1) - \bar X_\alpha(u_2) \le  -\bar X_\alpha(v) - \mu(u_1) - \mu(v) + \mu(u_2) = -\bar X_\alpha(v) + O(1).
\end{align*}
Here the equality is a straightforward Taylor expansion computation. We can bound the lower tail of $\bar X_\alpha(v)$ using Theorem \ref{T:half-space-full-tails}, to get that for any $\eps \in (0, 1)$,
\[
	\P(\bar X_\alpha(u_1) - \bar X_\alpha(u_2) \ge \eps r) \le 2 \exp(-c\eps^{3}r^2),
\]
For the lower bound of $\bar X_\alpha(u_1) - \bar X_\alpha(u_2)$, we use the metric composition law. 
For each $z\in\II{1-y_1, r+y_2-y_1}$, define 
\begin{align*}
f_1(z) &= \mu[x, 0; z + y_1, n]_\textsf{L} - \mu(u_1), \\
f_2(z) &= \mu[z + y_1-1, n+1; y_2, n + r]_\textsf{L} - \mu(v) \\
A(z) &= \bar X[x, 0; z + y_1, n]_\textsf{L} - \bar X(u_1), \\
B(z) &= \bar X[z + y_1-1, n+1; y_2, n + r]_\textsf{L}.
\end{align*}
We can write
\begin{align*}
\bar X(u_2) &= - \mu(u_2) + \max_{z \in\II{1-y_1, r+y_2-y_1}} X[x, 0; z + y_1, n]_\textsf{L} + X[z + y_1 - 1, n+1; y_2, n + r]_\textsf{L} \\
&=\bar X(u_1) + O(1) + \max_{z \in\II{1-y_1, r+y_2-y_1}} [A(z) + B(z) + f_1(z) + f_2(z)] \\
&\le \bar X(u_1) + O(1) + \max_{z \in\II{1-y_1, r+y_2-y_1}}[A(z) - \delta_0 z^2/r] + \max_{z \in\II{1-y_1, r+y_2-y_1}} [B(z) - \delta_0 z^2/r]. 
\end{align*}
where we have extracted the $O(1)$-term and established the final inequality by straightforward Taylor expansion. Here $\delta_0$ is a universal constant. Now,  
by Lemma \ref{L:Xalpha-bd}, we have that for all $\eps \in (8((1/2 - \alpha)\vee 0)^2, C_*r^{-1/3})$: 
\begin{align*}
\P\left(\max_{z \in\II{1-y_1, r+y_2-y_1}} [B(z) - \delta_0 z^2/r] \ge \eps r\right) \le 2\exp(-c \eps^{3/2} r ).
\end{align*}
Here the lower bound on $\eps$ allows us to absorb the $\Delta_\alpha(r)$ term.
The above bound matches the upper bound in the lemma, so to complete the proof it suffices to bound the random variable $\max_{z \in\II{1-y_1, r+y_2-y_1}} [A(z) -\delta_0 z^2/r]$. First, we can write
\begin{multline*}
\P\left(\max_{z \in\II{1-y_1, r+y_2-y_1}}  [A(z) - \delta_0 z^2/r] \ge \eps r\right) \\ \le 
\P\left(\max_{z \in\II{1-y_1, r+y_2-y_1}, |z|\le n/200}  [A(z) - \delta_0 z^2/r] \ge \eps r\right)
+
\P\left(\max_{z \in\II{1-y_1, r+y_2-y_1}, |z|> n/200}  [A(z) - \delta_0 z^2/r] \ge \eps r\right).
\end{multline*}
The second term on the right-hand side above is at most $2\exp(-c \sqrt{n})$ by Lemma \ref{L:Xalpha-bd}, for any choice of $\eps$. For the first term on the right-hand side, assuming that $\eps$ satisfies the constraints of the lemma, and taking $w_0 = C_1\sqrt{\eps}r$ for a universal constant $C_1$ (to be determined and large enough depending on $\delta_0$), we can bound it by
\begin{align*}
& \sum_{k=1}^{\lceil\log_2(n/(200 w_0)) \rceil} \P\left (\max_{z \in\II{1-y_1, r+y_2-y_1}, |z| \le 2^k w_0 } A(z) \ge \eps r + \delta_0 (2^{k-1} -1)^2 w_0^2/r \right) \\
\le &\sum_{k=1}^{\lceil\log_2(n/(200 w_0)) \rceil} \P\left (\max_{z \in\II{1-y_1, r+y_2-y_1}, |z| \le 2^k w_0 } A(z) \ge C_1^{-1} 2^{2k-2}\sqrt{\eps} w_0 \right) \\
\le &\sum_{k=1}^{\lceil\log_2(n/(200 w_0)) \rceil} [2 \exp(-c (2^{2k} \eps \wedge 1) 2^k w_0) + 2 \exp(-c \sqrt{n})] \\
\le & 2 \exp(-c \eps^{3/2}r) + 2 \exp(-c \sqrt{n}).
\end{align*}
Here in the third line we have applied Proposition \ref{P:two-point-bd} with $z_0 =2^k w_0$ and $a = ( C_1^{-1} 2^{k-2} \sqrt{\eps}\wedge C_1) z_0$. We check that our parameters satisfy the conditions required for Proposition \ref{P:two-point-bd}. First, since $a \le C_1z_0$ the upper bound in \eqref{E:constraints} always holds. Next we check the lower bounds in \eqref{E:constraints}, which can be equivalently written as:
$$
C_2 \max \left(\frac{z_0}{n}, \frac{|y_1-x|}{n}, (1/2 - \alpha)\vee 0 \right) < C_1^{-1} 2^{k-2} \sqrt{\eps}\wedge C_1,
$$
for a universal constant $C_2 > 0$. 
Let $C_1>C_2$, and $C_0$ be large enough and $\eps_0$ be small enough, both depending on $C_1, C_2$.
Then we have $C_2z_0/n < C_2 < C_1$ and $C_2z_0/n = C_1C_22^{k}\sqrt{\eps}r/n<C_1C_22^{k}\eps_0\sqrt{\eps}< C_1^{-1} 2^{k-2} \sqrt{\eps}$.
Moreover, we have $C_2 \max \left(\frac{|y_1-x|}{n}, (1/2 - \alpha)\vee 0 \right)<C_2<C_1$,
and by \eqref{E:constraints-phrased} we have $C_2 \max \left(\frac{|y_1-x|}{n}, (1/2 - \alpha)\vee 0 \right)<C_2C_0^{-1/2}\sqrt{\eps}< C_1^{-1} 2^{k-2} \sqrt{\eps}$.
Thus the lower bounds in \eqref{E:constraints} hold.
\end{proof}

\section{Subsequential limits of exponential LPP}   
\label{sec:fixed}

In this section, we use the tail bounds in Section \ref{sec:tightness} to show that exponential LPP in half-space is tight when the boundary parameter has critical or subcritical scaling. We also prove properties of the subsequential limits that we need moving forward. The proofs in this section are either straightforward bookkeeping, or are simple generalizations of proofs relating to the directed landscape in full-space, mostly following \cite{DOV}. 

In Section \ref{SS:subsequential}, we prove tightness of exponential LPP and give a list of properties that fall out immediately from the prelimit. In Section \ref{ssec:kpzfp}, we prove that our subsequential limits have marginals which match the half-space KPZ fixed point marginals constructed in \cite{Zhang}. 
We also prove bounds on TASEP that imply tightness of the Poisson-avoiding metrics. Finally, in Section \ref{SS:prelandscapes}, we introduce the more general notion of (half-space) pre-landscapes.
We provide a list of properties of pre-landscapes, including modulus of continuity and shape theorems, which are required for our upcoming characterization of the directed landscape in half-space. 
As a byproduct of the results in this section, we will show that the half-space KPZ fixed point is a Feller process (Theorem \ref{T:Markov-process}).

\subsection{Tightness and basic properties}

\label{SS:subsequential}

We use the notations from Section \ref{SS:off-diagonal-tails}: for each $\alpha\in \R_+\cup\{\infty\}$, take the environment $X_\alpha:\Z^2_\ge \to \R_{\ge 0}$ consisting of independent exponential random variables 
$X_\alpha(i, j) \sim \operatorname{Exp}(1)$ for $i > j \in \Z$, and $X_\alpha(i, i) \sim \operatorname{Exp}(\alpha)$ for $i \in \Z$.
Also define the passage time $X_\alpha(u;v)$ for all $u\le v \in \Z^2_{\ge}$ as before.

For $x \ge 0$ and $s \in \R$, denote $(x, s)_n = (\lfloor ns + 2^{5/3} n^{2/3} x \rfloor, \lfloor ns \rfloor)$. 
Then for $\rho \in \R \cup \{-\infty\}$ and $n\in \N$ with $1/2 - 2^{-4/3} \rho n^{-1/3}>0$, we set
\begin{equation}
	\label{E:rescaled-elpp}
	\cQ_n^\rho(x, s; y, t) = \frac{X_{1/2 - 2^{-4/3} \rho n^{-1/3}}((x, s)_n; (y, t)_n) - 4n(t-s) - 2^{8/3}n^{2/3}(y-x)}{2^{4/3} n^{1/3}}.
\end{equation}
We can define $\cQ_n^\rho$ on all of $\mathbb H^2_\uparrow$ (recall from \eqref{eq:defnH2}) by simply setting it to be $-\infty$ outside of the set where the passage time in \eqref{E:rescaled-elpp} is defined. 
\begin{theorem}
	\label{T:elpp}
    Consider a sequence $\rho_n \in \R \cup \{-\infty\}$ such that $\lim_{n \to \infty} \rho_n = \rho \in\R \cup \{-\infty\}$. Then the sequence $\cQ_n^{\rho_n}$ is tight under the uniform-on-compact topology on functions from $\mathbb H^2_\uparrow \to \R$. 
\end{theorem}

\begin{proof}
We define $\mathcal J_n: \mathbb H^2_\uparrow \to \R$ by taking $\mathcal J_n(x, s; y, t)=\cQ_n^{\rho_n}(x, s; y, t)$ whenever $(x, s)_n, (y, t)_n \in \Z^2$, and extending to general points $(x, s; y, t)$ by a linear interpolation. 
For $b > 0$, define
\begin{equation}   \label{eq:Kbdef}
K_b = \{ u = (x, s; y, t) \in \mathbb{H}^2_\uparrow : \|u\|_\infty \le b, t -s \ge b^{-1}\}.
\end{equation}
It suffices to show that for all $b > 0$, $\mathcal J_n|_{K_b}$ is tight under the uniform topology. We start by translating the two-point bounds in Proposition \ref{P:two-point-bd} and \ref{P:temporal-est} to the limiting setting. In what follows, we use $c>0$ to denote a constant which is allowed to depend on $b$, $\rho$, and $\max_n \rho_n$, and may change from line to line. 

First, let $u_1 = (x, s; y, t), u_2 = (x, s; y + z, t) \in K_b$ for some $z \ge 1/(2^{5/3}n^{2/3})$. By rescaling Proposition \ref{P:two-point-bd}, for any $a \le \sqrt{z} n^{1/3}$ we have: 
\begin{equation}
	\label{E:y1y2}
	\P(|\mathcal J_n(u_1) - \mathcal J_n(u_2)| > a \sqrt{z}) \le 2 \exp(-c a^2) + 2 \exp(-c \sqrt{n}).
\end{equation} 
Here we removed the lower bound on $a$ in Proposition \ref{P:two-point-bd} since the bound can be made trivial in the small-$a$ regime by choosing $c$ sufficiently small (according to $b$, $\rho$, and $\max_n \rho_n$). By symmetry, the same bound \eqref{E:y1y2} holds if we let $u_2 = (x + z, s; y, t)$.

Now suppose $u_1 = (p; y, t), u_2 = (p; y', t + s) \in K_b$ for some $s \ge 1/n$, and that $p, (y, t), (y', t + s)$ are collinear. Then by rescaling Proposition \ref{P:temporal-est} we have that for any $a \le s^{1/3} n^{1/3}$ we have
\begin{equation}
	\label{E:y1y3}
	\P(|\mathcal J_n(u_1) - \mathcal J_n(u_2)| > a s^{1/3}) \le 2 \exp(-c a^{3/2}).
\end{equation}
Again, we eliminated the lower bound on $a$ by taking $c$ small enough (depending on $K$, $\alpha$ and $\max_n \rho_n$). Again by symmetry, the same bound holds for $u_1 = (y, t; q), u_2 = (y', t + s; q)$ with $(y, t), (y', t + s), q$ collinear.

Note that in \eqref{E:y1y2} and \eqref{E:y1y3}, the upper bounds on $a$ would be problematic when $z$ or $s$ is small, so we will give weak extensions of these estimates.

For \eqref{E:y1y2}, consider $\sqrt{z} n^{1/3}< a \le n^{1/4}$. Let $w = 4a \sqrt{z}/n^{1/3}$ and let $u_3 = (x, s; y + w, t)$. Note that $w>4z$, and $a\sqrt{z/(w-z)}/2<\sqrt{w-z}n^{1/3}$.
Then combining the bound \eqref{E:y1y2} for $u_1, u_3, a\sqrt{z/w}/2$ and $u_2, u_3, a\sqrt{z/(w-z)}/2$ and using the triangle inequality $|\mathcal J_n(u_1) - \mathcal J_n(u_2)| \le |\mathcal J_n(u_1) - \mathcal J_n(u_3)| + |\mathcal J_n(u_3) - \mathcal J_n(u_2)|$, we have:
\begin{equation}
	\label{E:y1y2strong}
	\P(|\mathcal J_n(u_1) - \mathcal J_n(u_2)| > a \sqrt{z}) \le 2 \exp(-c a n^{1/3} \sqrt{z}) + 2 \exp(-c \sqrt{n}) \le 2 \exp(-ca) + 2 \exp(-c \sqrt{n}),
\end{equation} 
where the final bound uses that $z \ge 1/(2^{5/3}n^{2/3})$. Similarly, for \eqref{E:y1y3}, consider $s^{1/3} n^{1/3}< a < n^{1/4}$. Take $r>0$ such that $4as^{1/3}=r^{2/3}n^{1/3}$.
Let $y''$ be chosen so that $p, (y, t), (y'', t + r)$ are collinear.
Let $u_3=(p; y'', t+r)$.
Note that $r>8s$, and $a(s/(r-s))^{1/3}/2< (r-s)^{1/3}n^{1/3}$.
Then combining the bound \eqref{E:y1y3} for $u_1, u_3, a(s/r)^{1/3}/2$ and $u_2, u_3, a(s/(r-s))^{1/3}/2$ and using the triangle inequality, we have:
\begin{equation}
	\label{E:y1y3strong}
	\P(|\mathcal J_n(u_1) - \mathcal J_n(u_2)| > a s^{1/3}) \le 2 \exp(-ca^{3/4}n^{1/4}s^{1/4}) \le 2 \exp(-c a^{3/4}),
\end{equation}
where in the final bound we have used that $s \ge 1/n$. Now, observe that by \eqref{E:y1y2}, \eqref{E:y1y3}, the bounds \eqref{E:y1y2strong}, \eqref{E:y1y3strong} actually hold for all $a < n^{1/4}$. Now, combining these bounds with the triangle inequality and taking a union bound to cover the $O(n^{20/3})$-many pairs $u_1, u_2 \in K_b$ coming from distinct lattice points, we have that for any $0<\eps<1/3$,
\begin{equation}
	\label{E:summarized-2-pt}
	\max_{u_1, u_2 \in K_b} \frac{|\mathcal J_n(u_1) - \mathcal J_n(u_2)|}{\|u_1 - u_2\|_2^{1/3 - \eps}},
\end{equation}
is a sequence of tight random variables. Since $\mathcal J_n(0,0; 0, 1)$ is tight by Theorem \ref{T:one-point-bds}, by the Arzel\'a-Ascoli theorem, this implies that the whole sequence $\mathcal J_n|_{K_b}$ is tight in the uniform topology, completing the proof.
\end{proof}

Many properties of half-space exponential LPP pass immediately to the limit. The next proposition records properties that follow immediately from corresponding properties in the prelimit. Note that the dependence on $\rho, t$ as $\rho, t \to \infty$ in the one-point tails (part $4$ below) is suboptimal off the diagonal. This can be sharpened, but we have not included the strengthening here.

\begin{prop}
	\label{P:properties-elpp}
	Fix $\rho \in \R \cup \{-\infty\}$, and let $\cQ^\rho$ be a subsequential limit in law of $\cQ_n^{\rho_n}$ along a sequence $\rho_n \to \rho$. Then $\cQ^\rho$ satisfies the following properties. In what follows, $c> 0$ are universal constants, and we define
		$\cG^\rho(x, s; y, t) = \cQ^\rho(x, s; y, t) + \frac{(x-y)^2}{(t-s)}$ for any $(x,s;y,t)\in \mathbb{H}^2_\uparrow$.
	\begin{enumerate}
		\item \textup{(Independent increments)} For any disjoint collection of time intervals $(s_i, t_i), i\in\II{1, k}$, the increments
		$\cQ^\rho(\cdot, s_i; \cdot, t_i)$
		are independent.
		\item \textup{(Metric composition)} For any fixed $r < s < t$ and $x, y \ge 0$, we have
		$$
		\cQ^\rho(x, s; y, t) = \max_{z \ge 0} \cQ^\rho(x, s; z, r) + \cQ^\rho(z, r; y, t).
		$$
		\item \textup{(KPZ scaling)} Fix $q > 0$, and consider the rescaled field
		$$
		\cQ^{\rho, q}(x, s; y, t) := q^{-1} \cQ^\rho(q^2x, q^3 s; q^2 y, q^3 t).
		$$
		Then $\cQ^{\rho, q}$ is a subsequential limit of the sequence $\cQ^{q \rho_n}_n$.
		\item \textup{(One-point tail bounds)} For $u = (x, s; y, s + t) \in \mathbb{H}^2_\uparrow$, and all $a > 0$,
		\begin{align*}
			\P( \cG^\rho(u)< - a t^{1/3}) &\le 2\exp(-c a^3), \\
			\P( \cG^\rho(u) > a t^{1/3} + t (\rho\vee 0)^2) &\le 2\exp\left(-c a^{3/2} \left(1 \wedge \frac{\sqrt{a}}{(\rho \vee 0) t^{1/3}}\right)\right).
		\end{align*}
		Furthermore, for $u = (0, s; 0, s + t)$ on the diagonal, we have the following Gaussian lower bound: for all $0 < a < t (\rho\vee 0)^2$,
		\begin{align*}
			\P( \cG^\rho(u) < -a + t (\rho\vee 0)^2) &\le 2\exp\left(- \frac{ca^2}{(\rho\vee 0)t}\right).
		\end{align*}
		\item \textup{(Two-point spatial tail bound)} 
        There is a universal constant $C_0>0$, such that for any $u = (x, s; y, s+t), u' = (x, s; y + z, s+t)  \in \mathbb H^2_\uparrow$, and 
        $a>C_0\max(z^{3/2}/t, \sqrt{z}|x-y|/t, \sqrt{z} (\rho\vee 0))$,
        we have:
		\begin{equation}
			\label{E:mm-2}
\P(|\cG^\rho(u) - \cG^\rho(u')| \ge a\sqrt{z}) \le 2 \exp(- c a^2).
		\end{equation}
		\item \textup{(Two-point temporal tail bound)} There is a universal constant $C_0>0$, such that for any $u = (x, s; y, s + t), u' = (x, s; y', s+t + r)  \in \mathbb H^2_\uparrow$, where $0<r<t$ and $(x, s), (y, s+t), (y', s + t + r)$ are all collinear, and $a > C_0\max(r^{2/3}(x-y)^2/t^2, r^{2/3}(\rho\vee 0)^2)$, we have:
		\begin{equation*}
			\P(|\cG^\rho(u) - \cG^\rho(u')| \ge a r^{1/3}) \le 2 \exp(- c a^{3/2}).
		\end{equation*}
		\item \textup{(Basic parabolic shape estimate)} For any $(x, s; y, s + t) \in \mathbb{H}^2_\uparrow$ and $a, \delta > 0$ we have:
		\begin{equation*}
            \begin{split}
                \P(\cG^\rho(x, s; y + z, s+ t) &\le t(\rho\vee 0)^2  + a t^{1/3} + \delta z^2 \text{ for all } z \in [-y, \infty)) \\
                &\ge 1 - 2\exp\left(-c_\delta a^{3/2} \left(1\wedge \frac{\sqrt{a}}{(\rho\vee 0) t^{1/3}}\right)\right),
            \end{split}
		\end{equation*}
        where $c_\delta$ is a constant depending on $\delta$.
		\item \textup{(Time stationarity and reversal)} For any $r \in \R$, we have
        $$
\cQ^\rho(x, s; y, t) \eqd \cQ^\rho(x, s + r; y, t + r) \eqd \cQ^\rho(y, -t; x, -s),
        $$
        where the equality in law is as continuous functions on $\mathbb H^2_\uparrow$. 

        \item \textup{(Quadrangle inequality)} For any $s < t$, and  $x_1 \le x_2$ and $y_1 \le y_2$, we have
        $$
       \cQ^\rho(x_1, 0;  y_1, t) + \cQ^\rho(x_2, 0;  y_2, t) - \cQ^\rho(x_1, 0;  y_2, t) - \cQ^\rho(x_2, 0;  y_1, t) \ge 0.
        $$
		\item \textup{(Half-space stationary horizon)} Recall from Theorem \ref{T:half-stationary-horizon} that the half-space stationary horizon with parameter $\rho \in \R \cup \{-\infty\}$ is a random function $\cH^\DL \in \mathcal D([2\rho \vee 0, \infty), \R, 2\rho \vee 0)$. (And denote that $\cH^\DL_\lambda(x) = \cH^\DL(\lambda,x)$.) For each $t > 0$, consider the process
        \begin{equation}
        \label{E:Hlambda}
\cH_\lambda^t(y) := \max_{x \ge 0} \cH^\DL_\lambda(x) + \cQ^\rho(x, 0;  y, t), \qquad y\ge 0.
        \end{equation}
Then $(\lambda, x) \mapsto \cH^t_\lambda(x) - \cH^t_\lambda(0)$ is in $\cD([2\rho \vee 0, \infty), \R_{\ge 0}, 2\rho \vee 0)$, and $\eqd\cH^\DL|_{[2\rho \vee 0, \infty) \times \R_{\ge 0}}$.	
	\end{enumerate}
\end{prop}

\begin{proof}
\textbf{Part $1$.} This is immediate from the same property in the prelimit when the closed intervals $[s_i, t_i]$ are disjoint. If only the open intervals $(s_i, t_i)$ are disjoint, then the same property holds by continuity of $\cQ^\rho$. 

\textbf{Part $2$.} In the prelimit we have the metric composition law
$$
\cQ^{\rho_n}_n(x, s; y, t) = \max_{z \ge 0} \cQ^{\rho_n}_n(x, s; z, r) + \cQ^{\rho_n}_n(z + 2^{-5/3} n^{-2/3}, r; y, t) + 2^{-4/3} n^{-1/3}.
$$
This passes to the limit if we can show that the location of the $\argmax$ is tight. This follows straightforwardly from the fact that the random variable
$
 \cQ^{\rho_n}_n(x, s; 0, r) +\cQ^{\rho_n}_n(2^{-5/3} n^{-2/3}, r; y, t)
 $ 
 is tight (Theorem \ref{T:half-space-full-tails}), and the fact that for some $\delta_0$ small enough (depending on $r, s, t$ and $x, y$), the shifted profile
 $$
 \max_{z \ge 0} \cQ^{\rho_n}_n(x, s; z, r) + \cQ^{\rho_n}_n(z + 2^{-5/3} n^{-2/3}, r; y, t) + \delta_0 z^2.
 $$
 is tight (by Lemma \ref{L:Xalpha-bd}).
 
 \textbf{Part $3$.} This is immediate from construction.
 
 \textbf{Part $4$.} This follow from passing the bounds in Theorems \ref{T:one-point-bds}, \ref{T:W-tail-bounds}, and \ref{T:half-space-full-tails} to the limit.

\textbf{Parts $5$ and $6$.}
These estimates follow from passing Proposition \ref{P:two-point-bd} and Proposition \ref{P:temporal-est} to the limit, respectively. Note that the upper bounds on $a$ or $\eps$ there become trivial in the limit.

\textbf{Part $7$.}  This follows from passing Lemma \ref{L:Xalpha-bd} to the limit.

 \textbf{Part $8$.} This is immediate from corresponding symmetries of the underlying lattice of exponential random variables. Note that the symmetry $\cQ^{\rho_n}_n(x, s; y, t) \eqd \cQ^{\rho_n}_n(y, -t; x, - s)$ is not exactly satisfied in the prelimit, because our spatial axes are not orthogonal to the diagonal. However, this differences wash away in the limit, since we have uniform-on-compact convergence.

 \textbf{Part $9$.} For $\cQ^{\rho_n}_n$ such a quadrangle inequality follows from a path-crossing argument. Then it also holds for $\cQ^\rho$ by passing $n\to\infty$.

 \textbf{Part $10$.} We check this in two steps. First, we check that for any finite set $L \subset [2\rho \vee 0, \infty)$, we have $\{\cH^t_\lambda - \cH^t_\lambda(0)\}_{\lambda\in L} \eqd \{\cH^\DL_\lambda\}_{\lambda\in L}$. For this, we translate Lemma \ref{lem:staLPP} into a stationarity statement for $\cQ^{\rho_n}_n$. Indeed, let $\{R^{(n)}_i\}_{i=1}^k$ be as defined at the beginning of Section \ref{ss:hfsh} with $\rho = \rho_n$ and a finite set of slopes 
$\lambda_1>\lambda_2>\cdots>\lambda_k \ge 2\rho \vee 0$. For $t \in n^{-1} \N$, define the process
 \begin{equation}
 \label{E:Hn-prelim}
    R^{(n), t}_i(y) = \max_{x \in 2^{-5/3} n^{-2/3} \Z, x> 0} R^{(n)}_i(x) + \cQ^{\rho_n}_n(x-2^{-5/3} n^{-2/3}, 0; y, t), 
 \end{equation}
 when $y \in 2^{-5/3} n^{-2/3} \Z$, $y\ge 0$, and by linear interpolation off of this set. Then
 $$
 (R^{(n), t}_i(y) - R^{(n),t}_i(0) : i\in \II{1, k}, y\ge 0) \eqd  (R^{(n)}_i(y) : i\in \II{1, k}, y\ge 0).
 $$
 This identity in law passes to the limit. This follows straightforward from tightness of the processes $\{R^{(n)}_i\}_{i=1}^k$ and $\cQ^{\rho_n}_n$, together with uniform-in-$n$ estimates on the growth of the processes $\{R^{(n)}_i\}_{i=1}^k$ and $\cQ^{\rho_n}_n$ (which guarantee that  argmax locations in \eqref{E:Hn-prelim} are tight). The estimate for $\{R^{(n)}_i\}_{i=1}^k$ follows from  standard random walk estimates, and the estimate for $\cQ^{\rho_n}_n$ is by Lemma \ref{L:Xalpha-bd}.

Given the finite dimensional distributions, to complete the proof, it remains to show that the process $(\lambda, x) \mapsto \cH^t_\lambda(x) - \cH^t_\lambda(0)$ 
is right-continuous.

Suppose $y_m \downarrow y$ and $\lambda_m \downarrow \lambda$. Let $M=M(\lambda, y)$ be the location of the $\argmax$ in \eqref{E:Hlambda}. Since $\cH^t_\lambda$ is increasing in $\lambda$, we have that $M$ is also increasing in $\lambda$. Moreover, the quadrangle inequality (part $8$) guarantees that $M$ is also increasing in $y$. Finally, right continuity of $\cH^\DL$ in $\lambda, y$ and these monotonicities imply that $M$ is also right-continuous.
Therefore
\begin{align*}
    \lim_{n \to \infty} \cH^t_{\lambda_m}(y_m) = \lim_{n \to \infty} \cH^\DL_{\lambda_m}(M(\lambda_m, y_m)) + \cQ^\rho(M(\lambda_m, y_m), 0;  y_m, t) = \cH^t_{\lambda}(y),
\end{align*}
by the right-continuity of $H^\DL$, the continuity of $\cQ^\rho$, and the right-continuity of $M$.
\end{proof}

\subsection{The half-space KPZ fixed point}
\label{ssec:kpzfp}

In \cite{Zhang}, X. Zhang proved that TASEP in half-space converges under $1:2:3$ scaling to a limit, the half-space KPZ fixed point. As in the full space setting, TASEP in half-space can be defined as an inverse function of half-space exponential LPP. This description, together with the uniform-on-compact tightness of exponential LPP implies that the any subsequential limit of half-space exponential LPP has KPZ fixed point marginals. 

A general framework for proving convergence of inverse functions under KPZ scaling was set up in \cite[Section 15]{dauvergne2021scaling}, and can be applied in the present setting. However, it is also straightforward to prove the convergence in our particular setting, and so we have included a proof here. We will also prove moderate deviation estimates for TASEP in half-space that will be used to show uniform-on-compact convergence of Poisson-avoiding metrics in later sections.

Recall from the introduction that half-space TASEP with boundary parameter $\alpha$. It can be defined as Markov process $(h_t, t \ge 0)$ on the state space
\begin{equation}   \label{eq:defSRW+}
\SRW_+ = \{h:\Z_{\ge 0} \to \Z \mid h(0) \in 2 \Z , \; h(i) - h(i+1) = \pm 1 \text{ for all i} \}.
\end{equation}
For $i \in \N$, if $h_t$ has a local minimum at $i$, then this flips to a local maximum $h_t(i) \mapsto h_t(i) + 2$ at rate $1$. For $i =0$, $h_t(0)$ flips from a local minimum to a local maximum at rate $\alpha$. All flips are independent. We can alternately define $h_t$ from exponential LPP $X=X_\alpha$ (recall from Section \ref{sssec:constru} or the previous subsection). First, let $R(x, y) = (x + y, y-x)/2$ be the rotation which translates from TASEP coordinates to exponential LPP coordinates. For $h\in \SRW_+$, define
$$
h^\alpha(t, x; h) := \max\left\{ g \in [2\Z + \mathds{1}(x \text{ is odd})] : \max_{y \in \Z} X(R(y, h(y)+2) ; R(x, g)) \le t\right\}.
$$
Then $(h^\alpha(t,\cdot; h), t\ge 0)$ equals $(h_t, t \ge 0)$ in law, with initial state $h_0=h$.

We now state a convergence result from \cite{Zhang}.
Recall that $\UC_+$ is the set of all upper semicontinuous functions from $\R_{\ge 0} \to \R$ with the topology of local hypograph convergence. 
For each $\eps>0$, define the rescaling map $\mathbf{A}_\eps: \SRW_+\to \UC_+$ via
\begin{equation}   \label{eq:bfAe}
\mathbf{A}_\eps h(x) = -\eps^{1/2}h(2 \eps^{-1}x),
\end{equation}
for each $x\in 2^{-1}\eps^{2/3}\Z_{\ge 0}$, and linearly interpolating between points in $2^{-1} \eps^{2/3}\Z_{\ge 0}$.
Then for $\rho<\eps^{-1/2}$, $t\ge 0$, and $f$ in the range of $\mathbf{A}_\eps$, we define 
\[
\mathfrak{h}^\rho_\eps(t, \cdot; f) = \mathbf{A}_\eps\left( h^{1/2 - \rho \eps^{1/2}/2}(2\eps^{-3/2}t, \cdot; \mathbf{A}_\eps^{-1}(f)) \right) + \eps^{-1}t.
\]
\begin{theorem}[\protect{Part of \cite[Theorem 1.2]{Zhang}}]
\label{T:half-space-fixed-point}
Let $\rho \in\R$ and $f \in \UC_+$ with $\limsup_{x\to\infty} f(x)/x < \infty$. Take any $h_\eps \in \operatorname{SRW}_+$, such that
$\mathbf{A}_\eps h_\eps \to f$
in $\UC_+$ as $\eps\to 0$, and for some constant $C > 0$,
\[
\mathbf{A}_\eps h_\eps(x) \le C(1 + x), \qquad \text{ for all } x \ge 0 \text{ and } \eps>0. 
\]
Then for any $t \ge 0$ and any finite set $Y \subset \R_{\ge 0}$, the limit
\[
( \mathfrak{h}^\rho(t, y ; f) : y \in Y) := \lim_{\eps \to 0} \Big( \mathfrak{h}^\rho_\eps(t, y; \mathbf{A}_\eps h_\eps) : y \in Y \Big)
\]
exists in law as a random variable in $\R^{|Y|}$. 
Moreover, for any $t \ge 0$ and $c>0$, and any finite set $Y \subset \R_{\ge 0}$, the limit
\[
( \mathfrak{h}^{-\infty}(t, y ; f) : y \in Y) := \lim_{n \to \infty} \Big( \mathfrak{h}^{-c\eps^{-1/2}}_\eps(t, y; \mathbf{A}_\eps h_\eps) : y \in Y \Big)
\]
exists in law and is independent of $c>0$.
\end{theorem}

We have omitted one of the major elements of Theorem 1.2 from \cite{Zhang} above, namely the explicit Fredholm Pffafian formula for $( \mathfrak{h}^\rho(t, y ; f) : y \in Y)$, as we only need the abstract existence of the limit above as an input to the present paper. 
We mention for the reader comparing this result to \cite[Theorem 1.2]{Zhang} that we have reversed the sign of the  boundary parameter, $\rho \mapsto - \rho$.

For $\rho \in \R \cup \{-\infty\}$, the process $\mathfrak{h}^\rho$ is the half-space KPZ fixed point with boundary parameter $\rho$. 
We note that \cite{Zhang} does not prove that the marginals of $\mathfrak{h}^\rho(\cdot, \cdot; f)$ identified above allow for an extension to a continuous function  in $(x, t)$, or that these marginals define a Feller Markov process. These details can be readily verified through standard probabilistic arguments (and will be implemented in Theorem \ref{T:Markov-process} below), using the following variational characterization.

\begin{prop}
    \label{P:KPZFP-marginals}
Take $\rho\in\R\cup\{-\infty\}$, and $\cQ^\rho$ a subsequential limit in law of $\cQ^\rho_n$ (when $\rho\in\R$) or $\cQ^{-c(2n)^{1/3}}_n$ for some $c>0$ (when $\rho=-\infty$).
For any $f \in \UC_+$ with $\limsup_{x\to\infty} f(x)/x < \infty$, a finite set $Y\subset \R_{\ge 0}$, and $s\in\R$, $t>0$, we have
		\begin{equation}
        \label{E:desired-fp}
		    \left(\max_{x \ge 0} f(x) + \cQ^\rho(x, s; y, s + t) : y \in Y\right) \eqd ( \mathfrak{h}^\rho(y, t ; f) : y \in Y).
		\end{equation}
\end{prop}

\begin{proof}
First, by translation invariance we may assume $s = 0$ to simplify notation. Next, it suffices to assume that $f \in \NW_+$, where $\NW_+ \subset \UC_+$ is the set of $f$ with $\{x\ge 0: f(x)>-\infty\}$ being finite. This follows since $\NW_+$ is dense in $\UC_+$ and both sides of \eqref{E:desired-fp} are continuous on the following set (for some $C>0$):
$$
\UC_C := \{f \in \UC_+: f(x) \le C(1 + x) \text{ for all } x\},
$$
For the left-hand side, this continuity follows from the continuity of $\cQ^\rho$ and the shape bound in Proposition \ref{P:properties-elpp}.7. For the right-hand side, this follows from Theorem \ref{T:half-space-fixed-point}  by the following diagonalization argument. Suppose $f_i \to f$ in $\UC_C$. For every $i$ we can find a sequence $h_{i,\eps} \in \SRW_+$ such that $\mathbf{A}_\eps h_{i,\eps}$ converge as $\eps \to 0$ to $f_i$, uniformly in $i$, and satisfy
$$
\mathbf{A}_\eps h_{i,\eps}(x) \le (C + 1)(1 + x), \qquad \text{ for all } x \ge 0.
$$
Let $d$ be any metric metrizing the weak topology on probability measures on $\R^{|Y|}$. Then by Theorem \ref{T:half-space-fixed-point} we can find a sequence $\eps_i \to 0$ such that
$$
d\left(\operatorname{Law} ( \mathfrak{h}^\rho(t, y ; f_i) : y \in Y) , \;\operatorname{Law}( \mathfrak{h}^\rho_{\eps_i}(t, y; \mathbf{A}_{\eps_i} h_{i,\eps_i}) : y \in Y)\right) < 1/i,
$$
where for $\rho=-\infty$, we replace $\mathfrak{h}^\rho_{\eps_i}$ by $\mathfrak{h}^{-c\eps^{-1/2}}_{\eps_i}$.
On the other hand, another application of Theorem \ref{T:half-space-fixed-point} implies that the random vectors $( \mathfrak{h}^\rho_{\eps_i}(t, y; \mathbf{A}_{\eps_i} h_{i,\eps_i}) : y \in Y)$ converge as $i \to \infty$ to $(\mathfrak{h}^\rho(t, y ; f) : y \in Y)$, which when combined with the above display gives the desired continuity for the right-hand side of \eqref{E:desired-fp}.

Now, let $f \in \operatorname{NW}_+$, and let $\{x_1, \dots, x_k\}$ be the set of points where $f \ne - \infty$. 
We shall take the scaling limits in Proposition \ref{P:properties-elpp} and Theorem \ref{T:half-space-fixed-point} jointly. 
For this, we take $\eps = (2n)^{-2/3}$ for $n\in \N$.
We set up a finite approximation of $f$. 
Let $a_n = 2^{5/3} n^{2/3}=2\eps^{-1}$ and $b_n = 2^{-1/3} n^{-1/3}=\eps^{1/2}$.
For $i\in\II{1, k}$ and $n \in \N$, consider narrow wedge initial conditions $f_{n, i} \in \operatorname{SRW}_+$ given by
$$
f_{n, i}(y) = -b_n^{-1} m_{n, i} + |y - a_n x_{n, i}|,
$$
where $x_{n, i} \to x_i$ and $m_{n, i} \to f(x_i)$ as $n \to \infty$. Let $f_n = \min(f_{n, 1}, \dots, f_{n, k})$. 
Specifying that the $f_n$ is of the exact above form is not strictly necessary, but avoids some complication later.

By combining the tightness in Theorem \ref{T:elpp} with the convergence in Theorem \ref{T:half-space-fixed-point}, we can find a subsequence $M \subset \N$ such that the processes $\cQ^\rho_n$ (or $\cQ^{-c\eps^{-1/2}}_n$ when $\rho=-\infty$) and $( \mathfrak{h}^\rho_{\eps}(t, y; \mathbf{A}_{\eps}f_n) : y \in Y)$ (or $( \mathfrak{h}^{-c\eps^{-1/2}}_{\eps}(t, y; \mathbf{A}_{\eps}f_n) : y \in Y)$ when $\rho=-\infty$) jointly converge in law to a coupling $(\cQ^\rho, ( \mathfrak{h}^\rho(t, y ; f) : y \in Y))$ along $M$. By Skorokhod's representation theorem, we may couple the processes to realize this convergence in law as almost sure convergence. To complete the proof, it is enough to show that in this coupling, for any fixed $y \in Y$ we have that almost surely,
\begin{equation}
\label{E:fixed-y}
\max_{x \ge 0} f(x) + \cQ^\rho(x, 0; y, t) = \mathfrak{h}^\rho(t, y ; f). 
\end{equation}
Below we ignore some rounding issues for clearer presentation.
We note that $ \mathfrak{h}^\rho_{\eps}(t, y; \mathbf{A}_{\eps}f_n)$ (or $\mathfrak{h}^{-c\eps^{-1/2}}_{\eps}(t, y; \mathbf{A}_{\eps}f_n)$ when $\rho=-\infty$) can be rewritten as
\begin{align}   \label{eq:bngX}
    -b_n\left[\max\left\{g : \max_{i\in\II{1, k}} X(R(a_n x_{n, i}, -b_n^{-1} m_{n, i}+2) ; R(a_n y, g)) \le 4nt\right\}  - 2 n t\right],
\end{align}
where $X=X_{1/2-\rho\eps^{1/2}/2}$ (if $\rho\in \R$) or $X=X_{1/2+c/2}$ (if $\rho=-\infty$).
Further note that $R(a_n x, m) = (x, [m-a_n x]/(2n))_n$ for any $m\in\R$ and $x>0$. Thus we have
\begin{multline*}
 X(R(a_n x_{n, i}, -b_n^{-1} m_{n, i}+2) ; R(a_n y, g)) \\
= 2 b_n^{-1} \cQ^\rho_n(x_{n, i}, [-b_n^{-1} m_{n, i}-a_n x_{n, i}+2]/(2n); y, [g-a_n y]/(2n)) + 2(g +b_n^{-1} m_{n, i}),
\end{multline*}
where $\cQ^\rho_n$ is again replaced by $\cQ^{-c\eps^{-1/2}}_n$ when $\rho=-\infty$.

Note that $ [-b_n^{-1} m_{n, i}-a_n x_{n, i}+2]/(2n)$ and $a_ny/(2n)$ converge to $0$ as $n\to\infty$.
Also note that the above expression is increasing in $g$, so the uniform-on-compact convergence to $\cQ^\rho$, we can see that for large enough $n$, the above expression will cross the threshold $4nt$ when $|g - 2nt| \le n^{1/3} \log n$. For $g$ in this range, again by the uniform-on-compact convergence, the $\cQ^\rho_n$ or $\cQ^{-c\eps^{-1/2}}_n$ term converges to $\cQ^\rho(x, 0; y, t)$. In summary, we have that
$$
 X(R(a_n x_{n, i}, -b_n^{-1} m_{n, i}); R(a_n y, g)) = 2g + 2 b_n^{-1}(\cQ^\rho(x_{n, i}, 0; y, t) + m_{n, i}) + o(n^{1/3}).
$$
Plugging this into \eqref{eq:bngX} gives
\begin{align*}
& -b_n\left[\max\left\{g : 2g + \max_{i\in\II{1, k}} 2 b_n^{-1}(\cQ^\rho(x_{n, i}, 0; y, t) + m_{n, i}) + o(n^{1/3}) \le 4nt\right\}  - 2 n t\right] \\
&= \max_{i\in\II{1, k}} \cQ^\rho(x_{n, i}, 0; y, t) + m_{n, i} + o(1),
\end{align*}
which converges to the right-hand side of \eqref{E:fixed-y} as $n \to \infty$.
\end{proof}

\subsubsection{One-point and two-point estimates for TASEP}
We prove moderate deviation estimates for TASEP from a narrow wedge initial condition, from the moderate deviation estimates for exponential LPP. Our main goal is to obtain tightness for Poisson-avoiding metrics (to be completed in Section \ref{sec:TASEP}), so we will work in less generality than for exponential LPP in order to simplify the bounds.

We start with one-point estimates. Below we write $\delta_y(x) = |x- y| + \mathds{1}(y \text{ is odd})$ for the narrow wedge condition started from $y\in\Z_{\ge 0}$. 

\begin{lemma}
\label{L:one-point-TASEP}
Fix $\alpha_0 \in \R$. 
Then there exists a constant $c > 0$ depending only on $\alpha_0$ such that for all $t > 0$, $\alpha \ge 1/2 - \alpha_0 t^{-1/3}$, $x, y \in \II{0, \alpha_0 t^{2/3}}$ we have
\begin{align*}
	\P(h^\alpha(t, x; \delta_y) \ge t/2 + \epsilon t) &\le 2\exp\left(-c\eps^{3}t^2\right), \qquad \eps \in (0, 1),\qquad \text{ and} \\
    \P(h^\alpha(t, x; \delta_y) \le t/2 - \epsilon t) &\le 2\exp\left(-c\eps^{3/2}t\right), \qquad \eps \in (0, t^{-1/3}).
    \end{align*}
\end{lemma}

\begin{proof}
We will ignore parity and rounding issues in $x, y$, as they are minor and do not affect the statements or proofs. We may also assume $\eps \ge C_0 t^{-2/3}$ for a constant $C_0>0$ depending only on $\alpha_0$, since otherwise the bounds are trivial.
 We can write
\begin{align}
\label{E:original-X}
 \P(h^\alpha(t, x; \delta_y) \ge t/2 + \epsilon t) = \P(X_\alpha(y/2, -y/2;  t/4 + \epsilon t/2 + x/2, t/4 + \epsilon t/2 - x/2) \le t).
\end{align}
Set
$$
n = t/4 + \eps t/2 + (x-y)/2, \quad m = t/4 + \eps t/2 - (x-y)/2, \quad \mu(m, n) = t + 2 \eps t + O(t^{1/3}), \quad p=(y/2, -y/2),
$$
where the constant in the $O(t^{1/3})$ term depends on $\alpha_0$. This allows us to rewrite \eqref{E:original-X} as 
$$
\P(X_\alpha(p; p + (n, m)) \le \mu(m, n) - 2\eps t + O(t^{1/3})), 
$$
which yields the desired bound from \eqref{E:half-space-lower-bd} in Theorem \ref{T:half-space-full-tails}.
For the second bound, the probability can be similarly written as
$$
\P(X_\alpha(p; p + (n, m)) \ge \mu(m, n) + 2\eps t + O(t^{1/3})), 
$$
which yields the desired bound from \eqref{E:half-space-upper-bd} in Theorem \ref{T:half-space-full-tails}, plus that $\eps \ge C_0 t^{-2/3}$ and the lower bound on $\alpha$.
\end{proof}

Two-point estimates for TASEP are slightly more delicate. We prove a slightly suboptimal bound that still yields tightness for Poisson-avoiding metrics. Denote
$$
\bar h^\alpha(x, t; \delta_y) = h^\alpha(x, t; \delta_y) - 2nt.
$$

\begin{lemma}
    \label{L:two-point-TASEP}
Fix $\alpha_0 \in \R$. 
There are constants $C_0, c > 0$ depending only on $\alpha_0$ such that the following holds. Take $0 < t_1 \le t_2 < \alpha_0 t_1$, and $x_1, x_2, y \in \II{0, \alpha_0 t_1^{2/3}}$, and $\alpha \ge 1/2 - \alpha_0 t_1^{1/3}$.  Then 
\begin{align*}
	\P\Big(|\bar h^\alpha(x_1, t_1; \delta_y) - \bar h^\alpha(x_2, t_2; \delta_y)| \ge a (|x_1 - x_2|^{1/2} + |t_2 - t_1|^{1/3})\Big) &\le 2\exp(-c a^{1/2}),
\end{align*}    
for any $a > 0$ satisfying
$$
C_0\log^2 \left( \frac{t_1^{1/3}}{|x_1 - x_2|^{1/2} + |t_2 - t_1|^{1/3}}\right) < a < t_1^{1/4}.
$$
\end{lemma}

\begin{proof}
First, we assume $|x_1 - x_2|^{1/2} + |t_2 - t_1|^{1/3} \ge 1$. Indeed, the only way for this to fail is if $x_1 = x_2$ and $t_2 - t_1 \le 1$. In this case, the bound follows from bounding the probability that the Poisson clock governing the height flips for TASEP at location $x_1 = x_2$ rings at least $a/2$ times in the interval $[t_1, t_2]$.
We also assume that $C_1(|x_1-x_2|^{1/2} + |t_2-t_1|^{1/2})< t_1^{1/3}$ for some constant $C_1>0$ large enough depending on $\alpha_0$. This is because, otherwise, the conclusion follows from applying Lemma \ref{L:one-point-TASEP} to $\bar h^\alpha(x_1, t_1; \delta_y)$ and $\bar h^\alpha(x_2, t_2; \delta_y)$ separately.

Below we use $C, c>0$ to denote large and small constants depending only on $\alpha_0$. Their values can change from line to line.

Set $Y_i =\bar h^\alpha(x_i, t_i; \delta_y), i = 1, 2$. As in the previous proof, we ignore parity and rounding issues. Let $m = m' a (|x_1 - x_2|^{1/2} + |t_2 - t_1|^{1/3})$ for some $m' \in [1/2, 1]$; we will choose the exact value of $m'$ later. 

Take $k_0 = \left\lceil\frac{t_1^{1/3}}{|x_1 - x_2|^{1/2} + |t_2 - t_1|^{1/3}}\right\rceil$.
Then $mk_0 \le t_1^{2/3}$, and by a union bound we have that
\begin{align}
\nonumber
	\P(|Y_1 - Y_2| \ge a (|x_1 - x_2|^{1/2} + &|t_2 - t_1|^{1/3})) \\ 
    \nonumber
\le \sum_{k = -k_0+1}^{k_0} \P(Y_1 \ge mk/2,& \;\; Y_2 \le m(k-1)/2) + \P(Y_1 \le m(k-1)/2, \;\; Y_2 \ge mk/2) \\
\nonumber
&+ \P(|Y_1| \ge mk_0/2) + \P(|Y_2| \ge mk_0/2) \\
\nonumber
    \le \sum_{k = -k_0+1}^{k_0} \P(Y_1 \ge mk/2,& \;\; Y_2 \le m(k-1)/2) + \P(Y_1 \le m(k-1)/2, \;\; Y_2 \ge mk/2) \\
\label{E:first-comp}
&+ 2 \exp(- cm^{3/2} t_1^{-1/2} k_0^{3/2}).
\end{align}
Here the second equality uses Lemma \ref{L:one-point-TASEP}. The final term in \eqref{E:first-comp} is $<2\exp(-c a^{3/4})$.

Since $k_0<\exp((a/C_0)^{1/2})+1$ (by the lower bound of $a$), it remains to bound each term in the sum by $2\exp(-c a^{1/2})$ for any $1<a<t_1^{1/4}$. We focus on $\P(Y_1 \ge mk/2, Y_2 \le m(k-1)/2)$, as the bound on the other term is identical. Let
$$
p = (y/2, -y/2), \quad q_i = (t_i/4 + m(k+ 1 - i)/4 + x_i/2, t_i/4 + m(k+ 1 - i)/4 - x_i/2), \quad i = 1, 2.
$$
Then similarly to the previous proof, we can write
\begin{align}
\label{E:YY-mK}
\P(Y_1 \ge mk/2, &\;\; Y_2 \le m(k-1)/2) = \P(X_\alpha(p; q_1) \le t_1,  X_\alpha(p; q_2) \ge t_2).
\end{align}
We have that $|\mu(q_1-p)-\mu(q_2-p)-t_1+t_2-m|<C(|t_1-t_2|t_1^{-2/3} + |x_1-x_2|t_1^{-1/3})$.
By taking $C_1$ large enough we have $|\mu(q_1-p)-\mu(q_2-p)|<m/2$.
Thus by denoting $\bar X_\alpha(u; v) = X_\alpha(u; v) - \mu(v- u)$ for any $u\le v\in \Z^2_{\ge}$, we have
\[
     \P(X_\alpha(p; q_1) \le t_1,  X_\alpha(p; q_2) \ge t_2) \le \P(|\bar X_\alpha(p; q_1) - \bar X_\alpha(p; q_2)| \ge m/2).
\]
Now take
\[
\tilde a = \frac{m}{t_1^{1/3} (|\hat x_1 - \hat x_2|^{1/2} + |\hat t_1 - \hat t_2 |^{1/3})  } = \frac{ m' a (|x_1 - x_2|^{1/2} + |t_1 - t_2|^{1/3})}{t_1^{1/3} (|\hat x_1 - \hat x_2|^{1/2} + |\hat t_1 - \hat t_2 |^{1/3})  },
\]
with $(\hat x_i, \hat t_i)$ taken to satisfy $(t_1 \hat t_i + 2^{5/3}t_1^{2/3}\hat x_i, t_1\hat t_i) = q_i$, for $i=1, 2$.
Observe that $2^{5/6}t_1^{1/3} |\hat x_1 - \hat x_2|^{1/2} = |x_1 - x_2|^{1/2}$, and $|\hat t_1 - \hat t_2| = |m + (t_1 - t_2) - (x_1 - x_2)|/t_1$.
By taking either $m'=1/2$ or $m'=1$ (where the choice depends only on $a, x_1, x_2, t_1, t_2$), we have
\[
\tilde a< \frac{Ca (|x_1 - x_2|^{1/2} + |t_1 - t_2|^{1/3})}{|x_1 - x_2|^{1/2} + |(t_1 - t_2)-(x_1-x_2)|^{1/3}} < Ca < Ct_1^{1/4}.
\]
Then we can use Propositions \ref{P:two-point-bd} and \ref{P:temporal-est}, and argue as in the proof of Theorem \ref{T:elpp} up to the bounds \eqref{E:y1y2strong} and \eqref{E:y1y3strong}, to conclude that $\P(|\bar X_\alpha(p; q_1) - \bar X_\alpha(p; q_2)| \ge m/2)< 2\exp(- c \tilde a^{3/4})$.
On the other hand, we have
\[
\tilde a > \frac{ca (|x_1 - x_2|^{1/2} + |t_1 - t_2|^{1/3})}{|x_1 - x_2|^{1/2} + |t_1-t_2|^{1/3} + m^{1/3}} > c(a\wedge m^{2/3}) > ca^{2/3}.
\]
Thus we get $\P(|\bar X_\alpha(p; q_1) - \bar X_\alpha(p; q_2)| \ge m/2)< 2\exp(- c a^{1/2})$.
\end{proof}

\subsection{Pre-landscapes and regularity properties}
\label{SS:prelandscapes}

In this subsection, we define half-space pre-landscapes, and establish their basic properties. As an upshot of the regularity results, we will prove that the KPZ fixed point in half-space is a continuous Markov process, when started from any initial condition of sub-parabolic growth.

For this subsection and throughout the next section, we will use the notation $f \diamond g$ for metric composition. More formally, consider $D \subset \R$, sets $D_1, D_2$ (usually a single point or subsets of $\R^k$ for some $k\in \N$), and functions $f:D_1 \times D \to \R, g:D \times D_2 \to \R$. Then $f \diamond g:D_1 \times D_2 \to \R$ is given by
$$
f \diamond g(x, y) = \sup_{z \in D} f(x, z) + g(z, y).
$$
\begin{definition}
\label{D:prelandscapes}
   A random function $\cM^\rho: \Q^4 \cap \mathbb{H}^2_\uparrow \to \R$ is a \textbf{(half-space) pre-landscape} with parameter $\rho \in \R \cup \{-\infty\}$ if the following three conditions hold:
\begin{enumerate}
    \item (Triangle inequality) For any $o = (x, s)$, $p = (y, r)$, $q = (z, t) \in \Q_{\ge 0} \times \Q$ with $s<r<t$, we have
				$\cM^\rho(o; p) + \cM^\rho(p; q) \le \cM^\rho(o; q)$.
				\item (Independent increments) For any $s_1<s_2<\cdots<s_k \in \Q$, the increments $\cM^\rho(\cdot, s_i; \cdot, s_{i+1}):\Q_{\ge 0}^2 \to \R, i\in\II{1, k-1}$ are independent.
                \item (KPZ fixed point marginals) Recall that $\NW_{+}$ is the set of all $f \in \UC_+$ such that $\supp(f) = \{x : f(x) > -\infty\}$ is finite. Consider any function $f \in \NW_{+}$ with $\supp(f) \subset \Q_{\ge 0}$ and $f(\R)\subset\Q\cup\{-\infty\}$.  Then for any $t, s \in \Q$, $s > 0$ and any finite set $Y \subset \Q_{\ge 0}$ we have that
				\[
				 (f \diamond \cM^\rho(\cdot, t; y, t + s) : y \in Y)
				\eqd (\fh^\rho (t, y; f) : y \in Y).
				\]
                Here the right-hand side above is given by Theorem \ref{T:half-space-fixed-point}.
\end{enumerate}
Moving forward, we will use the notation $\cM^\rho_{s, t}(\cdot, \cdot) = \cM^\rho(\cdot, s; \cdot, t)$ for the increment in a pre-landscape from time $s$ to time $t$, and the notation $\fh^\rho_t f = \fh^\rho (t, \cdot; f)$, which will help emphasize the semigroup structure under metric composition.
\end{definition}
Note that any subsequential limit $\cQ^\rho$ from Proposition \ref{P:KPZFP-marginals} is a pre-landscape of parameter $\rho$, by parts $1$ and $2$ of Proposition \ref{P:properties-elpp} and Proposition \ref{P:KPZFP-marginals}. 

Many of the properties in Proposition \ref{P:properties-elpp} have straightforward analogues for pre-landscapes. The important exception here is the stationary horizon property, Proposition \ref{P:properties-elpp}.10, which will only be realized as a consequence of our characterization theorem. The next two lemmas derive versions of parts $1$-$8$. We start with a pre-landscape analogue of the time-reversability in part $8$.

\begin{lemma}
			\label{L:Equation-QQ}
            For $\cM^\rho$ from Definition \ref{D:prelandscapes}, the time reversal $(x, s; y, t) \mapsto \cM^\rho(y, -s; x, -t)$ is also a half-space pre-landscape with parameter $\rho$.
		\end{lemma}
		
		\begin{proof}
        Properties $1, 2$ are immediate, so we just check property $3$.
			Take any $f, g \in \NW_+$ with $\supp(f), \supp(g) \subset \Q_{\ge 0}$, and $\supp(f), \supp(g)$ being finite, and $f(\R_{\ge 0}), g(\R_{\ge 0})\subset\Q\cup\{-\infty\}$. Then
			\begin{align*}
				\P\left [ \cM^\rho_{t-s, t} \diamond f \le - g \right] = \P\left [ g \diamond \cM^\rho_{t-s, t} \diamond f \le 0 \right] =\P\left [ g \diamond \cM^\rho_{t-s, t} \le - f \right] = \P[\fh^\alpha_s g \le -f].
			\end{align*}
Property $3$ then follows if we can show that $\P[\fh^\alpha_s g \le -f] =  \P[\fh^\alpha_s f \le -g]$. For this, consider the same chain of equalities, but with $\cM^\rho$ equal to a subsequential limit $\cQ^\rho$ from Proposition \ref{P:KPZFP-marginals}. In this case, the left-hand side above equals $\P\left [f \diamond \cQ_{-t, -t + s}^\rho  \le - g \right]$ by time-reversal invariance (Proposition \ref{P:properties-elpp}.8), which equals $\P[\fh^\rho_s f \le -g]$ by Proposition \ref{P:KPZFP-marginals}. 
		\end{proof}

        \begin{lemma}
            \label{L:M-lemma}
        For $\cM^\rho$ from Definition \ref{D:prelandscapes}, it has a continuous extension to $\mathbb H^2_\uparrow$, and this extension satisfies parts $1, 2, 4, 5, 6$, and $7$ of Proposition \ref{P:properties-elpp} (in the place of $\cQ^\rho$ there).    
        Moreover, for any $q > 0, c \in \R$ if we consider the rescaling
        \begin{equation}
        \label{E:KPZ-scale-quasi-invariance}
            \cM^{\rho, q, c}(x, s; y, t) := q^{-1} \cM^\rho(q^2 x, q^3 s + c; q^2 y, q^3 t + c)
        \end{equation}
        then $\cM^{\rho, q, c}|_{\Q^4 \cap \mathbb H^2_\uparrow}$ is a half-space pre-landscape of parameter $q \alpha$.
        \end{lemma}

        \begin{proof}
        First, $\cM^\rho$ satisfies the metric composition law (part $2$) on $\Q^4 \cap \mathbb H^2_\uparrow$ by the same proof as \cite[Lemma 3.3]{DZ24}. Here we use $\cQ^\rho$ from Proposition \ref{P:KPZFP-marginals} in the place of the full-space landscape, and the metric composition law and continuity of $\cQ^\rho$ from Proposition \ref{P:properties-elpp}.

        Next, $\cM^\rho$ satisfies parts $4, 5, 7$ on rational points since marginals of $\cM^\rho$ are the same as for any subsequential limit $\cQ^\rho$ from Proposition \ref{P:KPZFP-marginals}, by Definition \ref{D:prelandscapes}.3 and the time-reversibility in Lemma \ref{L:Equation-QQ}. For part $6$, using the metric composition law on rationals, for any points $(x, s; y, t), (x, s; y', t + r) \in \Q^4 \cap \mathbb H^2_\uparrow$ for $r > 0$ we can write
        $$
        (\cM^\rho_{s, t}(x, y), \cM^\rho_{s, t + r}(x, y')) = (\cM^\rho_{s, t}(x, y),\cM^\rho_{s, t} \diamond \cM^\rho_{t, t + r}(x, y')) \eqd (\cQ^\rho_{s, t}(x, y),\cQ^\rho_{s, t} \diamond \cQ^\rho_{t, t + r}(x, y')).
        $$
       Here the final equality in law uses the independent increment property of $\cM^\rho, \cQ^\rho$, and the fact that any marginal of the form $\cM^\rho_{s, t}(\cdot, x), \cM^\rho_{s, t}(x, \cdot)$ is the same as the corresponding one in $\cQ^\rho$ and $\cM^\rho$ by Definition \ref{D:prelandscapes}.3 and the time-reversibility in Lemma \ref{L:Equation-QQ}.

       Given parts $4, 5, 6$ for $\cM^\rho$ on rationals, the existence of a continuous extension follows from the Kolmogorov-Centsov criterion. The resulting continuous extension inherits parts $1, 2, 4, 5, 6, 7$ from $\cM^\rho$ restricted to rational points. Note that the metric composition law (which is a supremum on rationals) becomes a maximum almost surely over $\R$ by the parabolic decay bound in part $7$. Finally, the `Moreover' follows from the same rescaling property of the KPZ fixed point marginals, which itself requires parts $3$ and $8$ of Proposition \ref{P:properties-elpp}, together with Proposition \ref{P:KPZFP-marginals}.
           \end{proof}

We are now ready to give regularity estimates for pre-landscapes. Each of them is a complete analogue of a corresponding statement in \cite{DOV}, with a proof that goes through verbatim.

\begin{prop}
\label{P:only-easy-consequences}
Take any half-space pre-landscape $\cM^\rho$ continuously extended to $\mathbb{H}^2_\uparrow$. Let $\cK^\rho(x, s; y, t) = \cM^\rho(x, s; y, t) + \frac{(x-y)^2}{(t-s)}$ for any $(x,s;y,t)\in \mathbb{H}^2_\uparrow$.
\begin{enumerate}
\item \textup{(Modulus of continuity, analogue of \cite[Proposition 10.5]{DOV}).} For any $b \ge 2, \ep \le 1$, define
$$
K^\ep_b = \{(x, t; y, t + s) \in \mathbb H^2_\uparrow : s \in [\ep, b], \text{and } |t|, x, y \le b \}.
$$
For two points $u_i = (x_i, t_i; y_i, s_i)\in K^\ep_b$, $i = 1, 2$, let 
\begin{align*}
    \tau &= \tau(u_1, u_2) = \|(t_1, s_1) - (t_2, s_2)\|_2, \\
    \xi &= \xi(u_1, u_2) = \|(x_1, y_1) - (x_2, y_2)\|_2.
\end{align*}
Then if $\tau^{1/3} + \xi^{1/2} \le \min(\eps/b, 1/(\rho \vee 0))$, we have
$$
|\cK^\rho(u_1) - \cK^\rho(u_2)| \le M \lf(\tau^{1/3}\log^{2/3}(\tau^{-1}) + \xi^{1/2}\log^{1/2}(4b\xi^{-1}) \rg),
$$
with a random variable $M$  satisfying
$\p(M > a) \le b^{4}[(b/\ep)^{10} + (\rho \vee 0)^{10}] \exp(-ca^{3/2})$
for all $a > 0$. Here $c>0$ is a universal constant.
\item \textup{(Shape bound, analogue of \cite[Corollary 10.7]{DOV}).} Fix $\rho_0 > 1$ and assume $\rho < \rho_0$. Then there exists $c > 0$ depending on $\rho_0$ such that the following holds. For all $u = (x, t; y, t + s) \in \mathbb H^2_\uparrow$ with $s \le \rho_0$, we have
$$
\lf|\cK^\rho(u) \rg|\le M s^{1/3} \log^2\lf(\frac{2\|u\|_2}{s}\rg),
$$
where $M > 0$ is a random variable satisfying
$\p(M > a) \le 2\exp(-ca^{3/2})$
for all $a > 0$.
\item \textup{(Metric composition holds everywhere, analogue of \cite[Lemma 10.8]{DOV})}
For each $u = (x, r;y, t) \in \mathbb H^2_\uparrow$ and $s \in (r, t)$, define
	$f_{u, s}(z) = \cM^\rho(x,r;z,s)+\cM^\rho(z,s;y,t)$. Then on an event of probability $1$, $\cM^\rho$ satisfies the metric composition law
	\begin{equation}
	\label{E:L-met-strong}
	\cM^\rho(u)=\max_{z \in \mathbb R} f_{u, s}(z)
	\end{equation}
	for every $u = (x, r;y, t) \in \mathbb H^2_\uparrow$ and $s \in (r, t)$. Moreover, for any compact set $K \sset \mathbb H^2_\uparrow$, there exists a random variable $B_K$ such that for all $u = (x, r;y, t) \in K$ and $s \in (r, t)$, the set where $f_{u, s}$ attains its maximum is contained in the interval $[-B_K, B_K]$.
\end{enumerate}
\end{prop}

\begin{proof}
For part $1$,we first claim that for all $u_1 =(x_1, s_1; y_1, t_1), u_2=(x_2, s_2; y_2, t_2) \in K^\ep_b$ with $\tau^{1/3} + \xi^{1/2} \le \min(\eps/b, 1/(\rho \vee 0))$ we have 
\begin{equation}
\label{E:Ku}
\p\lf(|\cK^\rho(u_1) - \cK^\rho(u_2)|  \ge m\tau^{1/3} + \ell \xi^{1/2}\rg) \le Ce^{-cm^{3/2}} + Ce^{-c\ell^2}
\end{equation}
for universal constants $c$ and $C$. Indeed, we first apply Proposition \ref{P:properties-elpp}.6 to shift to the time coordinates in $u_1$ to $s_1, t_1$. The lower bound on $a$ in that proposition goes away because of our upper bound on $\tau$.  Note that when we do this, the collinearity condition in Proposition \ref{P:properties-elpp}.6 means that the spatial coordinates may shift by at most $\tau b/\eps$. Therefore to complete the bound we must apply Proposition \ref{P:properties-elpp}.5 with $z = \xi + \tau b/\eps$. At this point, the lower bound on $a$ in that proposition goes away because of our upper bound on both $\tau$ and $\xi$.

Now, we can think of $K^\ep_b$ as a subset of the box $S = [-b, b]^3 \X [\eps, 2b]$ with coordinates $x, y, t, s$. Therefore as in the proof of \cite[Proposition 10.5]{DOV}, we may complete the proof by appealing to a general lemma for converting tail bounds to modulus of continuity estimates. More precisely, we apply \cite[Lemma 3.3]{dauvergne2021bulk} with $\alpha_x = \alpha_y = 1/2, \alpha_s = \alpha_t = 1/3, \beta_x = \beta_y = 2, \beta_s = \beta_t = 3/2, r_x = r_y = \min(\eps/b, 1/(\rho \vee 0))^2, r_s = r_t = \min(\eps/b, 1/(\rho \vee 0))^3$. This gives the desired result after simplification.
The conditions that $b \ge 2$ and $\ep \le 1$ are used only in the simplification. 

Part $2$ follows verbatim as in the proof of \cite[Corollary 10.7]{DOV}. The two inputs required for that proof are \cite[Proposition 10.5]{DOV} and a one-point tail bound of the form
$$
\P(|\cK^\rho(x, t; y, t + s)| \ge as^{1/3}) \le 2 \exp(-ca^{3/2})
$$
for all $a > 0$. Part $1$ above gives the analogue here of \cite[Proposition 10.5]{DOV}, and Proposition \ref{P:properties-elpp}.4 (for $\cM^\rho$ by Lemma \ref{L:M-lemma}) gives the above bound for $\rho, s < \rho_0$, after allowing the constant $c$ to depend on $\rho_0$.

Part $3$ follows verbatim as in the proof of \cite[Lemma 10.8]{DOV}. The necessary inputs are metric composition for all rational $u, s$, which is provided by Proposition \ref{P:properties-elpp}.2 (for $\cM^\rho$ by Lemma \ref{L:M-lemma}) in our setting, and \cite[Corollary 10.7]{DOV}, which here is replaced by part $2$ above.
\end{proof}

We have now gathered all the necessary preliminary estimates to move forward with the proof of our characterization theorem, which we execute in the next section. We end this section by using the regularities in the previous proposition to extend the KPZ fixed point in half-space $\mathfrak{h}^\rho_t$ to a time-homogeneous Markov process.
Recall the space $\UC_{0,+}$ from \eqref{E:UC0+}.

\begin{theorem}
    \label{T:Markov-process}
 For $\rho \in \R \cup \{-\infty\}$, take any half-space pre-landscape $\cM^\rho$, continuously extended to $\mathbb{H}^2_\uparrow$.
 Consider a function $f \in \UC_{0,+}$.
For each $t > 0$, define the process
$\mathfrak{h}^\rho_t f = f \diamond \cM^\rho_{0, t}$, and set $\mathfrak{h}^\rho_0 f = f$.
Then $(\mathfrak{h}^\rho_t, t \ge 0)$ is a time-homogeneous Markov process on $\UC_{0,+}$, whose law does not depend on the choice of pre-landscape. Moreover, for almost surely $\cM^\rho$, $(t,f)\mapsto \mathfrak{h}^\rho_tf$ is continuous on $\UC_{0,+}$, and $\fh^\rho_t f$ is a continuous function for each $t > 0, f \in \UC_{0,+}$. Lastly, for any finite set $Y \subset \R_{\ge 0}, t > 0$ and any function $f \in \UC_{0,+}$  satisfying
\begin{equation}
    \label{E:linear-grwoth}
    \limsup_{x \to \infty}\frac{f(x)}{x} < \infty,
\end{equation}
the marginal
$
(\mathfrak{h}^\rho_t f(y) : y \in Y)
$
is given by Theorem \ref{T:half-space-fixed-point}. 
The Markov process $(\mathfrak{h}^\rho_t, t \ge 0)$ is the \textbf{KPZ fixed point in half-space} with parameter $\rho$.
\end{theorem}

\begin{proof}
The continuity and shape bound of $\cM^\rho$ in Proposition \ref{P:only-easy-consequences} imply that $\mathfrak{h}^\rho_t f(x)$ is continuous as a function of $t > 0$ and $x\ge 0$, and in the initial condition $f$. Continuity as $t \to 0$ of $\fh^\rho_t f$ in $\UC_{0,+}$ also follows from the continuity and shape bound of $\cM^\rho$. The fact that the process $(\fh^\rho_t f, t \ge 0)$ is a time-homogeneous Markov process follows since for $0 < s < t$, we have
$$
\fh^\rho_t f = f \diamond \cM^\rho_{0, t} = f \diamond \cM^\rho_{0, s} \diamond \cM^\rho_{s, t} \eqd (f \diamond \cM^\rho_{0, s}) \diamond \hat\cM^\rho_{0, t-s} \eqd \fh^\rho_{t-s} (f \diamond \cM_{0, s}).
$$
where $\hat\cM^\rho$ is an independent copy of $\cM^\rho$. The second equality above uses the metric composition law of $\cM^\rho$, the third equality uses the independent increment property and KPZ fixed point marginal properties (Definition \ref{D:prelandscapes}.2, 3) for $\cM$ when $f \in \NW_+$ has rational support and range. The extension to arbitrary $f$ follows since both sides of the third equality are continuous in law over $f \in \UC_{0,+}$ (due to the continuity and shape bound of $\cM^\rho$ in Proposition \ref{P:only-easy-consequences}). The first and fourth equalities above are by definition. 

Since $\mathfrak{h}^\rho_t f$ is a time-homogeneous Markov process, its law is determined by the laws of $\mathfrak{h}^\rho_t f$ for each individual choice of $f$. 
Such laws are determined by Definition \ref{D:prelandscapes}.3 for $f\in\NW_+$ with rational support and range, and for general $f\in \UC_{0,+}$ via the continuity in law of $f$, as in the previous step. 
In particular, the law  does not depend on the choice of $\cM^\rho$.
For $f$ satisfying \eqref{E:linear-grwoth}, the marginal is given by Theorem \ref{T:half-space-fixed-point}, due to Proposition \ref{P:KPZFP-marginals}.
\end{proof}

We provide another quick consequence of Definition \ref{D:prelandscapes}.3 and the regularity estimates.
It is a one-slope marginal version of Proposition \ref{P:properties-elpp}.10 for pre-landscapes.

\begin{lemma}  \label{lem:Mstat}
 For $\rho \in \R \cup \{-\infty\}$, take any half-space pre-landscape $\cM^\rho$, continuously extended to $\mathbb{H}^2_\uparrow$, and  the half-space stationary horizon $\cH^\DL \in \mathcal D([2\rho \vee 0, \infty), \R, 2\rho \vee 0)$ from Theorem \ref{T:half-stationary-horizon} (denoting that $\cH^\DL_\lambda(x) = \cH^\DL(\lambda,x)$).
Then for any $\lambda\ge 2\rho \vee 0$, we have
$$
\cH^\DL_{\lambda}\diamond \cM^\rho_{0, t} - (\cH^\DL_{\lambda}\diamond \cM^\rho_{0, t})(0)\eqd \cH^\DL_{\lambda} \eqd \cM^\rho_{0, t}\diamond\cH^\DL_{\lambda} - (\cM^\rho_{0, t}\diamond\cH^\DL_{\lambda})(0) .
$$
Moreover, when $\lambda=2|\rho|$, $\cH^\DL_{\lambda}$ is a Brownian motion with slope $-2\rho$.
\end{lemma}

\begin{proof}
The first equality in law is a direct consequence of the independence of the choice of pre-landscape in Theorem \ref{T:Markov-process}, and the result for $\cQ^\rho$ from Proposition \ref{P:KPZFP-marginals}, which is in Proposition \ref{P:properties-elpp}.10. The second equality in law uses the time-reversal in Lemma \ref{L:Equation-QQ}.

    The `Moreover' part follows from the construction of half-space stationary horizon in Section \ref{ss:hfsh}. If $\rho \ge 0, \lambda = 2 \rho$, then by construction (see the discussion after the statement of Proposition \ref{P:half-space-horizon-marginals}) we have that $\cH^\DL_\lambda$ is a Brownian motion of diffusivity $2$ and slope $\lambda$. If $\rho < 0$ so that $\lambda = - 2 \rho$, we can write
\[
\cH^\DL_{\lambda}(x)=\max\left\{ \cB_2(x), \cB_2(x) - X(-1,2) + \max_{0\le y\le x} \cB_1(y)-\cB_2(y)  \right\},
\]
where $\cB_1$, $\cB_2:\R_{\ge 0} \to \R$ are independent Brownian motions with diffusivity $2$ and slopes $-\lambda$ and $\lambda$, respectively, and $X(-1,2)\sim \Exp(  -2\rho  )$. Then by Corollary \ref{cor:exp-bro}, $\cH^\DL_\lambda$ is a Brownian motion with slope $-\lambda=-2\rho$.
\end{proof}

\section{Characterization}   \label{sec:chara}

In this section we prove our characterization of the half-space directed landscape, i.e., Theorem \ref{thm:chawdl}, which we restate as follows.

\begin{theorem}
\label{T:half-space-land}
For every $\rho \in \R$, there exists a unique (in law) pre-landscape from $\Q^4 \cap \mathbb{H}^2_\uparrow \to \R$ of parameter $\rho$. 
It has a continuous extension $\cL^\rho: \mathbb{H}^2_\uparrow \to \R$, which we call the \textbf{half-space directed landscape} with parameter $\rho$.
\end{theorem}

As discussed in the introduction, we require a separate argument to handle the case of $\rho = -\infty$. 

\begin{corollary}
    \label{C:-infty-case}
There is a unique (in law) random continuous function $\cL^{-\infty}:\mathbb{H}^2_\uparrow \to \R$ such that:
\begin{enumerate}
    \item For each $\rho \in\R$, $\cL^{-\infty}$ is stochastically dominated by $\cL^\rho$,  the half-space directed landscape with parameter $\rho$.
    \item For each $u = (x, s; y, s + t) \in \Q^4\cap \mathbb{H}^2_\uparrow$, we have
    $\cL^{-\infty}(u) \eqd \mathfrak{h}^{-\infty}_t\delta_x(y)$,
    where $\mathfrak{h}^{-\infty}_t$ is the half-space KPZ fixed point of parameter $-\infty$, and $\delta_x\in \UC_+$ such that $\delta_x(x)=0$ and $\delta_x(z)=-\infty$ for each $z\neq x$.
\end{enumerate}
We call $\cL^{-\infty}$ the \textbf{half-space directed landscape} with parameter $-\infty$.   
\end{corollary}

\begin{remark}
\label{R:cts-input}
  Unlike with Theorem \ref{T:half-space-land}, Corollary \ref{C:-infty-case} is stated with a continuous function as our input, which is not always as easy to obtain from the prelimit. If one wants to remove this constraint and start with a function defined only on $\Q^4 \cap \mathbb{H}^2_\uparrow$, one workaround is to replace the second condition with the assumption that $\cL^{-\infty}$ is a pre-landscape of parameter $-\infty$, which extends to a continuous function on all of $\mathbb{H}^2_\uparrow$ by Lemma \ref{L:M-lemma}.
\end{remark}

The bulk of this section is devoted to the proof of Theorem \ref{T:half-space-land}, and we will outline the proof shortly. Before doing so, we prove Corollary \ref{C:-infty-case} and Theorem \ref{thm:expconv} from Theorem \ref{T:half-space-land}.

\begin{proof}[Proof of Corollary \ref{C:-infty-case}]
The existence is by any subsequential limit $\cQ^{-\infty}$ from Proposition \ref{P:KPZFP-marginals}, plus the continuous extension in Lemma \ref{L:M-lemma}. We next prove uniqueness.

By condition $1$, we can construct a coupling of $\cL^{-\infty}$ and $\cL^{-n}, n \in \N$ so that 
$$
\cL^{-1} \ge \cL^{-2} \ge \cdots \ge \cL^{-\infty}.
$$
Set $\hat\cL = \lim_{n \to -\infty} \cL^{-n}$. We claim that almost surely, for any finite set $F\subset \Q^4 \cap \mathbb{H}^2_\uparrow$ we have that $\cL^{-\infty}|_F = \hat\cL|_F$. This will imply the corollary, since the law of $\hat\cL|_F$ is the limit of the law of $\cL^{-n}|_F$, which does not depend on the choice of $\cL^{-\infty}$. We have that $\cL^{-\infty}|_F \le \hat\cL|_F$, so it suffices to show that for any point $u = (x, s; y, t)\in \Q^4 \cap \mathbb{H}^2_\uparrow$ we have that $\E\cL(u) = \E\hat\cL(u)= \lim_{n \to \infty} \E \cL^{-n}(u)$. Observe
\begin{equation}
\label{E:L-sides}
    \E \cL^{-n}(u) - \E\cL^{-\infty}(u) \le \E \cL^{-n}(0, s; 0, t) - \E\cL(0, s; 0, t).
\end{equation}
Indeed, since this equality only concerns one-point laws, by condition $2$ it suffices to prove this when $\cL^{-\infty}, \cL^{-n}$ are limits of exponential LPP. In this case, the bound follows from the second inequality in Corollary \ref{C:exp-quadrangle}. Now, the right-hand side of \eqref{E:L-sides} goes to $0$ as $n\to\infty$, see the discussion after Definition 2.9 in \cite{BBCS1}.
\end{proof}

\begin{proof}[Proof of Theorem \ref{thm:expconv}]
The processes $\cL_n$ are tight by Theorem \ref{T:elpp}, and any subsequential limit is a pre-landscape of parameter $\rho$ by parts $1$ and $2$ of Proposition \ref{P:properties-elpp} and Proposition \ref{P:KPZFP-marginals}. In the case when $\rho \in \R$, the uniqueness of the subsequential limit then follows from Theorem \ref{T:half-space-land}. If $\rho = -\infty$, then $\cL_n$ is stochastically dominated by itself with any finite $\rho$ (for $n$ large enough depending on $\rho$), via the coupling where we simply rescale the boundary random variables in the exponential environment. Therefore any subsequential limit of $\cL_n$ satisfies both conditions of Corollary \ref{C:-infty-case}, which implies uniqueness of the limit. 
\end{proof}

For Theorem \ref{T:half-space-land}, the existence is fulfilled by any subsequential limit $\cQ^\rho$ from Proposition \ref{P:KPZFP-marginals}, and the continuous extension is from Lemma \ref{L:M-lemma}. To prove uniqueness, we broadly follow the framework for characterizing the full-space directed landscape from \cite{DZ24}, which is described in Section \ref{SS:exchange}. Many steps are exactly as in that paper, and indeed, Section \ref{SS:exchange} and Section \ref{SS:coupling-and-branch} are mostly recap from \cite{DZ24}. Here most proofs are omitted as they go through verbatim.

The main new ingredient is the proof of the key quantitative branching estimate Proposition \ref{prop:key1}, to be given in Section \ref{ss:key1}. The proof is much more involved than its counterpart in the full-space setting (in \cite[Section 4]{DZ24}), as geodesic behavior and coalescence probabilities change near the boundary, and the half-space stationary horizon has a more complicated description than the full-space stationary horizon. 
Beyond the proof of the branching estimate, the proofs in \cite[Section 4]{DZ24} use a quantitative absolute continuity of the Airy$_2$ process with respect to Brownian motion from \cite{dauvergne2024wiener} (see also \cite{calvert2019brownian,HamBr}).
This result is not available for the corresponding marginals of the half-space KPZ fixed point, even when starting from a single narrow wedge at the boundary. To circumvent this, we analyze stationary initial and final conditions (instead of only the initial condition in \cite{DZ24}), which requires different Brownian removal arguments.

\subsection{The exchange strategy}
\label{SS:exchange}
From now until the end of this section, we fix $\rho\in \R$, and let $\cM=\cM^\rho$ denote the continuous extension of a half-space pre-landscape with parameter $\rho$.
We also let $\cQ=\cQ^\rho$ be a subsequential limit from Proposition \ref{P:KPZFP-marginals}.
We use the notations $\cM_{s, t} = \cM(\cdot, s; \cdot, t)$ and $\cQ_{s, t} = \cQ(\cdot, s; \cdot, t)$ for any $s<t$.

By the metric composition law and independence of increments of $\cM$, to prove Theorem \ref{T:half-space-land} it suffices to prove the following.
    \begin{prop}   \label{prop:MEeq}
        For any $s<t$, we have $\cM_{s, t} \eqd \cQ_{s, t}$.
    \end{prop}
        
 By the KPZ rescaling and time shift in \eqref{E:KPZ-scale-quasi-invariance} and  parts $3$ and $8$ of Proposition \ref{P:properties-elpp}, it suffices to prove Proposition \ref{prop:MEeq} with $s = 0, t = 1$.	By continuity of $\cM, \cQ$, for this it suffices to prove the following finite dimensional distribution equality.
		\begin{prop}  \label{prop:finid}
			Take any $k\in\N$, and $x_1,\ldots, x_k >0$, and $y_1,\ldots, y_k >0$.
			Then 
			\[
			\{\cM_{0,1}(x_i,y_i)\}_{i=1}^k\stackrel{d}{=}\{\cQ_{0,1}(x_i,y_i)\}_{i=1}^k.
			\]
		\end{prop}
    For technical reasons, it is more convenient to work with the following Brownian boundary version.

Take $\lambda_1>\lambda_2>\cdots>\lambda_k > 2\rho\vee 0$. For the half-space stationary horizon $\cH^{\DL}\in \mathcal D([2\rho \vee 0, \infty), \R, 2\rho \vee 0)$ from Theorem \ref{T:half-stationary-horizon}, 
denote $\cH_i=\cH^\DL(\lambda_i, \cdot)|_{\R_{\ge 0}}$ for each $i\in\II{1, k}$. Next, let $W > 0$ and $\cD_1, \ldots, \cD_k: \R_{\ge 0}\to \R$ be independent Brownian motions, each with slope $-W$. 

Let $\cE[W] = \cE_1[W, \cH] \cap \cE_2[W, \cD]$, where we define the events
\begin{align}
\label{E:HW}
\cE_1[W, \cH]:\quad \cH_i(x)&<\cH_i(W) + W + W(x-W)/3, \qquad \text{ for all } i \in \II{1, k}, x \ge W. \\
\label{E:DW}
  \cE_2[W, \cD]:\quad  \cD_i(x)&<\cD_i(W) + W - 2W(x-W)/3, \qquad \text{ for all } i \in \II{1, k}, x \ge W.
\end{align}
We will prove the following Brownian boundary version of Proposition \ref{prop:finid}.
\begin{prop}   \label{prop:ulteq}
For  $\{\cH_i\}_{i=1}^k, \cM, \cQ, \{\cD_i\}_{i=1}^k$ independent of each other, we have
\[
(\{\cH_i\}_{i=1}^k, \{ \cH_i\diamond \cM_{0,1} \diamond \cD_i \}_{i=1}^k, \{\cD_i\}_{i=1}^k ) \mathds{1}[\cE[W]]\eqd (\{\cH_i\}_{i=1}^k, \{ \cH_i\diamond \cQ_{0,1} \diamond \cD_i \}_{i=1}^k, \{\cD_i\}_{i=1}^k ) \mathds{1}[\cE[W]].
\]
\end{prop}

The proof of Proposition \ref{prop:ulteq} uses a Lindeberg exchange strategy, which is also used in the proof of the full-space characterization in \cite{DZ24}.

For each $t\in [0,1]$ and $i\in\II{1, k}$, we denote\footnote{We use the convention that $\cM(x,t;x,t)=\cQ(x,t;x,t)=0$ for any $x\ge 0$ and $t$, and $\cL(x,t;y,t)=\cM(x,t;y,t)=-\infty$ for any $t$ and $x, y\ge 0$, with $x\neq y$.}
\begin{equation}   \label{eq:cMdef}
\cM_i^t = \cH_i \diamond \cQ_{0,t} \diamond \cM_{t,1} \diamond \cD_i,
\end{equation}
        		where $\{\cH_i\}_{i=1}^k, \cQ, \cM, \{\cD_i\}_{i=1}^k$ are all independent. 
We note that $\{\cM_i^0\}_{i=1}^k = \{ \cH_i\diamond \cM \diamond \cD_i \}_{i=1}^k$ and $\{\cM_i^1\}_{i=1}^k = \{ \cH_i\diamond \cQ \diamond \cD_i \}_{i=1}^k$.

		Our basic strategy will be to show that for $0\le s < t\le 1$, we can couple the random variables $\{\cM^s_i\}_{i=1}^k$ and $\{\cM^t_i\}_{i=1}^k$ such that $\E |\cM^s_i - \cM^t_i | \mathds{1}[\cE[W]]$ is much smaller than $t-s$ when $t-s$ is small, for each $i\in\II{1, k}$. Note that this will typically not be true if $\cM$ and $\cQ$ are taken to be independent.

        More precisely, we will prove the following swap result. 
For this result and throughout the remainder of the section, we use $C,c>0$ to denote large and small constants that are allowed to depend on $\rho, \lambda_1, \lambda_2, \ldots, \lambda_k$ and $W$, and their values can change from line to line.
		\begin{prop}  \label{prop:clo}
			For each  $\theta \in (0, 1/4)$, there exists $\delta>0$ such that the following holds. For $s <  t \in [\theta, 1 - \theta]$, we can find a coupling of $\{\cH_i\}_{i=1}^k, \cM, \cQ, \{\cD_i\}_{i=1}^k$ such that each of the tuples $(\{\cH_i\}_{i=1}^k, \{\cM_i^s\}_{i=1}^k, \{\cD_i\}_{i=1}^k)$ and  $(\{\cH_i\}_{i=1}^k, \{\cM_i^t\}_{i=1}^k, \{\cD_i\}_{i=1}^k)$ has the same law as in \eqref{eq:cMdef}, and 
			$$
			\P\left[\big\{\{\cM_i^s\}_{i=1}^k \neq \{\cM_i^t\}_{i=1}^k \big\}\cap \cE[W] \right] \le C (t-s)^{2/3+\delta},
			$$
			where the constant $C$ can also depend on $\theta$.
		\end{prop}
Before proving this proposition, we note that it readily implies Proposition \ref{prop:ulteq}. More precisely, we can derive the following tail estimate for the difference $\cM^s_i - \cM^t_i$.
		\begin{lemma}  \label{lem:diff-tail}
			For any $0\le s<t \le 1$, under any coupling between $\{\cH_i\}_{i=1}^k, \cM, \cQ, \{\cD_i\}_{i=1}^k$, we have the following:
			for any $i\in \llbracket 1, k\rrbracket$ and $a>0$,
			\[
			\P\left[|\cM^s_i - \cM^t_i | >  a(1 - \log(t-s))^2(t-s)^{1/3}\right] < 2\exp(-ca).
			\]
		\end{lemma}
The proof of Lemma \ref{lem:diff-tail} is essentially verbatim from that of \cite[Lemma 3.10]{DZ24}, with the shape bound in Proposition \ref{P:only-easy-consequences}.2 as an input.
We can then derive Proposition \ref{prop:ulteq} using Proposition \ref{prop:clo} and Lemma \ref{lem:diff-tail}, following the same arguments as the proof of  \cite[Proposition 3.8]{DZ24}, by bounding the $L^1$-Wasserstein distance between the laws.
We omit the details of these proofs.

In now remains to (1) prove Proposition \ref{prop:clo}, and (2) derive Proposition \ref{prop:finid} from Proposition \ref{prop:ulteq}. We achieve these in the remaining subsections.

\subsection{Coupling and branch estimates}
\label{SS:coupling-and-branch}

In this and the next subsection we prove Proposition \ref{prop:clo}. We fix $\theta \in (0, 1/4)$, and all the constants $C,c>0$ are allowed to depend on $\theta$ (in addition to $\rho, \lambda_1, \lambda_2, \ldots, \lambda_k$ and $W$).

The main argument in the proof is a key branching estimate, which we state in Proposition \ref{prop:key1}.

Take $s < t = s + r \in [\theta, 1 - \theta]$, and throughout we assume that $r = t-s$ is sufficiently small. We take $\{\cH_i\}_{i=1}^k$, $\cQ_{0,s}$, $\cM_{t,1}$, and $\{\cD_i\}_{i=1}^k$ that are independent of each other.

We consider the following sets.
For each $i\in\II{1, k}$, denote
		\[
		\tilde{A}_i=\{x \in \R :\cH_i\diamond \cQ_{0,s}(x) + \cM_{t,1}\diamond \cD_i(x) > \cH_i\diamond \cQ_{0,s}\diamond \cM_{t,1}\diamond \cD_i - |\log(r)|^{10}r^{1/3} \},
		\]
and let $A_i$ be the $|\log(r)|^{10}r^{2/3}$-neighborhood of $\tilde{A}_i$. 
We note the following bound on each $A_i$ under $\cE[W]$.
\begin{lemma}  \label{lem:boundAi}
We have $\P\big[A_i \subset [0, a] \;\big\mid\; \cE[W] \big] > 1-2\exp(-ca^{3/2})$, for each $i\in \llbracket 1, k\rrbracket$ and $a>0$.
\end{lemma}
\begin{proof}
This follows from the shape bound in Proposition \ref{P:only-easy-consequences}.2.
\end{proof}

Next, for each $i, j \in \llbracket 1, k\rrbracket$, $i\neq j$, we let $\tilde{D}_{i,j}$ be the set where
		$\cH_i\diamond \cQ_{0,s}-\cH_j\diamond \cQ_{0,s}$
		is non-constant, i.e.,
		\begin{multline*}
		\tilde{D}_{i,j} = \big\{x \in \R : \text{ for each } \ep > 0, \text{ there exists } y \in (x-\ep, x + \ep) \text{ with }\\ \cH_i\diamond \cQ_{0,s}(x)-\cH_j\diamond \cQ_{0,s}(x) \ne \cH_i\diamond \cQ_{0,s}(y)-\cH_j\diamond \cQ_{0,s}(y) \big\}.		    
		\end{multline*}
		We then let $D_{i,j}$ be the $|\log(r)|^{10}r^{2/3}$ neighborhood of $\tilde{D}_{i,j}$. 
		\begin{prop}  \label{prop:key1}
        There is a universal constant $\delta>0$ such that the following is true. For each $i, j \in \llbracket 1,k\rrbracket$, $i\neq j$,
and any closed interval $I\subset [0, |\log(r)|^{2/3+0.0001}]$ of length $r^{2/3}$,
		we have $\P[\{A_i \cap A_j \cap D_{i,j} \cap I \neq \emptyset\}] < C (\max I)^{-1}r^{4/3+\delta}$.
		\end{prop}
By a union bound over $I$ and Lemma \ref{lem:boundAi}, this immediately implies the disjointness of $A_i$, $A_j$ and $D_{i,j}$.
Namely, if we let $\cE_*$ be the event where $A_i\cap A_j \cap D_{i,j}=\emptyset$, for each $i, j\in \llbracket 1,k\rrbracket$, $i\neq j$, then we have the following bound.
		\begin{corollary} \label{cor:key1}
        There is a universal constant $\delta>0$ such that $\P[\cE[W]\setminus \cE_*] <Cr^{2/3+\delta}$.
		\end{corollary} 
This is the half-space version of \cite[Proposition 3.11]{DZ24}. Given Corollary \ref{cor:key1}, the proof of Proposition \ref{prop:clo} is essentially the same as that of \cite[Proposition 3.9]{DZ24}. We sketch the arguments below.

First, we define the required coupling between $\cQ_{s,t}$ and $\cM_{s,t}$ in Proposition \ref{prop:clo}. This coupling depends on $\{\cH_i\}_{i=1}^k$, $\cQ_{0,s}$, $\cM_{t,1}$, and $\{\cD_i\}_{i=1}^k$. It is essentially the same coupling constructed in \cite[Section 3.3]{DZ24}.

For this next lemma, we use that under $\cE[W]$, each set $A_i$ is bounded, and consists of finitely many intervals $A_{i,1}, \ldots, A_{i,\ell(i)}$, each of length at least $2|\log(r)|^{10}r^{2/3}$. The boundedness of each $A_i$ follows from the tail bound in Lemma \ref{lem:boundAi}.

		\begin{lemma}  \label{lem:exiF-hf} 
        There exists a function $F:\R_{\ge 0}\to \R\cup\{-\infty\}$ that is $\sig(\{\cH_i\}_{i=1}^k, \cQ_{0,s}, \cM_{t,1}, \{\cD_i\}_{i=1}^k)$-measurable, and  satisfies the following conditions:
			\begin{itemize}
				\item On $\cE[W]\cap\cE_*$, for each $i\in \llbracket 1,k\rrbracket$, $F-\cH_i\diamond \cQ_{0,s}$ is constant on each interval $A_{i, j}$, $j\in\II{1, \ell(i)}$.
				\item $F=-\infty$ outside of $\bigcup_{i=1}^k A_i$,
				\item For any $a>0$, with probability at least $1 - 2\exp(-ca)$ the following is true:  
                $$
                |F(x)-F(y)| \le a \log(|x-y|^{-1} + 2)\sqrt{|x-y|}
                $$
                for all $x, y \in \bigcup_{i=1}^k A_i$, $|x-y|\le |\log(r)|^{10}r^{2/3}$. 
			\end{itemize}
		\end{lemma}
    This lemma is the analogue of \cite[Lemma 3.12]{DZ24} in our setting.
The proof is essentially verbatim. Instead of repeating all the details, we give a brief sketch, focusing on the differences in the proof.
\begin{proof}[Sketch proof of Lemma \ref{lem:exiF-hf}]
On $\cE[W]^c\cup\cE_*^c$, let $F = 0$ on $\bigcup_{i=1}^k A_i$ and $-\infty$ off this set. On $\cE[W]\cap\cE_*$, we construct $F$ on the intervals $A_{i,1}, \ldots, A_{i,\ell(i)}$, for each $i\in \llbracket 1, k\rrbracket$, inductively from left to right, as in the proof of \cite[Lemma 3.12]{DZ24}.

The only difference (from the proof of \cite[Lemma 3.12]{DZ24}) is in justifying the last condition.
Due to Lemma \ref{lem:boundAi}, and using the stationarity of $\cQ$ stated in Proposition \ref{P:properties-elpp}.10, it suffices to show that
\begin{equation}   \label{eq:cH1}
\P[|\cH_i(x)-\cH_i(y)| <a \log(|x-y|^{-1} + 2)\sqrt{|x-y|}, \quad \forall x, y \in [0, a], i \in \llbracket 1,k\rrbracket] > 1-2\exp(-ca).
\end{equation}
For this, note that for any $i\in \llbracket 1,k\rrbracket$ and $0\le x<y$, we have 
\[
|\cH_i(x)-\cH_i(y)| \le \sum_{j=1}^{2k} \max_{z\in [x,y]} \cB_j(z) - \min_{z\in [x,y]} \cB_j(z),
\]
where $\{\cB_j\}_{j=1}^{2k}$ are the Brownian motions used in the definition of the half-space stationary horizon marginals in Proposition \ref{P:half-space-horizon-marginals}.
Then from the fact that 
\[
\P[|\cB_j(x)-\cB_j(y)| <a \log(|x-y|^{-1} + 2)\sqrt{|x-y|}, \quad \forall x, y \in [0, a], j \in \llbracket 1,2k\rrbracket] > 1-2\exp(-ca),
\]
 \eqref{eq:cH1} follows.
\end{proof}

Then as in \cite[Section 3.3]{DZ24} (see the discussion around eq.~(25)), we define a coupling between $\cQ_{s,t}$ and $\cM_{s,t}$, such that $F\diamond \cQ_{s,t}=F\diamond \cM_{s,t}$ almost surely. 
The proof of Proposition \ref{prop:clo} is in terms of this coupling, and follows verbatim from the proof of \cite[Proposition 3.9]{DZ24}, with the only differences being:
\begin{enumerate}
    \item We bound $|\cH_i\diamond \cQ_{0,s}(x)-\cH_i\diamond \cQ_{0,s}(y)|$ for $0\le x < y \le |\log(r)|^2$ and $|\cH_i\diamond \cQ_{0,s}(x)-\cH_i\diamond \cQ_{0,s}(0)|$ for any $x>0$. 
    Due to  the stationarity of $\cQ$ stated in Proposition \ref{P:properties-elpp}.10, these are the same in law as $|\cH_i(x)-\cH_i(y)|$ for $0\le x < y \le |\log(r)|^2$ and $|\cH_i(x)|$ for $x>0$, respectively. The bounds for them follow from the same arguments as the proof of \eqref{eq:cH1}.
    \item We use Lemma \ref{lem:boundAi} to bound the probability that $\bigcup_{i=1}^k A_i \not\subset [0, |\log(r)|^{2}/2]$, under $\cE[W]$.
\end{enumerate}
We omit the remaining details.

\subsection{Proof of the branch estimate}   \label{ss:key1}
We next prove Proposition \ref{prop:key1}, where the arguments differ from those in \cite[Section 4]{DZ24} in the full-space setting. 

Without loss of generality, it suffices to prove Proposition \ref{prop:key1} for $i=1$ and $j=2$.
Again, below we use $C,c$ to denote large and small constants which may depend on $\rho, \lambda_1, \lambda_2$, and $W, \theta$, and may change from line to line.

Fix an interval $I\subset [0, |\log(r)|^{2/3+0.0001}]$ of length $r^{2/3}$. We define $I_*=[z_-, z_+]$ to be the closed $|\log(r)|^{10}r^{2/3}$-neighborhood of $I$ in  $\R_{\ge 0}$.
Next, for $i = 1, 2$ define the event
\[
\cE_i:\quad \max_{x\in I_*} \cH_i\diamond \cQ_{0,s}(x)+\cM_{t,1}\diamond \cD_i(x) > \max_{x\in [0, z_++1]}\cH_i\diamond \cQ_{0,s}(x)+\cM_{t,1}\diamond \cD_i(x) - |\log(r)|^{10}r^{1/3},
\]
and set
\[
\cE_{1,2}:\quad \cH_1\diamond \cQ_{0,s}(z_-)-\cH_2\diamond \cQ_{0,s}(z_-) \neq \cH_1\diamond \cQ_{0,s}(z_+)-\cH_2\diamond \cQ_{0,s}(z_+).
\]
By construction, $\{A_1 \cap I \neq \emptyset\}\subset \cE_1$ and $\{A_2 \cap I \neq \emptyset\}\subset \cE_2$.
As for $\cE_{1,2}$, note that  
\begin{equation}  \label{eq:hquad}
\cH_1\diamond \cQ_{0,s}(x) - \cH_2\diamond \cQ_{0,s}(x) \ge \cH_1\diamond \cQ_{0,s}(y) - \cH_2\diamond \cQ_{0,s}(y),
\end{equation}
for $0\le x < y$. This follows from the same inequality without the metric composition with $\cQ_{0,s}$ (since $\cH^{\DL}\in \mathcal D([2\rho \vee 0, \infty), \R, 2\rho \vee 0)$ by Theorem \ref{T:half-stationary-horizon}) and the stationarity in Proposition \ref{P:properties-elpp}.10. Therefore the event $\cE_{1,2}^c$ implies that $\cH_1\diamond \cQ_{0,s}-\cH_2\diamond \cQ_{0,s}$ is constant in $[z_-, z_+]$, thus $D_{1,2}\cap I =\emptyset$.
In summary,
\[
\{A_1 \cap A_2 \cap D_{1,2} \cap I \neq \emptyset\}\subset \cE_1\cap\cE_2\cap\cE_{1,2}.  
\]
It remains to bound $\P[\cE_1\cap\cE_2\cap\cE_{1,2}]$.

\subsubsection{A comparison between Brownian motions}
To bound $\P[\cE_1\cap\cE_2\cap\cE_{1,2}]$, we will repeatedly use the following quantitive absolute continuity bound between Brownian motions with different drifts. For this lemma, we use the notation $\cC([0,z])$ for the space of continuous functions on $[0,z]$ with the uniform topology.
\begin{lemma}  
\label{lem:bmcomp}
Let $\cX_1$ and $\cX_2$ be independent Brownian motions (with diffusivity $2$) and slopes $\theta_1, \theta_2$, respectively. Take any $0<\eps<0.1$ and $z>0$.
For any Borel set $\cE \subset \cC([0,z])$, if $\P[\cX_1|_{[0,z]} \in \cE] =\eps$ then 
\[
\P[\cX_2|_{[0,z]}\in \cE] \le C_0\eps \exp(C_0 |\log(\eps)|^{1/2}z^{1/2} |\theta_1-\theta_2| + C_0 z(\theta_1-\theta_2)^2),
\] where $C_0>0$ is a universal constant.
\end{lemma}
\begin{proof}
Without loss of generality we may assume that $\theta_1< \theta_2$. Moreover, by subtracting a linear function of slope $\theta_1$ from both $\cX_1, \cX_2$ we may assume $\theta_1 = 0 < \theta_2$.

Conditional on $\cX_1(z)$ (resp.~$\cX_2(z)$), $\cX_1|_{[0,z]}$ (resp.~$\cX_2|_{[0,z]}$) is a Brownian bridge (with diffusivity $2$). Therefore we can write
\begin{align*}
\P[\cX_2|_{[0,z]}\in \cE] &= \int \P[\cX_2|_{[0,z]}\in \cE \mid \cX_2(z)=h] \frac{\exp(-(h-\theta_2z)^2/(4z)) }{\sqrt{4\pi z} } dh \\ &= \int \P[\cX_1|_{[0,z]}\in \cE \mid \cX_1(z)=h]  \frac{\exp(-h^2/(4z)) }{\sqrt{4\pi z} } \exp(h\theta_2/2 - z\theta_2^2/4  ) dh.
\end{align*}
Via integration by parts in $h$, this equals
\begin{align*}
\frac{\theta_2}{2}&\int \P[\cX_1|_{[0,z]}\in \cE, \cX_1(z)>h]  \exp(h\theta_2/2 - z\theta_2^2/4  ) dh
\\
\le 
\frac{\theta_2}{2}&\int (\P[\cX_1(z)>h]\wedge \eps)   \exp(h\theta_2/2 - z\theta_2^2/4)dh.
\end{align*}
Take $h_*$ such that $\P[\cX_1(z)>h_*]=\eps$. Then by splitting the above integral into pieces on $(-\infty, h_*]$ and $[h_*, \infty)$, and integration by parts in $h$, it equals
\[
\int_{h_*}^\infty \P[\cX_1(z)=h] \exp(h\theta_2/2 - z\theta_2^2/4) dh.
\]
This can be further written as
\[
\P[\cX_2(z)> h_* + z\theta_2] +  \int_{h_*}^{h_* + z\theta_2} \P[\cX_1(z)=h] \exp(h\theta_2/2 - z\theta_2^2/4  ) dh.
\]
Noting that $\P[\cX_2(z)> h_* + z\theta_2]  = \P[\cX_1(z)>h_*]=\eps$, the above is at most
\[
\eps \left(1 + \exp(h_*\theta_2/2 + z\theta_2^2/4)\right).
\]
Finally, since $\P[\cX_1(z)>10|\log(\eps)|^{1/2}z^{1/2}]<\eps$, we have $h_* < 10|\log(\eps)|^{1/2}z^{1/2}$. The conclusion thereby follows.
\end{proof}

\subsubsection{Decomposition into events in Brownian motions}   \label{sssec:decompinto}
We now proceed to bound $\P[\cE_1\cap \cE_2\cap \cE_{1,2}]$. 
Using the stationarity of $\cQ$ from Proposition \ref{P:properties-elpp}.10, we can replace $\cH_1\diamond \cQ_{0,s}$ and $\cH_2\diamond \cQ_{0,s}$ in the above events by $\cH_1$ and $\cH_2$.

We will also replace $\cM_{t,1}\diamond \cD_1$ and $\cM_{t,1}\diamond \cD_2$, using Lemma \ref{lem:Mstat} and Lemma \ref{lem:bmcomp}. More precisely, we let  $\cD_1^*, \cD_2^*:[0,\infty)\to \R$ be independent Brownian motions with slope $-2\rho$. 
By Lemma \ref{lem:Mstat}, $\cM_{t,1}\diamond \cD_1^*-\cM_{t,1}\diamond \cD_1^*(0)$ and $\cM_{t,1}\diamond \cD_2^*-\cM_{t,1}\diamond \cD_2^*(0)$ are still Brownian motions with slope $-2\rho$ (although not independent). This second replacement comes with a Radon-Nikodym derivative cost, which we manage using Lemma \ref{lem:bmcomp}.

Let $\cE_i^*, i = 1, 2$ and  $\cE_{1,2}^*$ be the events where
\begin{align*}
   \cE_i^*&:\quad \max_{x\in I_*} \cH_i(x)+\cM_{t,1}\diamond \cD_i^*(x) > \max_{x\in [0, z_++1]}\cH_i(x)+\cM_{t,1}\diamond \cD_i^*(x) - |\log(r)|^{10}r^{1/3}, \\ 
   \cE_{1,2}^*&:\quad \cH_1(z_-)-\cH_2(z_-) \neq \cH_1(z_+)-\cH_2(z_+).
\end{align*}
It now suffices to upper bound $\P[\cE_1^*\cap\cE_2^*\cap\cE_{1,2}^*]$. 

We let the pair of processes $\{\cH_i\}_{i=1}^2$ be constructed as $\{R^\DL_i\}_{i=1}^2$ in Proposition \ref{P:half-space-horizon-marginals}. Namely, take four Brownian motions $\{\cB_i\}_{i=1}^4$ with slopes being $(\lambda_1, \lambda_2, -\lambda_2, -\lambda_1)$, and an array of exponential random variables $\{X(-i, j) : i\in\II{1, 2}, j \in \II{i +1, 5 - i}\}$, where
\begin{align*}
 &X(-1, 2) \sim \Exp(\lambda_1/2 - \lambda_2/2), \quad &&X(-1, 3) \sim \Exp(\lambda_1/2 + \lambda_2/2), \\
\quad &X(-1, 4) \sim \Exp(\lambda_1/2 - \rho), \quad, &&X(-2, 3) \sim \Exp(\lambda_2/2 - \rho). 
\end{align*}
All these Brownian motions and exponential random variables are independent of each other.
Recall from Section \ref{ss:hfsh} the cadlag LPP notations, and the half-space exponential-Brownian LPP, where $X\cB$ is the concatenation of the sequence of cadlag functions generated from $\{X(-i, j) : i\in\II{1, 2}, j \in \II{i +1, 5 - i}\}$, and $\cB$.
We then have that for each $i = 1, 2$,
\begin{equation}
\label{E:Hix}
   \cH_i(x) = X\cB(-i,i; x, 4) - X\cB(-i,i; 0, 4).
\end{equation}
Now consider the following events.
\begin{align*}
    \cE^*_{1,+}&:\quad 
\cB_3(z_+)+\cM_{t,1}\diamond \cD_1^*(z_+) > \max_{x\in [z_+, z_++1]} \cB_3(x)+\cM_{t,1}\diamond \cD_1^*(x)-2|\log(r)|^{10}r^{1/3},\\
\cE^*_{2,-}&:\quad \cH_2(z_-)+\cM_{t,1}\diamond \cD_2^*(z_-) > \max_{x\in [0, z_-]} \cH_2(x)+\cM_{t,1}\diamond \cD_2^*(x)-2|\log(r)|^{10}r^{1/3},\\
\cE^*_{2,+}&:\quad \cB_4(z_+)+\cM_{t,1}\diamond \cD_2^*(z_+) > \max_{x\in [z_+, z_++1]} \cB_4(x)+\cM_{t,1}\diamond \cD_2^*(x)-2|\log(r)|^{10}r^{1/3},
\end{align*}
and $\cE_{TV}$ where
\[
\quad \max_{x\in I_*} \cB_i(x) < \min_{x\in I_*} \cB_i(x)+ |\log(r)|^{10}r^{1/3}/20, \quad \forall i\in\{1, 2, 3, 4\},
\]
\[
\quad \max_{x\in I_*} \cM_{t,1}\diamond \cD_i^*(x) < \min_{x\in I_*} \cM_{t,1}\diamond \cD_i^*(x)+ |\log(r)|^{10}r^{1/3}/20, \quad \forall i\in\{1, 2\}.
\]
We can further reduce $\cE_1^*\cap\cE_2^*$ to these events.
\begin{lemma}   \label{lem:12subpmtv}
We have $\cE_1^*\cap\cE_2^*\cap\cE_{1,2}^*\cap\cE_{TV}\subset \cE^*_{1,+}\cap\cE^*_{2,-}\cap\cE^*_{2,+}\cap\cE_{1,2}^*$.
\end{lemma}
\begin{proof}
Under $\cE_{TV}$, for each $i=1, 2$ we have:
\begin{equation}  \label{eq:h1tv}
\max_{x\in I_*} \cH_i(x) < \min_{x\in I_*} \cH_i(x) + |\log(r)|^{10}r^{1/3}/5.
\end{equation}
Thus under $\cE_2^*\cap \cE_{TV}$, we
must have that $\cE_{2,-}^*$ holds, and
\begin{equation}  \label{eq:chbdcomp}
\cH_2(z_+)+\cM_{t,1}\diamond \cD_2^*(z_+) > \max_{x\in [z_+, z_++1]} \cH_2(x)+\cM_{t,1}\diamond \cD_2^*(x)-\frac54 |\log(r)|^{10}r^{1/3}.
\end{equation}
We note that from the construction of $\cH_2$ in Proposition \ref{P:half-space-horizon-marginals},
\[
\cH_2(x)-\cH_2(z_+) \ge \cB_4(x)-\cB_4(z_+),
\]
for any $x\ge z_+$. This with \eqref{eq:chbdcomp} implies $\cE_{2,+}^*$.

It remains to deduce $\cE_{1,+}^*$. For each $x\ge 0$,
from the construction of $\cH_1$ in Proposition \ref{P:half-space-horizon-marginals},
\[
\cH_1(x) = \max_{\ell = 1, 2, 3, 4} \cB(0,\ell; x, 4) - \sum_{i=\ell + 1}^4 X(-1, \ell),
\]
and we can write 
$\cB(0,\ell; x, 4)=\sum_{i=\ell}^4 \cB_i(x_{i+1})-\cB_i(x_{i})$,
for some $0=x_\ell\le \cdots \le x_4\le x_5=x$. Now, let $z_4$ be the jump point $x_4$ for $x = z_+$. Observe that under $\cE_{1,2}^*$, we have $z_4>z_-$.
Indeed, otherwise
\[
\cH_1(z_+)-\cH_1(z_-)=\cB_4(z_+)-\cB_4(z_-)\le \cH_2(z_+)-\cH_2(z_-).
\]
However, since $\cH^{\DL}\in \mathcal D([2\rho \vee 0, \infty), \R, 2\rho \vee 0)$, we must have $\cH_1(z_+)-\cH_1(z_-)\ge \cH_2(z_+)-\cH_2(z_-)$, therefore $\cH_1(z_+)-\cH_1(z_-)=\cH_2(z_+)-\cH_2(z_-)$, contradicting  $\cE_{1,2}^*$.

Next, for each $x\ge z_+ \ge z_4$, we have
\[
\cH_1(x) - \cH_1(z_4) > \cB_3(x) - \cB_3(z_4). 
\]
Then under $\cE_{1,2}^*\cap\cE_{TV}$, since $z_4\in I_*$, using \eqref{eq:h1tv} and $\cE_{TV}$ we have 
\begin{equation}   \label{eq:chb3ev}
\cH_1(x) - \cH_1(z_+) > \cB_3(x) - \cB_3(z_+) - |\log(r)|^{10}r^{1/3}/4. 
\end{equation}
Moreover, under $\cE_1^*\cap \cE_{TV}$, again using \eqref{eq:h1tv} we have
\[
\cH_1(z_+)+\cM_{t,1}\diamond \cD_1^*(z_+) > \max_{x\in [z_+, z_++1]} \cH_1(x)+\cM_{t,1}\diamond \cD_1^*(x)-\frac54 |\log(r)|^{10}r^{1/3}.
\]
This with \eqref{eq:chb3ev} implies $\cE_{1,+}^*$.
\end{proof}
Next, since $z_+-z_-<3|\log(r)|^{10}r^{2/3}$, from standard estimates on Brownian motion (applied to $\{\cB_i\}_{i=1}^4$ and $\cM_{t,1}\diamond \cD_1^*-\cM_{t,1}\diamond \cD_1^*(0)$ and $\cM_{t,1}\diamond \cD_2^*-\cM_{t,1}\diamond \cD_2^*(0)$, which are Brownian motions of slope $-2\rho$), we have
\begin{equation}   \label{eq:bdtv}
\P[\cE_{TV}]>1- 2\exp(-c|\log(r)|^{2}).
\end{equation}
Also, we have
\begin{equation}    \label{eq:bd1p}
\P[\cE^*_{2,+}] < C|\log(r)|^{10}r^{1/3}.
\end{equation}
This again simply follows from the fact that $\cB_4+\cM_{t,1}\diamond \cD_2^* - (\cB_4(z_+)+\cM_{t,1}\diamond \cD_2^*(z_+))$ is a Brownian motion (of diffusivity $4$ and slope $-\lambda_1-2\rho$).

As for $\cE^*_{1,+}$, there is a universal constant $\delta_0>0$, such that
\begin{equation}  \label{eq:key2+}
    \P[\cE^*_{1,+} \mid \cE^*_{2,-}\cap\cE^*_{2,+}\cap\cE^*_{1,2}] < Cr^{\delta_0}, \quad \text{ if }\quad  \P[\cE^*_{2,-}\cap\cE^*_{2,+}\cap\cE^*_{1,2}]\ge r^{20}.
\end{equation}
This follows from \cite[Lemma 4.5]{DZ24}. We will also prove the following bound. 
\begin{lemma}   \label{lem:key112-}
We have $\P[\cE^*_{2,-}\cap \cE^*_{1,2}] < Cz_+^{-1}r\exp(C|\log(r)|^{0.9})$.
\end{lemma}
Assuming Lemma \ref{lem:key112-}, we can now finish proving Proposition \ref{prop:key1}.
\begin{proof}[Proof of Proposition \ref{prop:key1}]
By Lemma \ref{lem:12subpmtv},
\begin{equation}   
\label{eq:bound1212cap}
\P[\cE_1^*\cap\cE_2^*\cap\cE_{1,2}^*] \le \P[\cE_{1,+}^*\cap\cE_{2,+}^*\cap\cE_{2,-}^*\cap\cE_{1,2}^*] + 1 - \P[\cE_{TV}].
\end{equation}
Using the fact that $\cE^*_{2,-}\cap \cE^*_{1,2}$ is independent of $\cE^*_{2,+}$, we can write the right-hand side of \eqref{eq:bound1212cap} as
\[
\P[\cE^*_{1,+} \mid \cE^*_{2,+}\cap\cE^*_{2,-}\cap\cE^*_{1,2}] \P[\cE^*_{2,-}\cap \cE^*_{1,2}] \P[\cE^*_{2,+}] + 1 - \P[\cE_{TV}] .
\]
If $\P[\cE^*_{2,+}\cap\cE^*_{2,-}\cap\cE^*_{1,2}]\ge r^{20}$, from \eqref{eq:bdtv}, \eqref{eq:bd1p}, \eqref{eq:key2+}, and Lemma \ref{lem:key112-}, we have
\begin{equation}  \label{eq:bdprob1212s}
\P[\cE_1^*\cap\cE_2^*\cap\cE_{1,2}^*] < Cz_+^{-1}r^{4/3+\delta_0}\exp(C|\log(r)|^{0.9}). 
\end{equation}
If $\P[\cE^*_{2,+}\cap\cE^*_{2,-}\cap\cE^*_{1,2}]< r^{20}$, we bound the right-hand side of \eqref{eq:bound1212cap} by $r^{20}+C\exp(-c|\log(r)|^{2})$, using \eqref{eq:bdtv}.
Thus \eqref{eq:bdprob1212s} holds as well.

We next estimate $\P[\cE_1\cap\cE_2\cap\cE_{1,2}]$ in terms of $\P[\cE_1^*\cap\cE_2^*\cap\cE_{1,2}^*]$.
By the stationarity of $\cQ$ from Proposition \ref{P:properties-elpp}.10, it remains to compare $\cM_{t,1}\diamond \cD_1$ and $\cM_{t,1}\diamond \cD_1$ to $\cM_{t,1}\diamond \cD_1^*$ and $\cM_{t,1}\diamond \cD_1^*$.

Take $S=2|\log(r)|^{2/3+0.0001}$. Let $\cD_{1,S}$ (resp. $\cD_{2,S}$, $\cD_{1,S}^*$, $\cD_{2,S}^*$) be equal to $\cD_1$ (resp. $\cD_2$, $\cD_1^*$, $\cD_2^*$) in $[0, S]$, and equal $-\infty$ in $(S, \infty)$.
Let $\cE_S$ be the event where
\[
\cM_{t,1}\diamond \cD_1|_{[0, z_++1]} = \cM_{t,1}\diamond \cD_{1,S}|_{[0, z_++1]}, \quad \cM_{t,1}\diamond \cD_2|_{[0, z_++1]} = \cM_{t,1}\diamond \cD_{2,S}|_{[0, z_++1]}, \]
\[
\cM_{t,1}\diamond \cD^*_1|_{[0, z_++1]} = \cM_{t,1}\diamond \cD^*_{1,S}|_{[0, z_++1]}, \quad
\cM_{t,1}\diamond \cD^*_2|_{[0, z_++1]} = \cM_{t,1}\diamond \cD^*_{2,S}|_{[0, z_++1]}.
\]
We note that $\cE_S$ is implied by the following events: 
\[
|\cD^*_1(x)|, |\cD^*_2(x)|, |\cD_1(x)|, |\cD_2(x)| < S^{1/2}x+S^{3/2}, \quad \forall x \ge 0,
\]
and
\[
\cM_{t,1}(x, y) < \cM_{t,1}(x,x) - 2S^{3/2}-S^{1/2}(x+y), \quad \forall 0<x\le z_++1 \text{ and } y>S.
\]
Then from Brownian estimates and the shape theorem in Proposition \ref{P:only-easy-consequences}.2, we have that $\P[\cE_S]>1-2\exp(-cS^2)$.

By using Lemma \ref{lem:bmcomp} to compare the laws of $\cD_1|_{[0,S]}$ and $\cD_2|_{[0,S]}$ against the laws of $\cD_1^*|_{[0,S]}$ and $\cD_2^*|_{[0,S]}$, and \eqref{eq:bdprob1212s}, we have
\[
\P[\cE_1\cap\cE_2\cap\cE_{1,2}\cap\cE_S] < Cz_+^{-1}r^{4/3+\delta_0}\exp(C|\log(r)|^{0.9}+C|\log(r)|^{1/2}S^{1/2} + CS).
\]
Then the conclusion follows, by taking $\delta<\delta_0$.
\end{proof}

\subsubsection{Reducing the number of Brownian motions}

It now remains to prove Lemma \ref{lem:key112-}, which is the last piece to prove of Proposition \ref{prop:ulteq}. Our first step is to reduce the problem to one involving only three Brownian motions.
From the construction of $\{\cH_i\}_{i=1}^2$ in Proposition \ref{P:half-space-horizon-marginals} (recalled above in Section \ref{sssec:decompinto}), we modify the slopes of the Brownian motions and the exponential parameters in a way that will allow us to get a simpler description.

Take
\[
\gamma =\min\big( \lambda_2/2 - \rho,(\lambda_1-\lambda_2)/2, (\lambda_1+\lambda_2)/4 \big)>0.
\]
Take $X^{(1)}(-1,2), X^{(1)}(-1,4), X^{(1)}(-2,3) \sim \Exp(\gamma)$, and $X^{(1)}(-1,3)\sim\Exp(2\gamma)$, all independent of each other.
Take independent Brownian motions $\cB^{(1)}_1, \cB^{(1)}_2, \cB^{(1)}_3, \cB^{(1)}_4$, with slopes $2\gamma, 0, -2\gamma, -2\gamma$, respectively. Define $\cH_i^{(1)},i = 1, 2$ as in \eqref{E:Hix} with $X\cB$ replaced by $X^{(1)}\cB^{(1)}$, which is the cadlag functions generated from $\{X^{(1)}(-i, j) : i\in\II{1, 2}, j \in \II{i +1, 5 - i}\}$, concatenated by $\cB^{(1)}$.

\begin{lemma}
\label{L:replacement-RN}
Lemma \ref{lem:key112-} holds if the same bound holds for $\{\cH^{(1)}_i\}_{i=1}^2$ (for a different $C$).
\end{lemma}

\begin{proof}
The law of $\{\cH_i\}_{i=1}^2$ is absolute continuous with respect to the law of $\{\cH^{(1)}_i\}_{i=1}^2$.
We can further bound the Radon-Nikodym derivative, via comparing the density of the exponential random variables, and using Lemma \ref{lem:bmcomp}. Indeed, the law of $X$ is absolutely continuous with respect to the law of $X^{(1)}$ with uniformly bounded Radon-Nikodym derivative (depending on $\lambda_1, \lambda_2, \rho$) since we have reduced all exponential parameters. Lemma \ref{lem:bmcomp} then ensures that if a set $A \subset \cC([0, |\log(r)|^{2/3 + 0.0001}])$ has measure $\eps$ under the law of $\{\cH^{(1)}_i\}_{i=1}^2$, then $A$ will have measure at most $\eps \exp(C|\log(\eps)|^{1/2} |\log(r)|^{1/3 + 0.00005})$ under the law of $\{\cH_i\}_{i=1}^2$. Taking 
$\eps = Cz_+^{-1} r \exp(C|\log(r)|^{0.9}) > cr$
then gives the result.
\end{proof}

Now, we can describe the law of $\{\cH^{(1)}_i\}_{i=1}^2$ using only three Brownian motions.

Take $X^{(2)}(-1,3), X^{(2)}(-1,4)\sim \Exp(\gamma)$, and independent of each other.
Take independent Brownian motions $\cB^{(2)}_2, \cB^{(2)}_3, \cB^{(2)}_4$, with slopes $2\gamma, 0, -2\gamma$, respectively.
Let $X^{(2)}\cB^{(2)}$ be their corresponding concatenation.
For $x \ge 0$ and $i = 1, 2$, let
\[
\cH^{(2)}_i(x) =   X^{(2)}\cB^{(2)}(-1,i+1; x, 4) - X^{(2)}\cB^{(2)}(-1,i+1; 0, 4).
\]
\begin{lemma}   \label{lem:eqH12}
We have $\{\cH^{(1)}_i\}_{i=1}^2\eqd\{\cH^{(2)}_i\}_{i=1}^2$.
\end{lemma}
\begin{proof}
This equality in distribution essentially follows by applying Lemma \ref{lem:exp-bro-swap} twice.
For this, we introduce the following intermediate object.

Take $X^{(3)}(-1,2)\sim \Exp(2\gamma)$ and $X^{(3)}(-1,3), X^{(3)}(-1,4)\sim \Exp(\gamma)$, and independent of each other.
Take independent Brownian motions $\cB^{(3)}_1, \cB^{(3)}_2, \cB^{(3)}_3, \cB^{(3)}_4$, with slopes $2\gamma, -2\gamma, 0, -2\gamma$, respectively. 
Let $X^{(3)}\cB^{(3)}$ be their corresponding concatenation.
Then for $x \ge 0$ and $i = 1, 2$, let
$$
\cH^{(3)}_i(x) =   X^{(3)}\cB^{(3)}(-1, 2i-1; x, 4) - X^{(3)}\cB^{(3)}(-1, 2i-1; 0, 4).
$$
Next, in Lemma \ref{lem:exp-bro-swap}, we take $\lambda_0=0$ and $\lambda_1=-2\gamma$, $\beta_{-1}=\gamma$, and send $\beta_{-2}\to 0$ from above.
This way we get a coupling between $\cB^{(1)}_2$, $\cB^{(1)}_3$, $X^{(1)}(-1,2), X^{(1)}(-2,3), X^{(1)}(-1,3)$ and $\cB^{(3)}_2$, $\cB^{(3)}_3$, $X^{(3)}(-1,2), X^{(3)}(-1,3)$, such that almost surely 
\[
X^{(1)}\cB^{(1)}(x, 2; y, 3) = X^{(3)}\cB^{(3)}(x, 2; y, 3),\quad \forall -2< x < y.
\]
When $x = - 2$, the exponential random variable of parameter $-\lambda_0/2 + \beta_{-2}$ is blowing up. In the limit we have that
\[
X^{(1)}\cB^{(1)}(-2, 2; y, 3) - X^{(1)}\cB^{(1)}(-2, 2; 0, 3) = X^{(3)}\cB^{(3)}(-1, 3; y, 3) - X^{(3)}\cB^{(3)}(-1, 3; 0, 3),\quad \forall 0 \le y.
\]
Here, on the right-hand side, we have first used that the blow-up random variable has location $(-2, 3)$, which allows us to move the starting location from $(-2, 2)\mapsto (-2, 3)$ as we take the limit, and then switched $(-2, 3)$ to $(-1, 3)$ by using that $X^{(3)}$ has no weight in $(-\infty, -1)$. 

Finally, $\cB^{(1)}_1, \cB^{(1)}_4, X^{(1)}(4,4)$ and $\cB^{(3)}_1, \cB^{(3)}_4, X^{(3)}(4,4)$ have the same distribution. Combining this with the previous two displays gives that $\{\cH^{(1)}_i\}_{i=1}^2\eqd\{\cH^{(3)}_i\}_{i=1}^2$.

On the other hand, by \eqref{eq:exp-bro-c2} in Corollary \ref{cor:exp-bro},
\[
x\mapsto X^{(3)}\cB^{(3)}(-1, 1;x,2) - X^{(3)}\cB^{(3)}(-1, 1;0,2)
\]
is a Brownian motion with slope $2\gamma$, i.e., it has the same distribution as $\cB^{(2)}_2$. Moreover, $\cB^{(3)}_3$,  $\cB^{(3)}_4$, $X^{(3)}(-1,3)$, $X^{(3)}(4,4)$ and $\cB^{(2)}_3$, $\cB^{(2)}_4$, $X^{(2)}(-1,3)$, $X^{(2)}(4,4)$ have the same distribution. Thus we conclude that  $\{\cH^{(3)}_i\}_{i=1}^2\eqd\{\cH^{(2)}_i\}_{i=1}^2$,
so the conclusion follows.
\end{proof}
It now remains to analyze $\{\cH^{(2)}_i\}_{i=1}^2$. 
In preparation, we define $\cE^{(2)}_{2,-}$ and $\cE^{(2)}_{1,2}$ with $\{\cH_i\}_{i=1}^2$ replaced by $\{\cH^{(2)}_i\}_{i=1}^2$ in the definition of $\cE^*_{2,-}$ and $\cE^*_{1,2}$.

For this, we take $\cB^{(4)}_3$, $\cB^{(4)}_4$ to be independent Brownian motions, with slopes $-2\gamma$, $0$, respectively.
Using Corollary \ref{cor:exp-bro}, we couple $\cB^{(2)}_3$, $\cB^{(2)}_4$, $X^{(2)}(-1,4)$ and $\cB^{(4)}_3$, $\cB^{(4)}_4$, such that for all $0\le x<y$,
\[
\cB^{(2)}(x ,3; y, 4) = \cB^{(4)}(x ,3; y, 4),
\]
and $\cB^{(4)}_4 = \cH^{(2)}_2$, and $X^{(2)}(-1,4)=\max_{x \ge 0} \cB^{(4)}_3(x) - \cB^{(4)}_4(x)$.
We also set $\cB^{(4)}_2=\cB^{(2)}_2$.

Next, since $X^{(2)}(-1,3)/2\sim \Exp(2\gamma)$, again using Corollary \ref{cor:exp-bro} we couple $\cB^{(4)}_2$, $\cB^{(4)}_3$, $X^{(2)}(-1,3)$ with Brownian motions $\cB^{(5)}_2$, $\cB^{(5)}_3$ of drift $2 \gamma$ and $-2 \gamma$, such that for all $0\le x < y$,
\[
\cB^{(4)}(x , 2; y, 3) = \cB^{(5)}(x , 2; y, 3),
\]
and for each $x\ge 0$,
\[
\cB^{(5)}_3(x) = \max\big( X^{(2)}(-1,3)/2 +\cB^{(4)}_3(x), \cB^{(4)}(0 , 2; x, 3) \big) - X^{(2)}(-1,3)/2,\]
and 
\begin{equation}  \label{eq:243}
X^{(2)}(-1,3)/2=\max_{x \ge 0} \cB^{(5)}_2(x) - \cB^{(5)}_3(x).
\end{equation}

We then consider the following events:
\begin{align*}
    \cE^{(4)}_{2,-}:&\quad \cB^{(4)}_4(z_-)+\cM_{t,1}\diamond \cD_2^*(z_-) > \max_{x\in [0, z_-]} \cB^{(4)}_4(x)+\cM_{t,1}\diamond \cD_2^*(x)-2|\log(r)|^{10}r^{1/3},
\\
\cE_{\operatorname{disj}:(-1,3)}:&\quad X^{(2)}(-1,3)<\max_{0\le x\le z_+}-\cB^{(4)}_3(x)+\cB^{(4)}_2(x),
\\
\cE_{\operatorname{disj}:34}:&\quad \max_{0\le x \le z_-} -\cB^{(4)}_4(x)+\cB^{(5)}_3(x) < \max_{z_-\le x \le z_+} -\cB^{(4)}_4(x)+\cB^{(5)}_3(x).
\end{align*}
Since $\cB^{(4)}_4 = \cH^{(2)}_2$ by the construction of Corollary \ref{cor:exp-bro} we have that $\cE^{(2)}_{2,-}=\cE^{(4)}_{2,-}$.
We also have the following inclusion.
\begin{lemma}   \label{lem:ce12inc}
We have $\cE^{(2)}_{1,2}\subset \cE_{\operatorname{disj}:(-1,3)} \cap \cE_{\operatorname{disj}:34}$.
\end{lemma}
\begin{proof}
First, we assume that $\cE_{\operatorname{disj}:(-1,3)}$ does not hold.
Then for each $x\in [0, z_+]$, we have
\[
X^{(2)}(-1,3) + \cB^{(4)}(0,3;x,4) \ge \cB^{(4)}(0,2;x,4).
\]
Thus
\[
X^{(2)}(-1,3) + \cB^{(2)}(0,3;x,4) \ge \cB^{(2)}(0,2;x,4),
\]
which implies that
\[
X^{(2)}\cB^{(2)}(-1,3;x,4)\ge X^{(2)}\cB^{(2)}(0, 2;x,4).
\]
This implies that for each $x\in [0, z_+]$,
\begin{equation}   \label{eq:0xz+}
\cH^{(2)}_1(x) = X^{(2)}\cB^{(2)}(-1,3;x,4) - X^{(2)}\cB^{(2)}(-1,3;x,4) = \cH^{(2)}_2(x),
\end{equation}
which contradicts $\cE^{(2)}_{1,2}$.
Thus $\cE^{(2)}_{1,2}\subset \cE_{\operatorname{disj}:(-1,3)}$.

Second, we assume that $\cE_{\operatorname{disj}:34}$ does not hold.
This can be written as
\begin{multline*}
\max_{0\le x \le z_-} -\cB^{(4)}_4(x)+\max\big( X^{(2)}(-1,3)/2 +\cB^{(4)}_3(x), \cB^{(4)}(0 , 2; x, 3) \big) \\ = \max_{0\le x \le z_+} -\cB^{(4)}_4(x)+\max\big( X^{(2)}(-1,3)/2 +\cB^{(4)}_3(x), \cB^{(4)}(0 , 2; x, 3) \big),
\end{multline*}
or equivalently,
\begin{multline}
\label{E:maxmax}
\max\big(X^{(2)}(-1,3)/2+\cB^{(2)}(0 , 3; z_-, 4), \cB^{(2)}(0 , 2; z_-, 4)\big)-\cB^{(2)}_4(z_-) + \cB^{(2)}_4(z_+)\\
=\max\big(X^{(2)}(-1,3)/2+\cB^{(2)}(0 , 3; z_+, 4), \cB^{(2)}(0 , 2; z_+, 4)\big).
\end{multline}
Note that
\begin{equation}  \label{eq:b34}
\cB^{(2)}(0 , 3; z_-, 4) -\cB^{(2)}_4(z_-) + \cB^{(2)}_4(z_+) \le \cB^{(2)}(0 , 3; z_+, 4),
\end{equation}
\begin{equation}  \label{eq:b24}
 \cB^{(2)}(0 , 2; z_-, 4)  -\cB^{(2)}_4(z_-) + \cB^{(2)}_4(z_+)\le  \cB^{(2)}(0 , 2; z_+, 4).
\end{equation}
If $X^{(2)}(-1,3)/2+\cB^{(2)}(0 , 3; z_+, 4)\ge \cB^{(2)}(0 , 2; z_+, 4)$, then by \eqref{E:maxmax} we would have that equality holds in \eqref{eq:b34}. Thus
\begin{multline}   \label{eq:b4322}
    \max\big(X^{(2)}(-1,3)+\cB^{(2)}(0 , 3; z_-, 4), \cB^{(2)}(0 , 2; z_-, 4)\big)-\cB^{(2)}_4(z_-) + \cB^{(2)}_4(z_+)\\
=\max\big(X^{(2)}(-1,3)+\cB^{(2)}(0 , 3; z_+, 4), \cB^{(2)}(0 , 2; z_+, 4)\big).
\end{multline}
That is, \eqref{E:maxmax} holds with $X^{(2)}(-1,3)$ in place of $X^{(2)}(-1,3)/2$.

If  $X^{(2)}(-1,3)/2+\cB^{(2)}(0 , 3; z_+, 4)< \cB^{(2)}(0 , 2; z_+, 4)$, then by \eqref{E:maxmax} we would have that equality holds in \eqref{eq:b24}, which implies that
\[
\max_{x\in [0,z_-]} -\cB^{(2)}_4(x) + \cB^{(2)}_3(x) = \max_{x\in [0,z_+]} -\cB^{(2)}_4(x) + \cB^{(2)}_3(x).
\]
Therefore equality also holds in \eqref{eq:b34}. Then from the equality in both \eqref{eq:b24} and  \eqref{eq:b34}, we still get \eqref{eq:b4322}.

From \eqref{eq:b4322}, we get $\cH^{(2)}_1(z_+)-\cH^{(2)}_1(z_-)=\cB^{(2)}_4(z_+)-\cB^{(2)}_4(z_-)$.
Also, from the fact that equality holds in \eqref{eq:b24}, we get $\cH^{(2)}_2(z_+)-\cH^{(2)}_2(z_-)=\cB^{(2)}_4(z_+)-\cB^{(2)}_4(z_-)$.
These contradict $\cE^{(2)}_{1,2}$.
Thus $\cE^{(2)}_{1,2}\subset \cE_{\operatorname{disj}:34}$.
\end{proof}

\begin{lemma}   \label{lem:cEdisjpr}
We have $\P[\cE^{(4)}_{2,-}\cap \cE_{\operatorname{disj}:(-1,3)}\cap \cE_{\operatorname{disj}:34}]<Cz_+^{-1}|\log(r)|^{32}r$.
\end{lemma}
\begin{proof}
We let $\cE_{TV}^{(45)}$ be the event where the following conditions hold:
\begin{align*}
    &\max_{z_-\le x \le z_+} -\cB^{(4)}_4(x)+\cB^{(5)}_3(x) +\cB^{(4)}_4(z_-) - \cB^{(5)}_3(z_-) < |\log(r)|^{10}r^{1/3},\\
&\max_{0\le x\le z_+}-\cB^{(4)}_3(x)+\cB^{(4)}_2(x) < z_+^{1/2}|\log(r)|^{2},\\
&- \cB^{(5)}_2(z_+) + \cB^{(5)}_3(z_+)< z_+^{1/2}|\log(r)|^{2}.
\end{align*}
Using that $z_+-z_-<3|\log(r)|^{10}r^{2/3}$,  $\P[\cE_{TV}^{(45)}]>1-2\exp(-c|\log(r)|^{2})$.
On the other hand (and using \eqref{eq:243}), we have that $\cE_{\operatorname{disj}:(-1,3)}\cap \cE_{\operatorname{disj}:34}\cap\cE_{TV}^{(45)}$ implies $\cE'_{\operatorname{disj}:(-1,3)}\cap \cE'_{\operatorname{disj}:34}$, where
\begin{align*}
\cE'_{\operatorname{disj}:(-1,3)}:&\quad  \max_{x \ge z_+} \cB^{(5)}_2(x) - \cB^{(5)}_3(x) - \cB^{(5)}_2(z_+) + \cB^{(5)}_3(z_+) <2z_+^{1/2}|\log(r)|^{2},
\\
\cE'_{\operatorname{disj}:34}:&\quad \max_{0\le x \le z_-} -\cB^{(4)}_4(x)+\cB^{(5)}_3(x) +\cB^{(4)}_4(z_-) - \cB^{(5)}_3(z_-) < |\log(r)|^{10}r^{1/3}.
\end{align*}
We note that $\cE^{(4)}_{2,-}\cap \cE'_{\operatorname{disj}:34}$ is independent of $\cE'_{\operatorname{disj}:(-1,3)}$.
Next, we have $\P[\cE'_{\operatorname{disj}:(-1,3)}]<Cz_+^{1/2}|\log(r)|^{2}$.
Finally, by Lemma \ref{lem:bm3} below (and Brownian scaling) we have $\P[\cE^{(4)}_{2,-}\cap \cE'_{\operatorname{disj}:34}]<Cz_-^{-3/2}|\log(r)|^{30}r$. 
In summary, we get
\[
\P[\cE^{(4)}_{2,-}\cap \cE_{\operatorname{disj}:(-1,3)}\cap \cE_{\operatorname{disj}:34}] < Cz_+^{1/2}|\log(r)|^{2}(1 \wedge z_-^{-3/2}|\log(r)|^{30}r) + 2\exp(-c|\log(r)|^{2}).
\]
If $z_+>6|\log(r)|^{10}r^{2/3}$, we would have $z_->z_+/2$, therefore the above is bounded by $Cz_+^{-1}|\log(r)|^{32}r$.
Otherwise, the upper bound is $Cz_+^{1/2}|\log(r)|^{2}$, which is at most $Cz_+^{-1}|\log(r)|^{20}r$.
\end{proof}
		\begin{lemma} \label{lem:bm3}
			Take any $D>1$ and $0<\ep<1$.
			Let $X, Y, Z$ be three independent Brownian motions, each with drift in $[-D, D]$.
			Then we have
			\[
			\P\left[ X(x)+Y(x) < \ep, \; Z(x)-Y(x) < \ep, \forall x\in[0,1]  \right] < C_0\ep^3 D^3.
			\]
            where $C_0>0$ is a universal constant.
		\end{lemma}
This lemma is a straightforward computation using the Karlin-McGregor formula, and was shown as part of the proof of \cite[Lemma 4.3]{DZ24}. We omit the details.

Finally, we put everything back together to bound the probability of the original event $\cE^*_{2,-}\cap \cE^*_{1,2}$.
\begin{proof}[Proof of Lemma \ref{lem:key112-}]
By Lemmas \ref{lem:ce12inc} and \ref{lem:cEdisjpr} and the fact that $\cE^{(2)}_{2,-}=\cE^{(4)}_{2,-}$, we have $\P[\cE^{(2)}_{2,-}\cap\cE^{(2)}_{1,2}]<Cz_+^{-1}|\log(r)|^{32}r$.
Then by Lemma \ref{lem:eqH12}, we also have $\P[\cE^{(1)}_{2,-}\cap\cE^{(1)}_{1,2}]<Cz_+^{-1}|\log(r)|^{32}r$,
where $\cE^{(1)}_{2,-}$ (resp.~$\cE^{(1)}_{1,2}$) is $\cE^*_{2,-}$ (resp.~$\cE^*_{1,2}$) with $\{\cH_i\}_{i=1}^2$ replaced by $\{\cH^{(1)}_i\}_{i=1}^2$. The result then follows from Lemma \ref{L:replacement-RN}.
\end{proof}

\subsection{From Brownian to multiple points}
In this subsection we finish the proof of Proposition \ref{prop:finid}, using the Brownian version (Proposition \ref{prop:ulteq}).
Our general strategy is similar to the proof of Proposition 3.7 from Proposition 3.8 in \cite{DZ24}. However, the analysis is slightly more delicate because of the more complex structure of the half-space stationary horizon in Section \ref{ss:hfsh} when compared to the full-space version.

\begin{proof}[Proof of Proposition \ref{prop:finid}]
Without loss of generality, we assume that $x_1> \cdots >x_k>0$.
Take $\{\cH_i\}_{i=1}^k$ as in the setup of Proposition \ref{prop:ulteq}, which by Proposition \ref{P:half-space-horizon-marginals} is constructed from an array of exponential random variables $\{X(-i, j) : i\in\II{1, k}, j \in \II{i +1, 2k +1 - i}\}$, and a sequence of independent Brownian motions $\cB_1, \dots, \cB_{2k}$, via $\cH_i(x)=X\cB(-i, i; x, 2k) - X\cB(-i, i; 0, 2k)$.

We construct a sequence of processes $\{\cH_{\tau,i}\}_{i=1}^k$, indexed by $\tau>0$, such that each $\{\cH_{\tau,i}\}_{i=1}^k$ is absolutely continuous with respect to $\{\cH_i\}_{i=1}^k$, and as $\tau\to 0$,
			\begin{equation}
				\label{eq:htau1}
				\sup_{x\ge 0, |x-x_i|>\tau}  \cH_{\tau,i}(x) - 2kx - \cH_{\tau,i}(x_i) \cvgp -\infty,    
			\end{equation}
			\begin{equation}   \label{eq:htau2}
				\sup_{x\ge 0, |x-x_i|\le \tau} \cH_{\tau,i} (x) - \cH_{\tau,i}(x_i) \cvgp 0.
			\end{equation}

The construction is as follows.
For each $i\in \llbracket 1, k\rrbracket$, define continous $\phi_{\tau,i}:\R_{\ge 0}\to \R$, such that
			\begin{itemize}[nosep]
            \item $\phi_{\tau,i}(0)=0$,
				\item $\phi_{\tau, i}'(x) = 0$ for $x < x_i - \tau$ or $x > x_k + 1$,
				\item $\phi_{\tau, i}'(x) = i\tau^{-2}$ for $x \in (x_i - \tau, x_i)$,
				\item $\phi_{\tau, i}'(x) = -2i\tau^{-2}$ for $x \in (x_i, x_k + 1)$.
			\end{itemize}
Moreover, for each $i\in \llbracket k+1, 2k\rrbracket$, we define $\phi_{\tau,i}:\R_{\ge 0}\to \R$, such that
			\begin{itemize}[nosep]
            \item $\phi_{\tau,i}(0)=0$,
				\item $\phi_{\tau, i}'(x) = 0$ for $x > x_k + 1$,
				\item $\phi_{\tau, i}'(x) = -4k\tau^{-2}$ for $x \in (0, x_k + 1)$.
			\end{itemize}
Let $\cB_{\tau,i}=\cB_i+\phi_{\tau,i}$, for each $\tau>0$ and $i\in\llbracket 1, 2k\rrbracket$, and define $\{\cH_{\tau,i}\}_{i=1}^k$ the same way as $\{\cH_i\}_{i=1}^k$, but with $\{\cB_i\}_{i=1}^{2k}$ replaced by $\{\cB_{\tau,i}\}_{i=1}^{2k}$. Note that for each $i\in\llbracket 1, 2k\rrbracket$ and $\tau>0$, we have
$\int_0^\infty (\phi_{\tau,i}'(x))^2 dx < \infty$. Therefore by the Cameron-Martin theorem, the law of $\{\cB_{\tau,i}\}_{i=1}^{2k}$ is absolutely continuous with respect to that of $\{\cB_i\}_{i=1}^{2k}$.

We next verify \eqref{eq:htau1} and \eqref{eq:htau2}. Below we assume that $\tau<x_k/2$, $\tau<(x_i-x_{i+1})/2$ for each $i\in \llbracket 1, k-1\rrbracket$, and we set $M=\max\{X(-i,j): 1 \le i \le k; i+1 \le j \le 2k-i\}$. 

For each $i\in\llbracket 1, k\rrbracket$, and $\tau>0$, we have
\begin{equation}  
\label{eq:chlbd}
\cH_{\tau,i}(x_i)\ge \phi_{\tau,i}(x_i)+\cB_i(x_i) + X \cB(-i,i;0,i) - X \cB(-i,i;0, 2k) \ge \phi_{\tau,i}(x_i)+\cB_i(x_i) -2kM,  
\end{equation}
and for any $x\ge 0$, we have
\[
\cH_{\tau,i}(x)\le \sum_{j=1}^{2k} \max_{y\in [0,x]} \cB_j(y) - \min_{y\in [0,x]} \cB_j(y) +  
\phi_\tau(0, i; x, 2k).
\]
We note that for each $i\in\llbracket 1, k\rrbracket$ and $x\ge 0$ with $|x-x_i|>\tau$, we have $\phi_\tau(0, i; x, 2k) \le \phi_{\tau,i}(x_i)-\tau^{-1}$. Since the non-$\phi_\tau$ terms on the right-hand side of \eqref{eq:chlbd} and the previous display are constant as $\tau\to 0$, this gives \eqref{eq:htau1}.

On the other hand, for any $x\in [x_i, x_i+\tau]$, we have
\[
\cH_{\tau,i}(x)- \cH_{\tau,i}(x_i)
\le \sum_{j=1}^{2k} \sup_{x_i \le y < z \le x} \cB_j(z) + \phi_{\tau,j}(z)- \cB_j(y)-\phi_{\tau,j}(y),
\]
therefore $\sup_{x\in [x_i, x_i+\tau]} \cH_{\tau,i} (x) - \cH_{\tau,i}(x_i) \cvgp 0$.
For any $x\in [x_i-\tau, x_i]$, let $m_x \in [-i, x]$ be the largest time such that the geodesic in $X \cB_\tau$ from $(-i, i)$ to $(x, 2k)$ goes through $(m_x, i)$. Then 
\begin{align*}
\cH_{\tau, i}(x) - \cH_{\tau, i}(x_i) &\le X \cB_\tau(m_x, i + 1; x, 2k) - (X \cB_\tau(x_i) - X \cB_\tau(m_x))  \\
&\le X \cB(m_x, i + 1; x, 2k) - (X \cB(x_i) - X \cB(m_x)) - ((x_i - m_x)\vee \tau) i \tau^{-2}.
\end{align*}

The supremum over $x \in [x_i - \tau, x_i]$ and $m_x \in [-i, x]$ of the right-hand side above tends to $0$ with $\tau$. This completes the proof of \eqref{eq:htau2}.

Next, we construct a modification of the Brownian motions $\cD_i, i \in \II{1, k}$.
For each $\tau>0$ define $\cD_{\tau,i} = \cD_i + \psi_{\tau, i}$ where
$\psi_{\tau,i}:[0,\infty)\to \R$ is a continuous function defined via:
\begin{itemize}
    \item $\psi_{\tau,i}(0)=0$,
    \item $\psi_{\tau,i}'(x)=0$ for $x<y_i-\tau$ or $x>y_i+\tau$,
    \item $\psi_{\tau,i}'(x)=\tau^{-2}$ for $x\in (y_i-\tau, y_i)$,
    \item $\psi_{\tau,i}'(x)=-\tau^{-2}$ for $x\in (y_i, y_i+\tau)$.
\end{itemize}
Again by the Cameron-Martin theorem, $\{\cD_{\tau,i}\}_{i=1}^k$ is absolutely continuous with respect to the law of $\{\cD_i\}_{i=1}^k$.
We then define $\cE_\tau[W] = \cE_1[W, \cH_\tau] \cap \cE_2[W, \cD_\tau]$ the same way as $\cE[W]$, but with $\{\cH_i\}_{i=1}^k$ and $\{\cD_i\}_{i=1}^k$ replaced by $\cH_\tau = \{\cH_{\tau,i}\}_{i=1}^k$ and $\cD_\tau = \{\cD_{\tau,i}\}_{i=1}^k$. By Proposition \ref{prop:ulteq}, we now have
\begin{equation}
    \label{E:tau-W-dist}
    (\cH_\tau, \{ \cH_{\tau,i}\diamond \cM_{0,1} \diamond \cD_{\tau,i} \}_{i=1}^k, \cD_\tau ) \mathds{1}[\cE_\tau[W]]\eqd (\cH_\tau, \{ \cH_{\tau,i}\diamond \cQ_{0,1} \diamond \cD_{\tau,i} \}_{i=1}^k, \cD_\tau ) \mathds{1}[\cE_\tau[W]]
\end{equation}
Next, we claim that 
\begin{equation}
\label{E:W-tau}
    \lim_{W \to \infty} \liminf_{\tau\to 0} \P[\cE_\tau[W]] = 1.
\end{equation}
Indeed, first observe that $\cE_2[W, \cD_\tau] = \cE_2[W, \cD]$, and that 
$$
\P(\cE_2[W, \cD]) = \P(B(x) \le W + Wx/3 \text{ for all } x \ge 0)^k,
$$
where $B$ is an undrifted Brownian motion. Hence $\P(\cE_2[W, \cD_\tau]) \to 1$ as $W \to 0$. Next, for all small enough $\tau$ we have that
$$
\cE_1[W, \cH_\tau] \supset \{X\cB(W, 1; W + x, 2k) \le W + Wx/3\}.
$$
The probability of the right-hand side tends to $1$ with $W$, completing the proof of \eqref{E:W-tau}.

Next, by \eqref{eq:htau1}, \eqref{eq:htau2}, the construction of $\cD_\tau$, and the shape bounds and continuity of $\cQ, \cM$ from Proposition \ref{P:only-easy-consequences}, for any $\eps > 0$ and $i \in \II{1, k}$ we have that
$$
\P(|\cH_{\tau,i}\diamond \cM_{0, 1} \diamond \cD_{\tau,i} - \cH_{\tau,i}(x_i) - \cM_{0, 1}(x_i, y_i) - \cD_{\tau,i}(y_i)| > \eps) \to 1
$$
as $\eps \to 0$, and similarly with $\cQ$ in place of $\cM$. Combining this with \eqref{E:tau-W-dist} and \eqref{E:W-tau} and taking $\tau \to 0$ and then $W \to \infty$ yields the result.
\end{proof}

\section{Multi-level Poisson avoiding metrics}   \label{sec:TASEP}

In this section we apply our characterization Theorem \ref{thm:chawdl} to multi-level Poisson-avoiding metrics, which generalize the Poisson-avoiding metric from the introduction.

\begin{definition}  \label{defn:Dd}
For each $d\in \N$, let $\Lambda_d$ be the following directed graph. Its vertices consist of all $(x, a)\in \Z_{\ge 0} \times \Z_{2d}$ with $x+a\in 2\Z$, where $\Z_{2d}=\Z/2d\Z=\{0,\ldots, 2d-1\}$.
For each vertex $(x,a)$, there is a directed edge from it to $(x+1, a+1)$, and a directed edge from it to $(x-1, a+1)$, if $x\ge 1$. (Here our space is periodic, so $a, a + 1$ are understood modulo $2d$).
We let $D_d:\Lambda_d^2\to \Z_{\ge 0}$ be the resulting directed metric on $\Lambda_d$, i.e., $D_d((x,a), (y,b))$ is the length of the shortest directed path from $(x,a)$ to $(y,b)$.
\end{definition}

\begin{definition}
For each $\alpha>0$ and $d\in \N$, the \textbf{$d$-level Poisson-avoiding metric} with boundary parameter $\alpha$ is a random function
$$
H_d: \{(u, s; v, t) \in (\Lambda_d \times \R)^2 : s < t\} \to \Z_{\ge 0}
$$
defined as follows. First, let $\Pi$ be a Poisson process on $\Lambda_d \times \R$ such that $\Pi|_{(x, a) \times \R}$ has intensity $\alpha$ if $x = 0$, and intensity $1$ if $x > 0$. From $\Pi$, we construct a random weighted directed graph where:
\begin{itemize}[nosep]
    \item $(u, t)$ is connected to $(v, t)$ by an edge of weight $1$ if $(u, v)$ is a directed edge in $\Lambda_d$.
    \item $(u, s)$ is connected to $(u, t)$ by an edge of weight $0$ if $(u, r) \notin \Pi$ for all $r \in (s, t]$.
\end{itemize}
Let $H_d:(\Lambda_d \times \R)^2 \to \Z_{\ge 0},$ be directed graph distance in this metric.
\end{definition}

A main motivation of the above metric is its connection to the following coupling of multiple half-space TASEPs (with the same boundary parameter).

		\begin{definition}   \label{def:TASEPcp}
		For each $\alpha>0$ and $d\in \N$, we defined the \textbf{$d$-coupling} of half-space TASEPs with boundary parameter $\alpha$, as follows.

        Take the Poisson process $\Pi$, and view this as a collection of Poisson clocks on $\Lambda_d$. These clocks generate half-space TASEPs as follows.
        Let $\pi:\{(x,a)\in \Z_{\ge 0}\times \Z, x+a\in 2\Z\} \to \Lambda_d$ be the projection map.
        Recall that we can define half-space TASEP as a Markov process $(h_t: t\ge 0)$ on the state space $\SRW_+$ (from \eqref{eq:defSRW+}). 
        We flip $h_t(i) \mapsto h_t(i) + 2$ according to the Poisson clock at the vertex $\pi(i, h_t(i))$. In words, there is a flip $h_{t}(i)= h_{t-}(i)+2$ if and only if the Poisson clock at $\pi(i, h_t(i))$ rings at times $t$, and $h_t(j)=h_{t-}(i)+1$ for $j\in \{i-1, i+1\}\cap \Z_{\ge 0}$.

        By using the same set of Poisson clocks, we get a coupling of half-space TASEPs with boundary parameter $\alpha$, started from all initial configurations and all times.
		\end{definition}

The $d=1$ coupling is also known as the basic coupling of half-space TASEPs, and the Poisson-avoiding metric with $d=1$ is precisely the half-space Poisson-avoiding metric in the introduction.
We have a generalization of \eqref{eq:h0Hvar}: to evolve any one of half-space TASEPs from time $s$ to $t$, we can write
\begin{equation}   \label{eq:TASEPdcoupling}
h_t(x) = \min_{y\in \Z_{\ge 0}, (x,a)\in \Lambda_d} h_s(y) + H_d( \pi(y, h_s(y)), s; x, a, t).
\end{equation}

We next present the $d$-level generalization of Theorem \ref{thm:colorTASEPconv}.
For this, note that while $H_d$ is a 6-dimensional random function, the $\Z_{2d}$ coordinates degenerate in our limiting regime. Indeed, for each
$(x,a), (x, a'), (y,b), (y, b') \in \Lambda_d$ and $s<t$, we have
\begin{equation}   \label{eq:Hddiff}
|H_d(x,a,s;y,b,t) - H_d(x,a',s;y,b',t)| \le 2d.
\end{equation}
In the limiting regimes we consider, $d$ will always grow more slowly than the fluctuation scale, so by \eqref{eq:Hddiff}, to study limits of $H_d$ it suffices to work with the following function:
\[
H_d^-(x,s;y,t) = \min_{(x,a), (y,b) \in \Lambda_d} H_d(x,a,s;y,b,t),
\]
defined for each $x,y\in\Z_{\ge 0}$ and $s<t$.
We also recall the map $\mathbf{A}_\eps:\SRW_+\to\UC_+$ from \eqref{eq:bfAe}.
\begin{theorem}  \label{thm:mPAM}
For each $\eps>0$, we take $d_\eps>0$, such that $d_\eps\eps^{1/2}|\log(\eps)|\to 0$ as $\eps\to 0$.
Then consider the following function defined on $\mathbb{H}^2_\uparrow$:
\begin{equation}   \label{eq:mPAMconv}
\cM^{\rho,\eps}(x, s; y, t):= -\eps^{1/2}H_{d_\eps}^-(\lfloor 2\eps^{-1}x \rfloor, 2\eps^{-3/2}s;  \lfloor 2\eps^{-1}y \rfloor, 2\eps^{-3/2}t) + \eps^{-1}(t-s),
\end{equation}
where either (1) fixed $\alpha\ge \frac{1}{2}$ and $\rho=-\infty$, or (2) $\alpha = \frac{1}{2}-\rho\eps^{1/2}/2$ for some $\rho\in\R$.
Then as $\eps\to 0$, the above function converges to $\cL^\rho$, weakly under the uniform-on-compact topology.  
\end{theorem}
\begin{proof}
We first prove the tightness of $\cM^{\rho,\eps}$ under the uniform-on-compact topology, using arguments similar to the proof of Theorem \ref{T:elpp}.

For each $b>0$, take $K_b\subset \mathbb{H}^2_\uparrow$ as in \eqref{eq:Kbdef}, and we will prove the tightness of $\cM^{\rho,\eps}|_{K_b}$ in the uniform topology.
Below, we use $c>0$ to denote small constants which may depend on $\rho$ and $b$, and the value can change from line to line.

By \eqref{eq:TASEPdcoupling} and \eqref{eq:Hddiff}, the bounds of Lemmas \ref{L:one-point-TASEP} and \ref{L:two-point-TASEP} imply that, for any $a>0$ and $(x,s;y,t)\in K_b$,
\begin{equation}   \label{eq:Mopt}
\begin{split}
    \P(\cM^{\rho,\eps}(x,s;y,t) < -a(t-s)^{1/3} - 2d_\eps\eps^{1/2} ) &< 2\exp\left( -c(a^3 \wedge \eps^{-3}) \right),\\
    \P(\cM^{\rho,\eps}(x,s;y,t) > a(t-s)^{1/3} + 2d_\eps\eps^{1/2} ) &< 2\exp\left( -c(a^{3/2} \wedge \eps^{-3/4}) \right),
    \end{split}
    \end{equation}
and for any $(x,s;y,t), (x,s;y',t')\in K_b$,
\begin{multline}   \label{eq:Mtpt}
	\P\Big(| \cM^{\rho,\eps}(x,s;y,t)  - \cM^{\rho,\eps}(x,s;y',t')  | \ge a (|y-y'|^{1/4} + |t-t'|^{1/4}) + 4d_\eps\eps^{1/2}\Big) \\ < 2\exp\left(-c(a^{1/2} \wedge \eps^{-3/16})\right).
\end{multline}    
Here the $1/4$-powers are not optimal, but easily allow us to remove the lower bound on $a$ from Lemma \ref{L:two-point-TASEP}.
From \eqref{eq:Mopt}, and the assumption that $d_\eps\eps^{1/2}|\log(\eps)|\to 0$ as $\eps\to 0$, we get the tightness of $\cM^{\rho,\eps}(0,0;0,1)$.
From \eqref{eq:Mtpt}, and the fact that $(x,s;y,t)\mapsto \cM^{\rho,\eps}(y,-t;x,-s)$ has the same law as $\cM^{\rho,\eps}$, we similarly have that for any $a>0$ and $(x,s;y,t), (x',s';y,t)\in K_b$,
\begin{multline}   \label{eq:Mtptr}
	\P\Big(| \cM^{\rho,\eps}(x,s;y,t)  - \cM^{\rho,\eps}(x',s';y,t)  | \ge a (|x-x'|^{1/4} + |s-s'|^{1/4}) + 4d_\eps\eps^{1/2}\Big) \\ < 2\exp\left(-c(a^{1/2} \wedge \eps^{-3/16})\right).
\end{multline}    
Take any $0<\delta<1/4$, and $\tilde{a}>0$.
We take a union bound of \eqref{eq:Mtpt}, over all $m\in \N$, $m\le |\log_2(\eps)/2|+1$, and $(x,s;y,t), (x,s;y',t')\in K_b\cap (2^{-2m}\times 2^{-3m})^2$ with $|y-y'|\le 2^{-2m}$ and $|t-t'|\le 2^{-3m}$, and $a=2^{\delta m}\tilde{a}$.
We also take a union bound of \eqref{eq:Mtptr}, over all $m\in \N$, $m\le |\log_2(\eps)/2|+1$, and $(x,s;y,t), (x',s';y,t)\in K_b\cap (2^{-2m}\times 2^{-3m})^2$ with $|x-x'|\le 2^{-2m}$ and $|s-s'|\le 2^{-3m}$, and $a=2^{\delta m}\tilde{a}$.
From these union bounds and the triangle inequality, we have that the random variable
\[
\sup_{u_1, u_2\in K_b} \frac{ |\cM^{\rho,\eps}(u_1)-\cM^{\rho,\eps}(u_2)|  }{\|u_1-u_2\|_2^{1/4-\delta} + d_\eps \eps^{1/2} |\log(\eps\|u_1-u_2\|_2^{-2/3}) |},
\]
is tight as $\eps\to 0$. By the Arzel\'a-Ascoli theorem, and the fact that $d_\eps\eps^{1/2}|\log(\eps)|\to 0$ as $\eps\to 0$, we get the tightness of $\cM^{\rho,\eps}|_{K_b}$ in the uniform topology.

Now take $\cM^\rho: \mathbb{H}^2_\uparrow \to \R$ to be any subsequential limit of $\cM^{\rho,\eps}$ as $\ep\to 0$.
We show that $\cM^\rho$ is a half-space pre-landscape with parameter $\rho$, by checking that its restriction to $\Q^4 \cap \mathbb{H}^2_\uparrow$ satisfies the three conditions in Theorem \ref{thm:chawdl}.

The independence of increment holds for each $H_{d_\eps}^-$, from its definition, thus holds for $\cM^\rho$.
As for the triangle inequality, if we take $o = (x, s)$, $p = (y, r)$, $q = (z, t) \in \Q_{\ge 0} \times \Q$ with $s<r<t$, we would have
\begin{multline*}
H_{d_\eps}^-(\lfloor 2\eps^{-1}x \rfloor, 2\eps^{-3/2}s;  \lfloor 2\eps^{-1}y \rfloor, 2\eps^{-3/2}r) + H_{d_\eps}^-(\lfloor 2\eps^{-1}y \rfloor, 2\eps^{-3/2}r;  \lfloor 2\eps^{-1}z \rfloor, 2\eps^{-3/2}t) \\ \le H_{d_\eps}^-(\lfloor 2\eps^{-1}x \rfloor, 2\eps^{-3/2}s;  \lfloor 2\eps^{-1}z \rfloor, 2\eps^{-3/2}t) + 4d_\eps.
\end{multline*}
By sending $\eps\to 0$, and using that $d_\eps \eps^{1/2}\to 0$, we get
\[
\cM^\rho(o; p) + \cM^\rho(p; q) \ge \cM^\rho(o; q).
\]

It now remains to verify the KPZ fixed point marginals of $\cM^\rho$ stated in Theorem \ref{thm:chawdl}.
Take any $s<t\in \Q$ and $f:\R_{\ge 0}\to \R\cup\{-\infty\}$ with a finite set $P\subset\Q_{\ge 0}$, such that $f(P)\subset \Q$, and $f(\Q_{\ge 0}\setminus P)=\{-\infty\}$.
For each $\eps>0$, consider the half-space TASEP from time $2\eps^{-3/2}s$ to $2\eps^{-3/2}t$, with initial condition $h_{2\eps^{-3/2}s}^\eps$ given by
\[
h_{2\eps^{-3/2}s}^\eps( \lfloor 2\eps^{-1}z\rfloor) = -\eps^{-1/2}f(z) - \iota(f,z),
\]
for each $z\in P$, where $\iota(f,z)\in [0, 2)$ is the number such that $\lfloor 2\eps^{-1}z\rfloor -\eps^{-1/2}f(z) - \iota(f,z) \in 2\Z$.
For general $x\in \Z_{\ge 0}$, we let
\[
h_{2\eps^{-3/2}s}^\eps (x) = \min_{z\in P} h_{2\eps^{-3/2}s}^\eps( \lfloor 2\eps^{-1}z\rfloor) + |x-\lfloor 2\eps^{-1}z\rfloor|.
\]
As $\eps\to 0$, the function $\mathbf{A}_\eps h_{2\eps^{-3/2}s}^\eps$ converges to $f$ (in $\UC_+$). Then by Theorem \ref{T:half-space-fixed-point}, for any finite $Y\subset \Q_{\ge 0}$, we have that $(\mathbf{A}_\eps h_{2\eps^{-3/2}t}^\eps(y)+\eps^{-1}(t-s): y\in Y)$ converges to 
\[
( \mathfrak{h}^\rho(t-s, y; f) : y \in Y)\eqd(f\diamond \cL^\rho(\cdot,s;y, t): y\in Y),\]
in distribution.
On the other hand, by \eqref{eq:TASEPdcoupling} we have
\[
h_{2\eps^{-3/2}t}^\eps (x) = \min_{z\in P, (x,a)\in \Lambda_{d^\eps}} h_{2\eps^{-3/2}s}^\eps( \lfloor 2\eps^{-1}z\rfloor) + H_{d^\eps}(\pi(\lfloor 2\eps^{-1}z\rfloor, h_{2\eps^{-3/2}s}^\eps( \lfloor 2\eps^{-1}z\rfloor)), ; x, a, t).
\]
Then by \eqref{eq:Hddiff} and that $d_\eps \eps^{1/2}\to 0$ as $\eps\to 0$, an almost sure subsequential limit of $(\mathbf{A}_\eps h_{2\eps^{-3/2}t}^\eps(y)+\eps^{-1}(t-s): y\in Y)$ would be $(f\diamond \cM^\rho(\cdot,s;y, t): y\in Y)$.

Now that $\cM^\rho$ is a half-space pre-landscape, by Theorem \ref{thm:chawdl} the conclusion for $\rho\in\R$ follows.

Finally, for any $\rho\in\R$, we have that $\cM^{-\infty, \eps}$ is stochastically dominated by $\cM^{\rho,\eps}$, when $\eps$ is small enough depending on $\rho$.
This follows from a coupling of the Poisson processes $\Pi$ with two different boundary parameters $\alpha$, where the one with the smaller $\alpha$ is contained in the other one.
Thus by taking a subsequential limit as $\eps\to 0$, $\cM^{-\infty}$ is stochastically dominated by $\cL^\rho$ for any $\rho\in\R$.
Then by Corollary \ref{C:-infty-case} the conclusion for $\rho=-\infty$ holds.
\end{proof}

We can also get convergence of half-space TASEPs under the $d$-coupling.
\begin{theorem}  \label{thm:TASEPexo}
			Take $f_1, \ldots, f_k \in\UC_+$.
            For each $\eps>0$, we take $d_\eps>0$, such that $d_\eps\eps^{1/2}|\log(\eps)|\to 0$ as $\eps\to 0$.
            Take half-space TASEPs $(h_t^{\eps, i}, t\ge 0)$ for $i\in\llbracket 1, k\rrbracket$ under the $d_\eps$-coupling, with each $\mathbf{A}_\eps h_0^{\eps, i}\to f_i$ in $\UC_+$ as $\eps\to 0$.
            The boundary parameter is taken either (1) fixed $\alpha> \frac{1}{2}$ and $\rho=-\infty$, or (2) $\alpha = \frac{1}{2}-\rho\eps^{1/2}/2$ for some $\rho\in\R$.
            Then 
                            \[
                \big\{ \mathbf{A}_\ep h^{\eps,i}_{2\eps^{-3/2}t} (x) + \eps^{-1}t \big\}_{(i,x,t)\in\llbracket 1, k\rrbracket\times \R_{\ge 0}\times \R_+} \cvgd \{f_i\diamond \cL^\rho(\cdot, 0; x, t)\}_{(i,x,t)\in\llbracket 1, k\rrbracket\times \R_{\ge 0}\times \R_+},
                \]
            in the following senses:
            \begin{itemize}
                \item[(i)] if there is a constant $C>0$, such that $\mathbf{A}_\ep h^{\eps,i}_0(x)+2\eps^{-1/2}x$ is constant for $x\ge C, \ep > 0, i \in \II{1, k}$, then the convergence is in the uniform-on-compact topology;
                \item[(ii)] if there is a constant $C>0$, such that $\mathbf{A}_\ep h^{\eps,i}_0(x) \le C(1+x)$ for each $i\in\llbracket 1,k\rrbracket$ and $x\ge 0$, then the convergence is in the sense of finite dimensional distributions.
            \end{itemize}
\end{theorem}
We remark that the above convergence can be upgraded, to e.g., uniform-on-compact convergence assuming $\limsup_{x\to \infty}\sup_{\eps>0, i\in \llbracket 1,k\rrbracket} x^{-2}\mathbf{A}_\ep h^{\eps,i}_0(x)=0$, using the TASEP moderate deviation estimates (Lemmas \ref{L:one-point-TASEP} and \ref{L:two-point-TASEP}). We choose not to expand on this technical upgrade here.

\begin{proof}[Proof of Theorem \ref{thm:TASEPexo}]
Part (i) follows immediately from \eqref{eq:TASEPdcoupling} and \eqref{eq:Hddiff} and the uniform-on-compact convergence in Theorem \ref{thm:mPAM}.

We then prove (ii). Take any finite $P\subset \llbracket 1, k\rrbracket\times \R_{\ge 0} \times \R_+$, and consider $\big\{ \mathbf{A}_\ep h^{\eps,i}_{2\eps^{-3/2}t} (x) + \eps^{-1}t \big\}_{(i,x,t) \in P}$.
Tightness under the product topology follows from Theorem \ref{T:half-space-fixed-point}. Let $\{\cP(i,x,t)\}_{(i,x,t) \in P}$ be any subsequential limit in law.
We may also take this limit jointly with \eqref{thm:mPAM}, so that the subsequential limit is in the same probability space as $\cL^\rho$.

By Theorem \ref{T:half-space-fixed-point}, for each $(i,x,t) \in P$, we have that $\cP(i,x,t)\eqd f_i\diamond \cL^\rho(\cdot, 0; x, t)$.
We claim that  $\cP(i,x,t) \ge  f_i\diamond \cL^\rho(\cdot, 0; x, t)$ almost surely.
Assuming this claim, we must have  $\cP(i,x,t)=f_i\diamond \cL^\rho(\cdot, 0; x, t)$ almost surely, and the conclusion follows.

It remains to prove the claim. Take any $y\in\R_{\ge 0}$, we can find $y_\eps\in \Z_{\ge 0}$, such that $(\eps/2)y_{\eps}\to y$ and $-\eps^{1/2}h^{\eps,i}_0(y_{\eps})\to f(y)$ as $\eps\to 0$.
By \eqref{eq:TASEPdcoupling} and \eqref{eq:Hddiff}, we would have
\[
h^{\eps,i}_{2\eps^{-3/2}t} (\lfloor 2\eps^{-1}x\rfloor) \le  h^{\eps,i}_0 (y_\eps) + H_{d_\eps}^-(y_\eps, 0; \lfloor 2\eps^{-1}x\rfloor, 2\eps^{-3/2}t) + 2d_\eps.
\]
Then after sending $\eps\to 0$, using that $d_\eps \eps^{1/2}\to 0$, and applying the uniform-on-compact convergence in Theorem \ref{thm:mPAM} we have that $\cP(i,x,t) \ge f_i\diamond \cL^\rho(\cdot, 0; x, t)$ almost surely. 
\end{proof}

\section{Convergence of geodesics}   \label{sec:convgeo}

For the half-space directed landscape $\cL^\rho$ with $\rho\in \R\cup\{-\infty\}$, and a continuous function $\pi:[s,t]\to \R$, henceforth a \textbf{path}, define its length
\begin{align*}
	\|\pi\|_{\mathcal{L}^\rho} := \inf_{k\in \N} \inf_{s=t_0 < t_1 <... < t_k = t}\sum_{i=1}^k \mathcal{L}^\rho(\pi(t_{i-1}), t_{i-1}; \pi(t_i),t_i).
\end{align*}
We say $\pi$ is a \textbf{geodesic} from $p = (\pi(s),s)$ to $q = (\pi(t), t)$, or more succinctly a $(p; q)$-geodesic, if $\|\pi\|_{\mathcal{L}^\rho} = \mathcal{L}^\rho(\pi(s),s; \pi(t),t)$. 

In this section, we prove convergence of geodesics in half-space exponential LPP and $d$-level Poisson avoiding metric to geodesics in $\cL^\rho$. Many of the arguments in the present section are deterministic, following ideas laid out in \cite{DOV} and \cite{dauvergne2021scaling}, which imply that subsequential limits of geodesics are geodesics. We start by showing that geodesics in $\cL^\rho$ exist.

\begin{prop}
	\label{P:geos-exist}
	For all $(p; q) = (x, s; y, t) \in \mathbb H^2_\uparrow$, there exists at least one $(p;q)$-geodesic $\pi$ in $\cL^\rho$. Moreover, all geodesics in $\cL^\rho$ are H\"older-$(2/3- \eps)$ continuous for all $\eps > 0$.

\end{prop}

\begin{proof}
	Let $(p; q) = (x, s; y, t) \in \mathbb H^2_\uparrow$, and for $k \in \N$, let $D_k = \{s + (t-s)i/2^k : i = 0, \dots, 2^k\}$. We first recursively define $\pi$ on $D:=\bigcup_{k \in \N} D_k$. For this, we let $\ep_k = (t-s)/2^k$ and write $\bar \pi(r) = (\pi(r), r)$.
	
	Set $\pi(s) = x, \pi(t) = y$, which defines $\pi$ on $D_0$. Then given $\pi$ on $D_{k-1}$, for $r \in D_k \setminus D_{k-1}$, define
	$$
	\pi(r) = \argmax \{z: \cL^\rho(\bar \pi(r - \ep_k); \bar \pi(r + \ep_k)) = \cL^\rho(\bar \pi(r - \ep_k); z, r) +\cL^\rho(z, r; \bar \pi(r + \ep_k))\}.
	$$
	If there are multiple points in the above argmax, we may choose $\pi(r)$ arbitrarily. Note that the argmax above is always non-empty by Proposition \ref{P:only-easy-consequences}.3. With this construction, by the shape bound in Proposition \ref{P:only-easy-consequences}.2, for any $\delta > 0$ there exists a random constant $C_\delta > 0$ such that for all $k \in \N, r \in D_k$,
	$$
	|\pi(r) - \pi(r \pm \ep_k)| \le C_\delta \ep_k^{2/3 - \delta}.
	$$
	Hence $\pi$ has a continuous extension from $D$ to all of $[s, t]$ which is H\"older-$(2/3- \eps)$ continuous for all $\eps > 0$. This extension is a geodesic by continuity of $\cL^\rho$. Finally, any geodesic can be constructed this way, which gives the first part of the theorem.
\end{proof}

Given Proposition \ref{P:geos-exist}, we can show that geodesics in prelimiting models converge to geodesics in the half-space directed landscape. The arguments here are essentially contained in either \cite[Section 13]{DOV} or \cite[Section 9]{dauvergne2021scaling} in the setting of full-space convergence to the directed landscape. We have included a quick proof here to accommodate the (slight) extra generality needed for $d$-level Poisson-avoiding metrics. 

For this next lemma, let $\Lambda = \R_{\ge 0} \times [0, 1]$ and consider the domain
$$
\mathbb{H}^{2+}_\uparrow = \{(u, s; v, t) : u, v \in \Lambda, s < t\in \R\}.
$$
We consider functions $f:\mathbb{H}^{2+}_\uparrow \to \R \cup \{-\infty\}$ satisfying the following metric composition law: for every $(u, s; v, t) \in \mathbb{H}^{2+}_\uparrow$ and every $r \in (s, t)$, let
\begin{equation}
	\label{E:MC-lambda}
	f(u, s; v, t) = \max_{w \in \Lambda} f(u, s; w, r) + f(w, r; v, t).
\end{equation}
We say that a path $\pi:[s, t] \to \Lambda$ is a \textit{geodesic} for $f$ if for any times $s=t_0 < t_1 <... < t_k = t$, we have
\begin{equation}
	\label{E:f-equality}
	f(\pi(s), s; \pi(t), t) = \sum_{i=1}^k f(\pi(t_{i-1}), t_{i-1}; \pi(t_i),t_i).
\end{equation}
We also let $\iota(x, a) = x, \iota:\Lambda \to \R_{\ge 0}$ be the projection onto the first coordinate. 

Any directed metric on $\mathbb{H}^2_\uparrow$ can be extended to a function satisfying \eqref{E:MC-lambda} by making the dependence on the $[0, 1]$-coordinates trivial. In particular, we can apply this to $\cL^\rho$ or rescalings of exponential LPP. On the other hand, adding in the extra coordinate allows us to easily embed rescalings of $d$-level Poisson-avoiding metrics.
\begin{lemma}
	\label{L:abstract-convergence}
	Let $f_n:\mathbb{H}^{2+}_\uparrow \to \R \cup \{-\infty\}$ be a sequence of functions satisfying \eqref{E:MC-lambda} which converge uniformly on compact sets to continuous limit $f:\mathbb{H}^{2+}_\uparrow \to \R$ which also satisfies \eqref{E:MC-lambda}. Suppose that additionally $f$ satisfies the following shape bound: 
	\begin{itemize}
		\item For all $c_0 > 0$, there exists a constant $C > 0$ such that for all $u = (x, a, t; y, b; t + s) \in \mathbb{H}^{2+}_\uparrow$ with $s < c_0$ we have that
		$$
		\lf|f(u) + \frac{(x-y)^2}{s} \rg|\le C s^{1/3} \log^2\lf(\frac{{2\|u\|_2}}{s}\rg).
		$$
	\end{itemize}
	Suppose $\pi_n$ is a sequence of continuous geodesics for $f_n$ between points $(u_n, s)$ to $(v_n, t)$, where the sequence $(u_n, s; v_n, t)$ is precompact in $\mathbb{H}^{2+}_\uparrow$.
	Then:
	\begin{enumerate}
		\item The sequence $\pi_n$ is precompact in the uniform topology on functions from $[s, t] \to \Lambda$.
		\item If $\pi:[s, t] \to \Lambda$ is a subsequential limit of $\pi_n$, then $\pi$ satisfies \eqref{E:f-equality} for any finite set $\{s = t_0 < t_1 < \cdots < t_k = t\} \subset [s, t]$.
	\end{enumerate}
\end{lemma}

Our restriction that the endpoints $s$ and $t$ are the same for all $n$ is not necessary, and is only done so we can work with the topology of uniform convergence of functions (rather than, say, Hausdorff convergence of graphs).  When we apply Lemma \ref{L:abstract-convergence} with $f$ being the half-space directed landscape, with the $[0, 1]$-coordinates being trivial, the points $1$ and $2$ together will say that $\iota \pi_n$ will converge to a geodesic in the directed landscape. Note that we could also weaken the shape condition; we have chosen the above form so as to exactly match Proposition \ref{P:only-easy-consequences}.2.

\begin{lemma}
	\label{L:limsup-princ}
	Let $f_n, f$ be as in Lemma \ref{L:abstract-convergence}, and fix a bounded set $B = [-b, b]^6 \cap \mathbb{H}^{2+}_\uparrow$. Then there exists a positive constant $C_b$, such that for all $\ep \in (0, 1)$ there exists $n_\ep \in \N$ such that for all $n \ge n_\ep$ and $u = (x, a, t; y, a'; t + s) \in B$ we have
	$$
	f_n(u) \le C_b -\frac{(x -y)^2}{s + \ep}.
	$$
\end{lemma}

\begin{proof}
	This lemma is essentially \cite[Lemma 13.3]{DOV}, with $\Rd$ replaced by $\mathbb{H}^{2+}_\uparrow$. The proof goes through verbatim.
\end{proof}

\begin{proof}[Proof of Lemma \ref{L:abstract-convergence}]
	By possibly taking a subsequence, we may assume $u_n \to u$ and $v_n \to v$ in $\Lambda$.
	We first check that $\limsup_{n\to\infty} (\sup \iota\pi_n) < \infty$. For $m > 0$, define
	\begin{align*}
		a(m) &= \sup_{a \in [0, 1], r \in (s, t)} f(u, s; m, a, r) + f(m, a, r; v, t),
	\end{align*}
	and define $a_n$ similarly but with $f_n$ in place of $f$ and $u_n, v_n$ in place of $u, v$. We have that
	\begin{equation*}
		\sup \iota \pi_n > m \qquad \text{implies} \qquad a_n(m) \ge f_n(u_n, s; v_n, t).
	\end{equation*}
	On the other hand, by the uniform-on-compact convergence of $f_n \to f$, the shape bound on $f$, and Lemma \ref{L:limsup-princ}, we have that $a_n(m) \to a(m)$ for all fixed $m$. Also, $f_n(u_n, s; v_n, t) \to f(u, s; v, t)$. Therefore
	$$
	\limsup_{n\to\infty} (\sup \iota\pi_n) \le \sup \{ m > 0 : a(m) \ge f(u, s; v, t) \}.
	$$
	The right-hand side above is finite by the shape bound on $f$.
	
	Now, since $\limsup_{n\to\infty} (\sup \iota\pi_n) < \infty$, the sequence of graphs $\{( \pi_n(r), r) : r\in [s, t]\}$ is tight in the Hausdorff topology on closed subsets of $\R_{\ge 0} \times [s, t]$. Let $A$ be any subsequential limit. To complete the proof of part $1$ it suffices to show that $A = \{(\pi(r), t) : r \in [s, t]\}$ for some continuous function $\pi:[s, t]\to \Lambda$. The only way for this to fail is if there exists a time $r \in [s, t]$ and $x_1 < x_2$ and $a_1, a_2\in [0,1]$, such that $(x_1, a_1, r), (x_2, a_2, r) \in A$, which implies that along a subsequence of $n\in\N$, there exists $r_n, r_n' \to r$ such that $\pi_n(r_n) \to (x_1, a_1), \pi_n(r_n') \to (x_2, a_2)$. By possibly passing to a further subsequence we may assume $r_n < r_n'$ for all $n$. Then writing $\bar \pi_n(\cdot) = (\pi_n(\cdot), \cdot)$ we have
	\begin{align*}
		f_n(u_n, s; v_n, t) =f_n(u_n, s; \bar \pi_n(r_n)) + f_n(\bar \pi_n(r_n); \bar \pi_n(r_n')) + f_n(\bar \pi_n(r_n'); v_n, t).
	\end{align*}
	Since $x_1 \ne x_2$, Lemma \ref{L:limsup-princ} implies that the right-hand side converges to $-\infty$, whereas the left-hand side is finite. This is a contradiction. This completes the proof of part $1$.
	
	Part $2$ follows from part $1$, and the uniform-on-compact convergence of $f_n \to f$.
\end{proof}

We next apply Lemma \ref{L:abstract-convergence} to show that limit points of geodesics in $d$-level Poisson-avoiding metrics and in exponential LPP are geodesics in $\cL^\rho$, in particular proving Theorem \ref{thm:expLPPgeoconv}. We start by stating geodesic convergence for $d$-level Poisson-avoiding metrics with $d>1$. 

To set things up, observe that we can alternately define the $d$-level Poisson-avoiding metric $H_d$ as follows. First, for a path $f:[s, t] \to \Lambda_d$, define 
$$
\|f\|_{H_d} = \sup_{s=t_0 < t_1 < \cdots < t_k = t} \sum_{i=0}^{k-1} D_d(f(t_i), f(t_{i+1})),
$$
where $D_d$ is as in Definition \ref{defn:Dd}. Then
$$
H_d(u, s; v, t) = \min \|f\|_{H_d},
$$ 
where the minimum is over all paths $f:[s,t] \to \Lambda_d$ with $f(s) = u, f(t) = v$ and such that $(f(r), r) \notin \pi$ for all $r \in (s, t)$. If $(v, t) \in \Pi$ we impose the additional constraint that $f(t^-) \ne v$. We call $f$ a geodesic if it achieves the minimum above.

Next, for $(x, a, s) \in \Lambda \times \R$ define $(x, a, s)_{d,\eps} = (\lfloor 2\eps^{-1}x \rfloor, \lfloor a/(2d-1) \rfloor, 2\eps^{-3/2}s)$, and for $(u, s; v,t) \in \mathbb{H}^{2+}_\uparrow$ let
$$
H_{d, \eps}(u, s; v,t) = -\eps^{1/2}H_d((u, s)_{d, \eps}; (v,t)_{d, \eps}) + \eps^{-1}(t - s).
$$
We say that $f$ is a geodesic in $H_{d, \eps}$ if $(\mathfrak{g} f)_{d, \eps} = \mathfrak{g} \tilde f$ for some geodesic $\tilde f$ in $H_d$. This is similar to the definition introduced before Theorem \ref{thm:expLPPgeoconv} (and recall the notation $\mathfrak{g}f$ from there).

Next, for a half-space directed landscape $\cL^\rho$, define $\bar \cL^\rho: \mathbb{H}^{2+}_\uparrow \to \R$ by
$$
\bar \cL^\rho(x, a, s; y, b, t) = \cL^\rho(x, s; y, t).
$$
Theorem \ref{thm:mPAM} and the bound \eqref{eq:Hddiff} guarantees that with $d_\eps$ and boundary parameters as in that theorem, we have that $H_{d_\eps, \eps} \cvgd \bar \cL^\rho$ in the uniform-on-compact topology. We use this as an input to state and prove geodesic convergence.
\begin{theorem}  \label{thm:multigeoconv}
	Take $\rho \in \R\cup \{-\infty\}$, and a sequence $\eps_n \to 0$. With $d_\eps$ and boundary parameters as in Theorem \ref{thm:mPAM}, consider a coupling where $H_{d_{\eps_n}, \eps_n} \to \bar \cL^\rho$ almost surely in the uniform-on-compact topology. Then there exists a set $\Omega$ of probability $1$ such that on $\Omega$, the following assertions hold.
	\begin{enumerate}[nosep]
		\item Consider a sequence $(u_n, s; v_n, t)$ such that $(\iota u_n, s; \iota v_n, t) \to (x, s; y, t) \in \mathbb H^2_\uparrow$ and a sequence of geodesics $\pi_n$ in $H_{d_{\eps_n}, \eps_n}$ from $(u_n; s)$ to $(v_n; t)$. Then $\iota \pi_n$ is precompact in the uniform topology and any subsequential limit of $\iota \pi_n$ is a geodesic in $\cL^\rho$ from $(x, s)$ to $(y, t)$.
		\item If $\cL^\rho$ contains a unique geodesic $\pi$ from $(x, s)$ to $(y, t)$, then $\iota \pi_n \to\pi$ almost surely.
	\end{enumerate}
	\end{theorem}
	
	\begin{proof}
	First, our definitions of $H_{d_\eps, \eps}, \bar \cL^\rho$ satisfy \eqref{E:MC-lambda}, and our notion of geodesic in $H_{d_\eps, \eps}$ matches the definition via \eqref{E:f-equality}. Moreover, $\bar \cL^\rho$ satisfies the bullet point in Lemma \ref{L:abstract-convergence} almost surely by Proposition \ref{P:only-easy-consequences}.2. Finally, observe that $\pi$ is a geodesic in $\bar \cL^\rho$ if and only if $\iota \pi$ is a geodesic in $\cL^\rho$. Therefore we may apply Lemma \ref{L:abstract-convergence} to obtain part $1$ above, as long as our initial sequence $\pi_n$ is continuous. To remove the continuity constraint, simply observe that in $H_{d_{\eps_n}, \eps_n}$, for any geodesic $\tau$, there is a sequence of continuous geodesics $\tau_k$ such that the graphs $\mathfrak{g} \tau_k$ converge in the Hausdorff topology to a limiting set $S$ containing $\mathfrak{g} \tau$. This allows us to conclude the discontinuous case of the part $1$ from the continuous case by a diagonalization argument. 
    
    Part $2$ is immediate from part $1$.
	\end{proof}
	
	\begin{proof}[Proof of Theorem \ref{thm:expLPPgeoconv}]
	The case of Poisson-avoiding metrics is immediate from the previous theorem, since in the single-level setting our metric $H_{1, \eps}$ ignores the coordinates in $[0, 1]$, and so reduces to a metric on $\mathbb H^2_\uparrow$.
	
	The exponential LPP case requires a slight modification of the definition of $\cL_n$. Indeed, in the original last passage definition, let
	$
	X^-(u; v) = X(u, v) - X(u)
	$
	denote LPP where we do not include the initial vertex, and let $\cL_n^-$ equal $\cL_n$ with $X^-$ in place of $X$. By an exponential tail bound on the maximum size of single entries and the Borel-Cantelli lemma, in the setup of Theorem \ref{thm:expLPPgeoconv} we still have that $\cL_n^- \to \cL^\rho$ almost surely. Then $\cL_n^-$ satisfies \eqref{E:MC-lambda}, and geodesics in $\cL_n$ with the definition introduced prior to Theorem \ref{thm:expLPPgeoconv} satisfy \eqref{E:f-equality}. Next, by trivially extending $\cL_n^-$ to $\bar \cL_n^-$ on all of $\mathbb H^{2+}_\uparrow$ as we extended $\cL^\rho$ to $\bar \cL^\rho$, and applying Lemma \ref{L:abstract-convergence}, the conclusion follows.
	\end{proof}

It remains to prove almost sure uniqueness of geodesics (for fixed endpoints) in $\cL^\rho$, which guarantees that Theorem \ref{thm:multigeoconv}.2 is not vacuous.  In the full-space setting, Brownian absolute continuity of the profiles $\cL_{s, t}(x, \cdot)$ implies uniqueness immediately. Here we cannot appeal to this property a priori, so we take a different approach based around resampling and a full-space comparison.

\begin{prop}
	\label{P:geo-unique}
	Let $\rho \in \R \cup \{-\infty\}$. For any fixed $(p; q) = (x, s; y, t) \in \mathbb{H}^2_\uparrow$, there is almost surely a unique $(p;q)$-geodesic $\pi$ in $\cL^\rho$.
\end{prop}

\begin{proof}
	Fix $r \in (s, t)$. It is enough to show that if $\pi, \tau$ are two $(p; q)$-geodesics, then almost surely $\pi(r) = \tau(r)$. For this, for an interval $I\subset \R_{\ge 0}$, define the length
	$$
	\cL^{\rho}(p; q \mid I) = \max_{z \in I} \cL^\rho(p; z, r) + \cL^{\rho}(z, r; q).
	$$
	We will show that for any disjoint intervals $I, J\subset \R_{\ge 0}$, almost surely 
	\begin{equation}
		\label{E:Lalpha-interval-noneq}
		\cL^{\rho}(p; q \mid I) \ne \cL^{\rho}(p; q \mid J).
	\end{equation}
	If we can prove this, then \eqref{E:Lalpha-interval-noneq} holds almost surely, simultaneously for all rational disjoint intervals $I, J$. On the other hand, if $\pi(r) \ne \tau(r)$, then letting $I, J$ be disjoint intervals containing $\pi(r), \tau(r)$ respectively we would have
	$$
	\cL^{\rho}(p; q \mid I) = \cL^{\rho}(p; q \mid J) = \cL^{\rho}(p; q),
	$$
	which is a contradiction.
	
	We now prove \eqref{E:Lalpha-interval-noneq} by coupling with the full-space directed landscape. Let $I = [a, b]$ and $J = [c, d]$ for $0\le a < b < c < d$.
	
	We work in the setting of Section \ref{SS:subsequential}. Fix $n$, and an environment $X^{(n)} := X_{1/2 - 2^{-4/3} \rho n^{-1/3}}$ of exponential random variables as in that section. Let $Y$ be a field of i.i.d.\ $\Exp(1)$ random variables on $\Z^2$, independent of $X^{(n)}$. Fix $\delta > 0$ such that $c - \delta > b$, and define a third field $Z^{(n)}$ on $\Z^2_\ge$ by letting
	$$
	Z^{(n)}((x, r')_n) = \begin{cases}
		Y((x, r')_n), \qquad (x, r') \in [c-\delta, d+\delta] \times [r-\delta, r + \delta],\\
		X^{(n)}((x, r')_n), \qquad \text{else}.
	\end{cases}
	$$
	Now, we take rescaling as \eqref{E:rescaled-elpp}.
    Namely, we let $\cK_n$ (resp., $\cM_n$, $\cL_n$) be defined on $\mathbb H^2_\uparrow$ as $\cQ^\rho_n$ in \eqref{E:rescaled-elpp}, with $X_{1/2-2^{-4/3}\rho n^{-1/3}}$ replaced by $X^{(n)}$ (resp., $Z^{(n)}$, $Y$).
    Take a joint subsequential limit $(\cK, \cM, \cL)$ of $(\cK_n, \cM_n, \cL_n)$. Here $\cL_n$ converges in the uniform-on-compact topology to a full-space directed landscape (restricted to $\mathbb H^2_\uparrow$), see \cite[Theorem 1.7]{dauvergne2021scaling} for details. The fields $\cK, \cM$ are half-space directed landscapes of parameter $\rho$. We will check \eqref{E:Lalpha-interval-noneq} for $\cM$. 
	First, with probability tending to $1$ as $\eps \to 0$ we have that
	\begin{align*}
		&\cM(p; q \mid I) \\
		&= \max_{z_1, z_3 \in [(a - \ep^{1/2}) \vee 0, b + \ep^{1/2}], z_2 \in I} \cM(p; z_1, r-\eps) + \cM(z_1, r-\eps; z_2, r) + \cM(z_2, r; z_3, r + \eps) + \cM(z_3, r + \eps; q) \\
		&= \max_{z_1, z_3 \in [(a - \ep^{1/2}) \vee 0, b + \ep^{1/2}], z_2 \in I} \cM(p; z_1, r-\eps) + \cK(z_1, r-\eps; z_2, r) + \cK(z_2, r; z_3, r + \eps) + \cM(z_3, r + \eps; q).
	\end{align*} 
	Here the first equality uses the shape bound and the metric composition law in Proposition \ref{P:only-easy-consequences}. The second equality uses that the $Z^{(n)}$-weights in an open set $U$ around $[a, b] \times \{r\}$ match those of $X^{(n)}$, and with probability $\to 1$ as $\eps\to 0$, geodesics in the prelimit from $[(a - \ep^{1/2}) \vee 0, b + \ep^{1/2}] \times \{r - \eps\}$ to $[(a - \ep^{1/2}) \vee 0, b + \ep^{1/2}] \times \{r +\eps\}$ will not leave $U$. This can be derived from the one-point tails bounds in Theorem \ref{T:half-space-full-tails} (see e.g., \cite[Proposition C.9]{BGZ}).    
    See also the proof of \cite[Proposition 2.6]{dauvergne2023non} for more details. Similarly, with probability $\to 1$ as $\ep \to 0$ we have that
	\begin{align*}
		&\cM(p; q \mid J) \\
		&= \max_{z_1, z_3 \in [c - \ep^{1/2}, d + \ep^{1/2}], z_2 \in J} \cM(p; z_1, r-\eps) + \cL(z_1, r-\eps; z_2, r) + \cL(z_2, r; z_3, r + \eps) + \cM(z_3, r + \eps; q).
	\end{align*} 
    Denote the last lines of the above two displays by $\cM_{\cK}$ and $\cM_{\cL}$, respectively. 
	Therefore to complete the proof, it suffices to show that almost surely $\cM_\cK\neq \cM_\cL$. Observe that $\cM_\cK$ is $\cF$-measurable, where $\cF$ is the $\sigma$-algebra generated by $\cK$, and all distances in $\cM$ off of the strip $[r-\eps, r + \eps]$. On the other hand, we can write
	$$
	\cM_\cL = \max_{z \in J} f \diamond \cL(\cdot, r-\ep; z, r) + \cL(z, r; \cdot, r + \ep) \diamond g,
	$$
	where $f, g$ are $\cF$-measurable functions. By \cite[Theorem 1.2]{sarkar2020brownian}, for fixed $f, g$, both $f \diamond \cL(\cdot, r-\ep; \cdot, r) - f \diamond \cL(\cdot, r-\ep; c, r)$ and $\cL(\cdot, r; \cdot, r + \ep) \diamond g - \cL(c, r; \cdot, r + \ep) \diamond g$ in $J$ are absolutely continuous with respect to a Brownian motion started at $0$. Hence conditional on $\cF$, $\cM_\cL$ has a Lebesgue density, and is therefore not equal to $\cM_\cK$ almost surely.
\end{proof}

\section{Open Problems}   \label{sec:openproblems}

We end with several open problems. First, we believe that there should be a unique natural coupling of half-space directed landscapes $\cL^\rho$ for all boundary parameters $\rho$. This is closely related to the following convergence problem.

\begin{prob}
	\label{P:five-parameter}
Let $X_1:\Z^2_\ge \to \R_{\ge 0}$ be a family of independent $\Exp(1)$ random variables, and for every $\alpha > 0$ define $X_\alpha(i, j) = X_1(i, j)$ for $i>j$ and $X_\alpha(i, i) = \alpha^{-1} X_1(i, i)$ for all $i$. For $\rho \in \R \cup \{-\infty\}$ and $n\in\N$, define $\cL_n^\rho$ as in \eqref{eq:cLn}, but using the field
$X_{1/2 - 2^{-4/3} \rho n^{-1/3}}$ when $\rho \ne -\infty$ and the field $X_1$ when $\rho = - \infty$. Prove that
$$
(\cL_n^\rho : \rho \in \R \cup \{-\infty\}),
$$
has a joint scaling limit, as functions on $\mathbb H^2_\uparrow \times (\R \cup \{-\infty\})$.
\end{prob}

Closely related to Proposition \ref{P:five-parameter} is the question of uniqueness of couplings of $\cL^\rho$. We propose one suggestion that seems promising, and natural given Corollary \ref{C:exp-quadrangle}.
\begin{prob}
Is there a unique coupling of $\cL^\rho$ for all $\rho \in \R \cup \{-\infty\}$ such that for any $\rho \le \rho'$, $0\le x_1 \le x_2$, $0\le y_1 \le y_2$, and $s < t$, we have the quadrangle inequality
 $$
 \cL^\rho(x_1, s; y_1, t) + \cL^{\rho'}(x_2, s; y_2, t) \le \cL^{\rho'}(x_1, s; y_1, t) + \cL^{\rho}(x_2, s; y_2, t).
 $$
\end{prob}

We believe that the boundary should not contribute any randomness in the limit, and that the half-space directed landscapes should all be measurable functions of the full-space directed landscape. 

\begin{prob}
	\label{P:boundary-no-impact}
Under the setup of Proposition \ref{P:five-parameter}, we take $\tilde X_1:\Z^2_\ge\to\R_{\ge 0}$ such that $\tilde X_1(i,j)=X_1(i,j)$ for $i>j$, and all $\tilde X_1(i,i)\sim\Exp(1)$ are i.i.d., and independent from $X_1$. For $\rho\in\R$ and $n\in\N$, define $\tilde \cL^\rho_n$ from $\tilde X_1$ the same way as $\cL^\rho_n$ from $X_1$.
Show that as $n \to \infty$,
$$
d(\cL^\rho_n, \tilde \cL^\rho_n) \cvgp 0,
$$
where $d$ is (any) metric metrizing the uniform-on-compact topology on functions from $\mathbb H^2_\uparrow \to \R$.
\end{prob}

\begin{prob}
\label{prob:measurability}
Under the setup of Proposition \ref{P:five-parameter}, we extend $X_1$ to be i.i.d.~$\Exp(1)$ from $\Z^2\to\R_{\ge 0}$.
Define $\cL_n:\Rd\to \R$ (for $\Rd=\{(x,s;y,t): x,s,y,t\in\R, s<t\}$) by expression \eqref{eq:cLn}, with the field $X_1$.
Let $(\cL, \cL^{\rho_1}, \cL^{\rho_2})$ be a subsequential limit (in law) of $(\cL_n, \cL^{\rho_1}_n, \cL^{\rho_2}_n)$. 
Is there an explicit measurable function $f$ such that $\cL^{\rho_1} = f(\cL)$? Is there an explicit measurable function $f$ such that $\cL^{\rho_1} = f(\cL^{\rho_2})$?
\end{prob}

The measurable function $f$ for $\cL^{-\infty} = f(\cL)$ should simply be given by finding optimal length paths in $\cL$, restricted to stay non-negative. It is less clear how exactly to define $f$ when $\rho_1 > -\infty$. We expect that lengths in $\cL^{\rho_1}$ can be constructed from $\cL^{-\infty}$ by allowing paths to add in a local time at the boundary. However, the exact mechanism is unclear.

We end with a suggestion for an improvement to our characterization theorem which is easier to verify. We keep this problem intentionally vague.

\begin{prob}
\label{P:strong-chara}
Show that the half-space directed landscape $\cL^\rho$ can be characterized by:
\begin{itemize}[nosep]
	\item One-point Baik-Rains laws between points on the boundary, and
	\item Full-space behavior in the bulk. More precisely, there exists a coupling with the full-space directed landscape $\cL$ such that for any compact set $K \subset (0, \infty) \times \R$, there exists a (random) $\eps > 0$ such that $\cL^\rho(u; v) = \cL(u; v)$ whenever $u, v \in K$ and $\|u - v\|_2 < \eps$.	
\end{itemize}
\end{prob}

\let\oldbibliography\thebibliography
\renewcommand{\thebibliography}[1]{%
  \oldbibliography{#1}%
  \setlength{\itemsep}{0.32em}%
}

        \bibliographystyle{alpha}
		\bibliography{bibliography}
	\end{document}